\documentclass[12pt,a4paper]{amsart}
\makeatletter
\renewcommand\normalsize{%
    \@setfontsize\normalsize{11.7}{14pt plus .3pt minus .3pt}%
    \abovedisplayskip 10\p@ \@plus4\p@ \@minus4\p@
    \abovedisplayshortskip 6\p@ \@plus2\p@
    \belowdisplayshortskip 6\p@ \@plus2\p@
    \belowdisplayskip \abovedisplayskip}
\renewcommand\small{%
    \@setfontsize\small{9.5}{12\p@ plus .2\p@ minus .2\p@}%
    \abovedisplayskip 8.5\p@ \@plus4\p@ \@minus1\p@
    \belowdisplayskip \abovedisplayskip
    \abovedisplayshortskip \abovedisplayskip
    \belowdisplayshortskip \abovedisplayskip}
\renewcommand\footnotesize{%
    \@setfontsize\footnotesize{8.5}{9.25\p@ plus .1pt minus .1pt}
    \abovedisplayskip 6\p@ \@plus4\p@ \@minus1\p@
    \belowdisplayskip \abovedisplayskip
    \abovedisplayshortskip \abovedisplayskip
    \belowdisplayshortskip \abovedisplayskip}
\ifdefined\pdfpagewidth
\else
\fi
\calclayout
\makeatother

\usepackage{tikz-cd}

\usepackage{amsmath,leftidx,amsthm}
\usepackage[all]{xy}
\usepackage{xspace,amsthm}
\usepackage[psamsfonts]{amssymb}
\usepackage[latin1]{inputenc}
\usepackage{graphicx,color,mathrsfs}
\usepackage{hyperref,fancyhdr,appendix}
\usepackage{xcolor}

\newtheorem*{remark*}{\bf Remark}

\newtheorem{theorem}{\bf Theorem}[section]
\newtheorem{proposition}[theorem]{\bf Proposition}
\newtheorem{definition}[theorem]{\bf Definition}
\newtheorem{Theorem}{\bf Theorem}

\newtheorem*{claim}{\bf Claim}
\newtheorem{lemma}[theorem]{\bf Lemma}
\newtheorem{corollary}[theorem]{\bf Corollary}

\newtheorem{remark}[theorem]{\bf Remark}

\def\C{{\mathbb C}}

\def\Q{{\mathbb Q}}

\def\B{\mathbb{B}}

\def\p{\mathbb{P}}
\def\Aut{\mathrm{Aut}}
\def\supp{\textup{supp}}
\def\id{\mathrm{Id}}

\def\P{{\mathbb P}}

\usepackage{enumitem}

\newcommand{\Db}{\mathbb{D}}
\newcommand{\Pb}{\mathbb{P}}
\newcommand{\Cb}{\mathbb{C}}
\newcommand{\Nb}{\mathbb{N}}
\newcommand{\Zb}{\mathbb{Z}}
\newcommand{\Qb}{\mathbb{Q}}
\newcommand{\Rb}{\mathbb{R}}
\newcommand{\Sb}{\mathbb{S}}

\newcommand{\codim}{\text{codim}}
\newcommand{\crit}{{\rm Crit}}
\newcommand{\dist}{{\rm dist}}
\newcommand{\ddc}{{dd^c}}

\title[Sparsity of postcritically finite maps of $\mathbb{P}^k$ and beyond]{Sparsity of postcritically finite maps of $\mathbb{P}^k$ and beyond: A complex analytic approach }
\author{Thomas Gauthier}
\address{Laboratoire de Math\'ematiques d'Orsay, B\^atiment 307, Universit\'e Paris-Saclay, 91405 Orsay Cedex, France}
\email{thomas.gauthier@universite-paris-saclay.fr}
\author{Johan Taflin}
\address{Institut de Math\' ematiques de Bourgogne, UMR 5584 CNRS, Universit\' e de Bourgogne, F-21000 Dijon, France}
\email{johan.taflin@u-bourgogne.fr}
\author{Gabriel Vigny}
\address{LAMFA, Universit\'e de Picardie Jules Verne, 33 rue Saint-Leu, 80039 AMIENS Cedex 1, FRANCE}
\email{gabriel.vigny@u-picardie.fr}
\thanks{The first author is partially supported by the Institut Universitaire de France.}
\thanks{The second author is partially supported by the EIPHI Graduate School, ANR-17-EURE-0002.}
\thanks{The third author is partially supported by the ANR QuaSiDy, grant no ANR-21-CE40-0016.}

\begin{document}
\begin{abstract}
An endomorphism $f:\mathbb{P}^k\to\mathbb{P}^k$ of degree $d\geq2$ is said to be postcritically finite (or PCF) if its critical set $\mathrm{Crit}(f)$ is preperiodic, i.e. if there are integers $m>n\geq0$ such that $f^m(\mathrm{Crit}(f))\subseteq f^n(\mathrm{Crit}(f))$. When $k\geq2$, it was conjectured in \cite{IRS} that, in the space $\mathrm{End}_d^k$ of all endomorphisms of degree $d$ of $\mathbb{P}^k$, such endomorphisms are not Zariski dense.
We prove this conjecture. Further, in the space $\mathrm{Poly}_d^2$ of all regular polynomial endomorphisms of degree $d\geq2$ of the affine plane $\mathbb{A}^2$, we construct a dense and Zariski open subset where we have a uniform bound on the number of preperiodic points lying in the critical set.

The key object in the article are the complex bifurcation measure and its properties. The proofs are a combination of the theory of heights in arithmetic dynamics and methods from real dynamics to produce open subsets with maximal bifurcation.  
\end{abstract}

%
%

\maketitle

%
\setcounter{tocdepth}{1} 
\tableofcontents

\section*{Introduction}

Let $\pi:\mathcal{X}\to S$ be a family of complex projective varieties, where $S$ is a smooth complex projective variety, and let $\mathcal{L}$ be a nef and relatively ample line bundle on $\mathcal{X}$. We let $f:\mathcal{X}\dashrightarrow\mathcal{X}$ be a rational map such that $(\mathcal{X},f,\mathcal{L})$ is a family of polarized endomorphisms of degree $d\geq2$ over a Zariski open subset $S^0$ of $S$, i.e. for all $t\in S^0(\C)$, $X_t:=\pi^{-1}\{t\}$ is normal, $L_t:=\mathcal{L}|_{X_t}$ is ample and $f_t^*L_t\simeq L_t^{\otimes d}$. We further assume that the generic fiber is smooth. If $\mathcal{X}^0=\pi^{-1}(S^0)$, the family $\mathcal{X}^0\to S^0$ is the \emph{regular part} of $(\mathcal{X},f,\mathcal{L})$. The purpose of the article is to study maximal instability phenomena in both complex and arithmetic dynamics, each viewpoint giving deep insights into the other.   

From the arithmetic viewpoint, we are mainly interested in the notion of canonical height of a subvariety. Such height is a function meant to measure the arithmetic dynamical complexity of the orbit of the subvariety. Studying such objects in family,
we are particularly interested in two cases:
\begin{itemize}
 \item the moduli space $\mathscr{M}_d^k$ of degree $d$ endomorphisms of the projective space $\mathbb{P}^k$,
 \item the moduli space $\mathscr{P}_d^k$ of degree $d$ regular endomorphisms of the affine space $\mathbb{A}^k$.
 \end{itemize}
In both cases, we study a family which is finite to one over a Zariski open subset of the moduli space, and the family of subvarieties we consider is the critical set, see \S~ \ref{sec:moduli-space} for more details. More precisely, we 
\begin{itemize}
	\item show that this height is in fact a moduli height on a Zariski open set $U$ of $\mathscr{M}_d^k$.
	\item use that height to show that  postcritically finite maps --  PCF maps for short -- (see below) are not Zariski dense in $\mathscr{M}_d^k$ nor in $\mathscr{P}_d^2$. 
	\item prove a uniform bound on the number of preperiodic critical points for regular polynomial endomorphisms whose conjugacy lies in a Zariski open set of $\mathscr{P}_d^2$.
\end{itemize}
The complex analytic viewpoint is essential in that process to 
\begin{itemize}
	\item show that the support of the bifurcation measure (see below) has non-empty interior in both $\mathscr{M}_d^k$ and $\mathscr{P}_d^2.$
	\item prove that the correspondence between an endomorphism in $\mathscr{M}_d^k$ (or  $\mathscr{P}_d^2$) and the collection of the multipliers of its periodic points is finite-to-one outside a Zariski closed set. 
\end{itemize} 

We are strongly inspired by the recent results on families of abelian varieties where similar type of results have been established, as well as by the recent uniform bounds on the number of common preperiodic points for rational maps of $\mathbb{P}^1$, initiated by DeMarco, Krieger and Ye \cite{DKY1,DKY2} in the cases of flexible Latt\`es maps and quadratic polynomials, and developed since then by Mavraki and Schmidt~\cite{Mavraki_Schmidt} and DeMarco and Mavraki~\cite{DeMarco-Mavraki}. Concerning families of abelian varieties, they naturally fall in the setting of family of polarized endomorphisms when taking the multiplication by $[n]$ morphism. In particular, we used ideas coming from the work of  
 Gao-Habegger \cite{Gao-Habegger} and Cantat-Gao-Habegger-Xie \cite{Cantat-Gao-Habegger-Xie} where the Geometric Bogomolov conjecture is proved in characteristic $0$ (note that Xie and Yuan recently managed the tour de force of proving it in arbitrary characteristic \cite{XY}). We also rely on the work of Dimitrov-Gao-Habegger \cite{Dimitrov-Gao-Habegger} where a uniform bound on the number of rational points of a curve $C$, defined over a number field, inside its Jacobian is established (Uniform Mordell-Lang) and the works of K\"uhne \cite{Kuhne}, generalized by Yuan in arbitrary characteristic \cite{yuan_chara}, and Gao-Ge-K\"uhne \cite{GGK} where the Uniform Mordell-Lang Conjecture is generalized to arbitrary subvariety of an abelian variety.

~

A crucial point in our work is to link the notion of dynamical stability in complex dynamics, which can be characterized by positive closed currents, with the notion of dynamical height. In \cite{GV_Northcott}, relying on the theory of DSH functions of Dinh and Sibony \cite{dinhsibony2}, the first and third authors established such link for the $(1,1)$ \emph{bifurcation current} of a family of subvarieties, here we need to deal with the \emph{bifurcation measure}, which measures higher bifurcation phenomena. Let us explain those terms.

Let $\omega$ be a smooth positive form representing the first Chern class $c_1(\mathcal{L})$ on $\mathcal{X}$. As $f^*\mathcal{L}\simeq \mathcal{L}^{\otimes d}$ on $\mathcal{X}^0$, there is a smooth function $g:\mathcal{X}^0\to\mathbb{R}$ such that $d^{-1}f^*\omega=\omega+dd^cg$ as forms on $\mathcal{X}^0$.
In particular, the following limit exists as a closed positive $(1,1)$-current with continuous potential on the quasi-projective variety $\mathcal{X}^0(\C)$:
\[ \widehat{T}_f := \lim_{n\to\infty} \frac{1}{d^n} (f^n)^*(\omega),\]
and can be written as $\widehat{T}_f=\omega+dd^cg_f$, where $g_f$ is continuous on $\mathcal{X}^0(\C)$. The 
current $\widehat{T}_f$ is the \emph{fibered Green current} of $f$  (note that for abelian varieties, $\widehat{T}_f$ is the Betti form).
Let $\mathcal{Y}\to S$ be a family of subvarieties of $\mathcal{X}$, i.e. $\mathcal{Y}$ is a subvariety of $\mathcal{X}$ and $\pi|_{\mathcal{Y}}:\mathcal{Y}\to S$ of $\pi$ is flat over $S^0$. If $q$ is the relative dimension of $\mathcal{Y}$, for $1\leq m\leq \dim S$, the $m$-\emph{bifurcation current} of $(\mathcal{X},f,\mathcal{L},\mathcal{Y})$ can be defined on $S^0(\C)$ as
\[T_{f,\mathcal{Y}}^{(m)}:=(\pi_{[m]})_*\left(\widehat{T}_{f^{[m]}}^{m(\dim Y_\eta+1)}\wedge[\mathcal{Y}^{[m]}]\right),\]
where $Y_\eta$ is the generic fiber of $\mathcal{Y}$, $\pi_{[m]}:\mathcal{X}^{[m]}\to S$ is the $m$-fiber product of $\mathcal{X}$, and $f^{[m]}$ is the map induced by the fiberwise diagonal action of $f$. The \emph{bifurcation measure} of $(\mathcal{X},f,\mathcal{L},\mathcal{Y})$ is then \[\mu_{f,\mathcal{Y}}:=T_{f,\mathcal{Y}}^{(\dim S)}.\]

We now focus on the case of a family of rational maps of $\P^k(\C)$, parametrized by a projective variety $S$. In this case, the regular part is $\mathcal{X}^0=\mathbb{P}^k\times S^0$, where $S^0$ is a Zariski open subset of $S$. We then are interested in the bifurcation of the critical set $\mathrm{Crit}(f):=\{(z,t)\in\mathbb{P}^k\times S^0\, : \ \det(D_zf_t)=0\}$. So, the bifurcation measure is
\[\mu_{f,\mathrm{Crit}}:=T_{f,\mathrm{Crit}(f) }^{(\dim S)}=  (\pi_{[\dim S]})_*\left(\widehat{T}_{f^{[\dim S ]}}^{k(\dim S)}\wedge[\mathrm{Crit}(f)^{[\dim S]}]\right),\]
since $\mathrm{Crit}(f)$ is a hypersurface of $\mathbb{P}^k\times S^0$. 

 When $k=1$, the bifurcation current has been introduced by DeMarco \cite{DeMarco1} and the bifurcation measure by Bassanelli-Berteloot \cite{BB1}. For families of endomorphisms of $\P^k$, the bifurcation current has been introduced by Bassanelli-Berteloot~\cite{BB1}. In this higher dimensional setting, Berteloot-Bianchi-Dupont showed it is the appropriate tool for studying bifurcations in the important work \cite{BBD} and the bifurcation measure was first considered by Astorg and Bianchi \cite{astorg-bianchi} in the very particular case of families of polynomial skew-product.

It is an important question in complex dynamics to understand what kind of phenomena these currents (or this measure) actually characterize. One way to explore this question is to prove that the measure $\mu_{f,\mathrm{Crit}}$ equidistributes specific type of dynamical behaviors (\cite{BB2, favredujardin, favregauthier, GOV}). 

~

We now come to stating precise results. Define the critical height of a degree $d$ endomorphism $f:\mathbb{P}^k\to\mathbb{P}^k$ defined over a number field as the canonical height of $f$ evaluated at the critical set of $f$:
\[h_{\mathrm{crit}}(f):=\widehat{h}_f(\mathrm{Crit}(f))\]
and remark that this quantity depends only on the conjugacy class. In particular, this defines a function
\[h_{\mathrm{crit}}:\mathscr{M}_d^k(\bar{\mathbb{Q}})\to\mathbb{R}_+.\]
Our first result here is the following 
\begin{Theorem}[The critical height is a moduli height]\label{tm:critical-ample-height}
	The critical height $h_{\mathrm{crit}}$ of the moduli space $\mathscr{M}_d^k$ of degree $d$ of endomorphisms of $\mathbb{P}^k$ is an ample height on a non-empty Zariski open subset $U$ of $\mathscr{M}_d^k$, i.e. for any ample line bundle $M$ on a projective model of $\mathscr{M}_d^k$, there are constants $C_1,C_2>0$ and $C_3,C_4\in\mathbb{R}$ such that
	\[C_1\cdot h_M([f]) +C_3\leq h_\mathrm{crit}([f]) \leq C_2 h_M([f])+C_4,\]
	for all $[f]\in U(\bar{\mathbb{Q}})$.
	Moreover, a subvariety $Z$ is an irreducible component of $\mathscr{M}_d^k\setminus U$ if and only if the bifurcation measure $\mu_{f,\mathrm{Crit},Z}$ of the family induced by $Z$ vanishes.
\end{Theorem}
In Theorem~\ref{tm:critical-ample-height}, $h_M$ stands for a Weil height on a projective model of $\mathscr{M}_d^k$, associated with the ample line bundle $M$.

We want to stress the fact that Yuan and Zhang already showed Theorem~\ref{tm:critical-ample-height} under the hypothesis that  $\mu_{f,\mathrm{Crit}} \neq 0$ on  $\mathscr{M}_d^k(\C)$ (\cite[Theorem 5.3.5 (2) and Problem 6.3.9]{YZ-adelic}). Their approach is arithmetic in nature and allows an optimal control on the constant $C_3$, our approach has a more complex geometric flavor and permits instead a control of the multiplicative constant $C_1$.  

\medskip

For $k=1$, Theorem~\ref{tm:critical-ample-height} is due to Ingram \cite{Ingram_moduli} (see also \cite{GOV2}). For $k\geq 1$, Ingram also proved explicit versions of the above theorem for specific families using convenient parametrizations (e.g. \cite{Ingram_minimally_critical2, Ingram_minimally_critical}). In dimension 1, McMullen's result \cite{McMullen-families} implies that the algebraic subvariety $\mathscr{M}_d^1\setminus U$ where we do not have the inequality in Theorem~\ref{tm:critical-ample-height} is exactly the flexible Latt\`es family. Characterizing that subvariety in higher dimension is one of the main questions in bifurcation theory in higher dimension. 

In order to prove Theorem~\ref{tm:critical-ample-height}, we follow Gao and Habbeger and Dimitrov in the abelian case \cite{Gao-Habegger, Dimitrov-Gao-Habegger} to prove an estimate in a family with positive (suitable) height which compares the height of a parameter with the heights of generic point in the $m$-fiber product of $m$-fiber product of $\mathcal{Y}$. Our arguments are based on the early work \cite{GV_Northcott} of the first and third authors (see Theorem~\ref{tm:Gao-Habegger}). We then use notably Zhang inequalities \cite{Zhang-positivity} to conclude.

~

Let $k\geq1$. Let $\mathrm{End}_d^k$ denote the set of endomorphisms $f:\mathbb{P}^k(\mathbb{C})\to\mathbb{P}^k(\mathbb{C})$ of degree $d$ (in homogeneous coordinates, $f$ is the data of $k+1$ homogeneous polynomials with no common factor and the same degree $d$).  Such $f$ is \emph{postcritically finite} (PCF for short) if its postcritical set
\[\mathrm{PC}(f):=\bigcup_{n\geq1}f^n(\mathrm{Crit}(f))\]
is an algebraic subvariety of $\mathbb{P}^k$, where $\mathrm{Crit}(f)=\{z\in\mathbb{P}^k(\mathbb{C})\, : \ \det(D_zf)=0\}$ is its \emph{critical set}. In dimension $1$, the critical set is a finite set of cardinality $2d-2$ so, for all $n$, $f^n(\mathrm{Crit}(f))$ is again a finite set of cardinality $2d-2$ (counting the multiplicity), so PCF maps are not so hard to exhibit and it turns out that PCF maps are in fact Zariski dense (e.g. \cite{favredujardin,  buffepstein,
	GOV, Good-height}). In higher dimension, the algebraic hypersurface $\mathrm{Crit}(f)$ has positive dimension, hence it is not finite, and this fact is responsible for several new phenomena arising in complex dynamics in several variables.
In particular, only few examples of PCF maps which do not derive directly from 1-dimensional PCF maps are known: the first examples were given by Fornaess and Sibony \cite{FS-PCF}, interesting examples were also produced by Rong \cite{Rong-PCF} and by Koch \cite{Koch-PCF} who used constructions from Teichm\"uller theory. The dynamical study of such maps was notably developed by Ueda \cite{Ueda-PCF} and latter by Astorg~\cite{Astorg-PCF}.

Furthermore, since being PCF is invariant by conjugacy, such maps define elements of $\mathscr{M}_d^k$ since being PCF is invariant by conjugacy, and are of the utmost arithmetic importance since they satisfy $h_{\mathrm{crit}}(f)=0$ when they are defined over $\bar{\Q}$ (in dimension $1$, it is known that the converse is true by  Northcott's property, for $k>1$, this is an open and important problem). This motivates the following theorem:
\begin{Theorem}[Sparsity of PCF maps]\label{tm:not-Zariski-dense}
Fix two integers $k,d\geq2$. There exists a strict subvariety $V_d^k\subsetneq \mathrm{End}_d^k$ such that any PCF endomorphism $f$ is contained in $V_d^k$.
\end{Theorem}
Such a result was conjectured by Ingram, Ramadas and Silverman in \cite{IRS}, who showed that $\{f \in\mathrm{End}_d^k, \  f^n(\mathrm{Crit}(f))=  f^m(\mathrm{Crit}(f))\}$ is not Zariski dense for $m\in \{0, 1, 2\}$ and $d\geq 3$. Our approach is inspired by K\"uhne's Relative Equidistribution~\cite{Kuhne} on families of abelian varieties defined over a number field $\mathbb{K}$ which we generalize to the setting of families of polarized endomorphisms using the arithmetic equidistribution theorem of the first author \cite{Good-height} and Theorem~\ref{tm:critical-ample-height} (again, we stress that the arithmetic equidistribution result that we need was already proved by Yuan and Zhang \cite{YZ-adelic} using the properties of metrics on the Deligne pairing on adelic line bundles). Thus, if PCF maps were Zariski dense, they would equidistribute the bifurcation measure. In order to get a contradiction, we inject the following crucial theorem working at the complex place. 
\begin{Theorem}[Robust strong bifurcations]\label{tm:mu-interior}
	Fix two integers $k,d\geq2$. There exists a non-empty analytic open subset $\Omega\subset \mathscr{M}_d^k(\mathbb{C})$ (resp. $\Omega\subset \mathscr{P}_d^2$) such that
	\begin{itemize}
		\item the open set $\Omega$ is contained in $\mathrm{supp}(\mu_{f,\mathrm{Crit}})$,
		\item the open set $\Omega$ contains no PCF conjugacy class.
	\end{itemize}
\end{Theorem}
Observe that we do not prove the theorem for the moduli space $\mathscr P_d^k$ when $k\geq3$. This is a technical issue that comes from the fact that our proof of the generic finiteness of the multiplier maps only works on $\mathscr M_d^k$ and $\mathscr P_d^2$. The same result probably holds on $\mathscr P_d^k$ in all dimensions but our main motivation for the polynomial case is Theorem \ref{tm:uniformity} whose counterpart (see Theorem \ref{prop-empty-interior}) is weaker when $k\geq3.$

In dimension $1$, the works of Lyubich \cite{lyubich} and  Ma\~n\'e-Sad-Sullivan~\cite{MSS} imply that the bifurcation locus (i.e. the support of the bifurcation current) has empty interior. In higher dimension,
Bianchi and the second author first gave an example (the Desboves family) where this is not the case \cite{Bianchi_Taflin}. In general, the fact that the support of the bifurcation current has non-empty interior is due to the seminal work of Dujardin \cite{Dujardin_blender}, whose ideas were expanded upon by the second author in \cite{Taflin_blender} to produce open sets of bifurcation in other situations. A different approach, due to Biebler, was to construct an open set of bifurcations around Latt\`es maps \cite{Biebler_Lattes}. Turning to the bifurcation measure, the firsts to prove the non-emptiness of the support of the bifurcation measure were Astorg and Bianchi \cite{astorg-bianchi} in the very particular case of the family of polynomial skew-products (with given base dynamics satisfying certain additional assumptions) of $\C^2$.

An important ingredient in the proof of Theorem \ref{tm:mu-interior} is a mechanism called \emph{blender} in smooth dynamics. It was introduced by Bonatti-D\'\i az in \cite{bonatti-diaz-blender} to obtain new examples of robustly transitive diffeomorphisms. Since then, it was used in a wide range of contexts in real dynamics (see e.g. \cite{new-crit}, \cite{bonatti-diaz-tan} or \cite{berger-parablender}). A remarkable feature about blenders is that they are much easier to construct than other mechanisms given robust intersections (like the Newhouse phenomenon). This characteristic is of particular importance in the rigid setting of holomorphic dynamics, where they were first introduced by Dujardin in \cite{Dujardin_blender}.

Roughly speaking, in our context, a blender for a map $f$ is a repelling hyperbolic set (typically a Cantor set) that intersects an open family of (local) hypersurfaces and this property persists for small perturbations of $f$. Dujardin constructed in \cite{Dujardin_blender} a map with a blender for which a part of the postcritical set belongs to the associated family of hypersurfaces. This provides a robust intersection between the blender and the postcritical set which turns out to be sufficient to have an open set in the bifurcation locus. The same strategy can be followed in order to prove Theorem \ref{tm:mu-interior} except that instead of a single intersection we need as many as possible (i.e. the dimension of the moduli space) independent intersections, i.e. which satisfy the transversality condition of Definition \ref{def-trans}. To that end, we consider a map $f$ with a blender $\Lambda(f)$ and a saddle point $p(f)$ whose unstable manifold intersects robustly $\Lambda(f)$. Observe that in the terminology of smooth dynamics, this corresponds in our non-invertible context to a \emph{robust heterodimensional cycles} (see \cite{bonatti-diaz-hetero} for the interplays between these cycles and blenders in the $C^1$-setting). As the critical set has to intersect the stable manifold of $p(f)$, the inclination lemma gives infinitely many intersections between $\Lambda(f)$ and the postcritical set of $f$. All the difficulty in the proof is to check that they provide enough independent intersections. This brings us to prove (very) partial generalizations to higher dimension of several results from one-dimensional complex dynamics, like extension of local conjugacies \cite{buff-epstein}, the rigidity of stable algebraic families \cite{McMullen-families} or the fact that multipliers of periodic points provide (generically) local coordinates in the moduli space \cite{McMullen-families,gorbovickis-rat,JX}.

\bigskip

In the particular case of regular polynomial endomorphisms of the affine plane, we consider bifurcations of the finite part of the critical set, i.e. of the closure $C_f$ in $\mathbb{P}^2$ of the set $\{z\in\mathbb{C}^2, \det(D_zf)=0\}$. In this case, the non-negativity of the Green function at every place allows us to prove the following uniform result.

\begin{Theorem}[Uniformity]\label{tm:uniformity}
	Fix an integer $d\geq2$. There exists a constant $B(d)\geq1$ and a non-empty Zariski open subset $U\subset \mathrm{Poly}_d^2$ such for any $f\in U(\mathbb{C})$, we have
	\[\#\mathrm{Preper}(f)\cap C_f\leq B(d).\]
\end{Theorem}


 As mentioned before, similarly flavored uniform results already exist in complex dynamics and are a very important source of inspiration for us: the quotient $\tilde{S}:=(\mathrm{End}_d^1\times\mathrm{End}_d^1)/\mathrm{PGL}(2)$ by the diagonal action by conjugacy is a quasiprojective variety. On a suitable subvariety $S$ of $\tilde{S}$, one wants to show that outside a Zariski closed set of $S$, then all pairs of rational maps $f, g : \P^1 \to \P^1$ will have at most $B$ preperiodic points in
common. This is a statement very similar to Theorem~\ref{tm:uniformity} where one considers $\mathcal{Y}$ to be the fibered diagonal in $\mathbb{P}^1\times\mathbb{P}^1$ instead of the critical set. Then, such a result was first shown by DeMarco, Krieger and Ye in the Legendre family  \cite{DKY1} and in the quadratic family \cite{DKY2}. Then, Mavraki and Schmidt proved it in the case of any algebraic curve $S$ \cite{Mavraki_Schmidt}. Finally, DeMarco and Mavraki have shown very recently the optimal result that there is a uniform bound $B$, depending only on the degree $d$, so that for a Zariski open and dense
set in the space of all pairs of rational maps $f, g : \P^1 \to \P^1$ 
with degree $d$, $f$ and $g$ have at most $B$ preperiodic points in
common \cite{DeMarco-Mavraki}. 

\medskip 

To prove Theorem~\ref{tm:uniformity}, we show in Theorem~\ref{tm:gap} that there is a height gap (there exists a $\varepsilon>0$ such that all points of canonical height $\leq \varepsilon$ are contained in a Zariski closed proper subset of the fibered critical locus). For that, we follow the idea of Gao, Ge and K\"uhne on abelian varieties \cite{GGK} (first introduced by Ullmo \cite{Ullmo-Bogomolov} and Zhang~\cite{Zhang-Bogomolov}) to overfiber the dynamics (see also \cite{Mavraki_Schmidt} and \cite{DeMarco-Mavraki} where this strategy is used). The fact that local heights are all non-negative in the polynomial setting allows us to get the bound from the complex place.  

\begin{remark*}\normalfont
 To prove Theorem~\ref{tm:uniformity}, we give a positive answer to Problem 6.3.9 of \cite{YZ-adelic} in the moduli space $\mathscr{P}_d^2$ of regular polynomial endomorphisms of $\mathbb{P}^2$. To obtain a similar statement on $\mathscr{P}_d^k$ with $k\geq3$, the only missing piece is to prove that multiplier maps are generically finite to one on $\mathscr{P}_d^k$. This follows from a work in progress of the second author and Gorbovickis. 
\end{remark*}
\subsection*{Organization of the article}
 Section~\ref{height_and_current} is devoted to the construction of bifurcation currents and the corresponding volume we will need to construct the $m$-order canonical height.
 The proof of Theorem \ref{tm:mu-interior} occupies the next three sections. In Section~\ref{sec-rigidity}, we establish that the eigenvalues of the periodic points determine a conjugacy class up to finitely many choices generically in $\mathscr{M}_d^k$ and $\mathscr{P}_d^2$. In Section~\ref{sec-blender}, we prove that if an open subset of $\mathscr{M}_d^k$ or $\mathscr{P}_d^k$ satisfies a certain set of assumptions and is not contained in the support of the bifurcation measure then it has to contain lots of families where the eigenvalues of most of the periodic points are constant. Open subsets verifying this large set of assumptions are constructed in Section~\ref{sec-existence}, for all $k\geq2$ and $d\geq2$. In Section~\ref{section3}, we prove important height inequalities, and in particular Theorem~\ref{tm:critical-ample-height}. Section~\ref{section_equi} is devoted to the proof of the needed Relative Equidistribution theorem. In Section~\ref{section_harvest}, we prove Theorems \ref{tm:not-Zariski-dense} and \ref{tm:uniformity}.

 \subsection*{Acknowledgment} We would like to express our gratitude to Laura DeMarco and Myrto Mavraki for many discussions during the elaboration of this work and to Christian Bonatti and Romain Dujardin for sharing their insight on blenders. The first author would also like to thank Jean-Beno\^it Bost for many helpful discussions on Geometric Invariant Theory and the IHES for its hospitality. Finally, we would like to thank the referees for their very helpful comments and questions.

\section{The dynamical volumes of a family of subvarieties}\label{height_and_current}

\subsection{The canonical height over a function field of characteristic zero}\label{sec:def-currents}
We refer to \cite{GV_Northcott} for more details on the material of this Section. We treat here the case of general families of polarized endomorphisms and not only of families of endomorphisms of $\mathbb{P}^k$ since we will latter on need to take fibered product of such families (see \S~\ref{sec:fibered-product}).

In the whole section, we let $(\mathcal{X},f,\mathcal{L})$ be a family of polarized endomorphisms of degree $d\geq2$ parametrized by a normal projective variety $S$ with regular part $\mathcal{X}^0\to S^0$, all defined over $\C$. We also let $\mathcal{Y}\subsetneq\mathcal{X}$ be a subvariety such that $\pi(\mathcal{Y})=S$. As $\pi|_\mathcal{Y}$ is surjective, there is a non-empty Zariski open subset $S^0_\mathcal{Y}\subset S^0$ such that $\pi|_\mathcal{Y}$ is flat and projective over $S^0_\mathcal{Y}$. Up to replacing $S^0_\mathcal{Y}$ by a larger Zariski open of $S^0$ over which $\pi|_\mathcal{Y}$ is flat and projective, we can assume it is maximal for the inclusion with this property.

\begin{definition}
We say the tuple $(\mathcal{X},f,\mathcal{L},\mathcal{Y})$ is a \emph{dynamical pair} parametrized by $S$ and with regular part $S^0_\mathcal{Y}$.
\end{definition}
It is convenient for us to work in the more general setting of families of polarized endomorphisms since we shall be considering fibered products of several families. Nevertheless, as the first and third authors showed in \cite[\S 3.1]{GV_Northcott}, one can always reduce to the case of families of endomorphisms of a projective space. Indeed, by a result of Fakhruddin~\cite[Corollary 2.2]{Fakhruddin}, any family of polariezd endomorphisms  $(\mathcal{X},f,\mathcal{L})$, there exist $N\geq1$, an embedding $\iota:\mathcal{X}\to \mathbb{P}^N_\mathbb{C}\times S$, an integer $e\geq1$ and a family of endomorphisms $\mathcal{F}:\mathbb{P}^N_\mathbb{C}\times S\to \mathbb{P}^N_\mathbb{C}\times S$ such that
\[\iota \circ f = \mathcal{F} \circ \iota \quad \text{on} \ \mathcal{X}^0 \]
and $\mathcal{L}^{\otimes e}=\iota^*\mathcal{O}(1)$ on $\mathcal{X}^0$.
%
%
%
%

\bigskip

In particular, we can construct the fibered Green current $\widehat{T}_f$ of a family of polarized endomorphisms by pulling back the Green current of $\mathcal{F}$ to $\mathcal{X}^0(\C)$. This shows that $\widehat{T}_f$ is the restriction to a (possibly singular) subvariety of a current with continuous potentials in the ambient space $\mathbb{P}^N_\mathbb{C}\times S^0(\C)$. Demailly~\cite{Demailly_SMF} theory of intersection on singular variety thus applies to our case.   

The current $\widehat{T}_f$ is a closed positive $(1,1)$-current of finite mass on $\mathcal{X}^0(\C)$ and, for any $\lambda\in S^0(\C)$, the slice $T_{\lambda}:=\widehat{T}_{f}|_{X_\lambda}$ is the Green current of $f_\lambda$, see e.g.~\cite{GV_Northcott} for more details. Furthermore, if $k:= \dim X_\lambda$, we let $\mu_\lambda:=\deg_{L_\lambda}(X_\lambda)^{-1}\cdot T_\lambda^{k}$ be the unique maximal entropy measure of $f_\lambda$, we call its support the \emph{small Julia set}, and denote it by $J_k$ (or $J_k(f_\lambda)$ to stress the dependence of $f_\lambda$) \cite{briendduval, dinhsibony2}.  

\medskip

Let $\mathcal{M}$ be an ample $\mathbb{Q}$-line bundle on $S$. Let $Y_\eta$ be the generic fiber of the family $\mathcal{Y}\to S^0_\mathcal{Y}$. For any $n\geq1$, the map $f^n:\mathcal{X}\dashrightarrow \mathcal{X}$ has a priori indeterminacy points in $\mathcal{X}\setminus\mathcal{X}^0$ and we can find a projective variety $\mathcal{X}_n$, birational projective morphism $\psi_n:\mathcal{X}_n\to\mathcal{X}$ which is an isomorphism above $\mathcal{X}^0$ and a morphism $F_n:\mathcal{X}_n\to\mathcal{X}$ such that the following diagram commutes
\[\xymatrix{\mathcal{X}_n\ar[d]_{\psi_n} \ar[rd]^{F_n}& \\
\mathcal{X}\ar@{-->}[r]_{f^n}  & \mathcal{X}}
\]
Moreover, one can choose $\mathcal{X}_n$ as a finite sequence of blow-ups of $\mathcal{X}_{n-1}$.
Following \cite{Zhang-positivity}, we then define
\[\widehat{h}_{f_\eta}(Y_\eta):=\lim_{n\to\infty}d^{-n(\dim Y_\eta+1)}\frac{\left( (F_n)_*\left(\psi_n^*\{\mathcal{Y}\}\right)\cdot c_1(\mathcal{L})^{\dim Y_\eta+1}\cdot c_1(\pi^*\mathcal{M})^{\dim S-1}\right)}{(\dim Y_\eta+1)\deg_{Y_\eta}(L_\eta)}.\]

The next lemma follows from~\cite{GV_Northcott}:
\begin{lemma}\label{lm:GV}
For any $\mathcal{Y}$ as above, $\widehat{h}_{f_\eta}(Y_\eta)$ is well-defined and satisfies
\[\widehat{h}_{f_\eta}((f_\eta)_*(Y_\eta))=d\widehat{h}_{f_\eta}(Y_\eta).\]
In addition, we can compute $\widehat{h}_{f_\eta}(Y_\eta)$ as
\[\widehat{h}_{f_\eta}(Y_\eta)=\frac{1}{(\dim Y_\eta+1)\deg_{Y_\eta}(L_\eta)}\int_{\mathcal{X}^0(\C)}\widehat{T}^{\dim Y_\eta+1}_f\wedge[\mathcal{Y}]\wedge(\pi^*\omega_S)^{\dim S-1},\]
where $\omega_S$ is any smooth form representing $c_1(\mathcal{M})$.
\end{lemma}

\begin{remark*}\normalfont
As explained above, the integral can be lifted to $\P^N(\C) \times S$ where the map $f$ is lifted to a family of endomorphisms $F$. So, we can apply the theory of psh functions and currents on singular complex varieties of Demailly~\cite{Demailly_SMF} and the integral is well defined. Note also that, in what follows, we only consider the wedge product of currents which satisfy this lift property and that those currents either have continuous potential or are integration currents on closed subvarieties.
\end{remark*}

\begin{proof}
Let $q:=\dim Y_\eta$ and $p:=\dim S$. The fact that it is well-defined and the formula relating the limit of $d^{-n(q+1)}\left((F_n)_*\left(\psi_n^*\{\mathcal{Y}\}\right)\cdot c_1(\mathcal{L})^{q+1}\cdot c_1(\pi^*\mathcal{M})^{p-1}\right)$ with $\widehat{T}_f^{q+1}\wedge[\mathcal{Y}]\wedge (\pi^*\omega_S)^{p-1}$ are contained in \cite[Theorem~B]{GV_Northcott}. We then can compute
\begin{align*}
\widehat{h}_{f_\eta}((f_\eta)_*(Y_\eta)) & =\frac{1}{(q+1)\deg_{Y_\eta}(f_\eta^*L_\eta)}\int_{\mathcal{X}^0(\C)}\widehat{T}_f^{q+1}\wedge\left(f_*[\mathcal{Y}]\right)\wedge (\pi^*\omega_S)^{p-1}\\
& = \frac{1}{(q+1)d^q\deg_{Y_\eta}(L_\eta)}\int_{\mathcal{X}^0(\C)}\widehat{T}_f^{q+1}\wedge\left(f_*[\mathcal{Y}]\right)\wedge (\pi^*\omega_S)^{p-1}\\
& = \frac{1}{(q+1)d^q\deg_{Y_\eta}(L_\eta)}\int_{\mathcal{X}^0(\C)}\left(f^*\left(\widehat{T}_f^{q+1}\wedge (\pi^*\omega_S)^{p-1}\right)\right)\wedge [\mathcal{Y}]\\
& = \frac{d^{q+1}}{(q+1)d^q\deg_{Y_\eta}(L_\eta)}\int_{\mathcal{X}^0(\C)}\widehat{T}_f^{q+1}\wedge [\mathcal{Y}]\wedge (\pi^*\omega_S)^{p-1}=d\widehat{h}_{f_\eta}(Y_\eta),
\end{align*}
where we used that $\dim Y_\eta=q$ and $\pi\circ f=\pi$.
\end{proof}

 In particular, the last part of the lemma states that the height $\widehat{h}_{f_\eta}(Y_\eta)$ is $>0$ if and only if the bifurcation current  $T_{f,\mathcal{Y}}:=\pi_*\left(\widehat{T}_{f}^{\dim Y_\eta+1}\wedge[\mathcal{Y}]\right) $  is not identically $0$ in $S^0$, since 
 \[ \int_{\mathcal{X}^0(\C)}\widehat{T}^{\dim Y_\eta+1}_f\wedge[\mathcal{Y}]\wedge (\pi^*\omega_S)^{\dim S-1} =  \int_{S^0(\C)} \pi_*(\widehat{T}^{\dim Y_\eta+1}_f\wedge[\mathcal{Y}])\wedge \omega_S^{\dim S-1}.\]

\subsection{The higher bifurcation currents of a pair}\label{sec:fibered-product}
As above, let $(\mathcal{X},f,\mathcal{L},\mathcal{Y})$ be a dynamical pair parametrized by $S$ with regular part $S^0_\mathcal{Y}$.

 Let $\mathcal{M}$ be an ample $\mathbb{Q}$-line bundle on $S$ of volume $1$.
For any $m\geq1$, let $\mathcal{X}^{[m]}:=\mathcal{X}\times_S\cdots\times_S\mathcal{X}$ and $\mathcal{Y}^{[m]}:=\mathcal{Y}\times_S\cdots\times_S\mathcal{Y}$ be the respective $m$-fiber power of $\mathcal{X}$ and $\mathcal{Y}$. Denote also by $\pi_{[m]}:\mathcal{X}^{[m]}\to S$ the morphism induced by $\pi$. We define $f^{[m]}$ as
\[f^{[m]}(x)=(f_t(x_1),\ldots,f_t(x_m)), \quad x=(x_1,\ldots,x_m)\in X_t^m=\pi_{[m]}^{-1}\{t\}.\]
For any $1\leq j\leq m$, we let $p_j:\mathcal{X}^{[m]}\to\mathcal{X}$ be the projection onto the $j$-th factor of the fiber product and $\mathcal{L}^{[m]}:=p_1^*\mathcal{L}+\cdots+p_m^*\mathcal{L}$. By construction and using $\widehat{T}_f^{\dim X_\eta+1}=0$, we have
\begin{equation}\label{eq_factoriel}
\widehat{T}_{f^{[m]}}=p_1^*(\widehat{T}_f)+\cdots+p_m^*(\widehat{T}_f) \quad \text{and} \quad \widehat{T}_{f^{[m]}}^{m\dim X_\eta}= C(m,\dim X_\eta) \bigwedge_{j=1}^mp_j^*\left(\widehat{T}_f^{\dim X_\eta}\right),
\end{equation}
where $C(m,\dim X_\eta):=  \prod_{j=1}^m \binom{j \dim X_\eta}{\dim X_\eta}$.
We define higher bifurcation currents as follows:
\begin{definition}
For $1\leq m\leq \dim S$, the $m$-\emph{bifurcation current} of $(\mathcal{X},f,\mathcal{L},\mathcal{Y})$ is the closed positive $(m,m)$-current on $S^0_\mathcal{Y}(\C)$ given by
\[T_{f,\mathcal{Y}}^{(m)}:=(\pi_{[m]})_*\left(\widehat{T}_{f^{[m]}}^{m(\dim Y_\eta+1)}\wedge[\mathcal{Y}^{[m]}]\right).\]
The \emph{bifurcation measure} of $(\mathcal{X},f,\mathcal{L},\mathcal{Y})$ is 
\[\mu_{f,\mathcal{Y}}:=T_{f,\mathcal{Y}}^{(\dim S)}.\]
\end{definition}
As we will see below, the current $T_{f,\mathcal{Y}}^{(m)}$ $-$ which can be a priori non-zero even though $T_{f,\mathcal{Y}}^{m}$ is zero $-$is the current which characterizes a condition of non-degeneracy \`a la Yuan-Zhang~\cite{YZ-adelic}. In addition, it is more more practical, in order
to use dynamical arguments, to consider a bifurcation current associated to a
single (fibered) map rather than intersecting several bifurcation currents (e.g.
\cite{AGMV}) and we will consider, in a crucial way, families of fibered maps in the sequel.
We give basic properties of those currents.

\begin{proposition}\label{prop-higher-currents}
The following properties hold
\begin{enumerate}
\item For any $1\leq m\leq \dim S$ and any $j\leq m$
\begin{center}
$T_{f,\mathcal{Y}}^{(m)}\geq T_{f,\mathcal{Y}}^{(j)}\wedge T_{f,\mathcal{Y}}^{(m-j)}$, 
\end{center}
\item For all $m$, $T_{f,\mathcal{Y}}^{(m)}\neq 0$ implies $T_{f,\mathcal{Y}}^{(m-1)}\neq 0$.
Similarly, for all $m \geq \dim S$,
\[(\pi_{[m]})_*\left(\widehat{T}_{f^{[m]}}^{\dim\mathcal{Y}^{[m]}}\wedge [\mathcal{Y}^{[m]}]\right)\geq \mu_{f,\mathcal{Y}}.\] 
\item if $\dim Y_\eta=\dim X_\eta-1$, we have 
\[\mu_{f,\mathcal{Y}}=  \prod_{j=1}^{\dim S} \binom{j \dim X_\eta}{\dim X_\eta}\cdot\left( T_{f,\mathcal{Y}}^{(1)}\right)^{\wedge \dim S},\]
and for any $m\geq \dim S$, there is a constant $C_m\geq1$ such that
\[(\pi_{[m]})_*\left(\widehat{T}_{f^{[m]}}^{\dim\mathcal{Y}^{[m]}}\wedge [\mathcal{Y}^{[m]}]\right)=C_m\cdot \mu_{f,\mathcal{Y}}.\]
\end{enumerate}
\end{proposition}

\begin{proof} 
	Fix $1\leq m \leq \dim S$. For the sake of simplicity, let us only consider the case where $j=1$. Let $p_i:\mathcal{X}^{[m]}\to\mathcal{X}$ be the projection onto the $i$-th factor and $\pi_{[m]}: \mathcal{X}^{[m]}\to S$ be the canonical projection (so that $\pi_{[1]}$ is the projection $ \mathcal{X}\to S$). Let also $\tau_m:\mathcal{X}^{[m]}\to\mathcal{X}^{[m-1]}$ be the projection forgetting the first factor. Then 
	\begin{align*}
	\widehat{T}_{f^{[m]}}^{m(\dim Y_\eta+1)} & =\left(p_1^*(\widehat{T}_f)+\cdots+p_m^*(\widehat{T}_f)\right)^{m(\dim Y_\eta+1)}\\
	&\geq  p_1^*\left(\widehat{T}_f^{\dim Y_\eta+1}\right)\wedge \left(p_2^*(\widehat{T}_f)+\cdots+p_m^*(\widehat{T}_f)\right)^{(m-1)(\dim Y_\eta+1)}
	\end{align*}
as $(p_i^*\widehat{T}_f)^{\dim Y_\eta+1}= p_i^*(\widehat{T}_f^{\dim Y_\eta+1})$. Using the equality 
\[[\mathcal{Y}^{[m]}]= p_1^*( [\mathcal{Y}])\wedge \tau_m^*( [\mathcal{Y}^{[m-1]}])\]
and the equality
\[ \tau_m^*\left(\widehat{T}_{f^{[m-1]}}^{(m-1)(\dim Y_\eta+1)}\right)=\left(p_2^*(\widehat{T}_f)+\cdots+p_m^*(\widehat{T}_f)\right)^{(m-1)(\dim Y_\eta+1)},\]
we deduce that for a positive test form $\phi$ of bidimension $(m,m)$ on $S_\mathcal{Y}^0$,
\begin{align*}
\langle T_{f,\mathcal{Y}}^{(m)}, \phi \rangle & =\langle (\pi_{[m]})_*\left(\widehat{T}_{f^{[m]}}^{m(\dim Y_\eta+1)}\wedge[\mathcal{Y}^{[m]}]\right), \phi \rangle\\
 &= \langle \widehat{T}_{f^{[m]}}^{m(\dim Y_\eta+1)}\wedge[\mathcal{Y}^{[m]}], \pi_{[m]}^*\phi \rangle\\
&\geq\left \langle p_1^*\left(\widehat{T}_f^{\dim Y_\eta+1}\right)\wedge \tau_m^*\left(\widehat{T}_{f^{[m-1]}}^{(m-1)(\dim Y_\eta+1)}\right), \pi_{[m]}^*\phi \right\rangle\\
&\geq  \left\langle \widehat{T}_f^{\dim Y_\eta+1} \wedge [\mathcal{Y}]\wedge (p_1)_*\left(\tau_m^*\left(\widehat{T}_{f^{[m-1]}}^{(m-1)(\dim Y_\eta+1)}\right)\right), \pi_{[1]}^*\phi \right\rangle
\end{align*}
 where we used $\pi_{[m]}^*\phi=p_1^*(\pi_{[1]}^*\phi)$. By construction, we have
 \[(p_1)_*\left(\tau_m^*\left(\widehat{T}_{f^{[m-1]}}^{(m-1)(\dim Y_\eta+1)}\right)\right)=  \pi_{[1]}^*\left((T_{f,\mathcal{Y}}^{(m-1)})\right) \]
so that 	
\begin{align*}
\langle T_{f,\mathcal{Y}}^{(m)}, \phi \rangle &\geq \left\langle \widehat{T}_f^{\dim Y_\eta+1} \wedge [\mathcal{Y}]\wedge \pi_{[1]}^*\left(T_{f,\mathcal{Y}}^{(m-1)}\right), \pi_{[1]}^*\phi \right\rangle
\end{align*}	
which gives the first point.

We prove the second point. Assume $T_{f,\mathcal{Y}}^{(m)}\neq 0$. Then, we develop the product 
\[\widehat{T}_{f^{[m]}}^{m(\dim Y_\eta+1)}=\left(p_1^*(\widehat{T}_f)+\cdots+p_{m}^*(\widehat{T}_f)\right)^{m(\dim Y_\eta+1)}\] 
in $T_{f,\mathcal{Y}}^{(m)}$. We deduce that there exists $(\alpha_1, \dots \alpha_{m})$ with $\alpha_1 + \dots + \alpha_m= m(\dim Y_\eta+1) $ such that 
\[  \bigwedge_{i=1}^{m} p_i^*\left(\widehat{T}_{f}^{\alpha_j}\wedge[\mathcal{Y}]\right) >0.\]
By symmetry, we can assume that $\alpha_2 + \dots + \alpha_m \geq (m-1)(\dim Y_\eta+1)$. Take $\phi$ a test form on $S^0_\mathcal{Y}$ so that
\begin{align*}
0\neq &\left\langle  \bigwedge_{i=1}^{m} p_i^*\left(\widehat{T}_{f}^{\alpha_i}\wedge[\mathcal{Y}]\right) , \pi_{[m]}^*(\phi)\right\rangle\\
&\qquad = \left\langle \widehat{T}_{f}^{\alpha_1}\wedge[\mathcal{Y}]\wedge 
	(p_1)_*\left(\bigwedge_{i=2}^{m}  p_i^*\left(\widehat{T}_{f}^{\alpha_i}\wedge[\mathcal{Y}]\right) \right) , \pi_{[1]}^*(\phi) \right\rangle.
\end{align*}
In particular, we deduce that $
(p_1)_*\left(\bigwedge_{i=2}^{m}  p_i^*\left(\widehat{T}_{f}^{\alpha_i}\wedge[\mathcal{Y}]\right) \right)$ is non zero, which in turn implies $
(p_1)_*\left(\bigwedge_{j=2}^{m}  p_i^*\left(\widehat{T}_{f}^{\alpha'_i}\wedge[\mathcal{Y}]\right) \right)$ is non-zero, where the $\alpha'_i$ are non-negative integers such that $\alpha'_i\leq \alpha_i$ for all $i$ and  $\alpha'_2 + \dots + \alpha'_m = (m-1)(\dim Y_\eta+1)$. This in turn implies that $T_{f,\mathcal{Y}}^{(m-1)}\neq0$. The case $m\geq \dim S$ is similar.

The proof of the third point is now similar to that of the first point. Indeed, assume $\dim Y_\eta=\dim X_\eta -1$, then $p_i^*(\widehat{T}_f)^{\dim Y_\eta + 2} \wedge p_i^*( [\mathcal{Y}])= p_i^*(\widehat{T}_f^{\dim Y_\eta + 2} \wedge [\mathcal{Y}])=0$. In particular,  
\[\widehat{T}_{f^{[\dim S]}}^{\dim S(\dim Y_\eta +1)}\wedge[\mathcal{Y}^{[\dim S]}]=\prod_{i=1}^{\dim S} \binom{i \dim X_\eta}{\dim X_\eta}\bigwedge_{i=1}^{\dim S}p_i^*\left(\widehat{T}_f)^{\dim Y_\eta + 1}\wedge [\mathcal{Y}]\right) \]
and the rest follows. 
To conclude, take $m \geq \dim S$:
\begin{align*}(\pi_{[m]})_*\left(\widehat{T}_{f^{[m]}}^{\dim\mathcal{Y}^{[m]}}\wedge [\mathcal{Y}^{[m]}]\right)&=  (\pi_{[m]})_*\left(  \left(p_1^*(\widehat{T}_f)+\cdots+p_m^*(\widehat{T}_f)\right)^{\dim\mathcal{Y}^{[m]}}\wedge [\mathcal{Y}^{[m]}]\right)  \\ 
                          & = (\pi_{[m]})_*\left(  \left(p_1^*(\widehat{T}_f)+\cdots+p_m^*(\widehat{T}_f)\right)^{m \dim Y_\eta + \dim S}\wedge [\mathcal{Y}^{[m]}]\right).          
\end{align*}
Developing all terms in the product and using  $ \widehat{T}_f^{\dim Y_\eta + 2} \wedge [\mathcal{Y}]=0$, we end up, up to permutations, to a sum of terms of the form 
\begin{align*}(\pi_{[m]})_*\left( \bigwedge_{i=1}^{\dim S} \left(p_i^*(\widehat{T}_f^{\dim Y_\eta +1} \wedge [\mathcal{Y}]\right) \wedge  \bigwedge_{\ell=\dim S+1}^{m} \left(p_\ell^*(\widehat{T}_f^{\dim Y_\eta} \wedge [\mathcal{Y}]) \right) \right).          
\end{align*}
The assertion now follows by Fubini.
\end{proof}

\subsection{The dynamical volumes of a pair}\label{Sec:dyn-volumes}
As above, let $(\mathcal{X},f,\mathcal{L},\mathcal{Y})$ be a dynamical pair parametrized by $S$ with regular part $S^0_\mathcal{Y}$.
We now can define the dynamical volumes of $\mathcal{Y}$ as follows
\begin{definition} 
For any $m\geq\dim S$, we define the $m$-\emph{dynamical volume} $\mathrm{Vol}_{f}^{(m)}(\mathcal{Y})$ of $\mathcal{Y}$ for $(\mathcal{X},f,\mathcal{L})$ as the non-negative real number
\[\mathrm{Vol}_{f}^{(m)}(\mathcal{Y}):=\int_{(\mathcal{X}^{[m]})^0(\C)}\widehat{T}_{f^{[m]}}^{\dim \mathcal{Y}^{[m]}}\wedge[\mathcal{Y}^{[m]}].\]
For any ample $\mathbb{Q}$-line bundle $\mathcal{M}$ on $S$, we also define the $m$-\emph{parametric degree} $\deg_{f,\mathcal{M}}^{(m)}(\mathcal{Y})$ of $\mathcal{Y}$ relative to $\mathcal{M}$ as 
\[\deg_{f,\mathcal{M}}^{(m)}(\mathcal{Y}):=\int_{(\mathcal{X}^{[m]})^0(\C)}\widehat{T}_{f^{[m]}}^{\dim\mathcal{Y}^{[m]}-1}\wedge[\mathcal{Y}^{[m]}]\wedge \pi_{[m]}^*\left(\omega_S\right),\]
where $\omega_S$ is any smooth form on $S$ representing $c_1(\mathcal{M})$.
\end{definition}

\begin{remark}\normalfont
 When $\dim S=1$, a computation gives
\[\deg^{(1)}_{f,\mathcal{M}}(\mathcal{Y})=\deg_{Y_\eta}(L_\eta)\cdot \deg_\mathcal{M}(S)>0.\]
In particular, if $\deg_\mathcal{M}(S)=1$, then
\[\widehat{h}_{f_\eta}(Y_\eta)=\frac{\mathrm{Vol}^{(1)}_{f}(\mathcal{Y})}{(\dim Y_\eta+1)\cdot \deg^{(1)}_{f,\mathcal{M}}(\mathcal{Y})}.\]
\end{remark}

We can relate the non-vanishing of the bifurcation measure $\mu_{f,\mathcal{Y}}$ with the non vanishing of the volumes $\mathrm{Vol}^{(m)}_{f}(\mathcal{Y})$ for all $m\geq \dim S$. 

\begin{proposition}\label{prop-volumes} The following properties hold:
\begin{enumerate}
\item We have $\mu_{f,\mathcal{Y}}$ is non-zero if and only if for all $m\geq \dim S$, $\mathrm{Vol}_{f}^{(m)}(\mathcal{Y})>0$. 
\item For any integer $m\geq \dim S$, and for any ample $\mathbb{Q}$-line bundle $\mathcal{M}$ on $S$ of volume $1$, we have
\[\deg_{f,\mathcal{M}}^{(m)}(\mathcal{Y})\geq m\deg_{Y_\eta}(L_\eta)^{m-\dim S+1}\int_{S^0(\C)}T_{f,\mathcal{Y}}^{(\dim S-1)}\wedge \omega_S,\]
for any smooth form $\omega_S$ which represents $c_1(\mathcal{M})$.
\end{enumerate}
In particular, if $\mu_{f,\mathcal{Y}}\neq0$, then for all $m\geq\dim S$, and all $\mathcal{M}$, we have $\mathrm{Vol}^{(m)}_{f}(\mathcal{Y})>0$ and $\deg_{f,\mathcal{M}}^{(m)}(\mathcal{Y})>0$.
\end{proposition}

\begin{proof}
The first point follows from Proposition~\ref{prop-higher-currents} (2).
Let $p:=\dim S$. To prove the second point, we remark that
\[ \widehat{T}_{f^{[m]}}^{\dim \mathcal{Y}^{[m]}-1}\geq \left(\bigwedge_{j=p}^{m}p_{\ell_j}^*\widehat{T}_f^{\dim Y_\eta}\right)\wedge\bigwedge_{j=1}^{p-1}\left(p_j^*\widehat{T}_f^{\dim Y_\eta+1}\right).\]
Let $\pi_p:\mathcal{X}^{[m]}\to\mathcal{X}^{[p-1]}$ be the projection forgetting the $m-p+1$ last variables and for $1\leq j\leq p-1$, let $p_j':\mathcal{X}^{[p-1]}\to\mathcal{X}$ be the projection onto the $j$-th factor. The measure
\[\left(\bigwedge_{j=p}^{m}p_{\ell_j}^*\widehat{T}_f^{\dim Y_\eta}\right)\wedge\bigwedge_{j=1}^{p-1}\left(p_j^*\widehat{T}_f^{\dim Y_\eta+1}\right)\wedge[\mathcal{Y}^{[m]}]\wedge \pi_{[m]}^*(\omega_S)\]
rewrites as
\[\bigwedge_{\ell=p}^mp_\ell^*\left(\widehat{T}_f^{\dim Y_\eta}\wedge[\mathcal{Y}]\right)\wedge(\pi_p)^*\left(\bigwedge_{j=1}^{p-1}(p_j')^*\left(\widehat{T}_f^{\dim Y_\eta+1}\wedge[\mathcal{Y}]\right)\wedge \pi_{[p-1]}^*(\omega_S)\right)\]
where we used that $\pi_{[p]}^*\omega_S=(\pi_p)^*\left(\pi_{[p-1]}^*\omega_S\right)$. In particular, its volume is that of its push-forward by $\pi_p$, which is the measure
\[(\pi_p)_*\left(\bigwedge_{\ell=p}^mp_\ell^*\left(\widehat{T}_f^{\dim Y_\eta}\wedge[\mathcal{Y}]\right)\right)\wedge\bigwedge_{j=1}^{p-1}(p_j')^*\left(\widehat{T}_f^{\dim Y_\eta+1}\wedge[\mathcal{Y}]\right)\wedge \pi_{[p-1]}^*(\omega_S).\]
We now remark that $\pi_p$ has fibers of dimension $k(m-p+1):=\dim X_\eta^{m-p+1}$, where $k=\dim X_\eta$, and that 
$\widehat{T}_f^{\dim Y_\eta}\wedge[\mathcal{Y}]$ is a $(k,k)$-current on $\mathcal{X}^0(\C)$, so that the current 
\[T:=(\pi_p)_*\left(\bigwedge_{\ell=p}^{m}p_\ell^*\left(\widehat{T}_f^{\dim Y_\eta}\wedge[\mathcal{Y}]\right)\right)\]
 is a $(0,0)$-current on $(\mathcal{X}^{[p-1]})^0(\C)$ which is nothing but the constant $\deg_{Y_\eta}(L_\eta)^{m-p+1}$. 
Therefore, the volume of the studied measure is exactly 
\[\deg_{Y_\eta}(L_\eta)^{m-p+1}\cdot \int_{S^0(\C)}T_{f,\mathcal{Y}}^{(\dim S-1)}\wedge\omega_S.\]
 As the wedge product is symmetric, proceeding similarly for the other terms of the sum, we find
\[\deg_{f,\mathcal{M}}(\mathcal{Y})\geq m\cdot \deg_{Y_\eta}(L_\eta)^{m-p+1}\cdot \int_{S^0(\C)}T_{f,\mathcal{Y}}^{(\dim S-1)}\wedge\omega_S,\]
and the proof of the second point is complete (observe that the second point of Proposition~\ref{prop-higher-currents} guarantees that $T_{f,\mathcal{Y}}^{(\dim S-1)}\neq0$). 
\end{proof}

\subsection{Dynamical volume as limits of iterated intersection numbers}
Let $(\mathcal{X},f,\mathcal{L})$ be a family of polarized endomorphisms of degree $d$ and $\mathcal{Y}\subsetneq\mathcal{X}$ be a subvariety with $\pi(\mathcal{Y})=S$. Let also $m\geq1$ be an integer and let $(\mathcal{X}^{[m]},f^{[m]},\mathcal{L}^{[m]})$ be the polarized endomorphism induced on $\mathcal{X}^{[m]}:=\mathcal{X}\times_S\cdots\times_S\mathcal{X}$ as above with induced morphism $\pi_{[m]}:\mathcal{X}^{[m]}\to S$ and let $\mathcal{Y}^{[m]}:=\mathcal{Y}\times_S\cdots\times_S\mathcal{Y}$. One can check that $\pi_{[m]}(\mathcal{Y}^{[m]})=S$ and we have the following.
\begin{lemma}\label{lm:goodmodel}
For any $m\geq1$, there is a sequence $(\mathcal{X}_n^{(m)})_{n\geq0}$ of projective varieties, a sequence $\psi_n^{(m)}:\mathcal{X}_n^{(m)}\to\mathcal{X}^{[m]}$ of birational projective morphisms which are isomorphisms above the regular part of $(\mathcal{X}^{[m]})^0$ and a sequence of morphisms $F_n^{(m)}:\mathcal{X}_n^{(m)}\to \mathcal{X}^{[m]}$ such that $\mathcal{X}_0^{(m)}=\mathcal{X}^{[m]}$ and the following diagram commutes
\[\xymatrix{\mathcal{X}_n^{(m)}\ar[d]_{\psi_n^{(m)}} \ar[rd]^{F_n^{(m)}}& \\
\mathcal{X}^{[m]}\ar@{-->}[r]_{(f^{[m]})^n}  & \mathcal{X}^{[m]}}.
\]
Moreover, one can choose $\mathcal{X}_n^{(m)}$ as a finite sequence of blow-ups of $\mathcal{X}_{n-1}^{(m)}$.
\end{lemma}

 Relying on estimates from \cite{GV_Northcott} we can deduce
\begin{lemma}\label{cor:GV}
For any $m\geq\dim S$, there is a constant $C_m\geq1$ depending only on $(\mathcal{X},f,\mathcal{L},\mathcal{Y})$ and $m$ such that for any $n\geq1$, if $\mathcal{Y}_{n}^{(m)}:=(F_n^{(m)})_*(\psi_n^{(m)})^*\mathcal{Y}^{[m]}$, then
\[\left|\frac{\left(\{\mathcal{Y}_{n}^{(m)}\}\cdot c_1(\mathcal{L}^{[m]})^{\dim\mathcal{Y}^{[m]}}\right)}{d^{n\dim\mathcal{Y}^{[m]}}}-\mathrm{Vol}^{(m)}_{f}(\mathcal{Y})\right|\leq C_md^{-n},\]
and,  for any ample $\mathbb{Q}$-line bundle $\mathcal{M}$ on $S$ of volume $1$,
\[\left| \frac{\left(\{\mathcal{Y}_{n}^{(m)}\}\cdot c_1(\mathcal{L}^{[m]})^{\dim\mathcal{Y}^{[m]}-1}\cdot c_1(\pi_{[m]}^*\mathcal{M})\right)}{d^{n(\dim\mathcal{Y}^{[m]}-1)}}-\deg_{f,\mathcal{M}}^{(m)}(\mathcal{Y})\right|\leq C_md^{-n}.\]
\end{lemma}

\begin{proof}
Let  $\omega_S$ be a smooth form on $S(\C)$ which represents $c_1(\mathcal{M})$ (it has mass $\deg_S(\mathcal{M})=c_1(\mathcal{M})^{\dim S}=1$) and $\omega_\mathcal{L}$ be a smooth form on $\mathcal{X}(\C)$ which represents $c_1(\mathcal{L})$. For $m\geq\dim S$, define $\omega^{[m]}:=\sum_jp_j^*\omega$. Let $(\mathcal{X}^{[m]})^0:=\pi_{[m]}^{-1}(S^0_\mathcal{Y})$ as above. By definition, we have
\begin{align*}
\left(\{\mathcal{Y}_{n}^{(m)}\}\cdot c_1(\mathcal{L}^{[m]})^{\dim\mathcal{Y}^{[m]}}\right) & =\int_{(\mathcal{X}^{[m]})^0(\C)}\left(\omega^{[m]}\right)^{\dim\mathcal{Y}^{[m]}}\wedge[\mathcal{Y}_{n}^{(m)}]\\
&=\int_{(\mathcal{X}^{[m]})^0(\C)}\left(((f^{[m]})^n)^*\omega^{[m]}\right)^{\dim\mathcal{Y}^{[m]}}\wedge[\mathcal{Y}^{[m]}].
\end{align*}
We rely on Proposition~3.4 of \cite{GV_Northcott}: we have
\[d^{-n\dim\mathcal{Y}^{[m]}}\left(\{\mathcal{Y}_{n}^{(m)}\}\cdot c_1(\mathcal{L}^{[m]})^{\dim\mathcal{Y}^{[m]}}\right)=\int_{(\mathcal{X}^{[m]})^0(\C)}\widehat{T}_{f^{[m]}}^{\dim\mathcal{Y}^{[m]}}\wedge[\mathcal{Y}^{[m]}]+O\left(\frac{1}{d^n}\right).\]
This is the first assertion we want to prove.
Similarly, 
\begin{align*}
I_{n,m}:&=\bigg(\{\mathcal{Y}_{n}^{(m)}\}\cdot c_1(\mathcal{L}^{[m]})^{\dim\mathcal{Y}^{[m]}-1} \cdot c_1(\pi_{[m]}^*\mathcal{M})\bigg)  \\
& =\int_{(\mathcal{X}^{[m]})^0(\C)}\left(((f^{[m]})^n)^*\omega^{[m]}\right)^{\dim\mathcal{Y}^{[m]}-1}\wedge[\mathcal{Y}^{[m]}]\wedge(\pi_{[m]}^*\omega_S)
\end{align*}
and the same argument using Proposition~3.4 of~\cite{GV_Northcott} gives
\begin{align*}
d^{-n(\dim\mathcal{Y}^{[m]}-1)}I_{n,m} & =\int_{(\mathcal{X}^{[m]})^0(\C)}\widehat{T}_{f^{[m]}}^{\dim\mathcal{Y}^{[m]}-1}\wedge[\mathcal{Y}^{[m]}]\wedge\pi_{[m]}^*(\omega_S)+O\left(\frac{1}{d^n}\right).
\end{align*}
This concludes the proof.
\end{proof}

\subsection{A sufficient criterion for positive volume}
To finish this section, we give a sufficient criterion for a parameter to belong to the support of the measure $\mu_{f,\mathcal{Y}}$. The existence of such a parameter implies in particular that $\mathrm{Vol}_f(\mathcal{Y})>0$.

\begin{definition}\label{def-trans}
Pick an integer $m\geq1$. We say that $\mathcal{Y}$ is $m$-\emph{transversely $J_k$-prerepelling (resp. properly $J_k$-prerepelling)} at a point $z=(z_1,\ldots,z_{m})\in \mathcal{X}^{[m]}$ with $\lambda_0:=\pi_{[m]}(z)\in S^0$ if $z_1,\ldots,z_{m}$ are $J_k(f_{\lambda_0})$-repelling periodic points of $f_{\lambda_0}$ and if there exist an integer $N\geq1$ and a neighborhood $U$ of $\lambda_0$ such that, if $z_j(\lambda)$ is the natural continuation of $z_j$ as a repelling periodic point of $f_\lambda$ in $U$, then
\begin{enumerate}
\item $z_j\in f_{\lambda_0}^N(Y_{\lambda_0})$ for all $1\leq j\leq m$,
\item $z_j(\lambda)\in J_k(f_\lambda)$ for all $\lambda\in U$ and all $1\leq j\leq m$,
\item the image of the local section $Z:\lambda\in U\mapsto(z_1(\lambda),\ldots,z_{m}(\lambda))\in (\mathcal{X}^{[m]})^0$ of $\pi_{[m]}$ intersects transversely, as local submanifolds, a local branch of $(f^{[m]})^{N}(\mathcal{Y}^{[m]})$ at $z$ (resp. $z$ lies in an proper intersection between the image of $Z$ and a local branch of $(f^{[m]})^{N}(\mathcal{Y}^{[m]})$ of pure dimension $\dim S-m$).
\end{enumerate}
\end{definition}

 In some sense, this definition is equivalent to the existence of $m$ \emph{independent} Misiurewicz intersections. The case of single Misiurewicz intersections corresponds to \emph{Misiurewicz parameters} in \cite{BBD}.  The third point in the definition seems a bit technical but in the examples we will construct, we cannot a priori exclude the case where $\mathcal{Y}^{[m]}$ is not locally irreducible and the periodic points lie persistently in a local branch of $\mathcal{Y}$ but transversely to another local branch. Another important remark for what follows is that, as observed by Dujardin (see \cite[Proposition-Definition 2.5]{Dujardin_blender}), the repelling periodic points can be replaced by points in a repelling hyperbolic set contained in $J_k.$ Finally, notice that when $m=\dim S$ and $\mathcal{Y}$ is locally irreducible near $z_1,\ldots,z_m$, Definition \ref{def-trans} is exactly what DeMarco and Mavraki \cite{DeMarco-Mavraki} call a rigid $m$-repeller.

We prove the following, which is a general criterion in the spirit of \cite[Proposition~4.8]{DeMarco-Mavraki}.
\begin{proposition}\label{tm:densitypart1}
Let $(\mathcal{X},f,\mathcal{L})$ be a family of polarized endomorphisms parametrized by $S$ and let $\mathcal{Y}\subsetneq\mathcal{X}$ be a hypersurface which projects dominantly to $S$. Let $1\leq m\leq \dim S$ and assume $\mathcal{Y}$ is $m$-properly $J_k$-prerepelling at $z\in (\mathcal{X}^{[m]})^0$. Then 
\begin{center}
$z\in\mathrm{supp}\left(\widehat{T}_{f^{[m]}}^{m(\dim Y_\eta+1)}\wedge [\mathcal{Y}^{[m]}]\right)$.
\end{center}
In particular, $\pi_{[m]}(z)\in\mathrm{supp}(T^{(m)}_{f,\mathcal{Y}})$.
\end{proposition}

The proof of this result is an adaptation of the strategy of Buff and Epstein~\cite{buffepstein} and the strategy of Berteloot, Bianchi and Dupont~\cite{BBD}, see also~\cite{Article1,AGMV,gauthier-abscont,GV_Northcott}.

\begin{proof}[Proof of Proposition~\ref{tm:densitypart1}]
As the statement is purely local, we let $\mathbb{B}\subset S^0$ be a ball in a local coordinate centered at $\lambda_0$. Since $\widehat{T}_f$ has continuous potentials, for any analytic submanifold $\Lambda\subset\mathbb{B}$ of dimension $m$ with $\lambda_0\in \Lambda$, we have
\[\mathrm{supp}\left(\widehat{T}_{f^{[m]}}^{m(\dim Y_\eta+1)}|_{\pi_{[m]}^{-1}(\Lambda)}\wedge [\mathcal{Y}^{[m]}\cap\pi_{[m]}^{-1}(\Lambda)]\right)\subset \mathrm{supp}\left(\widehat{T}_{f^{[m]}}^{m(\dim Y_\eta+1)}\wedge [\mathcal{Y}^{[m]}]\right)\]
by e.g.~\cite[Lemma~6.3]{Article1}. In particular, we can replace $\mathbb{B}$ with the intersection between $\mathbb{B}$ with a subspace $\mathbb{B}\cap V$ where $V$ is a linear subspace of dimension $m$ such that the intersection between the image of the local section $Z:\lambda\in U\cap V\mapsto(z_1(\lambda),\ldots,z_{m}(\lambda))\in (\mathcal{X}^{[m]})^0$ of $\pi_{[m]}$ and a local branch of $(f^{[m]})^{N}(\mathcal{Y}^{[m]})$ at $z$ is isolated in $(\mathcal{X}^{[m]})^0\cap\pi_{[m]}^{-1}(\mathbb{B}\cap V)$. In the rest of the proof, we thus can assume $m=\dim S$ and let $k$ be the relative dimension of $\mathcal{X}$ over $S$ so that $\dim Y_\eta +1=k$.  To simplify notations, write $F:=f^{[m]}:(\mathcal{X}^{[m]})^0\to(\mathcal{X}^{[m]})^0$ and we let $\mu:=T_{f,\mathcal{Y}}^{(m)}|_{\mathbb{B}}$.

\bigskip

Our aim here is to exhibit a basis of neighborhood $\{\Omega_n\}_n$ of $\lambda_0$ in $\B$ with $\mu(\Omega_n)>0$ for all $n$.
For a Borel subset $B\subset\mathbb{B}$, let $(\mathcal{X}^{[m]})_B:=\pi_{[m]}^{-1}(B)$, where $\pi_{[m]}:\mathcal{X}^{[m]}\to S$ is the map induced by $\pi:\mathcal{X}\to S$. Then, since $F^*\widehat{T}_F=d\widehat{T}_F$, we have
\begin{eqnarray*}
(F^n)^*\left(\widehat{T}^{km}_F\right)=d^{mkn}\widehat{T}^{km}_F \quad \text{and} \quad \mu(B) & = & d^{-kmn}\int_{(\mathcal{X}^{[m]})_B}\widehat{T}_F^{mk}\wedge(F^n)_*\left[\mathcal{Y}^{[m]}\right].
\end{eqnarray*}

Since $\mathcal{Y}$ is properly $J_k$-prerepelling at $\lambda_0$, there are $z_1,\ldots,z_m\in X_{\lambda_0}$, $J_k(f_{\lambda_0})$-repelling periodic points and $N\geq1$ such that $(z_1,\ldots,z_m)\in F^N(\mathcal{Y}^{[m]})^0$.
Let $p\geq1$ be such that $f^p_{\lambda_0}(z_i)=z_i$ for all $i$. We let $\mathcal{Y}_0$ be the local branch of $(F^{N})(\mathcal{Y}^{[m]})^0$ satisfying the hypothesis of the Proposition. For any integer $n\geq1$, we let $\mathcal{Y}_n:=(F^{np})(\mathcal{Y}_0)$, so that $\dim(\mathcal{Y}_n)= mk$ and
\[I_n:=\int_{(\mathcal{X}^{[m]})_\mathbb{B}}\widehat{T}_F^{mk}\wedge\left[\mathcal{Y}_n\right]\leq d^{knm}\mu(\mathbb{B}).\]
By our choice of $\mathbb{B}$, we have that $z_i(\lambda)$ is $J_k(f_{\lambda})$-repelling for all $\lambda\in \mathbb{B}$ and that there is $K>1$ such that
\[d(f_\lambda^{p}(z),f_\lambda^{p}(w))\geq K\cdot d(z,w)\]
for all $z,w\in D(z_j,\epsilon)\subset\mathcal{X}$ and all $\lambda\in \mathbb{B}$ for some given $\epsilon>0$ with $\pi(D(z_j,\epsilon))\subset\mathbb{B}$ \cite{BBD}, where $D(z_j,\epsilon)$ is a polydisk of polyradius $(\epsilon,\ldots,\epsilon)$.
Thus, if we denote $\mathrm{z}:=(z_1,\ldots,z_m)\in\mathcal{X}^{[m]}$ and by $S_n$ the connected component of $\mathcal{Y}_n\cap D_\epsilon$ containing $z$ where $D_\epsilon:=D_{\mathcal{X}^{[m]}}(\mathrm{z},\epsilon)$, the current $[S_n]$ is vertical-like in $B_\epsilon$ (i.e. $\pi_{[m]}(\mathrm{supp}([S_n])\cap D_\epsilon)$ is relatively compact in $\pi_{[m]}(D_\epsilon)$), and there exist $n_0\geq1$ and a basis of neighborhood $\Omega_n$ of $\lambda_0$ in $\mathbb{B}$ such that for all $n\geq n_0$
\[\supp([S_n])=S_n\subset\mathcal{X}^{[m]}_{\Omega_{n}}\cap D_\epsilon.\]

 Let $S$ be any weak limit of the sequence $[S_n]/\|[S_n]\|$, where the mass $\|[S_n]\|$ is computed with respect to some K\"ahler form $\alpha$ on $\mathcal{X}^{[m]}_{\mathbb{B}}$. Then $S$ is a closed positive $(k,k)$-current of mass $1$ in $D_\epsilon$ whose support is contained in the fiber  $X_{\lambda_0}^m$ of $\pi_{[m]}$. Hence $S=M\cdot[X_{\lambda_0}^m\cap D_\epsilon]$, where $M^{-1}>0$ is the volume of $D_\epsilon$ for the volume form $\alpha|_{X_{\lambda_0}^m}$.

As a consequence, $[S_n]/\|[S_n]\|$ converges weakly to $S$ as $n\to\infty$ and, since the $(mk,mk)$-current $\widehat{T}_F^{km}$ is the $mk$-times wedge product of a closed positive $(1,1)$-current with continuous potential (since $\widehat{T}_f$ has continuous potential),
\[\widehat{T}_F^{km}\wedge\frac{[S_n]}{\|[S_n]\|}\longrightarrow \widehat{T}_F^{km}\wedge S\]
as $n\to+\infty$. Hence the above gives
\begin{align*}
\liminf_{n\to\infty}\left(\|[S_n]\|^{-1}\cdot I_n\right)& \geq\liminf_{n\to\infty}\int\widehat{T}_F^{km}\wedge\frac{[S_n]}{\|[S_n]\|}\geq \int\widehat{T}_F^{km}\wedge S\\
& \geq  \ M\cdot \int \widehat{T}_F^{km}\wedge[X_{\lambda_0}^m\cap D_\epsilon].
\end{align*}
In particular, there exists $n_1\geq n_0$ such that for all $n\geq n_1$,
\[2\|[S_n]\|^{-1}\cdot I_n\geq M\int \widehat{T}_F^{km}\wedge[X_{\lambda_0}^m\cap D_\epsilon].\]
Finally, by construction of $S_n$, we have $\liminf_{n\to\infty}\|[S_n]\|\geq \mathrm{Vol}(D_\varepsilon)>0$, where the volume is computed with respect to the K\"ahler form $\alpha|_{X_{\lambda_0}^m}$ on $X_{\lambda_0}^m$.
Up to increasing $n_0$, we may assume $\|[S_n]\|\geq c>0$ for all $n\geq n_0$. Letting $\gamma=Mc/4>0$, we find
\[\mu(\mathbb{B})=d^{-km(np+N)}I_n\geq d^{-km(np+N)}\gamma\int\widehat{T}_F^{km}\wedge[X_{\lambda_0}^m\cap D_\epsilon],\]
for all $n\geq n_1$. To conclude, we need to prove the last integral is non-zero. 
By construction, the set $X_{\lambda_0}^m\cap D_\epsilon$ is an open neighborhood of $\mathrm{z}$ in $X_{\lambda_0}^m$ hence it contains $B(z_1,\delta)\times\cdots\times B(z_m,\delta)\subset X_{\lambda_0}^m$ for some $\delta>0$ (with a slight abuse of notations since here the balls are meant in $X_{\lambda_0}$). Moreover, the current $\widehat{T}_F^{km}$ restricts to $X_{\lambda_0}^{km}$ as the measure 
\begin{center}
$\widehat{T}_F^{km}|_{X_{\lambda_0}^m}=\mu_{F_{\lambda_0}}=\mu_{f_{\lambda_0}}^{\otimes m}$.
\end{center}
In particular, we can apply Fubini Theorem to find
\begin{align*}
\int\widehat{T}_F^{km}\wedge[X_{\lambda_0}^m\cap D_\epsilon] \geq & \int_{X_{\lambda_0}}\widehat{T}_F^{km}\wedge[B(z_1,\delta)\times\cdots\times B(z_m,\delta)]\\
&=\prod_{j=1}^m\mu_{f_{\lambda_0}}(B(z_j,\delta))>0,
\end{align*}
where we used that $z_j\in \mathrm{supp}(\mu_{f_{\lambda_0}})$ by assumption.
\end{proof}

\section{Rigidity of some stable families}\label{sec-rigidity}
\subsection{Spaces of endomorphisms, moduli spaces, stable families}\label{sec:moduli-space}
\subsubsection{The spaces $\mathrm{End}_d^k$ and $\mathrm{Poly}_d^k$}
As an endomorphism $f$ of $\Pb^k$ of degree $d$ is given by $k+1$ homogeneous polynomials of degree $d$, the coefficients of these polynomials allow us to see $f$ as a point in $\Pb^{N_d^k}$ where $N_d^k:= (k+1)\binom{k+d}{d}-1$. The condition on the coefficients to ensure that the associated map is an endomorphism of $\Pb^k$ is algebraic so there exists a Zariski open set $\mathrm{End}_d^k\subset\Pb^{N_d^k}$ corresponding to degree $d$ endomorphisms. More precisely, the variety $\mathrm{End}_d^k$ is the complement of the hypersurface in $\mathbb{P}^{N_d^k}$ defined by the vanishing
of the Macaulay resultant. Indeed, there is a unique homogeneous polynomial $\mathrm{Res}:\mathbb{P}^{N_d^k}\to\p^1$ defined over $\mathbb{Q}$ such that
\begin{itemize}
\item if $f=[F_0:\cdots:F_k]$, then $\mathrm{Res}(f):=\mathrm{Res}(F_0,\ldots,F_k)=0$ if and only if the polynomial map $(F_0,\ldots,F_k)$ is degenerate,
\item $\mathrm{Res}(z_0^d,\ldots,z_k^d)=1$.
\end{itemize}
See, e.g.,~\cite[Proposition~1.1]{BB1} for more details, see also \cite[\S~3]{szpiro}. In particular, the variety $\mathrm{End}_d^k$ is an irreducible smooth quasi-projective variety defined over $\mathbb{Q}$. 
Moreover, the map
\[f:\mathbb{P}^k_{\mathrm{End}_d^k}\longrightarrow\mathbb{P}^k_{\mathrm{End}_d^k}\]
 is a family $(\mathbb{P}^k_{\mathbb{P}^{N_d^k}},f,\mathcal{O}_{\mathbb{P}^k}(1))$ of degree $d$ endomorphisms of $\mathbb{P}^k$ parametrized by $\mathbb{P}^{N_d^k}$ -- if we follow the notations introduced above -- which is defined over $\mathbb{Q}$.

\medskip

A \emph{regular polynomial endomorphism} $f:\mathbb{A}^k\to\mathbb{A}^k$ of degree $d\geq2$ is a polynomial map which extends to a degree $d$ endomorphism $f:\mathbb{P}^k\to\mathbb{P}^k$.
 For such a morphism, if $H_\infty$ is the hyperplane at infinity of $\mathbb{A}^k$ in $\mathbb{P}^k$, we have $f^{-1}(H_\infty)=H_\infty$, see, e.g.,~\cite{bedford-jonsson}.
  The space $\mathrm{Poly}_d^k$ of regular polynomial endomorphisms of degree $d$ of $\mathbb{A}^k$ is a smooth closed subvariety of $\mathrm{End}_d^k$ of dimension $k\binom{k+d}{d}$ -- which is the intersection of $\mathrm{End}_d^k$ with a linear subspace of $\mathbb{P}^{N_d^k}$ defined over $\mathbb{Q}$. In particular, $\mathrm{Poly}_d^k$ is also a smooth quasi-projective variety defined over $\mathbb{Q}$ and the map \[f:\mathbb{P}^k_{\mathrm{Poly}_d^k}\longrightarrow\mathbb{P}^k_{\mathrm{Poly}_d^k}\]
 is a family $(\mathbb{P}^k_{S},f,\mathcal{O}_{\mathbb{P}^k}(1))$ of degree $d$ endomorphisms of $\mathbb{P}^k$ parametrized by the closure $S$ of $\mathrm{Poly}_d^k$ in $\mathbb{P}^{N_d^k}$ -- if we follow the notations introduced above -- which is defined over $\mathbb{Q}$.

\subsubsection{The moduli spaces $\mathscr{M}_d^k$ and $\mathscr{P}_d^k$ and good families}
The space which is really adapted to our investigations is the moduli space $\mathscr{M}_d^k$ of degree $d$ endomorphisms of the projective space $\mathbb{P}^k$ of dimension $k$: it is the quotient space of the space $\mathrm{End}_d^k$ of endomorphisms of degree $d$ of $\mathbb{P}^k$ by the action by conjugacy of $\mathrm{PGL}(k+1)$. It is known to be an irreducible affine variety of dimension
\[\mathcal{N}_d^k:=\dim\mathscr{M}_d^k=(k+1)\binom{k+d}{d}-(k+1)^2\] defined over $\mathbb{Q}$, see \cite{Silverman} when $k=1$ and \cite[\S~3.2 $\&$ 3.3]{szpiro} when $k>1$, hence there is a proper closed subvariety $V$ defined over $\mathbb{Q}$ such that the canonical projection
\[\Pi:\mathrm{End}_d^k\setminus V\to \mathscr{M}_d^k\setminus\Pi(V)\]
is a locally trivial $\mathrm{PGL}(k+1)$-principal bundle. As this is merely a coarse moduli space, there is no universal family. However, we can cook up a good family which can play a sufficiently similar role.

\begin{lemma}\label{good-family-end}
There is family $(\mathbb{P}^k_{S},f,\mathcal{O}_{\mathbb{P}^k}(1))$ defined over $\mathbb{Q}$, with $\dim S=\mathcal{N}_d^k$ whose maximal regular part $\mathcal{U}_d^k$ satisfies the following properties:
\begin{enumerate}
\item the set $\mathscr{U}_d^k:=\Pi(\mathcal{U}_d^k)$ contains a dense Zariski open subset of $\mathscr{M}_d^k$,
\item the map $\Pi|_{\mathcal{U}_d^k}:\mathcal{U}_d^k\to\mathscr{U}_d^k$ is finite.
\end{enumerate}
\end{lemma}
\begin{proof}
Let $D_k:=d^k+d^{k-1}+\cdots+d+1$. Let $\mathrm{End}_d^{k,\mathrm{fix}}$ be the space of all fixed marked degree $d$ endomorphisms of $\mathbb{P}^k$, i.e. the space of all couples $(f,\{x_1,\ldots,x_{D_k}\})$ where $f\in\mathrm{End}_d^k$ is an endomorphism and $\{x_1,\ldots,x_{D_k}\}$ is an unordered $D_k$-tuple of points in $\mathbb{P}^k$ which identifies with the collection of all fixed points of $f$ counted with multiplicities. This is a quasi-projective variety since it is closed in $\mathrm{End}_d^k\times \mathrm{Sym}^{D_k}(\mathbb{P}^k)$. Remark that, since an endomorphism has $D_k$ fixed points, the canonical projection $\rho:\mathrm{End}_d^{k,\mathrm{fix}}\to\mathrm{End}_d^k$ has finite fibers and is proper. Hence $\dim\mathrm{End}_d^{k,\mathrm{fix}}=N_d^k$.

Let now $U\subset \mathrm{End}_d^{k,\mathrm{fix}}$ be the Zariski open subset consisting of couples $(f,\{x_1,\ldots,x_{D_k}\})$ where $x_i\neq x_j$ for all $i\neq j$.
Let $e_1:=[1:0:\cdots:0]$, $e_2:=[0:1:0:\cdots:0]$, $\ldots$, $e_{k+1}:=[0:\cdots:0:1]$ and $e_{k+2}:=[1:\cdots:1]$. As these points do not lie in the same hyperplane in $\mathbb{P}^k$, any $\varphi\in\mathrm{PGL}(k+1)$ is uniquely determined by the values it takes on the set $\{e_1,\ldots,e_{k+1}\}$. Define $U_1\subset U$ as the subset of $U$ consisting of those couples $(f,\{x_1,\ldots,x_{D_k}\})\in U$ with $e_i\in\{x_1,\ldots,x_{D_k}\}$ for $1\leq i\leq k+2$. Note that this condition is closed in $U$ and that it is not vacuous since $D_k\geq d^k+2\geq k+2$, as $d\geq2$. The quasi-projective variety $U_1\subset\mathrm{End}_d^{k,\mathrm{fix}}$ has dimension $\mathcal{N}_d^k$  and the map $(\Pi\circ\rho)|_{U_1}:U_1\to\mathscr{M}_d^k$ is finite onto its image. Let $S$ be the Zariski closure of $\rho(U_1)$ in $\mathbb{P}^{N_d^k}\supsetneq \mathrm{End}_d^k$. The family $(\mathbb{P}^k_{S},f,\mathcal{O}_{\mathbb{P}^k}(1))$ has the expected properties.
\end{proof}

\medskip

The second family we will be interested in is the moduli space $\mathscr{P}_d^k$ of degree $d$ regular polynomial endomorphisms of the affine space $\mathbb{A}^k$: it is the quotient of the space $\mathrm{Poly}_d^k$ of regular polynomial endomorphisms of degree $d$ of the affine space $\mathbb{A}^k$ by the action by conjugacy of the group of affine transformations $\mathrm{Aut}(\mathbb{A}^k)=\mathrm{GL}(k)\ltimes \mathbb{A}^k$. 
 The same proof as those given in \cite{silverman-spacerat, szpiro,Levy} ensures that the moduli space $\mathscr{P}_d^k$ is also a coarse moduli space and is an irreducible affine variety defined over $\mathbb{Q}$ of dimension
\[\mathcal{P}_d^k:=\dim\mathscr{P}_d^k=k\binom{k+d}{d}-(k^2+k)
>\mathcal{N}_d^{k-1}=\dim\mathscr{M}_d^{k-1}.\]
As before, there is a proper closed subvariety $V$ such that the canonical projection
\[\Pi:\mathrm{Poly}_d^k\setminus V\to \mathscr{V}_d^k:=\mathscr{P}_d^k\setminus\Pi(V)\]
is a locally trivial $\mathrm{Aut}(\mathbb{A}^k)$-principal bundle. Proceeding as above, we have
\begin{lemma}\label{good-family-poly}
There is family $(\mathbb{P}^k_{S},f,\mathcal{O}_{\mathbb{P}^k}(1))$ of regular polynomial endomorphisms defined over $\mathbb{Q}$, with $\dim S=\mathcal{P}_d^k$ whose maximal regular part $\mathcal{V}_d^k$ satisfies the following properties:
\begin{enumerate}
\item the set $\mathscr{V}_d^k:=\Pi(\mathcal{V}_d^k)$ contains a dense Zariski open subset of $\mathscr{P}_d^k$,
\item the map $\Pi|_{\mathcal{V}_d^k}:\mathcal{V}_d^k\to\mathscr{V}_d^k$ is finite.
\end{enumerate}
\end{lemma}
%
\subsubsection{Stable families of endomorphisms of $\mathbb{P}^k$ following Berteloot-Bianchi-Dupont}\label{BBD}
Let $M$ be a connected complex manifold. An analytic family of endomorphisms of $\mathbb{P}^k$ parametrized by $M$ can be described as a surjective holomorphic map $f:(z,t)\in \mathbb{P}^k\times M\longmapsto(f_t(z),t)\in\mathbb{P}^k\times M$. In particular, for any $t\in M$, the induced map $f_t:\mathbb{P}^k\to\mathbb{P}^k$ is an endomorphism of degree $d$ (independent of $t\in M$). 

Following Berteloot, Bianchi and Dupont~\cite{BBD}, we say that such an analytic family of endomorphisms of $\mathbb{P}^k$ is $J_k$-\emph{stable} if the function
\[t\in M\longmapsto L(f_t):=\int_{\mathbb{P}^k(\C)}\log|\det(Df_t)|\mu_{f_t},\]
is a pluriharmonic function on $M$, i.e. $dd^c_tL(f_t)\equiv0$, where $L(f_t)$ is the sum of Lyapunov exponents of the unique maximal entropy measure $\mu_{f_t}$ of $f_t$. Berteloot-Bianchi-Dupont gave several equivalent description of this notion of stability and showed it is the higher-dimensional equivalent to the notion of stability introduced by Ma\~n\'e, Sad and Sullivan~\cite{MSS} for families of rational maps of $\mathbb{P}^1$. 

\medskip

When $M$ is a quasi-projective variety, and $f$ is a morphism (i.e. $f$ defines an algebraic family) and, if $S$ is a projective model of $M$, then $(\mathbb{P}^k\times S,f,\mathcal{O}_{\mathbb{P}^k}(1))$ is a family of endomorphisms as above with regular part $M$. In this case, one can show that the family $M$ is $J_k$-stable if and only if the function $t\mapsto L(f_t)$ is constant on $M$. The next section implies that $t\mapsto L(f_t)$ is constant on $M$ if and only if the multipliers (i.e., the eigenvalues associated to periodic points) are constant on $S$. By \cite{BB1}, one also has
\[dd^cL=\pi_*\left(\widehat{T}_f^k\wedge[\mathrm{Crit}(f)]\right)=T_{f,\mathrm{Crit}},\]
as currents on $M$, so that the family is $J_k$-stable if and only if $T_{f,\mathrm{Crit}(f)}=0$ on $M$. Here $\mathrm{Crit}(f)=\{(z,t)\in \mathbb{P}^k\times M, \ \det(Df_t)(z)=0\}$.

\medskip

In the proof of Theorem \ref{tm:mu-interior}, we will make crucial use of the fact that the multipliers  are generically finite-to-one on $\mathscr{M}_d^k$. This is established in Corollary \ref{cor-finite}, whose key step is the observation that, if this were not the case, then $\mathscr{M}_d^k$ would be covered by positive-dimensional algebraic families on which all multipliers are constant. Such families must be stable in the sense of Berteloot-Bianchi-Dupont, and we rule out this possibility by exhibiting rigid Latt\`es maps in $\mathscr{M}_d^k$ (see Lemma \ref{le-rigid}). The case of $\mathscr{P}_d^2$ is addressed separately in Section \ref{sec_c2}.

Note also that for Theorem \ref{tm:mu-interior}, we actually require a slightly stronger statement: instead of using the multipliers of all periodic points, it suffices to consider almost all of them. This motivates the introduction of the sequence of periodic points $(x_n)_{n\geq1}$ below.

\subsection{Families with many constant multipliers}\label{sec-finite}
Let $d\geq2$ and $S$ be an irreducible complex projective variety. Let $(\mathbb{P}^k_S,f,\mathcal{O}_{\mathbb{P}^k}(1))$ be a family of endomorphisms of $\Pb^k$ of degree $d$, with regular part $S^0\subseteq S$. Let $t_0\in S^0$ be an arbitrary parameter in this family. We consider a non-decreasing sequence $(m_n)_{n\geq1}$ of positive integers and a sequence of distinct points $(x_n)_{n\geq1}$ in $\Pb^k$ such that
\begin{itemize}
\item for each $n\geq1$, $x_n$ is a repelling periodic point for $f_{t_0}$ of exact period $m_n$,
\item if for $s\geq1$ we set $M_s:=\#\{n\geq1\,;\,m_n| s\}$ then $M_s/d^{sk}$ converges to $1$ when $s$ goes to $\infty$.
\end{itemize}
In other words, the last point says that most of the periodic points of $f_{t_0}$ are in $(x_n)_{n\geq1}$. Note that the existence of such a sequence follows from the equidistribution theorem of Briend-Duval \cite{briendduval}, simply by listing all the periodic points obtained in their construction.

From these data, for each $n\geq1$ we consider the analytic set
\[\tilde{X_n}:=\left\{(t,z_1,\ldots,z_n)\in S^0\times(\mathbb{P}^k)^n\,; \ f_t^{m_s}(z_s)=z_s\ \text{for all} \ 1\leq s\leq n\right\}.\]
Observe that, since the points in the sequence $(x_n)_{n\geq1}$ are repelling, the point $(t_0,x_1,\ldots,x_n)$ is regular in $\tilde X_n$ and we denote by $X_n$ the irreducible component of $\tilde X_n$ which contains it. The natural projection $\pi_n\colon X_n\to S^0$ is surjective and finite. We also have a family of \emph{multiplier maps} $\Lambda_n\colon X_n\to\Cb^n$ defined by $\Lambda_n(t,z_1,\ldots,z_n)=(\det D_{z_s}f_t^{m_s})_{1\leq s\leq n}$.
\begin{proposition}\label{prop-finite}
Assume that there exists $t_1\in S^0$ such that there is no algebraic curve $Z\subset S^0$ passing through $t_1$ such that $Z$ is $J_k$-stable.
Then for $n\geq1$ large enough the multiplier map $\Lambda_n$ is generically finite-to-one.
\end{proposition}


\begin{proof}
Observe first that, the maps $\Lambda_n$ contain more and more information, if the result holds for one $n_0\geq1$ then it is also the case for all $n\geq n_0.$
Assume by contradiction that for each $n\geq1$ the map $\Lambda_n$ is not generically finite. In particular, for each $n\geq1$ the set $Y_n=\Lambda_n^{-1}(\Lambda_n(t_1))$ has positive dimension. The sequence of algebraic set $(Z_n)_{n\geq1}$ defined by $Z_n:=\pi_n(Y_n)$ is decreasing so there exists $N\geq1$ such that $Z_n=Z_N$ for all $n\geq N$. From this, the key observation is that, relying on the equidistribution of repelling orbits \cite{briendduval}, we have by \cite[Theorem 1.5]{BDM}  (see also \cite[Theorem 4.1]{BD_distortion}) for all $t\in S^0$, 
	\[ \lim_{n \to + \infty}  {1 \over d^{kn}} \sum_{p \in \mathrm{RPer}_n(f_t)}  \log  |\det(D {f_t})(p)|  = L(f_t),\]
	where $\mathrm{RPer}_n(f_t)$ is the set  of $n$-periodic repelling points of $f_t$ and $L(f_t)$ the sum of the Lyapunov exponents of its equilibrium measure. This implies by the chain rules that
\[ \lim_{n \to + \infty}  {1 \over n d^{kn}} \sum_{p \in \mathrm{RPer}_n(f_t)}  \log  |\det(D {f_t^n})(p)|  = L(f_t),\]
or equivalently
\[ \lim_{n \to + \infty}  {1 \over n d^{kn}} \sum_{p \in \mathrm{Per}_n(f_t)}  \log^+  |\det(D {f_t^n})(p)|  = L(f_t),\]
where $\log^+x =\max (\log x ,0)$ and $\mathrm{Per}_n(f_t)$ is the set of all $n$-periodic points of $f_t$. In particular, as we have assume that $M_s/d^{sk}\to 1$ with $s$ where $M_s:=\# \{n\geq1\,;\,m_n| s\}$, the fact that all the functions $\Lambda_n$ are constant on $Y_n$ implies that $t\mapsto L(f_t)$ is also constant on $Z_N$. In particular $Z_N$ is a $J_k$-stable family containing $t_1$. Contradiction.
\end{proof}
Using rigid Latt\`es maps, we have the following result which answers by the positive to the first part of \cite[Question 19.4]{doyle-silverman}. A description of the set $\Gamma$ in the next corollary, however, remains a much more difficult question. Note that in dimension 1, much sharper results are known. On one hand, when $(x_n)_{n\geq1}$ corresponds to all periodic cycles, McMullen \cite{McMullen-families} proved this corollary with the optimal $\Gamma$, i.e., the set of flexible Latt\`es maps. A stronger statement using only the modulus of the multipliers has been recently given in \cite{JX}. On the other hand, Gorbovickis \cite{gorbovickis-rat} obtained Corollary \ref{cor-finite} when $k=1$ using only the multipliers of an (almost) arbitrary set of $2d-2$ cycles.

\medskip

Consider now the family $(\Pb^k_S,f,\mathcal{O}_{\Pb^k}(1))$ with regular part $\mathcal{U}_d^k$ given by Lemma~\ref{good-family-end}. For this family, the map $\Lambda_n$ is defined on the corresponding variety $X_n$ as above. Recall that the canonical projection $\pi_n:X_n\to \mathcal{U}_d^k$ is surjective and finite.


\begin{corollary}\label{cor-finite}
Let $d\geq2$ and $k\geq1$. If $(m_n)_{n\geq1}$, $(x_n)_{n\geq1}$ and $(\Lambda_n)_{n\geq1}$ are as above with $X=\mathcal{U}_d^k$, then there exist $N\geq1$ and a Zariski closed proper subset $\Gamma$ of $\mathcal{U}_d^k$ such that $\Lambda_n$ is finite-to-one on $X_n\setminus \pi_n^{-1}(\Gamma)$ for all $n\geq N$.
\end{corollary}
\begin{proof}
We simply apply Proposition \ref{prop-finite} with $f_{t_1}$ equal to a rigid Latt\`es map in the family given by Lemma~\ref{good-family-end}. By Berteloot and Dupont~\cite{BertelootDupont}, Latt\`es maps are the only minimum of the Lyapunov function $L$ so if $f_{t_1}$ is a rigid Latt\`es map, there is no stable family in $\mathcal{U}_d^k$ containing it. The next two lemmas conclude the proof.
\end{proof}

Our proof of this result follows Berteloot-Loeb \cite{BL}, although there might be a quicker argument using properties of Abelian varieties of CM-type. The existence of rigid Latt\`es maps (i.e., not contained in a holomorphic family of such maps) arise from the next lemma together with Lemma \ref{le-it}, which removes the technical difficulty related to the iterate $f^p_t$.

\begin{lemma}\label{le-rigid}
Let $(f_t)_{t\in\Db}$ be a holomorphic family of Latt\`es maps of $\Pb^k$ such that $f_0$ is the symmetric product of a rigid Latt\`es map of $\Pb^1$. Then, there exist $p\geq1$ and a continuous deformation $(\psi_t)_{t\in\Db}$ of $\id$ in $\mathrm{Aut}(\Pb^k)$ such that, for all $t\in\Db$, $f^p_t$ is conjugate to $f^p_0$ by $\psi_t.$
\end{lemma}
\begin{proof}
Let $g_0$ the rigid Latt\`es map of $\Pb^1$ such that $f_0$ is the symmetric product of $g_0.$ There exist an elliptic curve $E_0$, a finite branched cover $\rho\colon E_0\to\Pb^1$ and a complex number $a$ such that $\rho$ semi-conjugates $g_0$ to the multiplication by $a.$ As $g_0$ is rigid, $a$ has to be an imaginary quadratic integer (see \cite{Milnor}).

Let $\pi\colon(\Pb^1)^k\to\Pb^k$ be the symmetrization map, i.e. the quotient map for the action by permutation of coordinates of the symmetric group $\mathfrak{S}_k$, and $m\colon E_0^k\to E_0^k$ be the multiplication by $a.$ As the periodic points of $m$ are dense, we can choose a periodic point $x_0,$ of period  denoted by $p$, such that $z_0:=\pi\circ\rho^{(k)}(x_0)$ is not in the ramification values of $\pi\circ\rho^{(k)}$ and such that $a^p$ is still not real (i.e., an imaginary quadratic integer). In particular, $f_0^p(z_0)=z_0$ and, using the map $\pi\circ\rho^{(k)},$ we obtain that  $D_{z_0}f_0^p=a^p\id$ and that the Green current of $f_0$ is smooth and strictly positive in a neighborhood of $z_0.$ From now on, we replace the family $(f_t)_{t\in\Db}$ by $(f_t^p)_{t\in\Db}$ and assume that $p=1.$

A family of Latt\`es maps has constant sum of Lyapunov exponents so it is stable in the sense of \cite{BBD}. Hence, $z_0$ can be followed as a repelling point $z_t$ which stay outside of the postcritical set. Therefore, by a result of Berteloot--Loeb \cite[Proposition 4.2]{BL}, for each $t\in \Db$ there exists a holomorphic map $\phi_t\colon \Cb^k \to \Pb^k,$ locally injective at $0,$ and a linear self-map $D_t\colon \Cb^k \to \Cb^k$ of the form $D_t(x)=\sqrt d\, U_t(x)$, where $U_t$ is a diagonal unitary matrix, such that $\phi_t(0)=z_t$ and
$$
\phi_t \circ D_t = f_t \circ \phi_t.
$$
As explained in the proof of \cite{BL}, the map $\phi_t$ corresponds to the Poincar\'e linearization map, precomposed with a linear change of coordinates, in an  basis of eigenvectors of $D_{x_t}f_t.$ Both can be chosen to depend holomorphically on $t$. Moreover, the eigenvalues are of modulus $\sqrt d$ and depends holomorphically on $t,$ thus they are constant, i.e., the map $D_t$ is independent of $t.$ As for $t=0$ the map $D_0$ corresponds to the multiplication by an imaginary quadratic integer $a \in \Cb \setminus \Rb,$ this holds for all $t \in \Db.$

On the other hand, Berteloot--Loeb also proved in \cite{BL} that
\[
G_t = \{ (U,b) \in \mathbb{U}_k \ltimes \mathbb{C}^k \;\; ;\;\; \phi_t(U \cdot z + b) = \phi_t(z) \}
\]
is a crystallographic group of $\Cb^k,$ for each $t \in \Db.$ By Bieberbach's theorem, the translation part $L_t$ is a lattice, and thus $A_t = \Cb^k / L_t$ is a complex torus. Moreover, $aL_t \subset L_t,$ and multiplication by $a$ on $A_t$ induces the endomorphism above $f_t.$ Observe that the injectivity of $\phi_t$ in a neighborhood (locally uniform in $t$) of $0$ prevents collisions when following the elements of $G_t$. Hence, they can be followed holomorphically for $t \in \Db$.

In suitable coordinates of $\Cb^k,$ after a change of variables given by Siegel normal form, there exists $\tau(t)$ in the Siegel upper half-space such that
\[
L_t = \Zb^k + \tau(t) \Zb^k.
\]
Since $a$ is an imaginary quadratic integer and $aL_t \subset L_t$ for all $t \in \Db,$ the continuity of $\tau(t)$ implies that it must be constant.
Hence, there exists $P_t\in\mathrm{GL}_k(\Cb)$, depending holomorphically on $t$ since $L_t$ does, such that $P_t(L_t) = L_0$.

Furthermore, $P_tG_tP_t^{-1} / L_0$ is a finite subgroup of $\Aut(A_0,0)$, which is a discrete group, thus it must also be independent of $t,$ always equal to $H_0:=G_0/L_0.$ To summarize, we obtain the following commutative diagram:
\[
\begin{tikzcd}[column sep=huge, row sep=large]
{\mathbb{C}^k}
 \arrow[rr,"\phi_t"]
  \arrow[d, shift left=.6ex, "P_t"]
&& \mathbb{P}^{k}
\\
{\mathbb{C}^k}
 \arrow[rr,bend right,"\phi_0"] \arrow[r,]
& A_{0}:={\mathbb{C}^k/L_0}
  \arrow[r,]
& A_{0}/H_0=\mathbb{P}^{k}
  \arrow[u, swap, "\psi_t"]
\end{tikzcd}
\]
where the map $\psi_t$ comes from the fact that $\phi_t \circ P_t^{-1}$ passes to the quotient. As the group $G_t$ acts transitively on the fibers of $\phi_t$ (see \cite[Proposition 5.1]{BL}), this map $\psi_t$ is an automorphism. In particular, if $z\in\Pb^k$ and $x\in\Cb^k$ are such that $\phi_t(x)=z$ then, on one hand,
$$f_t(z)=\phi_t(ax)=\psi_t\circ\phi_0\circ P_t(ax)=\psi_t\circ\phi_0(aP_t(x))=\psi_t\circ f_0(\phi_0(P_t(x))),$$
and, on the other hand,
$$\phi_0(P_t(x))=\psi_t^{-1}(z).$$
This gives $f_t=\psi_t\circ f_0\circ\psi_t^{-1}.$
\end{proof}

\begin{lemma}\label{le-it}
Let $(f_t)_{t\in\Db}$ be a holomorphic family in $\mathrm{End}_d^k$.  
Assume there exist $p \geq 1$ and a continuous deformation $(\psi_t)_{t\in\Db}$ of $\id$ in $\mathrm{Aut}(\Pb^k)$ such that, for all $t \in \Db$, the iterate $f_t^p$ is conjugate to $f_0^p$ by $\psi_t$.  
Then $f_t = \psi_t \circ f_0 \circ \psi_t^{-1}$ for all $t \in \Db$.
\end{lemma}
\begin{proof}
For each $t \in \Db$, define $g_t := \psi_t^{-1} \circ f_t \circ \psi_t,$ so that $g_t^p = f_0^p$. This equality implies that $g_t$ and $f_0$ have the same set of periodic points. In particular, if $x$ is a periodic point of $f_0$, then $y_t := g_t(x)$ is again a periodic point of $f_0$. Since this set is discrete and $y_t$ depends continuously on $t$, we must have $y_t \equiv y_0 = f_0(x).$ Thus $f_0 = g_t$ on the set of periodic points, which is Zariski dense; hence $f_0 = g_t$ on $\Pb^k$ for all  $t \in \Db.$
\end{proof}

\begin{proposition}[From rigid to isolated Latt\`es maps]\label{prop-lattes-isole}
For all $k \geq 2$ and $d \geq 2$, there exists a Latt\`es map whose class in $\mathscr{M}_d^k$ is isolated among all classes of Latt\`es maps.
\end{proposition}
\begin{proof}
Let $[f_0]$ be the class of the symmetric product of a rigid Latt\`es map on $\Pb^1$. By Lemma \ref{le-rigid}, no stable algebraic family passes through $[f_0]$. Hence, by the proof of Corollary \ref{cor-finite}, the multiplier map $\Lambda_n$ is finite-to-one above a neighborhood of $[f_0]$ for sufficiently large $n$. It follows that no stable algebraic family intersects this neighborhood. Therefore, for a sufficiently small compact neighborhood $B$ of $[f_0]$, the set $K$ of Latt\`es classes in $B$ is compact (since the Lyapunov function $L$ is continuous) and disjoint from any stable algebraic family.

Suppose, for contradiction, that $K$ contains no isolated class. As $K$ is compact, it must be uncountable. Since there are only countably many PCF relations of the form ${f^p(\crit_f)=f^q(\crit_f)}$, one such relation must contain infinitely many classes in $K$. Consequently, a positive-dimensional component of classes satisfying this relation must intersect $B$. This gives a contradiction since such a component is a stable algebraic family.
\end{proof}

\begin{proposition}[Isolated Latt\`es maps imply non-vanishing of the bifurcation measure]\label{Joe'suggestion} The bifurcation measure $\mu_{f,\mathrm{Crit}}$ is non-zero on $\mathscr{M}_{d}^{k}(\mathbb{C})$.
\end{proposition}
\begin{proof} Recall, from §~\ref{BBD} that $T_{f,\mathrm{Crit}}=dd^cL$ so by Proposition~\ref{prop-higher-currents}, $\mu_{f,\mathrm{Crit}} \geq (dd^cL)^{\wedge\mathcal{N}_d^k}$. Since at an isolated Latt\`es map, the Lyapunov function $L$ admits a strict minimum, its Monge Amp\`ere $(dd^cL)^{\wedge\mathcal{N}_d^k}$ does not vanish (see \cite{BedfordTaylorbis}).
\end{proof}
The argument goes back to Bassanelli and Berteloot (see \cite[Proposition 6.3]{BB1} when $k=1$). To the best of our knowledge, no such simple arguments hold in the polynomial case and we only know that $\mu_{f,\mathrm{Crit}}\neq0$ on $\mathscr P_d^k$ when $k=2$ thanks to Theorem \ref{tm:mu-interior}.


\subsection{Families of regular polynomial endomorphisms of the affine plane}\label{sec_c2}
In order to apply Proposition \ref{prop-finite} to the moduli space $\mathscr P_d^2$, we prove the following rigidity result.
\begin{theorem}\label{th-rigidity}
Let $d\geq2$. Let $g$ be a rational map of $\Pb^1$ of degree $d$ which is not a flexible Latt\`es map. Moreover, assume that
\begin{itemize}
\item[\textit{i)}] $g$ possesses at least $3$ postcritical repelling periodic points (possibly in the same critical orbit),
\item[\textit{ii)}] for one of these postcritical repelling periodic points $y_1$ we have that
\[\{a\in\Pb^1\, ; \text{ there exists} \ n\geq1,\ g^n(a)=y_1\ \text{ and }\ a\ \text{ is not in the critical set of }\ g^n\}\]
is dense in the Julia set $J_g$ of $g$.
\end{itemize}
Let $(\mathbb{P}^2\times Z,f,\mathcal{O}_{\mathbb{P}^2}(1))$ be a stable family of regular polynomial endomorphisms of $\Cb^2$ of degree $d$, parametrized by an irreducible algebraic curve $Z$. If there exists $\lambda_0\in Z$ such that $f_{\lambda_0}$ is equal to the lift of $g$ to $\Cb^2$ then the family $(\mathbb{P}^2\times Z,f,\mathcal{O}_{\mathbb{P}^2}(1))$ is isotrivial.
\end{theorem}
\begin{remark*}\normalfont
Observe that to find such a rational map $g$, it suffices to take a polynomial map with a postcritical repelling point of period $5$.

To the best of our knowledge, this is the first rigidity result in higher dimensions that is not a direct consequence of one-dimensional results.
\end{remark*}

\begin{proof}
The plan of the proof is first to show that the family $(f_\lambda|_{L_\infty})_{\lambda\in Z}$ is constant - up to conjugacy - using McMullen's rigidity theorem~\cite{McMullen-families}. In a second time, we show that for all $\lambda\in Z$, the map $f_\lambda$ is a lift of the rational map $f_\lambda|_{L_\infty}$.

\medskip

We will use several results of Bedford-Jonsson on regular polynomial endomorphisms of $\Cb^2$ obtained in \cite{bedford-jonsson}.

Let $\lambda$ be in $Z$. If $\crit_{f_\lambda}$ denotes the critical set of $f_\lambda$ in $\Pb^2$, we set $C_{\lambda}:=\overline{\crit_{f_\lambda}\setminus L_\infty}$, where $L_\infty$ is the line at infinity in $\Pb^2$. The critical measure of $f_\lambda$ is $\mu_{c,\lambda}:=T_\lambda\wedge [C_\lambda]$, where $T_\lambda$ is the Green current of $f_\lambda$. In $\Cb^2$, $T_\lambda=\widehat{T}|_{\mathbb{P}^2\times\{\lambda\}}$ is equal to the $\ddc$ of the Green function $G_\lambda$ of $f_\lambda$, which is non-negative on $\Cb^2$ and positive precisely outside the set $K_\lambda$ of points of $\Cb^2$ with bounded orbit. Bedford-Jonsson proved in particular that the sum $L(\lambda)$ of the Lyapunov exponents of the equilibrium measure $\mu_\lambda:=T_\lambda\wedge T_\lambda$ verifies
\[L(\lambda)=\log d+\ell(\lambda)+\int G_\lambda\mu_{c,\lambda},\]
where $\ell(\lambda)$ is the Lyapunov exponent associated to $f_{\lambda|L_\infty}.$

As the family is stable, $\lambda\mapsto L(\lambda)$ is a harmonic function on $Z$. Since it is positive, it must be constant. Actually, both maps $\lambda\mapsto\ell(\lambda)$ and $\tilde\ell\colon\lambda\mapsto \int G_\lambda\mu_{c,\lambda}$ are also constant. To see this, first observe that $\ell$ is subharmonic, since it is the Lyapunov exponent of the family $(f_{\lambda|L_\infty})_{\lambda \in Z}$. Moreover, by definition, $\tilde\ell$ is non-negative, and the constancy of $L$ implies that $\ell$ is bounded from above. It follows that $\ell$, which extends as a subharmonic function to an algebraic compactification of $Z$, must be constant.

This has two consequences. First, at the parameter $\lambda_0$ the map $f_{\lambda_0}$ is the lift of $g$ to $\Pb^2$ so $L(\lambda_0)=\log d+\ell(\lambda_0)$, i.e. $\tilde\ell(\lambda_0):=\int G_{\lambda_0}\mu_{c,\lambda_0}=0$. Hence, $\tilde\ell(\lambda)=0$ for all $\lambda\in Z$. In other words, the critical measure $\mu_{c,\lambda}$ is supported in $K_\lambda$. On the other hand, $(f_{\lambda|L_\infty})_{\lambda\in Z}$ is an algebraic stable family of rational maps on $\Pb^1$ so by \cite{McMullen-families} it must be isotrivial since $f_{\lambda_0|L_\infty}$ is not a flexible Latt\`es map. Up to a finite cover of $Z$, we can then perform a family of affine conjugations in order to have for each $\lambda\in Z$
\begin{itemize}
\item $f_{\lambda|L_\infty}=g,$
\item $0\in\Cb^2$ is a fixed point of $f_\lambda$, which is the continuation of the center of the pencil of curves preserved by $f_{\lambda_0}.$
\end{itemize}
Note that there remains one degree of freedom, corresponding to the homothety of center $0$, which will be used later.

We denote by $X'$ the set of preperiodic critical points of $g$ and by $Y'$ its set of postcritical periodic points. The subset $Y\subset Y'$ corresponds to repelling postcritical periodic points and $X\subset X'$ to points eventually mapped into $Y$. We also choose two integers $N\geq1$ and $m\geq1$ such that $g^N(X')\subset Y'$ and $g^m(y)=y$ for each $y\in Y'$. Observe that by i) the set $Y$ has at least $3$ points. Based on this, the proof proceeds in four main steps. Notice that in what follows, we identify $\Pb^1$ with $L_\infty$. 
\begin{itemize}
\item[\textbf{(1)}] For each $\lambda$, each irreducible component of the critical set of $f_\lambda$ containing a point of $X$ has to be preperiodic.
\item[\textbf{(2)}] The periodic irreducible components of the postcritical set of $f_\lambda$ passing through points of $Y$ are lines containing $0.$
 In other words, there exists a set of at least $3$ lines $\mathcal L=\{L_y\, ;\, y\in Y\}$ where each $L_y$ is $f_\lambda$-periodic for all $\lambda\in Z.$
\item[\textbf{(3)}] The pencil of lines $\mathcal P$ passing through $0$ has to be preserved by each $f_\lambda.$
\item[\textbf{(4)}] Up to homothety, there is a unique regular polynomial endomorphism of $\Cb^2$ preserving $\mathcal P$ acting as $g$ on $L_\infty.$
\end{itemize}
Let us now prove these four claims. Note that the delicate one is \textbf{(3)}, and that our proof is strongly inspired by \cite{McMullen-families}, where the difficulties coming from \emph{unlabelled} holomorphic motion are highlighted. In our very special situation, we use the lamination coming from \cite[Theorem 8.8]{bedford-jonsson} to overcome possible monodromy problems.

\medskip

\noindent\textbf{Proof of (1).} Let $x$ be in $X$, i.e, a critical point of $g$ whose image under $g^N$ is a repelling $m$-periodic point $y$. Let $\lambda$ be in $Z$ and let $C$ be an irreducible component of $C_\lambda$ passing through $x$. The point $y$ must be of saddle type for $f_\lambda$, repelling in the direction of $L_\infty$ and super-attracting in the transverse direction. In particular, it admits a local stable manifold $W_{y,loc}^s$. On the other hand, $\mu_{c,\lambda}=T_\lambda\wedge[C_\lambda]$ vanishes near $x$ so $(f^n_{\lambda|C})_{n\geq0}$ is a normal family near $x$ (see e.g. \cite[Theorem 1]{ueda}). The saddle nature of $y$ gives that the only possible limit value of $(f^{N+nm)}_{\lambda|C})_{n\geq0}$ near $x$ is the constant function equal to $y$. Indeed, if $v$ is such a limit value, then $L_\infty$ contains its image, $v(x) = y$, and the sequence $(f_\lambda^{nm} \circ v)_{n \geq 0}$ is also normal. Since $f^m_{\lambda|L_\infty}$ is repelling at $y$, the map $v$ must be constant. Therefore, $(f^{N+nm)}_{\lambda|C})_{n\geq0}$ converges to the constant function equal to $y$ on a neighborhood $V$ of $x$ in $C$. This implies that $f^N_\lambda(V)\cap W^s_{y,loc}$ is a neighborhood of $y$ in $W^s_{y,loc}$ and thus, $f_\lambda^N(C)$ is $m$-periodic.

\medskip

\noindent\textbf{Proof of (2).}  Let $\mathcal L'=\{L_y\,;\, y\in Y'\}$ be the periodic postcritical lines for $f_{\lambda_0}$ in the pencil $\mathcal P$. Observe that $f_{\lambda_0|L_y}^m$ is conjugate to $z^{d^m}$ with a Julia set $S_y$ which is uniformly hyperbolic.

Let $\lambda$ be sufficiently close to $\lambda_0$. Let $y \in Y$ and let $D$ be an irreducible component of the postcritical set of $f_\lambda$ that contains $y$. As previously noted, $D$ is $m$-periodic and locally coincides with the stable manifold of $y$ near this point. In particular, $D$ intersects transversely $L_\infty$ at $y$. Moreover, by \cite[Lemma 6.2]{dinhsibony}, $f_{\lambda|D}^m$ has topological degree $d^m$ and thus, using a normalization, it has a unique measure of maximal entropy $m\log d$.

On the other hand, since $\lambda$ is close to $\lambda_0$, the curve $D$ is close to a union $\cup_{y' \in I} L_{y'}$, with $I \subset Y'$. For each $y'\in I$, the holomorphic motion of $S_{y'}$ gives a hyperbolic set of entropy $m\log d$ for $f_{\lambda|D}^m$ close to $S_{y'}.$ Since these sets are pairwise disjoint, the uniqueness of the measure of maximal entropy implies that $I$ consists of a single point and thus $I=\{y\}$. In particular, the intersection points of $D$ with $L_\infty$ are postcritical periodic points close to $y$. As this set is discrete in $L_\infty$ and independent of $\lambda$, $D\cap L_\infty$ is reduced to $\{y\}$ for $\lambda$ sufficiently close to $\lambda_0$. Furthermore, we have seen that this intersection is transverse. Thus, $D$ has degree $1$.

Finally, as $\lambda$ is close to $\lambda_0$, $0$ is attracting and $D$ intersects its basin of attraction. By invariance of $D$, $0$ is in $D$. 

This proves the result in a Euclidean neighborhood of $\lambda_0$. Since it is a closed property in the Zariski topology, it holds for all $\lambda\in Z$.

\medskip

\noindent\textbf{Proof of (3).} Let $L_1$ denote the line in $\mathcal P$ that contains $y_1\in Y$, the postcritical periodic point of $g$ given by \textit{ii)}. As we have seen, $L_1$ is $m$-periodic for every $f_\lambda$ with $\lambda \in Z$. Since the Julia set of $f^m_{\lambda_0|L_1}$ is contained in the small Julia set of $f_{\lambda_0}$, by \cite[Theorem 1.1]{BBD} the family $(f^m_{\lambda|L_1})_{\lambda\in Z}$ is stable. Indeed, in dynamics in dimension $2$ $J$-stability is equivalent to the fact that $J$-repelling periodic points move holomorphically and remain repelling. In particular, all repelling periodic points of $f^m_{\lambda|L_1}$ move holomorphically and remain repelling. By \cite{McMullen-families}, it has to be isotrivial. Hence, each $f^m_{\lambda|L_1}$ is holomorphically conjugate to $f^m_{\lambda_0|L_1}$, i.e. to $z\mapsto z^{d^m}$. In particular, up to a finite cover of $Z$ and using a family of homotheties of $\Cb^2$, we can assume that $f^m_{\lambda|L_1}$ is independent of $\lambda$.

Fix $\lambda\in Z$ for a moment and denote by $A_\lambda$ the basin of $L_\infty$, i.e. $A_\lambda:=\Pb^2\setminus K_\lambda$. As $\int G_\lambda\mu_{c,\lambda}=0$, by \cite[Theorem 8.8]{bedford-jonsson}, there exists a $f_\lambda$-invariant lamination by holomorphic discs $\{W_{a,\lambda}\,|\, a\in J_g\}$ in $A_\lambda$ parametrized by the Julia set $J_g$ of $g$, such that $W_{a,\lambda}\setminus\{a\}$ is either contained in the critical set of $f_\lambda$ or disjoint from it. Moreover, $W_{a,\lambda}$ is contained in the stable manifold of $a$ for a generic $a\in J_g$. Here, genericity is with respect to the equilibrium measure of $f_{\lambda|L_\infty}=g$. For the point $y_1$ defined above, $W_{y_1,\lambda}$ corresponds to the basin of attraction of $y_1$ for $f^m_{\lambda|L_1}$. As we have seen, this set is independent of $\lambda$ and we denote it by $W_{y_1}$.
 
Let $n\geq1$ and let $a_n\in g^{-n}(y_1)$ be such that $a_n$ does not lie in the critical set of $g^n$ and $a_n\notin Y$. Observe that the set $W_{a_n,\lambda}$ satisfies $f_\lambda^n(W_{a_n,\lambda})=W_{y_1}$ (and thus is contained in the algebraic set $f_\lambda^{-n}(L_1)$) and that $W_{a_n,\lambda}\setminus\{a_n\}$ is disjoint from the critical set of $f^n_\lambda$. Let $w\in W_{y_1}\setminus\{y_1\}$. The set $P_w:=\{z\in W_{a_n,\lambda_0}\,|\, f^n_{\lambda_0}(z)=w\}$ has exactly $d^n$ points. Moreover, if $\gamma$ is a loop in $Z$ and $z\in P_w$ then the fact that $W_{a_n,\lambda}\setminus\{a_n\}$ is disjoint from the critical set of $f^n_\lambda$ ensures that we can follow $z$ along $\gamma$ as a point in $f_\lambda^{-n}(w)\cap W_{a_n,\lambda}$. This gives an action of $\pi_1(Z,\lambda_0)$ on $P_w$ by permutations, whose kernel $H_n$ has finite index ($\leq d^n!$) in $\pi_1(Z,\lambda_0).$ On the finite branched cover $Z_n$ associated to $H_n$, the points in $P_w$ can be followed holomorphically, i.e., there exists a family $(\phi_z)_{z\in P_w}$ of holomorphic maps from $Z_n$ to $\Pb^2$ such that $\phi_z(\lambda_0)=z$ and $f^n_\lambda(\phi_z(\lambda))=w$ for all $z\in P_w$ and $\lambda\in Z_n$.

\smallskip

From this, there are two key observations. First, $W_{a_n,\lambda}$ is disjoint from $L_y$ for $y\in Y$ since $a_n\notin Y$. Hence, if $\pi\colon\Pb^2\setminus\{0\}\to L_\infty$ denotes the linear projection, $\psi_z:=\pi\circ\phi_z$ defines maps from $Z_n$ to $\Pb^1\setminus Y$. The other important observation is that the kernel $H_n$ defined above is independent of the choice of $w\in W_{y_1}\setminus\{y_1\}$. Actually, if $V$ is a small neighborhood of $w$ then the set $P_{w'}$ can be followed holomorphically for $w'\in V$ and the action of $\pi_1(Z,\lambda_0)$ is compatible with this motion. Hence, the kernel is the same for all $w'\in V$ and thus, by connectedness, for all $w'\in W_{y_1}\setminus\{y_1\}.$ Thus, for each $z\in W_{a_n,\lambda_0}$ we can associate $\phi_z\colon Z_n\to\Pb^2\setminus\{L_y\,;\, y\in Y\}$ and $\psi_z:=\pi\circ\phi_z\colon Z_n\to\Pb^1\setminus Y$. As the set of non-constant holomorphic maps from $Z_n$ to $\Pb^1\setminus\{y_1,\ldots,y_k\}$ is finite and since $z\mapsto\psi_z(\lambda)$ is continuous for each $\lambda\in Z_n$, the maps $\psi_z$ are either all constant or all equal. In both cases, the fact that $\phi_z(\lambda)$ converges to $a_n$ when $z\to a_n$ implies that each $\phi_z$ is identically equal to $a_n$. In other words, $W_{a_n,\lambda}$ is contained in $L_{a_n}:=\overline{\pi^{-1}(a_n)}$. Since the set of all possible $a_n$ for all possible $n\geq1$ is dense in $J_g$, each map $f_\lambda$ satisfies $\pi\circ f_\lambda=g\circ\pi$ on $\pi^{-1}(J_g)$. This set is not pluripolar in $\Pb^2\setminus\{0\}$ so $\pi\circ f_\lambda=g\circ\pi$ on $\Pb^2\setminus\{0\}$, i.e. $f_\lambda$ must preserve the pencil of lines $\mathcal P$ defined by $\pi$.

\medskip

\noindent\textbf{Proof of (4).} Let $\lambda\in Z$. Since $f_\lambda$ preserves $\mathcal P$, it must be of the form
\[f_\lambda[x:y:z]=[P(x,y):Q(x,y):R_\lambda(x,y,z)],\]
where $g[x:y]=[P(x,y):Q(x,y)]$. But $f_\lambda$ is also a regular polynomial endomorphism of $\Cb^2$ so $R_\lambda(x,y,z)=c_\lambda z^d$.
\end{proof}

In particular, Corollary \ref{cor-finite} also holds on the good family $(\Pb^2_S,f,\mathcal{O}_{\Pb^2}(1))$ with regular part $\mathcal{V}_d^2$ given by Lemma~\ref{good-family-poly}:
\begin{corollary}\label{cor-finite-poly}
Let $d\geq2$. Let $(m_n)_{n\geq1}$, $(x_n)_{n\geq1}$ and $(\Lambda_n)_{n\geq1}$ be as in Section \ref{sec-finite} with $X=\mathcal{V}_d^2$. Then, there exist $N\geq1$ and a Zariski closed proper subset $\Gamma$ of $\mathcal{V}_d^2$ such that $\Lambda_n$ is finite-to-one on $X_n\setminus\pi_n^{-1}(\Gamma)$ for all $n\geq N$.
\end{corollary}
\begin{proof}
The proof is the same than Corollary \ref{cor-finite} except that for the map $f_{t_1}$ we take the lift $f_{\lambda_0}$ from Theorem \ref{th-rigidity}. As mentioned after this theorem, it suffices to take for $g$ a polynomial map with a postcritical repelling point of period $5$.
\end{proof}

\section{Blenders and the bifurcation measures}\label{sec-blender}
Our goal here is to prove that an open set $\Omega$ of $\mathscr{M}_d^k$ or $\mathscr{P}_d^2$, satisfying a large set of assumptions (see \S~\ref{sec-ass}), must be contained in the support of the bifurcation measure. More precisely,  if this is not the case, then by Theorem \ref{th-isotrivial} below, $\Omega$ contains in a dense way positive dimensional subvarieties where the eigenvalues of all the periodic points on the small Julia set are constant. This contradicts either Corollary \ref{cor-finite} or Corollary \ref{cor-finite-poly}. Note that, unlike in the rest of the article, the families considered here may be transcendental and are not necessarily closed.

Apart from Theorem \ref{th-isotrivial} which only holds for $k=2$ in the polynomial case, all the proofs in this section are the same for both $\mathscr{M}_d^k$ and for $\mathscr{P}_d^k$. We therefore focus on the case of $\mathscr{M}_d^k$. The only difference for $\mathscr{P}_d^k$ is that the dimensions below, $N_d^k$ and $\mathcal N_d^k$, should be replaced by the dimension of $\mathrm{Poly}_d^k$ (i.e. $k\binom{k+d}{d}$) and the dimension of $\mathscr{P}_d^k$ (i.e. $k\left(\binom{k+d}{d}-(k+1)\right)$) respectively. 

In fact, throughout this section and in Section~\ref{sec-existence}, we do not work directly on $\mathscr{M}_d^k$ or $\mathscr{P}_d^k$. One key reason is that in Section~\ref{sec-existence} we consider degenerations outside the space of endomorphisms of $\Pb^k$. In order to obtain a non-empty open subset in the support of the bifurcation measure on $\mathscr M_d^k$, we provide a non-empty open subset $\Omega\subset \mathrm{End}_d^k$ in the support of the current  $T_{f,\mathrm{Crit}}^{\mathcal N_d^k}$ and use the fact that this current is the pullback of $\mu_{f,\mathrm{Crit}}$ under the canonical projection $\Pi\colon\mathrm{End}_d^k\to\mathscr{M}_d^k$.

\medskip

At the beginning of the section, after introducing some basic notations, we present in \S~\ref{sec-ass} a long list of assumptions that will be required in what follows. Then, in \S~\ref{sec-state} we state the main results of the whole section, Theorem \ref{th-constant} and Theorem \ref{th-isotrivial}, and explain how, combined with Corollary \ref{cor-finite} or Corollary \ref{cor-finite-poly} and Theorem \ref{th-existence}, they imply Theorem  \ref{tm:mu-interior}. In \S~\ref{sec-sketch}, we outline the proof strategy for these two theorems and describe the structure of the remainder of the section.

%
%

\subsection{Notations}\label{sec-not}
If $a>0$ then $\Db_a$ is the disc of center $0$ and radius $a$ in $\Cb$ and we set $\Db:=\Db_1.$

If $A\subset\Cb$ and $B\subset\Cb^{k-1}$ are two connected open subsets then $\Gamma\subset A\times B$ is a \emph{vertical graph} if there exists a holomorphic function $g\colon B\to A$ with $\overline{g(B)}\Subset A$ and such that $\Gamma=\{(g(w),w)\ ;\ w\in B\}$. One way to measure the verticality of a graph is to consider cone fields. As we will only work on $\Cb^k$ where the tangent bundle is trivial, for $\rho>0$ we say that a vertical graph $\Gamma$ as above is \emph{tangent} to the cone field
$$C_\rho:=\left\{(u_1,\ldots,u_k)\in\Cb^k\ ;\ \rho|u_1|\leq\max_{2\leq i\leq k}|u_i|\right\}$$
if the tangent bundle $T\Gamma$ is contained in $\Gamma\times C_\rho$. If $\Gamma=\{(g(w),w)\ ;\ w\in B\}$ as above, then this is equivalent to the fact that the partial derivatives of $g$ are uniformly bounded by $1/\rho$. Observe that the larger $\rho$ is, the more vertical $\Gamma$ becomes. The case $\rho=+\infty$ corresponds to vertical hyperplanes. We say that a map $f$ \emph{contracts} the cone field $C_\rho$ if there exists $\rho'>\rho$ such that the image of $C_\rho$ under the differential of $f$ at each point is contained in $C_{\rho'}.$
%
%
%
%

\subsection{Assumptions}\label{sec-ass}

Let $\Omega$ be a non-empty open subset of $\mathrm{End}_d^k$ or of $\mathrm{Poly}_d^k$ such that each $f\in\Omega$ fulfills the properties described below. Observe that most of these objects (all except $J_k$) are assumed to depend holomorphically on $f\in\Omega$ and our notation reflects this dependence. For example, $p\colon f\mapsto p(f)$ is the holomorphic motion of the saddle point given in Assumption \ref{saddle} and $\Lambda$ corresponds to the holomorphic motion of the hyperbolic set $\Lambda(f)$ from Assumption \ref{hyperbolic}, i.e. $x\in\Lambda$ is a function $f\mapsto x(f)$ given by this holomorphic motion. Another observation is that the most important case is $k=2$. When $k\geq3$, the dynamics on the last $k-2$ coordinates is not very important. However, item iii) in Assumption \ref{10} prevents us from taking product maps. The reader may refer to Figure \ref{fig:1} for an illustration of some of the assumptions.
\begin{enumerate}
\item\label{verti} There exist two disjoint holomorphic discs $U_+,U_-\subset\Cb$ and two constants $R>2$, and $\rho>10$ such that $f$ contracts the cone field $C_\rho$ on $\mathcal U:=\mathcal U_+\cup\mathcal U_-$ where
$$\mathcal U_\pm:=\Db_R\times U_\pm\times\Db^{k-2}.$$
\item\label{hyperbolic} There exist two disjoint holomorphic discs $V_+,V_-\subset\Cb$ such that, if we set
$$\mathcal V_+:=\Db_R\times V_+\times\Db^{k-2},\ \ \mathcal V_-:=\Db_R\times V_-\times\Db^{k-2}\ \text{ and }\ \mathcal V:=\mathcal V_+\cup\mathcal V_-$$
then
\begin{itemize}
\item $\overline V_\pm\subset U_\pm,$
\item $f^2$ contracts the cone field $C_\rho$ on $\mathcal V,$
\item $f^2$ is injective on $\mathcal V_\pm$ and $\overline{\mathcal V}\subset\overline{\mathcal U}\subset f^2(\mathcal V_\pm).$
\end{itemize}
Moreover,
$$\Lambda(f):=\cap_{n\geq0}f^{-2n}(\mathcal V)$$
is a repelling hyperbolic set for $f^2$, contained in $J_k(f).$
\item\label{saddle} $f$ has a non-critical saddle fixed point $p(f)\in \Db\times U_-\times\Db^{k-2}$ with one stable direction and $k-1$ unstable directions. We ask that its unstable manifold contains a vertical graph through $p(f)$ (denoted $W^u_{p(f),loc}$) in $\Db\times V_-\times\Db^{k-2}$ tangent to $C_\rho$. In what follows, $W^s_{p(f),loc}$ stands for a holomorphic disc in the stable manifold where $f$ is conjugate to a contraction. Finally, we assume that $f$ is $C^1$-linearizable near $p(f)$, with a linearization map depending continuously on $f$ in the $C^1$-topology.
\item\label{blender} Each vertical graph in $\Db\times V_+\times\Db^{k-2}$ (resp. $\Db\times V_-\times\Db^{k-2}$)  tangent to $C_\rho$ intersects $\Lambda(f)$. 
\item\label{free} The intersections between $W^u_{p,loc}$ and $\Lambda$ are not persistent in $\Omega$ (i.e. if $x\in\Lambda$ and $f\in\Omega$ satisfy $x(f)\in W^u_{p(f),loc}$ then there exists $g\in\Omega$ close to $f$ such that $x(g)$ is not in $W^u_{p(g),loc}$).
\item\label{repelling} There exists a repelling $2$-periodic point $r(f)\in\Db\times V_-\times\Db^{k-2}$ such that the eigenvalues of $D_{r(f)}f^2$ are all simple with no resonance. In particular, $f^2$ is holomorphically linearizable near $r(f)$ and we assume that the domain of linearization contains $\overline{\mathcal U_-}.$
\item\label{critical} There exists $n_0\geq1$ such that $f^{n_0}(\crit(f))$ has a transverse intersection with $W^s_{p(f),loc}\setminus\{p(f)\}.$
\item\label{homoclinic} There exist $K\in\Nb$ and $\tilde q(f)\in W^u_{p(f),loc}$ which is not a critical point for $f^K$ and such that $q(f):=f^K(\tilde q(f))\neq p(f)$ is a transverse homoclinic intersection in $W^s_{p(f),loc}.$
\item\label{exceptional} The exceptional set of $f$ is disjoint from its small Julia set.
\end{enumerate}
We now introduce one last, slightly more technical assumption. It will serve as the starting point for the induction on the dimension. Let $\chi_{p(f)}$ (resp. $\chi_{r(f)}$) be the eigenvalue of $D_{p(f)}f$ (resp. $D_{r(f)}f^2$) with the smallest modulus.
\begin{enumerate}[resume]
\item\label{10} For every non-empty open subset $\Omega'\subset\Omega$, there exists $f\in\Omega'$, $m\in\Nb$ and $x\in\Lambda$ such that
\begin{itemize}
\item[i)] $f^m(x(f))=r(f),$
\item[ii)] $x(f)\in W^u_{p(f),loc},$
\item[iii)] $D_{x(f)}f^m(T_{x(f)}W^u_{p(f),loc})$ is a ``generic'' hyperplane for $D_{r(f)}f^2$, i.e. contains no eigenvector of $D_{r(f)}f^2,$
\item[iv)] the subgroup $\langle\chi_{p(f)},\chi_{r(f)}\rangle$ of $\Cb^*$ generated by $\chi_{p(f)}$ and $\chi_{r(f)}$ is dense,
\end{itemize}
\end{enumerate}

\bigskip

From a non-technical point of view, the main ingredients to prove that $\Omega\subset\supp(T_{f,\mathrm{Crit}}^{\mathcal N_d^k})$ are Assumptions \ref{hyperbolic} to \ref{blender}. They should be sufficient for the proof. Assumption \ref{blender} says that $\Lambda(f)$ satisfies a blender property and by Assumption \ref{saddle}, there exists a connection between this blender $\Lambda(f)$ and the saddle point $p(f)$. If the critical set has a transverse intersection with the stable manifold of $p(f)$, this gives rise, by the inclination lemma, to infinitely many intersections between the postcritical set and $\Lambda(f)$. Very likely, all these intersections should provide as many independent bifurcations as possible. Most of the remaining assumptions aim to ensure several transversality properties which eventually give the existence of these independent bifurcations. In particular, a transverse intersection between $W^s_{p(f),loc}$ and the postcritical set is given by Assumption \ref{critical}. Observe that, in addition, this assumption also implies that $\Omega$ contains no PCF maps (see the end of the proof of Theorem \ref{tm:mu-interior} or \cite[Corollary 2.5]{le-PCA} for a more precise result).

In the example we construct, all these assumptions are easy to check, except Assumption \ref{10}. This last assumption is the key technical point in proving that the support of the bifurcation measure has non-empty interior. Establishing it on $\Omega$ takes a large part of Section \ref{sec-existence} where we need to consider degenerations outside $\mathrm{End}_d^k.$

\smallskip

In order to give more explanations on this assumption, item iv) will be used to ensure that the postcritical set of $f$ can approximate any leaf of a foliation by hypersurfaces $\mathcal F_{f}$, defined in a neighborhood of $r(f)$ as the vertical fibration associated with the linearization map (i.e., the strong unstable foliation of $r(f)$). Point iii) implies in particular that the strong unstable hyperplane $T_{r(f)}W_{r(f)}^{uu}$ together with $(D_{x(f)}f^{m+i}(T_{x(f)}W^u_{p(f),loc}))_{1\leq i\leq k-1}$ form a basis of hyperplanes. Each of them is actually the tangent space of a dynamical foliation, $\mathcal F_f$ and $(\mathcal G_f^i)_{1\leq i\leq k-1}$ respectively, which thus define local coordinates near $r(f)$. A key point will be that, under suitable conditions labeled as $(\star)$ in what follows, these coordinates provided local conjugacies which turn out to extend to a neighborhood of the small Julia set. The results of Section \ref{sec-rigidity} ensure then that the conjugacies are generically global.

\smallskip

Finally, notice that it would be easier to work with a fixed point in Assumption \ref{repelling}. However, we were not able to obtain the open set $\Omega$ when $d=2$ with this additional constraint.

\subsection{Statements}\label{sec-state}
Here, we assume that $\Omega$ satisfies all the assumptions of \S~\ref{sec-ass}. The purpose of Assumption \ref{10} is to construct families with the following properties.
\begin{definition}\label{def-dag}
A subvariety $M\subset\Omega$ satisfies the condition $(\dag)$ if
\begin{itemize}
\item[\textbf{(1)}] $M$ is connected,
\item[\textbf{(2)}] there exists $x\in\Lambda$ and $m\in\Nb$ such that for all $f\in M$, $x(f)\in W_{p(f),loc}^u$, $f^m(x(f))=r(f)$ and $D_{x(f)}f^m(T_{x(f)}W^u_{p(f),loc})$ is a generic hyperplane for $D_{r(f)}f^2,$
\item[\textbf{(3)}] each intersection point in $W^u_{p(f)}\cap\Lambda(f)$ can be locally followed holomorphically,
\item[\textbf{(4)}] there exists $f_0\in M$ such that the subgroup $\langle\chi_r(f_0),\chi_p(f_0)\rangle$ is dense in $\Cb^*.$
\end{itemize}
\end{definition}
We also consider a stronger condition.
\begin{definition}\label{def-star}
A subvariety $M\subset\Omega$ satisfies the condition $(\star)$ if it is simply connected, fulfills $(\dag)$ and is a stable family in the sense of Berteloot-Bianchi-Dupont.
\end{definition}
The main purpose of this whole section is to show that these conditions combined with the assumptions on $\Omega$ lead to the following two results.
\begin{theorem}\label{th-constant}
If $M\subset\Omega$ satisfies $(\dag)$ then the functions $f\mapsto\chi_{p(f)}$ and $f\mapsto\chi_{r(f)}$ are constant on $M$. In particular, any connected analytic subset $M'\subset M$ also satisfies $(\dag).$
\end{theorem}

\begin{theorem}\label{th-isotrivial}
Let $M\subset\Omega$ be an analytic subset which satisfies $(\star)$. Let $f_0$ and $f_1$ be in $M$. Then, $f_0$ and $f_1$ are holomorphically conjugate in a neighborhood of their respective small Julia sets.

Furthermore, these conjugacies are compatible with the holomorphic motion of periodic points, i.e., every $n$-periodic point $x(f_0)$ of $J_k(f_0)$ can be followed along $M$ as a $n$-periodic point $x(f)$ in $J_k(f)$ and all eigenvalues of $D_{x(f)}f^n$ are constant as functions of $f$. In particular, the map $f\mapsto\det D_{x(f)}f^n$ is constant on $M.$
\end{theorem}

Anticipating the existence of the open set $\Omega$, established in Theorem \ref{th-existence}, we can conclude the proof of Theorem \ref{tm:mu-interior}.
\begin{proof}[Proof of Theorem \ref{tm:mu-interior}]
We only consider the case of $\mathscr M_d^k$. As we already said, the proof for $\mathscr P_d^2$ is exactly the same except that $\mathcal N_d^k$ has to be replaced by $2\binom{d+2}{d}-6$. Observe that we cannot conclude the proof on $\mathscr P_d^k$ when $k\geq3$ since Corollary \ref{cor-finite-poly} only holds on $\mathscr P_d^2.$

Let $k\geq2$ and $d\geq2$. Let $\Omega$ be the open subset of $\mathrm{End}_d^k$ given by Theorem \ref{th-existence}. Our goal is to show that $\Omega\subset\supp(T_{f,\mathrm{Crit}}^{\mathcal N_d^k})$. To that end, we consider a non-empty connected open subset $\Omega'\subset\Omega$ and we will prove that $\Omega'\cap\supp(T_{f,\mathrm{Crit}}^{\mathcal N_d^k})$ is not empty, proceeding through the following steps.

\bigskip

\noindent\emph{\textbf{Step 1: Reducing $\Omega'$ with respect to multipliers.}} First, fix an arbitrary element $f'\in\Omega'$. If we apply Corollary \ref{cor-finite} to the sequence $(x_n)_{n\geq1}$ of all repelling periodic points of $f'$ in $J_k(f')$ then there exists $N_0\geq1$ such that the corresponding multiplier map $\Lambda_{N_0}$ is generically finite on a branched cover of $\mathscr M_d^k$. As the periodic points $(x_n)_{1\leq n\leq N_0}$ are repelling and in $J_k(f')$, they can be followed holomorphically as repelling points in $J_k(f)$ in a small neighborhood of $f'$ in $\mathrm{End}_d^k$. Since $\Lambda_{N_0}$ is generically finite, the fibers of the corresponding map on a small open subset $\Omega''\subset\Omega'$ close to $f'$ have codimension $\mathcal N_d^k$. Hence, by Theorem \ref{th-isotrivial}, any analytic subset of $\Omega''$ satisfying $(\star)$ must have codimension at least $\mathcal N_d^k.$

\smallskip

Now, by Assumption \ref{10}, there exists $f_0\in\Omega''$, $m\in\Nb$ and $x_1(f_0)\in\Lambda(f_0)$ such that
\begin{itemize}
\item[i)] $f_0^m(x_1(f_0))=r(f_0),$
\item[ii)] $x_1(f_0)\in W^u_{p(f_0),loc},$
\item[iii)] $D_{x_1(f_0)}f_0^m(T_{x_1(f_0)}W^u_{p(f_0),loc})$ is a generic hyperplane for $D_{r(f_0)}f_0^2,$
\item[iv)] the subgroup $\langle\chi_{p(f_0)},\chi_{r(f_0)}\rangle$ of $\Cb^*$ is dense.
\end{itemize}
In particular, $f_0$ belongs to
$$A_1:=\{f\in\Omega''\ ;\ x_1(f)\in W^u_{p(f),loc}\},$$
which is a hypersurface by Assumption \ref{free}.

\bigskip

\noindent\emph{\textbf{Step 2: Move to a smooth point.}} Our goal now is to fix new relations between $\Lambda(f_0)$ and $W^u_{p(f_0)}$ until condition $(\dag)$ holds and Theorem 3.3 can be applied. A minor technical issue arises from the fact that $A_1$ may be non-irreducible at $f_0$, with no control over the number of components. Hence, a relation between $\Lambda(f_0)$ and $W^u_{p(f_0)}$ might persist on some components but not on others, potentially disrupting the properness of the intersections. Although it is possible to prove that such a situation cannot occur, we will, for simplicity, instead pass to a smooth point of $A_1$. More precisely, the parameters $f$ in $A_1$ such that  $\langle\chi_{p(f)},\chi_{r(f)}\rangle$ is dense in $\Cb^*$ is itself dense in any irreducible component of $A_1$ containing $f_0.$ Actually, if $P(f)$ and $R(f)$ denote logarithms of $\chi_{p(f)}$ and $\chi_{r(f)}$ and if we write $R(f)=t(f)P(f)+\theta(f)2i\pi$ with $t(f),\theta(f)\in\Rb$ then this condition on the subgroup is equivalent to the fact that $1,$ $t(f)$ and $\theta(f)$ are independent over $\Qb.$ This holds outside a countable union of real analytic subsets of $A_1.$ Hence, we can take a smooth point $f_1$ of $A_1$ such that this condition is satisfied and such that $f_1$ is close enough to $f_0$ to ensure that item iii) above also holds for $f_1.$ Let $X_1$ be the irreducible component of $A_1$ containing $f_1$ and let $\Omega_1\subset\Omega''$ be a small open neighborhood of $f_1$ such that $A_1\cap\Omega_1=X_1\cap\Omega_1$ and iii) is satisfied on $\Omega_1.$

\bigskip

\noindent\emph{\textbf{Step 3: Building a sequence of relations.}} The set $W^u_{p(f_1)}\cap\Lambda(f_1)$ is infinite and we use it to define a sequence $(f_i)_{1\leq i\leq N}$ in $\Omega''$, a decreasing sequence $(\Omega_{i})_{1\leq i\leq N}$ of open subsets of $\Omega''$, and a decreasing sequence $(X_i)_{1\leq i\leq N}$ of smooth irreducible analytic sets such that $\mathrm{codim}(X_i)=i$ and $f_i\in X_i.$ The construction goes as follows. Assume that $(f_i)_{1\leq i\leq i_0},$ $(\Omega_i)_{1\leq i\leq i_0}$ and $(X_i)_{1\leq i\leq i_0}$ are defined. If all the intersection points in $W^u_{p(f_i)}\cap\Lambda(f_i)$ can be followed holomorphically on $X_{i_0}$ then we set $N:=i_0$ and the construction ends. Otherwise, there exist $n_{i_0+1}\geq0$ and $x_{i_0+1}\in\Lambda$ such that $x_{i_0+1}(f_{i_0})\in f_{i_0}^{n_{i_0+1}}(W^u_{p(f_{i_0}),loc})$ and this relation does not persist on $X_{i_0}$. Here, since $f_{i_0}^{n_{i_0+1}}(W^u_{p(f_{i_0}),loc})$ is not a closed analytic set, we mean that there exists a small neighborhood $\tilde\Omega_{i_0+1}\subset\Omega_{i_0}$ of $f_{i_0}$ such that the set
$$A_{i_0+1}:=\{f\in X_{i_0}\cap\tilde\Omega_{i_0+1}\ ;\ x_{i_0+1}(f)\in f^{n_{i_0+1}}(W^u_{p(f),loc})\},$$
is a closed hypersurface in $X_{i_0}\cap\tilde\Omega_{i_0+1}.$ As in Step 2, the set $A_{i_0+1}$ might be non-irreducible but we can choose a smooth point $f_{i_0+1}$ on it such that $\langle\chi_{p(f_{i_0+1})},\chi_{r(f_{i_0+1})}\rangle$ is dense in $\Cb^*.$ We then choose $X_{i_0+1}$ to be the irreducible component of $A_{i_0+1}$ which contains $f_{i_0+1}$ and take $\Omega_{i_0+1}\subset\tilde\Omega_{i_0+1}$ to be a small enough neighborhood of $f_{i_0+1}$ to have $A_{i_0+1}\cap\Omega_{i_0+1}=X_{i_0+1}\cap\Omega_{i_0+1}.$ Observe that we always have $N\leq\mathcal N_d^k$ and that $X_N$ satisfies $(\dag).$
 
\bigskip
 
\noindent\emph{\textbf{Interlude: From the unstable manifold to the postcritical set.}} Another important observation is that $X_N$ corresponds to $N$ independent intersections between $W^u_{p(f)}$ and $\Lambda(f)$ and, since the $J_k$-repelling periodic points are dense in $\Lambda$ and since by Assumption \ref{critical} some parts of the postcritical set approximate $W^u_{p(f)}$, a small perturbation of $X_N$ gives rise to an analytic set which corresponds to $N$-properly $J_k$-prerepelling parameters.
Since this point is important, we now provide more details. Let $\Omega_N,$ $(x_i)_{1\leq i\leq N}$ and $(n_i)_{1\leq i\leq N}$ be as above with the convention that $n_1=0.$ 
 We now consider the sets
 $$W_N:=\{(f,z_1,\ldots,z_N)\in\Omega_N\times(\Pb^k)^N\ ;\ z_i\in f^{n_i}(W^u_{p(f),loc})\text{ for }1\leq i\leq N\}$$
 and
 $$Y_N:=\{(f,z_1,\ldots,z_N)\in\Omega_N\times(\Pb^k)^N\ ;\ z_i=x_i(f)\text{ for }1\leq i\leq N\}.$$
 What we have proved so far is that the projection of $W_N\cap Y_N$ on $\Omega_N$, which is equal to $X_N\cap\Omega_N$, has codimension $N$. Since the projection of $Y_N$ on $\Omega_N$ is a biholomorphism, we have that $W_N\cap Y_N$ has pure dimension $N_d^k-N$, where $N_d^k:=\dim(\mathrm{End}_d^k)$. On the other hand, $\dim(Y_N)=N_d^k$ and $\dim(W_N)=N_d^k+(k-1)N$ so we have
 $$\dim(W_N\cap Y_N)=\dim(W_N)+\dim(Y_N)-\dim(\Omega_N\times(\Pb^k)^N),$$
which shows that the intersection is proper. As the repelling periodic points are dense in $\Lambda$, the set $Y_N$ is approximated by sets $Y'_{N,n}$ defined in the same way replacing each $x_i$ by repelling periodic points $x'_{i,n}$ converging to $x_i$. Moreover, the inclination lemma and Assumption \ref{critical} also give that $W_N$ is approximated by sets $W'_{N,n}$ defined as $W_N$ but using a local branch of some iterate of the critical set instead of $W_{p(f),loc}^u$. The persistence of proper intersections (see e.g. \cite[\textsection 12.3]{Chirka}) gives that $W'_{N,n}\cap Y'_{N,n}$ is proper when $n$ is large enough, i.e., $W'_{N,n}\cap Y'_{N,n}$ corresponds to $N$-properly $J_k$-prerepelling points in $\Omega_N\times(\Pb^k)^N$. 

\bigskip
 
\noindent\emph{\textbf{Step 4: Finishing the induction.}} Now, we continue the construction and define by induction $(f_i)_{N+1\leq i\leq N'},$ $(\Omega_i)_{N+1\leq i\leq N'}$ and  $(X_i)_{N+1\leq i\leq N'}$ in the following way. Assume the construction has been carried out for $N\leq i\leq i_0$. If the family defined by $X_{i_0}$ is stable then we set $N':=i_0$. Otherwise, there exists a non-persistent Misiurewicz relation on $X_{i_0}$ and we define $A_{i_0+1}$ to be the analytic hypersurface of $X_{i_0}$ where this relation persists. Then, we choose a smooth point $f_{i_0+1}$ on $A_{i_0+1}$ and a small neighborhood $\Omega_{i_0+1}\subset\Omega_{i_0}$ such that $X_{i_0+1}:=\Omega_{i_0+1}\cap A_{i_0+1}$ is smooth, connected and simply connected.

As above, at the end we have
\begin{itemize}
\item $\codim(X_{N'})=N'\leq\mathcal N_d^k,$
\item all the Misiurewicz relations in $X_{N'}$ are persistent, i.e., this family is stable,
\item by Theorem \ref{th-constant} $X_{N'}$ satisfies $(\dag)$ and thus $(\star).$
\end{itemize}
The construction of $\Omega''$ and Theorem \ref{th-isotrivial} then ensure that $N'\geq\mathcal N_d^k$ and thus $N'=\mathcal N_d^k$. On the other hand, exactly as in the interlude above, the points of $X_{N'}$ are approximated by $N'$-properly $J_k$-prerepelling parameters in $\mathrm{End}_d^k\times(\Pb^k)^{N'}$. By Proposition \ref{tm:densitypart1}, $X_{N'}$ is contained in the support of the current $T_{f,\mathrm{Crit}}^{\mathcal N_d^k}$. Moreover, the bifurcation measure $\mu_{f,\mathrm{Crit}}$ of the moduli space $\mathscr{M}_d^k$ satisfies $\Pi^*(\mu_{f,\mathrm{Crit}})=T_{f,\mathrm{Crit}}^{\mathcal{N}_d^k}$, where $\Pi:\mathrm{End}_d^k\to\mathscr{M}_d^k$ is the natural projection, see~\cite{BB1}. This implies $\hat\Omega:=\Pi(\Omega'')\subset\mathrm{supp}(\mu_{f,\mathrm{Crit}})$ and $\hat\Omega$ is open since $\Pi$ is an open map.

\bigskip

\noindent\emph{\textbf{Step 5: Absence of PCF maps.}} Finally, notice that Assumption \ref{critical} gives that the open set $\Omega$ obtained by Theorem \ref{th-existence} possesses no PCF maps. More precisely, let $f\in\Omega$. The inclination lemma applied to the portion of $f^{n_0}(\crit_f)$ that is transverse to $W^s_{p(f),loc}$ given by Assumption \ref{critical} implies that the postcritical set contains infinitely many disjoint (local) hypersurfaces converging to $W^u_{p(f),loc}$. Therefore, the postcritical set is not algebraic.
\end{proof}
\begin{remark}
For the absence of PCF maps in $\Omega$ we could have used a result of Le \cite[Corollary 2.5]{le-PCA}, which states  that a PCF map of $\Pb^k$ cannot have a non-critical saddle periodic point.
\end{remark}
\subsection{Sketch of the proofs of Theorems \ref{th-constant} and \ref{th-isotrivial}}\label{sec-sketch}
Let $M\subset\Omega$ be a subvariety satisfying $(\star)$. The assumptions of \S~\ref{sec-ass} are used in the following way.
\begin{enumerate}
\item\label{pt-1} As the family is stable, by \cite[Theorem C]{bianchi-poly-like} there exists an equilibrium lamination $\mathcal L$ for the family $(f)_{f\in M}$ (see Definition \ref{def-lami}).
\item\label{pt-2} Points \textbf{(2)} and \textbf{(3)} in Definition \ref{def-dag}, which come from Assumption \ref{10}, ensure that $r(f)\in W^u_{p(f)}$ persistently in the family.
\item\label{pt-3} Since the postcritical set intersects transversely $W^s_{p(f)}$ (Assumption \ref{critical}), the inclination lemma and the assumption $\overline{\langle\chi_{p(f_0)},\chi_{r(f_0)}\rangle}=\Cb^*$ imply that the postcritical set of $f_0$ can approximate any leaf of a foliation by hypersurfaces $\mathcal F_{f_0}$, defined in a neighborhood of $r(f_0)$ as its strong unstable foliation.
\item\label{pt-4} The previous point, the stability of $(f)_{f\in M}$ and the blender property from Assumption \ref{blender} imply, at first, a persistent identity $\chi_{r(f)}=\zeta\chi_{p(f)}^{-\omega}$ on $M$, which actually gives, combined with Assumption \ref{homoclinic}, that both these functions are constant.
\item\label{pt-5} The genericity part of \textbf{(2)} in Definition \ref{def-dag} allows us to construct $k-1$ other local foliations, $\mathcal G^1_f,\ldots,\mathcal G^{k-1}_f$ whose leaves are also approximated by the postcritical set and such that $(\mathcal F_f,\mathcal G^1_f,\ldots,\mathcal G^{k-1}_f)$ provides local holomorphic coordinates (depending holomorphically on $f$) near $r(f).$
\item\label{pt-6} The fact that the equilibrium lamination $\mathcal L$ is acritical implies that if $\gamma\in\mathcal L$ then the coordinates of $\gamma(f)$ with respect to $(\mathcal F_f,\mathcal G^1_f,\ldots,\mathcal G^{k-1}_f)$ are independent of $f.$
\item\label{pt-7} Since $\{\gamma(f)\ |\ \gamma\in\mathcal L\}$ is not contained in a proper analytic set, these local coordinates respecting $\mathcal L$ give a local conjugacy near $r(f).$
\item\label{pt-8} This local conjugacy extends to a neighborhood of the small Julia set, forcing the multipliers to be constant in the family.
\end{enumerate}
Now, in \S~\ref{sec-loc} we set notations and basic results for the family $(f)_{f\in\Omega}$. \S~\ref{sec-rela} and \S~\ref{sec-stereo} are devoted to obtain the points \eqref{pt-3} and \eqref{pt-4} which actually imply Theorem \ref{th-constant}. The conjugacy, which corresponds to points \eqref{pt-5} to \eqref{pt-8}, is constructed in \S~\ref{sec-iso}.

\subsection{Semi-local dynamics}\label{sec-loc}
First, we fix an arbitrary $f_0\in \Omega$.  As Theorem \ref{th-isotrivial} is essentially a local result, we will, when necessary, replace $\Omega$ by a smaller connected open neighborhood of $f_0$ in $\Omega.$

Since the fibration of $\Cb^k$ by vertical hypersurfaces will play an important role in what follows, we denote by $\pi \colon \Cb^k \to \Cb$ the projection onto the first coordinate, and we write points as $(z, w) \in \Cb \times \Cb^{k-1}$ to indicate the corresponding coordinates.

Let $r(f)$ be the repelling $2$-periodic point given by Assumption \ref{repelling}. We denote by  $\chi_{r(f)}$ the eigenvalue of $D_{r(f)}f^2$ with the smallest modulus. Since the eigenvalues of $D_{r(f)}f^2$ have no resonance, there exist (see e.g.  \cite{berger-reinke}) a holomorphic family of holomorphic maps $(\phi_f)_{f\in \Omega}$ from $\Cb^k$ to $\Pb^k$ and a holomorphic family $(\tilde L_f)_{f\in\Omega}$ of linear self-maps of $\Cb^{k-1}$ such that $\phi_f(0)=r(f)$ and
$$\phi_f^{-1}\circ f^2\circ\phi_f(z,w)=(\chi_{r(f)}z,\tilde L_f(w))=:L_f(z,w)$$
for every $(z,w)\in\Cb\times\Cb^{k-1}$ near $0$. In particular, the vertical linear fibration of $\Cb^k$ defined by $\pi$ is sent on the strong unstable fibration of $r(f)$ and $\phi_{f|\pi^{-1}(0)}$ provides a parametrization of the strongly unstable manifold of $r(f)$. Moreover, Assumption \ref{repelling} implies that there exists a neighborhood $A$ of $0$ in $\Cb^k$ such that $\phi_{f_0}$ is injective on $A$ with $\overline{\mathcal U_-}\subset\phi_{f_0}(A)$. The cone condition in Assumption \ref{hyperbolic} ensures that there is an open set $\tilde A\subset\Cb^{k-1}$ such that $\{0\}\times \tilde A\subset A$ and $\phi_{f_0}(\{0\}\times \tilde A)$ is a vertical graph passing through $\Db\times V_-\times\Db^{k-2}$, i.e., is a closed analytic subset of $\Db\times V_-\times\Db^{k-2}$ whose closure in $\mathcal U_-$ is disjoint from $(\partial\Db)\times V_-\times\Db^{k-2}.$  As these properties are stable under small perturbations, there exists $\nu>0$ such that, possibly by reducing $\Omega$ and slightly $A,\tilde A$, for each $f\in\Omega$
\begin{itemize}
\item $\phi_f$ is injective on $A$ with $\overline{\mathcal U_-}\subset\phi_f(A),$
\item for all $c\in\Db_\nu$, $\{c\}\times\tilde A\subset A$ and $\phi_f(\{c\}\times\tilde A)$ is a vertical graph passing through $\Db\times V_-\times\Db^{k-2}.$
\end{itemize}
We denote by $\delta_f\colon\phi_f(A)\to A$ the associated inverse map. Observe that the second point above combined with Assumption \ref{blender} implies that each $\phi_f(\{c\}\times\tilde A)$ intersects $\Lambda(f)$. In what follows, it will be convenient to normalize the family $(\phi_f)_{f\in\Omega}$ in the following way. Consider a family $(u_f)_{f\in\Omega}$ of self-maps of $\Cb\times\Cb^{k-1}$ of the form $u_f(z,w)=(\tau(f)z,w)$ where $\tau(f)\in\Cb^*$ depends holomorphically on $f$ and is chosen so that
\begin{itemize}
\item $|\tau(f)|<\nu$ ensuring that $u_f(\Db\times\tilde A)\subset\Db_\nu\times\tilde A,$
\item there exists $r'\in\Lambda$ close enough to $r$ such that for all $f$ in $\Omega$, $\pi\circ u_f^{-1}\circ\delta_f(r'(f))\equiv 1.$
\end{itemize}
Hence, if we replace $\phi_f$ by $\phi_f\circ u_f$, we may assume that $\nu = 1$, and we then have $\pi \circ \delta_f(r'(f)) \equiv 1$. This normalization will only appear in Corollary \ref{cor-constant}, which is nevertheless a key ingredient in \S~\ref{sec-iso}.

We then set $B:=\Db\times\tilde A$ and $D_f:=\phi_f(B)$. The latter possesses a natural foliation $\mathcal F_f$ where $\mathcal F_f(c)=\phi_f(\{c\}\times\tilde A)$ for $c\in\Db$. As we have already seen, each leaf is a vertical graph intersecting $\Lambda(f)$. In particular, $\mathcal F_f(0)$ corresponds to the local strong unstable manifold $W^{uu}_{r(f),loc}$ of $r(f)$. We also denote by $W^{cu}_{r(f),loc}=\phi_f(\Db\times\{0\})$, the local weak unstable manifold of $r(f)$.

\begin{figure}[h]
    \centering
    \includegraphics[width=0.65\textwidth]{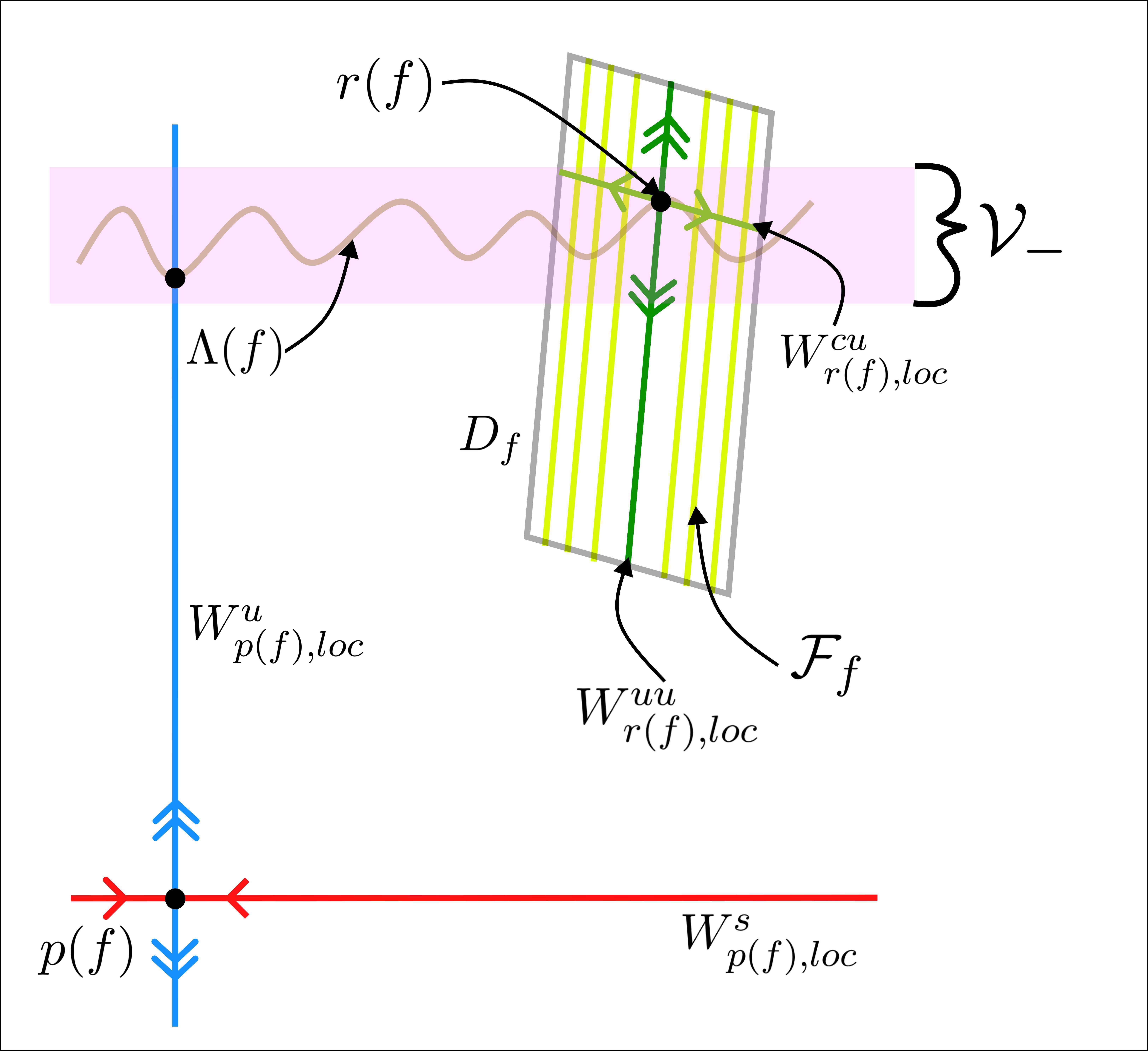}
    \caption{Summary of the notations. The whole picture is contained in $\mathcal U_-.$ In the example obtained in Section \ref{sec-existence}, the hyperbolic set $\Lambda(f)$ is a Cantor set but it intersects any sufficiently vertical graph in $\mathcal V_-.$}
    \label{fig:1}
\end{figure}

Finally, we also introduce some notation for the dynamics near the saddle fixed point $p(f)$ given by Assumption \ref{saddle}. Let $\chi_{p(f)}$ denote the eigenvalue of $D_{p(f)}f$ with the smallest modulus. Using holomorphic conjugacies separately on the stable and unstable directions, we first choose holomorphic local coordinates $v_f\colon V_f\to\Db^k$  on a neighborhood $V_f$ of $p(f)$ such that 
\begin{itemize}
\item $v_f^{-1}(\Db\times\{0\})$ is contained in the stable manifold of $p(f)$ and $v_f\circ f\circ v_f^{-1}(z,0)=(\chi_{p(f)}z,0),$
\item $v_f^{-1}(\{0\}\times\Db^{k-1})$ is contained in the unstable manifold of $p(f)$ and, whenever it is defined, $v_f\circ f\circ v_f^{-1}(0,w)=(0,A_f(w))$ where $A_f$ is an expanding square matrix of size $k-1.$
\end{itemize}
In what follows, the local stable manifold of $p(f)$ will be defined as $W^s_{p(f),loc}:=v_f^{-1}(\Db\times\{0\})$. For the local unstable manifold $W_{p(f),loc}^u$ of $p(f)$, we take the vertical graph in $\Db\times V_-\times\Db^{k-2}$ given by Assumption \ref{saddle}.

Moreover, this assumption also implies that there exists a $C^1$-family $(\theta_f)_{f\in M}$ of local $C^1$-diffeomorphisms such that $\theta_f\circ f\circ\theta_f^{-1}$ is the linear map $K_f(z,w)=(\chi_{p(f)}z,A_f(w))$. Observe that we can assume the domain of definition of $\theta_f$ contains $V_f$ and that $D_0(
\theta_f\circ v_f^{-1})=\id.$

\vspace{0.2cm}
\noindent\textbf{Consequences of the inclination lemma.} We will make extensive use of the inclination lemma for families of hypersurfaces transverse to $W^s_{p(f),loc}$ parametrized by a subset $M$ of $\Omega$. We will gradually strengthen the assumptions on $M$ until reaching condition $(\dag)$ in \S~\ref{sec-stereo} and $(\star)$ in \S~\ref{sec-iso}. For now, we just assume that $M$ is a connected analytic subset of $\Omega.$
\begin{definition}
We say that $\Gamma=(\Gamma(f))_{f\in M}$ is a family of polydiscs intersecting transversely $W^s_{p(f),loc}$ at $b(f)$ if
\begin{itemize}
\item each $\Gamma(f)$ is biholomorphic to $\Db^{k-1}$ and $f\mapsto\Gamma(f)$ is holomorphic,
\item each $\Gamma(f)$ intersects $W^s_{p(f),loc}$ in a unique point and this intersection is transverse,
\item the image by $v_f$ of this intersection point with $\Gamma(f)$ is $(b(f),0)\in\Db\times\{0\}.$
\end{itemize}
\end{definition}
From now on, we also assume that there exists $x\in\Lambda$ and $m\in\Nb$ such that for all $f\in M$, $x(f)\in W_{p(f),loc}^u$, $f^m(x(f))=r(f)$ and $D_{x(f)}f^m(T_{x(f)}W^u_{p(f),loc})$ is a generic hyperplane for $D_{r(f)}f^2$. By increasing $m$ if necessary, we can assume that $f^m$ maps biholomorphically a neighborhood of $x(f)$ in $W_{p(f),loc}^u$ to a vertical graph $W_m(f)$ in $D_f$, $C^1$-close to $W^{uu}_{r(f),loc}$ and thus tangent to $C_\rho.$

Let $\Gamma$ be a family of polydiscs intersecting transversely $W^s_{p(f),loc}$. By the inclination lemma, there exists $j_0\geq0$ such that for all $f\in M$ and all $j\geq j_0$, $f^{j}(\Gamma(f))$ contains a subset which is $C^1$-close to $W^u_{p(f),loc}$. In particular, $f^{j+m}(\Gamma(f))$ contains a subset $\Gamma_j(f)$ which is a vertical graph in $D_f$, tangent to $C_\rho$ and $C^1$-close to $W_m(f)$. From this, using again the inclination lemma but near $r(f)$, we can construct families of vertical graphs which turns out to be key objects to prove that $(\dag)$ implies a persistent resonance between $\chi_{p(f)}$ and $\chi_{r(f)}.$

\begin{definition}\label{def-gammakl}
Let $\Gamma$  and $j_0$ be as above. Let $f\in M$, $l_0\geq0$ and $j\geq j_0$. Let $c_j(f)$ denote the point of intersection between $\Gamma_j(f)$ and $W^{cu}_{r(f),loc}$. If for all $0\leq l\leq l_0$, $|\pi\circ\delta_f(f^{2l}(c_j(f)))|<1/2$ then we define $\Gamma_{j,l}(f)$ inductively by setting
\begin{itemize}
\item $\Gamma_{j,l}(f)$ is the vertical graph in $f^2(\Gamma_{j,l-1}(f))\cap D_f$ which contains $f^{2l}(c_j(f)).$
\end{itemize}
In this situation, we say that $\Gamma_{j,l}(f)$ is well-defined for all $l\leq l_0.$
\end{definition}

\begin{remark}\normalfont
\textbf{(1)} Observe that the injectivity in Assumption \ref{hyperbolic} implies that there is no ambiguity in the definition of $\Gamma_{j,l}(f).$\\
\textbf{(2)} A priori, it could happen that $\Gamma_{j,l}(f)$ is well defined for some $f=f_1$ and not for $f=f_2$, even if $|\pi\circ\delta_{f_1}(f_1^{2l}(c_j(f_1)))|$ is much smaller than $1/2$. However, we will see in Lemma \ref{le-normale} that under condition $(\dag)$, this doesn't happen, and thus $(\Gamma_{j,l}(f))_{f\in M}$ define holomorphic families of vertical graphs.
\end{remark}

Since $D_{x(f)}f^m(T_{x(f)}W^u_{p(f),loc})$ is a generic hyperplane for $D_{r(f)}f^2$, in particular $W^{cu}_{r(f),loc}$ is transverse to $W_m(f)$. Hence, by the inclination lemma there exist an integer $a>m$ and a holomorphic injective map $h_f\colon\Db\to V_f$ (where $V_f$ is the neighborhood of $p(f)$ defined above) such that
\begin{itemize}
\item $\Delta_f:=h_f(\Db)$ is transverse to $W^u_{p(f),loc},$
\item $\Delta_f$ is a graph above $W^s_{p(f),loc}$, more precisely the projection on the first coordinate of $v_f\circ h_f$ is the identity,
\item $f^a_{|\Delta_f}$ is injective and $f^a(\Delta_f)$ is a neighborhood of $r(f)$ in $W^{cu}_{r(f),loc}.$
\end{itemize}
We define $H_f\colon\Db\to\Db$ by $H_f=\pi\circ\delta_f\circ f^a\circ h_f$, which can be seen as a transition map between parametrizations of $\Delta_f$ and $W^{cu}_{r(f),loc}$ respectively. Observe that $H_f$ is injective with $H_f(0)=0$. Hence, there exists $\tilde\alpha(f)\neq0$, which depends holomorphically on $f$, such that
$$H_f(s)=\tilde\alpha(f)s+o(s),$$
where $o(s)$ is uniform in $f$. 

\begin{remark}\label{rk-homoclinic}\normalfont
Observe that a similar construction can be done where $W^{cu}_{r(f),loc}$ is replaced by a holomorphic disc $\Sigma_f$ transverse to $f^n(W^u_{p(f),loc})$ as long as the point in $W^u_{p(f),loc}$ sent to $\Sigma_f\cap f^n(W^u_{p(f),loc})$ is not critical for $f^n$. We will use such construction in Proposition \ref{prop-constant} for the homoclinic intersection given by Assumption \ref{homoclinic}, i.e., $\Sigma_f$ will be an open subset of $W^s_{p(f),loc}.$
\end{remark}

\begin{figure}[h]
    \centering
    \includegraphics[width=0.65\textwidth]{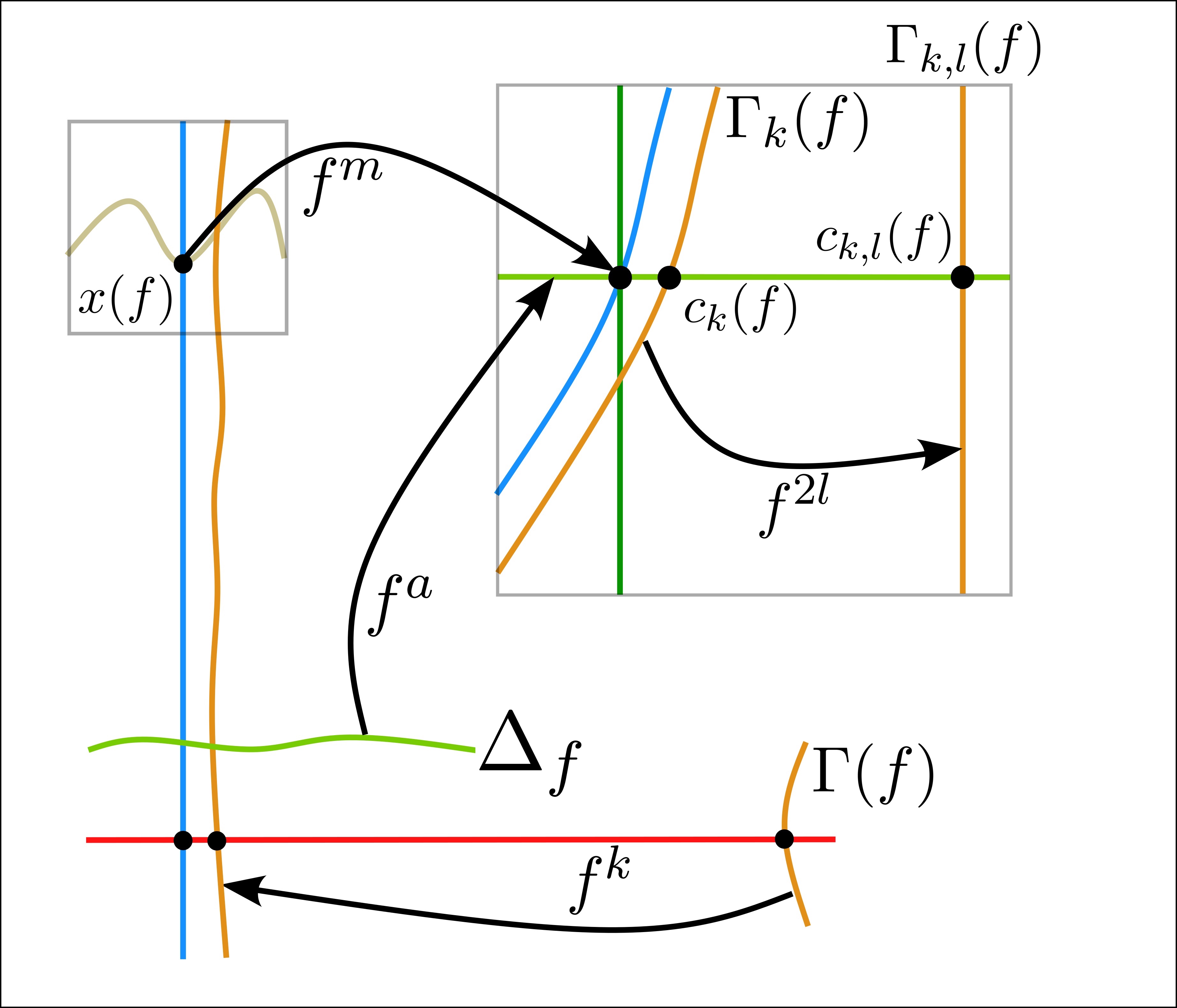}
    \caption{Definition of $\Gamma_{j,l}(f)$ where $x(f)\in W^u_{p(f),loc}$ is a preimage of $r(f).$ The integers $m$ and $a$ are constant but $j$ and $l$ can be large. The next two lemmas show that $c_j(f)$ and $c_{j,l}(f)$ are essentially equal to $\chi_{p}(f)^j$ and $\chi_p(f)^j\chi_r(f)^l$ respectively in the coordinates on $W^{cu}_{r(f),loc}$ given by $\phi_f.$}
    \label{fig:2}
\end{figure}

The two following lemmas can be seen as consequences of the inclination lemma or linearization results. Their purpose is to show that the vertical graphs $\Gamma_{j,l}(f)$ are close to leaves $\mathcal F_f(c_{j,l}(f))$ of the strong unstable foliation of $r(f)$, where $c_{j,l}(f)$ is essentially equal to $\chi_p(f)^j\chi_r(f)^l$. The first lemma focus on $\Gamma_j(f)$. It should be possible to prove it using distorsion estimates. Instead, we use $C^1$-linearization and this part of Assumption \ref{saddle} only appears here.
\begin{lemma}\label{le-Hf}
There exists a holomorphic function $\beta\colon M\to\Cb^*$ with the following property. Let $(\Gamma(f))_{f\in M}$ be a family of polydiscs intersecting transversely $W^s_{p(f),loc}$ at $b(f)$. For each $n\geq 0$ large enough there exists a holomorphic function $s_n\colon M\to\Cb$ such that for each $f\in M$ 
\begin{itemize}
\item $\Delta_f\cap f^{n}(\Gamma(f))=h_f(s_n(f)),$
\item $s_n(f)=\beta(f)b(f)\chi_{p(f)}^{n}+o(\chi_{p(f)}^n).$
\end{itemize}
In other words, for $j\geq0$ large enough, $\Gamma_j(f)\cap W^{cu}_{r(f),loc}\in\mathcal F_f(H_f(s_{j+m-a}(f)))$ with
$$H_f(s_{j+m-a}(f))=\alpha(f)b(f)\chi^j_{p(f)}+u_j(f),$$
with $\alpha(f):=\chi_{p(f)}^{m-a}\tilde\alpha(f)\beta(f)$ and such that $u_j(f)/\chi_{p(f)}^j$ converges to $0$, locally uniformly on $M.$
\end{lemma}
\begin{proof}
As $\theta_f$ is well defined on $V_f$, we can define a $C^1$-germ of $(\Cb,0)$ by
$$F_f(s):=\pi\circ v_f\circ\theta_f^{-1}\circ\tilde\pi_0\circ\theta_f\circ h_f(s),$$
where $\pi\colon\Cb^k\to\Cb$ is the first projection as above and $\tilde\pi_0\colon\Cb^k\to\Cb^k$ is defined by $\tilde\pi_0(z,w)=(z,0)$. The key point in the proof is that the differential of $F_f$ at $0$ is $\Cb$-linear. The maps $h_f$ and $\pi\circ v_f$ are holomorphic so we focus on $\theta_f^{-1}\circ\tilde\pi_0\circ\theta_f$ which corresponds to the projection in the unstable direction given by the linearization $\theta_f$. This maps is a priori not holomorphic but there exists a sequence $(U_n)_{n\geq0}$ of open neighborhoods of $W^u_{p(f),loc}\cap V_f$ such that
$$\pi_n := f^n\circ v_f^{-1}\circ\tilde\pi_0\circ v_f\circ f^{-n}$$
is defined on $U_n$. We claim that, for $u\in W^u_{p(f),loc}\cap V_f$, we have
$$\pi_n(u)=0=\theta_f^{-1}\circ\tilde\pi_0\circ\theta_f(u)\ \text{ and }\ D_u\pi_n\xrightarrow[n\to\infty]{}D_u(\theta_f^{-1}\circ\tilde\pi_0\circ\theta_f).$$
Indeed, the equality is obvious and for the convergence, if we set $\psi:=\theta_f\circ v_f^{-1}$ and $F:=v_f\circ f\circ v_f^{-1}$ then
$$\theta_f\circ\pi_n\circ\theta_f^{-1}=\psi\circ F^n\circ\tilde\pi_0\circ F^{-n}\circ\psi^{-1}=K_f^n\circ\psi\circ\tilde\pi_0\circ\psi^{-1}\circ K_f^{-n}.$$
It follows that, if $y$ satisfies $\theta_f(u)=(0,y)$ then, using that $D_0\psi=\id$ and writing $D_{(0,A^{-n}_f(y))}\psi^{-1}=\id+E_n$, we have
$$D_{(0,y)}(\theta_f\circ\pi_n\circ\theta_f^{-1})=K_f^n\circ\tilde\pi_0\circ(\id+E_n)\circ K_f^{-n}=K_f^n\circ\tilde\pi_0\circ K_f^{-n}+K_f^n\circ\tilde\pi_0\circ E_n\circ K_f^{-n}.$$
$K_f$ commutes with $\tilde\pi_0$ thus the first term is equal to $\tilde\pi_0$. The second one converges to $0$ since $\|K_f^{-n}\|\simeq\chi^{-n}_{p(f)}$, $\|K^n_f\circ\tilde\pi_0\|\simeq\chi_{p(f)}^n$ and $\|E_n\|$ converges to $0.$

This gives that $\theta_f^{-1}\circ\tilde\pi_0\circ\theta_f$ is $\Cb$-differentiable on $W^u_{p(f),loc}$ and so at $h_f(0)$. Hence, there exists $\gamma(f)\in\Cb$, which is non-zero since $\Delta_f$ and $W^s_{p(f),loc}$ are transverse to $W^u_{p(f),loc}$, such that $F_f(s)=\gamma(f)s+o(s)$. 

On the other hand, let $(\Gamma(f))_{f\in M}$ be a family of polydiscs intersecting transversely $W^s_{p(f),loc}$ at $b(f)$. For each $n\in\Nb$ large enough, there exists $s_n(f)\in\Db$ which depends holomorphically on $f\in M$ and such that $h_f(s_n(f))\in f^n(\Gamma(f))\cap\Delta_f$.   The set $\tilde\Gamma(f):=\theta_f(\Gamma(f))$ is locally a vertical graph $\{(g_f(y),y)\}$ where $g_f\colon(\Cb^{k-1},0)\to\Cb$ is a $C^1$-germ. Its image $\tilde\Gamma^n(f)$ by $K_f^n$ is given locally by $\{(g_f^n(y),y)\}$ where $g_f^n(y)=\chi_{p(f)}^ng_f(A^{-n}_f(y))$. Hence, since $A^{-n}_f$ contracts at an exponential speed, we have $g_f^n(y)=\chi_{p(f)}^ng_f(0)+o(\chi_{p(f)}^n)$ where the error term is uniform in $f$. Moreover, there exists $y_n(f)$ such that $\theta_f^{-1}(g_f^n(y_n(f)),y_n(f))=h_f(s_n(f))$. Therefore, the definitions of $v_f$ and $K_f$ give
\begin{align*}
F_f(s_n(f)) & =\pi\circ v_f\circ\theta_f^{-1}\circ\tilde\pi_0(g_f^n(y_n(f)),y_n(f)))\\
&=\pi\circ v_f\circ\theta_f^{-1}(g_f^n(y_n(f)),0)=\pi\circ v_f\circ\theta_f^{-1}(\chi_{p(f)}^ng_f(0),0)+o(\chi_{p(f)}^n)\\
&=\pi\circ v_f\circ\theta_f^{-1}(K_f^n(\theta_f(v_f^{-1}(b(f),0))))+o(\chi_{p(f)}^n)\\
&=\pi\circ v_f\circ f^n(v_f^{-1}(b(f),0))))+o(\chi_{p(f)}^n)\\
&=\chi_{p(f)}^nb(f)+o(\chi_{p(f)}^n).
\end{align*}
Hence, $s_n(f)=\beta(f)b(f)\chi_{p(f)}^{n}+o(\chi_{p(f)}^n)$ where $\beta(f):=\gamma(f)^{-1}$. To conclude, the sequence $s_n(f)/(b(f)\chi_{p(f)}^n)$ depends holomorphically on $f$ and this sequence converge locally uniformly to $\beta(f)$ which is then also holomorphic.
\end{proof}
The next lemma can be seen as a consequence of the inclination lemma in the presence of a dominated splitting. It can be proved using the linearization near $r(f)$ and the proof is left to the reader.
\begin{lemma}\label{le-coordonees}
Let $(\Gamma(f))_{f\in M}$ be a family of polydiscs  intersecting transversely $W^s_{p(f),loc}$ at $b(f)$. There exists a sequence $(\epsilon_n)_{n\geq0}$ converging to $0$ with the following property. If $f\in M$, $j\geq n$ and $l\geq n$ are such that $\Gamma_{j,l}(f)$ is well-defined then
$$d(\Gamma_{j,l}(f),\mathcal F_f(c_{j,l}(f)))\leq\epsilon_n,$$
where $c_{j,l}(f)=\alpha(f)b(f)\chi_{p(f)}^j\chi_{r(f)}^l.$
\end{lemma}

\subsection{Strong relations between the multipliers}\label{sec-rela}
From now on, we consider a subvariety $M\subset\Omega$ which satisfies \textbf{(1)}, \textbf{(2)} and  \textbf{(3)} in Definition \ref{def-dag}. Once again, we fix $f_0\in M.$

\begin{lemma}\label{le-normale}
Let $(j_n)_{n\geq0}$ and $(l_n)_{n\geq0}$ be two increasing sequences such that $(\chi_{p(f_0)}^{j_n}\chi_{r(f_0)}^{l_n})_{n\geq0}$ is a sequence in $\Db$ which converges. Then $\{f\mapsto\chi_{p(f)}^{j_n}\chi_{r(f)}^{l_n}\}_{n\geq 0}$ is a normal family in a neighborhood $M_0\subset M$ of $f_0.$
\end{lemma}
\begin{proof}
A preliminary observation is that since $\Lambda$ moves with respect to a holomorphic motion, the family of functions on $M$, $\{f\to x(f)\}_{x\in\Lambda}$ is a normal family. In particular, there exists a neighborhood $M_0\subset M$ of $f_0$ such that if $x(f_0)$ is in $D_{f_0}$ with $|\pi\circ\delta_{f_0}(x(f_0))|<\frac{1}{20}$ then for all $f\in M_0$, $x(f)\in D_f$ with $|\pi\circ\delta_{f}(x(f))|<\frac{1}{10}.$

By Assumption \ref{homoclinic}, there exists a family of polydiscs $\Gamma$ intersecting transversely  $W^s_{p(f),loc}$ at some $b(f)\neq0$ and such that $\Gamma(f)\subset W^u_{p(f)}$. By exchanging each $\Gamma(f)$ by an appropriate subset of $f^N(\Gamma(f))$, we can assume that, for all $f\in M_0$, $|\alpha(f)b(f)|<\frac{1}{20}.$

By Lemma \ref{le-Hf}, the first coordinate (with respect to $\delta_{f_0}$) of $c_{j_n}(f_0):=\Gamma_{j_n}(f_0)\cap W_{r(f_0),loc}^{cu}$ is $H_{f_0}(s_{j_n+m-a}(f_0))=\alpha(f_0)b(f_0)\chi^{j_n}_{p(f_0)}+u_{j_n}(f_0)$. Hence, 
$$\pi\circ\delta_{f_0}(f_0^{2l}(c_{j_n}(f_0)))=\chi_{r(f_0)}^{l}\left(\alpha(f_0)b(f_0)\chi^{j_n}_{p(f_0)}+u_{j_n}(f_0)\right).$$
Since $(\chi_{p(f_0)}^{j_n}\chi_{r(f_0)}^{l_n})_{n\geq0}$ converges to some $\chi(f_0)\in\Db$ and since $u_{j_n}(f_0)/\chi^{j_n}_{p(f_0)}$ converges to $0$, the sets $\Gamma_{j_n,l_n}(f_0)$ are well-defined for $n\geq n_0$ for some $n_0$ large enough.

By Assumption \ref{blender}, for $n\geq n_0$ there exists a point $x_n(f_0)\in\Lambda(f_0)$ which belongs to $\Gamma_{j_n,l_n}(f_0)$. On the other hand, by Lemma \ref{le-coordonees}, the sequence of analytic sets $(\Gamma_{j_n,l_n}(f_0))_{n\geq n_0}$ converges to $\mathcal F_{f_0}(\alpha(f_0)b(f_0)\chi(f_0))$. Since $|\alpha(f_0)b(f_0)|<\frac{1}{20}$, this implies that $|\pi\circ\delta_{f_0}(x_n(f_0))|<\frac{1}{20}$ for $n\geq n_1$ large enough and thus $|\pi\circ\delta_{f}(x_n(f))|<\frac{1}{10}$ for all $f\in M_0.$

As $M$ satisfies $\textbf{(3)}$ in Definition \ref{def-dag}, the persistence of proper intersections (see e.g. \cite[\textsection 12.3]{Chirka}) implies that the continuation $x_n(f)$ of $x_n(f_0)$ in $\Lambda(f)$ also lies on $\Gamma_{j_n,l_n}(f)$, which is thus well-defined. As observed above, all these functions $\{f\mapsto x_n(f)\}_{n\geq n_1}$ form a normal family. Hence, the same holds for the family
$$\left\{f\mapsto \frac{\pi\circ\delta_f(x_n(f))}{b(f)\alpha(f)}\right\}_{n\geq n_1}.$$
The result follows since, by Lemma \ref{le-Hf} and Lemma \ref{le-coordonees}, these functions are, locally on $M_0$, arbitrarily close to $\{f\mapsto\chi_{p(f)}^{j_n}\chi_{r(f)}^{l_n}\}_{n\geq 0}.$
\end{proof}

\begin{proposition}\label{prop-relation}
 There exists $\zeta\in\Sb^1$ and $\omega\in\Rb^*_+$ such that for all $f\in M$
$$\chi_{r(f)}=\zeta\chi_{p(f)}^{-\omega}.$$
\end{proposition}
\begin{proof}
Let $(j_n)_{n\geq0}$ and $(l_n)_{n\geq0}$ be two sequences as in Lemma \ref{le-normale} which we choose so that
$$\chi(f_0):=\lim_{n\to\infty}\chi_{p(f_0)}^{j_n}\chi_{r(f_0)}^{l_n}$$
is non-zero. By analytic continuation, it suffices to prove the result in a neighborhood of $f_0$. Let $M_0$ be the neighborhood of $f_0$ obtained by Lemma \ref{le-normale} where the family $\{f\mapsto\chi_{p(f)}^{j_n}\chi_{r(f)}^{l_n}\}_{n\geq 0}$ is normal. Let $\chi\colon M_0\to\Cb$ be a limit value and we can assume, up to take a subsequence, that for each $f\in M_0$, 
\begin{equation}\label{eq-conv}
\chi_{p(f)}^{j_n}\chi_{r(f)}^{l_n}\to\chi(f).
\end{equation}
Let $M_1$ be a simply connected neighborhood of $f_0$ on which $\chi$ does not vanish. Let $P(f)$ (resp. $R(f)$, resp. $Q(f)$) be a logarithm of $\chi_{p(f)}$ (resp. $\chi_{r(f)}$, resp. $\chi(f)$) on $M_1$. By \eqref{eq-conv}, the real parts of these functions satisfy on $M_1$
$$\lim_{n\to\infty}j_n\mathrm{Re}P(f)+l_n\mathrm{Re}R(f)=\mathrm{Re}Q(f)$$
and thus
$$\lim_{n\to\infty}\frac{j_n}{l_n}\mathrm{Re}P(f)+\mathrm{Re}R(f)=0.$$
Hence, if $\omega$ denotes a limit value of $(j_n/l_n)_{n\geq0}$ then $\mathrm{Re}R(f)=-\omega\mathrm{Re}P(f)$. This implies that there exists $t\in\Rb$ such that $R=-\omega P+it$ and so $\chi_{r(f)}=\zeta\chi_{p(f)}^{-\omega}$ with $\zeta:=e^{it}.$
\end{proof}
This gives precise information on the possible limit values for families of the form $(\Gamma_{j,l})$ obtained by Definition \ref{def-gammakl}.

\begin{lemma}\label{le-stereo}
There exist $t_0>0$ and a neighborhood $M_0\subset M$ of $f_0$ with the following property.
Let $(j_n)_{n\geq0}$ and $(l_n)_{n\geq0}$ be two increasing sequences such that $(\chi_{p(f_0)}^{j_n}\chi_{r(f_0)}^{l_n})_{n\geq0}$ converges to $\xi\chi_{p(f_0)}^t$ for some $t\in[t_0,+\infty[$ and $\xi\in\Sb^1$. Let $\Gamma$ be a family of polydiscs cutting $W^s_{p(f),loc}$ transversely at $b(f)$. Let $(\Gamma_{j,l})$ be the associated sequence of families of polydiscs obtained by Definition \ref{def-gammakl}. Then, there exists $n_0\in\Nb$ such that for $f\in M_0$, $(\Gamma_{j_n,l_n}(f))_{n\geq n_0}$ is well defined and converges to $\mathcal F(\alpha(f)b(f)\xi\chi_{p(f)}^t)$, uniformly on $M_0.$
\end{lemma}
\begin{proof}
Let $t_0>0$ be such that $|\alpha(f_0)\chi_{p(f_0)}^{t_0}|<\frac{1}{20}$ and let $M_0\subset M$ be a relatively compact neighborhood of $M$ such that $|\alpha(f)\chi_{p(f)}^{t_0}|<\frac{1}{10}$ for all $f\in M_0.$
By Proposition \ref{prop-relation}, there exist $\omega\in\Rb$ and $\zeta$ in $\Sb^1$ such that for all $f\in M$, $\chi_{r(f)}=\zeta\chi_{p(f)}^{-\omega}$. Hence, $\lim_{n\to\infty}\chi_{p(f_0)}^{j_n}\chi_{r(f_0)}^{l_n}=\xi\chi_{p(f_0)}^t$ implies that $(j_n-\omega l_n)_{n\geq0}$ and $(\zeta^{l_n})_{n\geq0}$ converge to $t$ and $\xi$ respectively. Therefore, for all $f\in M$, $(\chi_{p(f)}^{j_n}\chi_{r(f)}^{l_n})_{n\geq0}$ converges to $\xi\chi_{p(f)}^t$. As we assumed that $t\geq t_0$, we thus have that $|\alpha(f)b(f)\chi_{p(f)}^{j_n}\chi_{r(f)}^{l_n}|<\frac{1}{5}$ for all $f\in M_0$ and $n\geq n_1$ where $n_1\in\Nb$ is large enough. 

By Lemma \ref{le-Hf}, for all $f\in M$, $\Gamma_{j_n}(f)$ passes through $W^{cu}_{r(f),loc}\cap\mathcal F(H_f(s_{j_n+m-a}(f)))$,  where $H_f(s_{j_n+m-a}(f))=\alpha(f)b(f)\chi^{j_n}_{p(f)}+u_{j_n}(f)$ with $u_{j_n}(f)/\chi_{p(f)}^{j_n}$ converging to $0$, uniformly on $f\in M_0$. Thus, for $n_0\geq n_1$ large enough we have $|u_{j_n}(f)\chi_{r(f)}^{l_n}|<\frac{1}{5}$ on $M_0$ for all $n\geq n_0$. Hence, $|\chi_{r(f)}^{l_n}H_f(s_{j_n+m-a}(f))|<\frac{1}{2}$ and $(\Gamma_{j_n,l_n}(f))_{n\geq n_0}$ are well-defined. On the other hand, the convergence above implies that
$$\lim_{n\to\infty}\chi_{r(f)}^{l_n}H_f(s_{j_n+m-a}(f))=\alpha(f)b(f)\xi\chi_{p(f)}^t.$$
This, combined with Lemma \ref{le-coordonees}, implies that the sequence $(\Gamma_{j_n,l_n}(f))_{n\geq n_0}$ converges to $\mathcal F(\alpha(f)b(f)\xi\chi_{p(f)}^t)$.
\end{proof}

\subsection{Special holomorphic motion and constant multipliers}\label{sec-stereo}
From now on, we assume that $M\subset\Omega$ satisfies condition $(\dag)$ and we choose an element $f_0\in M$ such that $\overline{\langle\chi_{p(f_0)},\chi_{r(f_0)}\rangle}=\Cb^*$.

\begin{remark}\label{rk-dense}\normalfont
Observe that this last condition is equivalent to saying that $t$, $\omega$ and $1$ are linearly independent over $\Qb$, where $\chi_{r(f_0)}=e^{2i\pi t}\chi_{p(f_0)}^{-\omega}$. Hence, Proposition \ref{prop-relation} implies that  $\overline{\langle\chi_{p(f)},\chi_{r(f)}\rangle}=\Cb^*$ for all $f\in M$. This actually gives the last point in Theorem \ref{th-constant}.
\end{remark}

We now prove that this additional assumption on $M$ constrains the holomorphic motion of $\Lambda(f)$ to be very special.
\begin{proposition}\label{prop-stereo}
Let $(\Gamma(f))_{f\in M}$ be a family of polydiscs cutting $W^s_{p(f),loc}$ transversely at $b(f)$ such that $\Gamma(f)\subset W^u_{p(f)}$. Pick $x(f_0)\in\Lambda(f_0)\cap D_{f_0}$ such that $\pi(\delta_{f_0}(x(f_0)))=\alpha(f_0)b(f_0)\xi\chi_{p(f_0)}^t$ for some $t\in\Rb$ and $\xi\in\Sb^1$. Then, for all $f\in M$, the holomorphic continuation $x(f)$ in $\Lambda(f)$ of $x(f_0)$ lies in $D_f$ and satisfies $\pi(\delta_f(x(f)))=b(f)\alpha(f)\xi\chi_{p(f)}^{t}$.
\end{proposition}
\begin{proof}
Let $t_0>0$ and $M_0\subset M$ be as in Lemma \ref{le-stereo}. Let $x(f_0)\in\Lambda(f_0)\cap D_{f_0}$. By exchanging $x(f_0)$ by a preimage, we can assume that $\pi(\delta_{f_0}(x(f_0)))=\alpha(f_0)b(f_0)\xi\chi_{p(f_0)}^t$ with $t>t_0.$
 Since $\overline{\langle\chi_{p(f_0)},\chi_{r(f_0)}\rangle}=\Cb^*$, there exist $(j_n)_{n\geq0}$ and $(l_n)_{n\geq0}$ two increasing sequences such that $(\chi_{p(f_0)}^{j_n}\chi_{r(f_0)}^{l_n})_{n\geq0}$ converges to $\xi\chi_{p(f_0)}^t$. Let $\Gamma_{j_n,l_n}$ be the families of analytic sets associated to $\Gamma(f)\subset W^u_{p(f)}$. Lemma \ref{le-stereo} implies that $(\Gamma_{j_n,l_n}(f))_{\geq0}$ converges to $\mathcal F(\alpha(f)b(f)\xi\chi_{p(f)}^t)$, uniformly on $M_0$. Hence, if $x(f)$ intersects properly $\mathcal F_f(\alpha(f)b(f)\xi\chi_{p(f)}^t)$ then $x(f)$ would intersect properly $\Gamma_{j_n,l_n}(f)$ for $n\geq0$ large enough. This contradicts condition $(\dag)$ and thus $x(f)\in \mathcal F_f(\alpha(f)b(f)\xi\chi_{p(f)}^t)$ for all $f\in M_0$, i.e., $\pi(\delta_f(x(f)))=b(f)\alpha(f)\xi\chi_{p(f)}^{t}$. By analytic continuation, this equality holds on the whole space $M.$
\end{proof}

Using several homoclinic intersections, we obtain the following strong restriction on $\chi_p$ which, combined with Proposition \ref{prop-relation}, implies Theorem \ref{th-constant}
\begin{proposition}\label{prop-constant}
The function $\chi_p$ is constant on $M.$
\end{proposition}
\begin{proof}
Assume, by contradiction, that $\chi_p$ is not constant. This implies the existence of a small arc $\gamma\colon[0,1]\to M$ such that $\chi_{p(\gamma(s))}=re^{isa}$ for all $s\in[0,1]$, where $r\in\Cb^*$ and $a>0$ are two constants. In particular, for $n\in\Nb$ large, $s\mapsto\chi_{p(\gamma(s))}^n$ winds about $n/a$ times around $0$ on a small circle. We will use this fast variation in the argument together with Proposition \ref{prop-stereo} in order to obtain a contradiction.

By Assumption \ref{homoclinic}, there exist $K\geq1$ and $\tilde q(f)\in W^u_{p(f),loc}$ which is not a critical point for $f^K$ and such that $q(f):=f^K(\tilde q(f))\neq p(f)$ is a transverse homoclinic intersection in $W^s_{p(f),loc}$. We denote by $b_0(f)$ the point in $\Db$ such that $v_f(q(f),0)=b_0(f)$ and we have $b_0(f)\neq0$. As the homoclinic intersection is transverse, there exists a family of polydiscs $(\Gamma(f))_{f\in M}$ intersecting transversely $W^s_{p(f),loc}$ at $b_0(f)$ such that $\Gamma(f)\subset W^u_{p(f)}$. Moreover, as observed in Remark \ref{rk-homoclinic}, since $\tilde q(f)$ is not critical for $f^K$, 
 there exists a holomorphic injective map $g_f\colon\Db\to V_f$ such that
\begin{itemize}
\item $\Delta'_f:=g_f(\Db)$ is transverse to $W^u_{p(f),loc},$
\item $\Delta'_f$ is a graph above $W^s_{p(f),loc}$, i.e., the projection on the first coordinate of $v_f\circ g_f$ is the identity,
\item $f^K_{|\Delta'_f}$ is injective and $f^K(\Delta'_f)$ is a neighborhood of $q(f)$ in $W^{s}_{p(f),loc}.$
\end{itemize}
We also define $G_f\colon\Db\to\Db$ by $G_f=\pi\circ v_f\circ f^K\circ g_f$ where $\pi\colon\Cb^k\to\Cb$ is the first projection. Observe that $G_f$ is injective with $G_f(0)=b_0(f)$. Hence, there exists $\tilde\beta(f)\neq0$, which depends holomorphically on $f$, such that
$$G_f(s)=b_0(f)+\tilde\beta(f)s+o(s),$$
where $o(s)$ is uniform in $f$. Just as in Lemma \ref{le-Hf}, for $n\geq K$ large enough, there exist holomorphic functions $s'_n$ and $\delta_n$ such that
\begin{itemize}
\item $g_f(s'_n(f))\in f^{n-K}(\Gamma(f)),$
\item $s'_n(f)=b_0(f)\chi_{p(f)}^{n-K}+o(\chi_{p(f)}^{n}),$
\item $G_f(s'_n(f))=b_0(f)(1+\beta(f)\chi_{p(f)}^n+\delta_n(f))$, where $\delta_n(f)=o(\chi_{p(f)}^n)$ and $\beta(f):=\tilde\beta(f)\chi_{p(f)}^{-K}.$
\end{itemize}
We set $b_n(f):=G_f(s'_n(f))$ which corresponds to a transverse homoclinic intersection $f^n(\Gamma(f))\cap W^s_{p(f),loc}$ very close to $b_0(f)$. 

Now, if we apply Proposition \ref{prop-stereo} first to $b(f)=b_n(f)$ with $\xi=1$, $t=0$ and a second time to $b(f)=b_0(f)$, where $\xi_n$ and $t_n$ are chosen such that $\xi_n\chi_{p(f_0)}^{t_n}=1+\chi_{p(f_0)}^n\beta(f_0)+\delta_n(f_0)$, then we obtain for all $f\in M$
$$(b_0(f)+b_0(f)\chi_{p(f)}^n\beta(f)+b_0(f)\delta_n(f))\alpha(f)=b_0(f)\alpha(f)\xi_n\chi_{p(f)}^{t_n},$$
and thus
\begin{equation}\label{eq-tourne}
1+\chi_{p(f)}^n\beta(f)+\delta_n(f)=\xi_n\chi_{p(f)}^{t_n}.
\end{equation}
Observe that $\xi_n$ converges to $1$ and $t_n$ converges to $0$ since $1+\chi_{p(f_0)}^n\beta(f_0)+\delta_n(f_0)$ goes to $1.$

We choose an arc $\gamma\colon[0,1]\to M$ as above, small enough to insure that the argument of $s\mapsto\beta(\gamma(s))$ is almost constant and such that $\chi_{p(\gamma(s))}=re^{isa}$ for all $s\in[0,1]$, where $r\in\Cb^*$ and $a>0$ are two constants. In particular, when $n$ is large then $s\mapsto\chi_{p(\gamma(s))}^n\beta(\gamma(s))$ winds about $n/a$ times around $0$ on a small circle.

On the other hand, let $P\colon\Omega\to\Cb$ be a logarithm of $\chi_p$ on $\Omega$ and let $\theta_n\in\Rb$ converging to $0$ such that $\xi_n=e^{i\theta_n}$. Then $\xi_n\chi_{p(f)}^{t_n}-1=(t_nP(f)+i\theta_n)+o(t_nP(f)+i\theta_n)$ whose argument is essentially that of $t_nP(f)+i\theta_n$ which is never purely imaginary since $|\chi_{p(f)}|<1$ on $\Omega$. Hence, the equality \eqref{eq-tourne} cannot hold for $n\geq1$ large enough. This gives the desired contradiction. 
\end{proof}

The combination of Proposition \ref{prop-constant} with Proposition \ref{prop-stereo} says that the first coordinate (with respect to $\phi_f$) of the holomorphic motion of points in $\Lambda(f)\cap D_f$ is not only holomorphic in $f\in M$ but also in the starting point. In fact, in the coordinates given by $\phi_f$, this dependence is linear and our choice of normalization of $\phi_f$ implies that it is constant.
\begin{corollary}\label{cor-constant}
If $x(f_0)\in\Lambda(f_0)\cap D_f$ then, for all $f\in M$
$$\pi(\delta_f(x(f)))=\pi(\delta_{f_0}(x(f_0))).$$
\end{corollary}
\begin{proof}
As in the proof of Proposition \ref{prop-constant}, let $q(f)$ be the homoclinic intersection and $b_0(f)$ the corresponding point in $\Db$. By Proposition \ref{prop-constant} and Proposition \ref{prop-stereo}, if 
$$\pi(\delta_{f_0}(x(f_0)))=\alpha(f_0)b_0(f_0)s,$$
for some $s\in\Cb$, then
$$\pi(\delta_{f}(x(f)))=\alpha(f)b_0(f)s.$$
In other words
$$\pi(\delta_{f}(x(f)))=\pi(\delta_{f_0}(x(f_0)))\frac{b_0(f)\alpha(f)}{b_0(f_0)\alpha(f_0)}.$$
On the other hand, in order to normalize $\phi_f$ we had chosen $r'\in\Lambda$ in \S~\ref{sec-loc} close enough to $r$ such that
 $$\pi(\delta_{f}(r'(f)))=\pi(\delta_{f_0}(r'(f_0))).$$
 Hence, $\frac{b_0(f)\alpha(f)}{b_0(f_0)\alpha(f_0)}$ is constantly equal to $1.$
\end{proof}

\subsection{Construction of the conjugacy}\label{sec-iso}
We will first construct local conjugacies between elements of $M$ and then extend them in a neighborhood of the small Julia set $J_k$. 
This type of problem is classical in one variable complex dynamics. See in particular \cite{buff-epstein} where Buff-Epstein obtained at the end a global conjugacy outside the exceptional sets. In our context we have much less information on the dynamics outside the small Julia set and, even if the counterpart of \cite{buff-epstein} probably holds in higher dimension, our final argument relies strongly on the fact that we are working with a family.

Let $M\subset\Omega$ be a subvariety which satisfies $(\star)$. The difference with condition $(\dag)$ is that $(f)_{f\in M}$ is supposed to be simply connected and stable in the sense of Berteloot-Bianchi-Dupont \cite{BBD}. Observe that in \cite[Theorem 1.1]{BBD}, the parameter space has to be an open subset of $\mathrm{End}_d^k$. However, this restriction has been overcome by Bianchi in the broader setting of polynomial-like maps with large topological degree \cite{bianchi-poly-like}. The key notion in \cite{BBD,bianchi-poly-like} for what follows is the \emph{equilibrium lamination}. To introduce it, we first consider the set
$$\mathcal J:=\left\{\gamma\colon M\to\Pb^k\ \vrule\ \gamma\ \text{is holomorphic and}\ \gamma(f)\in J_k(f)\ \text{for every}\ f\in M\right\}.$$
The family $(f)_{f\in M}$ induces naturally a self-map $F$ of $\mathcal J$ by setting $F(\gamma)(f):=f(\gamma(f)).$
\begin{definition}\label{def-lami}
An \emph{equilibrium lamination} is a relatively compact subset $\mathcal L$ of $\mathcal J$ such that
\begin{itemize}
\item[\textbf{(1)}] $\gamma(f)\neq\gamma'(f)$ for all $f\in M$ if $\gamma,\gamma'\in\mathcal L$ with $\gamma\neq\gamma',$
\item[\textbf{(2)}] for every $f\in M$, the equilibrium measure of $f$ gives full mass to $\{\gamma(f)\ |\ \gamma\in\mathcal L\},$
\item[\textbf{(3)}] for every $f\in M$ and $\gamma\in\mathcal L$, $\gamma(f)$ is not a critical point of $f,$
\item[\textbf{(4)}] $\mathcal L$ is $F$-invariant and $F\colon\mathcal L\to\mathcal L$ is $d^k$ to $1.$
\end{itemize}
\end{definition}

One characterization of the stability of the family $(f)_{f\in M}$ given by \cite[Theorem C]{bianchi-poly-like} is that this family admits an equilibrium lamination.
A key step in proving that all elements of \( M \) are conjugate near their small Julia sets is to first construct local conjugacies near the repelling point \( r(f) \) that respect the equilibrium lamination.

\begin{lemma}\label{le-localconj}
Assume that \( M \) satisfies \((\star)\), and let \( \mathcal{L} \) denote the associated equilibrium lamination. Let \( f_0 \) and \( f_1 \) be two points in \( M \). For each \( i \in \{0,1\} \), there exist connected neighborhoods \( \tilde{U}_i \Subset U_i \) of \( r(f_i) \) with the following properties.
\begin{itemize}
\item $f_i(\tilde U_i)\cap U_i=\varnothing$ and $f_i^2$ is a biholomorphism between $\tilde U_i$ and $U_i,$
\item there exists a biholomorphism $\psi\colon U_0\cup f_0(\tilde U_0)\to U_1\cup f_1(\tilde U_1)$ such that
\begin{equation}\label{eq-conj}
f_0=\psi^{-1}\circ f_1\circ\psi \ \text{ and } \ f_0^2=\psi^{-1}\circ f_1^2\circ\psi \ \text{ on } \ \tilde U_0,
\end{equation}
\item if $\gamma\in\mathcal L$ verifies $\gamma(f_0)\in U_0$ then  $\gamma(f_1)=\psi(\gamma(f_0)),$
\end{itemize}
\end{lemma}
Observe that, since $f_0(\tilde U_0)\cap U_0=\varnothing$, the first equality in \eqref{eq-conj} is a consequence of the definition of $\psi$ on these sets. However, it will guarantee that $\psi$ gives a conjugacy between $f_0$ and $f_1$ on a neighborhood of the small Julia sets as soon as the same holds between $f_0^2$ and $f_1^2.$
\begin{proof}
For $i\in\{0,1\}$, let $\mathcal F_{f_i}$ be the foliation of $D_{f_i}$ defined in \S~\ref{sec-loc}. Observe that by Corollary \ref{cor-constant}, if $x(f_0)\in\Lambda(f_0)\cap D_{f_0}$ lies on the leaf $\mathcal F_{f_0}(c):=\phi_{f_0}(\pi^{-1}(c))$ then its continuation $x(f_1)$ lies on $\mathcal F_{f_1}(c):=\phi_{f_1}(\pi^{-1}(c))$. Replacing the family of polydisk $\Gamma(f)\subset W^u_{p(f)}$ by a similar family with $\Gamma(f)\subset\crit(f)$ given by Assumption \ref{critical}, we can extend this result to points in $J_k(f_0)\cap D_{f_0}$ coming from the equilibrium lamination. To be more precise, let $\gamma\in\mathcal L$ such that $\gamma(f_0)\in D_{f_0}$. Since sets of the form $(\Gamma_{j_n,l_n}(f_0))$ with $\Gamma(f_0)\subset\crit(f_0)$ can approximate every leaf $\mathcal F_{f_0}(c)$, and since $\gamma(f)$ is never in the postcritical set of $f$, a proper intersection argument shows that $\pi\circ\delta_f(\gamma(f))$ is independent of $f$.

In order to define $\psi$, it suffices to find, for $i\in\{0,1\}$, $k-1$ foliations $(\mathcal G^j_{f_i})_{1\leq j\leq k-1}$ near $r(f_i)$ which satisfy the same invariance property and such that $(\mathcal F_{f_i},\mathcal G^1_{f_i},\ldots,\mathcal G^{k-1}_{f_i})$ defines local coordinates near $r(f_i)$. For this last condition, it is sufficient to check that the $k$ tangent spaces at $r(f_i)$ of these $k$ foliations form a family of $k$ linearly independent hyperplanes.

To this aim, observe first that by Assumption \ref{repelling}, $p(f_i)$ is in the domain of linearization of $r(f_i)$ and thus, there exists $n_0\geq1$ such that $f_i^{n_0}$ sends biholomorphically an open subset $V_i\subset D_{f_i}$ to a neighborhood $V'_i$ of $p(f_i)$. We denote by $v_i\colon V'_i\to V_i$ the associated inverse branch of $f_i^{n_0}$. Moreover, the cone condition in Assumption \ref{hyperbolic} ensures that the leaves of $\mathcal G^0_{f_i}:=f_i^{n_0}(\mathcal F_{f_i|V_i})$ are transverse to $W^s_{p(f_i),loc}$. On the other hand, recall that $(\dag)$ holds on $M$, so there exist $m\in\Nb$ and $x(f_i)\in W_{p(f_i),loc}^u\cap\Lambda(f_i)$ such that $f_i^{m}(x(f_i))=r(f_i)$ and that, by increasing $m$ if necessary, we can assume that $f_i^m$ sends biholomorphically a neighborhood of $x(f_i)$ in $W_{p(f_i),loc}^u$ to a vertical graph $W_m(f_i)$ in $D_{f_i}$. Hence, by the inclination lemma, there exist $n_1\geq1$, a neighborhood $U_i\subset D_{f_i}$ of $r(f_i)$ and a small open set $V_i''\subset V_i'$ close to $p(f_i)$ 
 such that
\begin{itemize}
\item $U_i\subset f_i^2(U_i),$
\item $f^{n_1}_i\colon V_i''\to U_i$ is a biholomorphism whose inverse is denoted by $u_i,$
\item the leaves of $\mathcal G^1_{f_i}:=f_i^{n_1}(\mathcal G^0_{f_i|V''_i})$ are all $C^1$-close to $W_m(f_i)$.  In particular, the point \textbf{(2)} in Definition \ref{def-dag} implies that the tangent space of the leaf of $\mathcal G^1_{f_i}$ containing $r(f_i)$ is a generic hyperplane for $D_{r(f_i)}f^2_i$.
\end{itemize}

This last point, together with our choice of $m$, ensures that each leaf of $\mathcal G^1_{f_i}$ intersects $\Lambda(f_i)$. Moreover, $\mathcal G^1_{f_i}$ has the same invariance property as $\mathcal F_{f_i}$, i.e., if, for some $\gamma\in\mathcal L$, $\gamma(f_0)$ is in $U_0$ and lies on a certain leaf of $\mathcal G^1_{f_0}$ then $\gamma(f_1)$ lies on the corresponding leaf of $\mathcal G^1_{f_1}$. To be more precise, first observe that, by possibly reducing each $U_i$, we can assume that $U_1=\phi_{f_1}\circ\delta_{f_0}(U_0)$. Moreover, as in the beginning of this proof, the fact that each leaf of $\mathcal F_{f_i}$ can be approximated by $\Gamma_{j,l}(f_i)$ in the postcritical set of $f_i$ and properties \textbf{(3)} and  \textbf{(4)} in Definition \ref{def-lami} imply that if $\gamma(f_0)\in U_0$ then
\[\pi\circ\delta_{f_0}\circ v_0\circ u_0(\gamma(f_0))=\pi\circ\delta_{f_1}\circ v_1\circ u_1(\gamma(f_1)).\]

The other foliations are simply defined as $\mathcal G_{f_i}^{j}:=\left(f_i^{2(j-1)}(\mathcal G^1_{f_i})\right)_{|U_i}$. They also have the above invariance property since the same arguments imply
\[\pi\circ\delta_{f_0}\circ v_0\circ u_0\circ g_0^{j-1}(\gamma(f_0))=\pi\circ\delta_{f_1}\circ v_1\circ u_1\circ g_1^{j-1}(\gamma(f_1)),\]
if $\gamma\in\mathcal L$ with $\gamma(f_0)\in U_0$, where $g_i$ is the local inverse of $f_i^2$ near $r(f_i).$
Furthermore, the fact that the leaf of $\mathcal G^1_{f_i}$ containing $r(f_i)$ is a generic hyperplane for $D_{r(f_i)}f^2_i$ ensures that the tangent spaces of $\mathcal F_{f_i}$, $\mathcal G^1_{f_i},\ldots,\mathcal G^{k-1}_{f_i}$ at $r(f_i)$ are $k$ linearly independent hyperplanes. Hence, possibly by reducing $U_i$, these foliations define coordinates on $U_i$, i.e., since $\mathcal F_{f_i}$ (resp. $\mathcal G^j_{f_i}$) corresponds to the fibration defined by $\pi\circ\delta_{f_i}$ (resp. $\pi\circ\delta_{f_i}\circ v_i\circ u_i\circ g_i^{j-1}$), there exists an open subset $\hat U_i\subset\Cb^k$ such that the holomorphic map $\psi_i\colon U_i\to\hat U_i$ defined by
\[\psi_i(y)=\left(\pi\circ\delta_{f_i}(y),\pi\circ\delta_{f_i}\circ v_i\circ u_i(y),\cdots,\pi\circ\delta_{f_i}\circ v_i\circ u_i\circ g_i^{k-2}(y)\right)\]
is biholomorphic. 

Possibly by reducing again these sets, $\psi:=\psi_1^{-1}\circ\psi_0$ is a biholomorphism between $U_0$ and $U_1$. Furthermore, the discussion above and Corollary \ref{cor-constant} imply that if $\gamma\in\mathcal L$ satisfies $\gamma(f_0)\in U_0$ then $\gamma(f_1)\in U_1$ and $\gamma(f_1)=\psi(\gamma(f_0))$. In particular, $\psi(r(f_0))=r(f_1)$. Thus, using that these points are $2$-periodic and possibly by changing one last time $U_0$ and $U_1$, we can assume that, for $i\in\{0,1\}$, there exists a connected neighborhood $\tilde U_i\Subset U_i$ of $r(f_i)$ such that $f_i(\tilde U_i)\cap U_i=\varnothing$ and $f_i^2$ defines a biholomorphism between $\tilde U_i$ and $U_i$. This allows us to extend $\psi$ to $f_0(\tilde U_0)$ by $\psi(y):=f_1\circ\psi(f_0^{-1}(y))$ which artificially gives
\[f_0=\psi^{-1}\circ f_1\circ\psi \ \text{ on } \ \tilde U_0.\]
On the other hand, coming back to $f_0^2$ and $f_1^2$, the fact that $\psi(\gamma(f_0))=\gamma(f_1)$ for each $\gamma\in\mathcal L$ with $\gamma(f_0)\in U_0$ implies that
\begin{equation}\label{eq-conj2}
f_0^2=\psi^{-1}\circ f_1^2\circ\psi \ \text{ on } \ \tilde U_0\cap\Lambda(f_0).
\end{equation}
Since Assumption \ref{hyperbolic} guarantees that the points of the blender lie in the small Julia set which is not contained in an analytic subset of $\tilde U_0$, \cite{fs-cdhd2}, then $\tilde U_0\cap\{\gamma(f_0)\ |\ \gamma\in\mathcal L\}$ is not contained in an analytic subset and the equality \eqref{eq-conj2} holds throughout $\tilde U_0.$
\end{proof}
To emphasize the dependency of $\psi$ on $f_1$, in what follows we will denote by $\psi_f$ the corresponding map for $f\in M$ where $f_0$ stays fixed.

\begin{lemma}\label{le-unbranche}
The closure $\overline{\mathcal L}$ of $\mathcal L$ is an unbranched lamination.
 In particular, $\psi_f$ extends to a conjugacy between $J_k(f_0)$ and $J_k(f).$
\end{lemma}
\begin{proof}
Let $f_1\in M$ and let $U_i$, $i\in\{0,1\}$, be as in Lemma \ref{le-localconj}. Let $(\gamma_n)_{n\geq0}$ and $(\rho_n)_{n\geq0}$ be two sequences in $\mathcal L$ which converge toward two maps from $M$ to $\Pb^k$, $\gamma$ and $\rho$ respectively. Assume further that $\gamma(f_1)=\rho(f_1)$. Our first aim is to show that $\gamma=\rho$ on $M$.

Let $N\geq0$ be such that there exists $y\in f_1^{-N}(\gamma(f_1))\cap U_1$. For $n\geq0$ large enough, let $y_n\in U_1$ (resp. $z_n\in U_1$) be such that $f_1^N(y_n)=\gamma_n(f_1)$, $f_1^N(z_n)=\rho_n(f_1)$ and
$$\lim_{n\to\infty}y_n=\lim_{n\to\infty}z_n=y.$$
The point \textbf{(4)} in Definition \ref{def-lami} gives the existence of two sequences $(\tilde\gamma_n)_{n\geq0}$ and $(\tilde\rho_n)_{n\geq0}$ in $\mathcal L$ such that $F^N(\tilde\gamma_n)=\gamma_n$, $F^N(\tilde\rho_n)=\rho_n$ and $\tilde\gamma_n(f_1)=y_n$, $\tilde\rho_n(f_1)=z_n$. Up to a subsequence, we can assume that $(\tilde\gamma_n)_{n\geq0}$ and $(\tilde\rho_n)_{n\geq0}$ converge to two maps $\tilde\gamma$ and $\tilde\rho$. Since, by Lemma \ref{le-localconj} $\psi_{f_1}(\tilde\gamma_n(f_0))=\tilde\gamma_n(f_1)$ and $\psi_{f_1}(\tilde\rho_n(f_0))=\tilde\rho_n(f_1)$ we have $\psi_{f_1}(\tilde\gamma(f_0))=\psi_{f_1}(\tilde\rho(f_0))$ and thus $\tilde\gamma(f_0)=\tilde\rho(f_0)$ by injectivity of $\psi_{f_1}$. By applying $F^N$, we also have $\gamma(f_0)=\rho(f_0)$. The same arguments with an arbitrary map $f\in M$ give $\gamma=\rho$. This proves that $\overline{\mathcal L}$ is unbranched.

From this, we can define, for every $\gamma\in\overline{\mathcal L}$, $\psi_f(\gamma(f_0)):=\gamma(f)$. Since $\overline{\mathcal L}$ is unbranched, it extends $\psi_f$ as a conjugacy between $J_k(f_0)$ and $J_k(f)$.
\end{proof}
The extension of $\psi_f$ to a neighborhood of $J_k(f_0)$ comes from the following partial generalization of \cite{buff-epstein} to higher dimensions.

\begin{proposition}\label{prop-conj}
Let $f_0$ and $f_1$ be two endomorphisms of $\Pb^k$ of degree $d\geq2$. Assume there exist an open set $V_0$ and a continuous map $\psi\colon V_0\cup J_k(f_0)\to\Pb^k$ such that
\begin{itemize}
\item $\psi_{|J_k(f_0)}$ is a homeomorphism from $J_k(f_0)$ to $J_k(f_1)$ such that $\psi\circ f_0=f_1\circ\psi$ on $J_k(f_0),$
\item $V_0\cap J_k(f_0)\neq\varnothing$ and $\psi_{|V_0}$ is holomorphic.
\end{itemize}
Assume also that the exceptional set $\mathcal E(f_0)$ of $f_0$ is disjoint from $J_k(f_0)$. Then, there exist two open neighborhoods $N_1\subset N_2$ of $J_k(f_0)$ such that
\begin{itemize}
\item $f_0(N_1)\subset N_2$,
\item $\psi$ extends to a holomorphic map $\tilde\psi$ on $N_2$ such that $f_1\circ\tilde\psi=\tilde\psi\circ f_0$ on $N_1$.
\end{itemize}
\end{proposition}
As a first step, we show that the set of points where $\psi$ admits a local holomorphic extension is invariant under the dynamics. To this end, we need to lift small paths via $f_0$ in a controlled way. This is done through the following two elementary lemmas, where we denote by $C$ the critical set of $f_0$, set $A := f_0(C)$ its set of critical values, and define $B := f_0^{-1}(A)$.
\begin{lemma}\label{le-covering}
Let $y\in\Pb^k\setminus B$, and let $\gamma'\colon[0,1]\to\Pb^k\setminus A$ be a path such that $\gamma'(0)=f_0(y)$. Then there exists a unique path $\gamma\colon[0,1]\to\Pb^k\setminus B$ such that $\gamma(0)=y$ and $f_0(\gamma(t))=\gamma'(t)$ for all $t\in[0,1].$
\end{lemma}
\begin{proof}
This comes from the fact that $f_0\colon\Pb^k\setminus B\to\Pb^k\setminus A$ is a covering map.
\end{proof}

\begin{lemma}\label{le-vois}
Let $y\in\Pb^k$ and let $V$ be a neighborhood of $y$. There exists a connected open neighborhood $V_y$ (resp. $W_y$) of $y$ (resp. of $f_0(y)$) such that
\begin{itemize}
\item $V_y\subset V$ and $W_y\subset f_0(V_y),$
\item if $\gamma'\colon[0,1]\to W_y\setminus A$ is a path and $z\in f_0^{-1}(\gamma'(0))\cap V_y$ then there exists a path $\gamma$ in $V_y$ with $\gamma(0)=z$ and $f_0(\gamma(t))=\gamma'(t)$ for all $t\in[0,1].$
\end{itemize}
\end{lemma}
Notice that in the following proof we use a Lojasiewicz type inequality but the fact that $f$ is finite and open is sufficient.
\begin{proof}
Let $V_y\subset V$ be a connected open neighborhood of $y$ such that $f_0^{-1}(f_0(y))\cap\partial V_y=\varnothing$, i.e., $\dist(\partial V_y,f_0^{-1}(f_0(y)))=a$ with $a>0$. A Lojasiewicz type inequality (\cite[Corollary 4.12]{fs-cdhd2} when $k=2$) gives that there exists a constant $c>0$, depending only on $f_0$, such that
\[\dist(f_0(\partial V_y),f_0(y))\geq ca^{d^k}.\]
Since $f_0$ is an open mapping, there exists $\epsilon>0$ such that $\epsilon<ca^{d^k}$ and $W_y:=B(f_0(y),\epsilon)\subset f_0(V_y)$. In particular, $f_0(\partial V_y)\cap W_y=\varnothing.$
Hence, if $\gamma'\colon[0,1]\to W_y\setminus A$ is a path and $z\in f_0^{-1}(\gamma'(0))\cap V_y$ then by Lemma \ref{le-covering} we can lift $\gamma'$ to a path $\gamma$ in $\Pb^k$ such that $\gamma(0)=z$. Since $\gamma'([0,1])\subset W_y$ and $f_0(\partial V_y)\cap W_y=\varnothing$, we must have $\gamma([0,1])\subset V_y.$
\end{proof}

\begin{lemma}\label{le-image}
The set of points $y$ in $J_k(f_0)$ where $\psi$ admits a holomorphic extension in a neighborhood of $y$ is $f_0$-invariant.
\end{lemma}
\begin{proof}
Let $y\in J_k(f_0)$ be such a point and let $V$ be a neighborhood of $y$ where $\psi$ extends holomorphically. The interesting case is when $y$ is a critical point of $f_0$. Let $V_y$ and $W_y$ be the connected open neighborhood of $y$ and $f_0(y)$ respectively given by Lemma \ref{le-vois}. Let $z\in W_y\setminus A$ and let $z_1$ and $z_2$ be two points in $f_0^{-1}(z)\cap V_y$. The goal is to show that $f_1(\psi(z_1))=f_1(\psi(z_2))$ since in that case, we can define $\psi$ on $W_y\setminus A$ using local inverse branches of $f_0$ and the definition will extend to $W_y$.

Since $y$ is in $J_k(f_0)$, the same holds for $f_0(y)$. The fact that $J_k(f_0)$ is nowhere pluripolar \cite{fs-cdhd2} implies the existence of $w\in (W_y\cap J_k(f_0))\setminus A$. Let $\gamma'$ be a simple path in $W_y\setminus A$ between $z$ and $w$. By Lemma \ref{le-vois}, it admits two lifts $\gamma_1$ and $\gamma_2$ in $V_y$ such that $\gamma_1(0)=z_1$ and $\gamma_2(0)=z_2$. The end points $w_1:=\gamma_1(1)$ and $w_2:=\gamma_2(1)$ are preimages of $w$ and thus are in $J_k(f_0)$. Since $\gamma'$ is simple, by analytic continuation there exist a connected neighborhood $\Omega$ of $\gamma'([0,1])$ and two holomorphic maps, $g_1$ and $g_2$, from $\Omega$ to $V_y$ such that for $i\in\{1,2\}$,
\begin{itemize}
\item $g_i$ is an inverse branch of $f_0$, i.e., $f_0\circ g_i=\id_\Omega$,
\item $\gamma_i(t)=g_i(\gamma'(t))$ for all $t\in[0,1]$.
\end{itemize}
Since $\psi_{|J_k(f_0)}$ conjugates $f_0$ to $f_1$ on $J_k(f_0)$, we have $f_1\circ\psi\circ g_i=\psi$ on $\Omega\cap J_k(f_0)$ for $i\in\{1,2\}$. Hence, the fact that $\Omega\cap J_k(f_0)$ is not pluripolar and the connectedness of $\Omega$ implies that $f_1\circ\psi\circ g_1=f_1\circ\psi\circ g_2$ on $\Omega$. In particular $f_1(\psi(z_1))=f_1(\psi(g_1(z)))=f_1(\psi(g_2(z)))=f_1(\psi(z_2))$.
\end{proof}

From this, the proof of Proposition \ref{prop-conj} is identical to the one of \cite[Lemma 3]{buff-epstein} but we include it here for completeness.
\begin{proof}[Proof of Proposition \ref{prop-conj}]
Since $J_k(f_0)\cap\mathcal E(f_0)=\varnothing$, there exists $N\geq1$ such that $J_k(f_0)\subset f_0^N(U_1)$. Hence, Lemma \ref{le-image} implies that for all $y\in J_k(f_0)$ there is a holomorphic extension $\psi_y$ of $\psi$ in a neighborhood of $y$. In particular, there exists $r_y>0$ such that $\psi_y$ is defined on $B(y,3r_y)$. Observe that if $y,z\in J_k(f_0)$ are such that $r_y\leq r_z$ and $B(y,r_y)\cap B(z,r_z)\neq\varnothing$ then $B(y,r_y)\subset B(z,3r_z)$. In particular, by non-pluripolarity of $B(y,r_y)\cap J_k(f_0)$, we have that $\psi_y=\psi_z$ on $B(y,r_y)$. Hence, $\psi$ has a holomorphic extension $\tilde\psi$ on
\[N_2:=\bigcup_{y\in J_k(f_0)}B(y,r_y).\]
By continuity of $f_0$, there exists an open neighborhood $N_1\subset N_2$ of $J_k(f_0)$ such that $f_0(N_1)\subset N_2$. We can also assume that each connected component of $N_1$ intersects $J_k(f_0)$. If $N$ is such connected component then $f_1\circ\tilde\psi=\tilde\psi\circ f_0$ on $N\cap J_k(f_0)$ by definition of $\psi$ and thus by analytic continuation $f_1\circ\tilde\psi=\tilde\psi\circ f_0$ on $N$.
\end{proof}

This allows us to complete to proof of Theorem \ref{th-isotrivial}.
\begin{proof}[Proof of Theorem \ref{th-isotrivial}]
Let $M$ be an analytic subset of $\Omega$ satisfying $(\star)$. Let $f_0$ and $f_1$ be two elements of $M$. By Lemma \ref{le-localconj} and Lemma \ref{le-unbranche}, there exists a map $\psi$, given on $J_k(f_0)$ by the unbranched holomorphic motion $\overline{\mathcal L}$, which satisfies the assumptions of Proposition \ref{prop-conj}. Observe that we use Assumption \ref{exceptional} here, which ensures that $J_k(f_0)$ is disjoint from the exceptional set $\mathcal E(f_0)$. Hence, there are two neighborhoods $N_1\subset N_2$ of $J_k(f_0)$ and a holomorphic map $\tilde\psi$ on $N_2$ such that $f_1\circ\tilde\psi=\tilde\psi\circ f_0$ on $N_1$. This directly implies that all the periodic points in $J_k(f_0)$ can be followed holomorphically on $M$ and that their multipliers are constant on $M$.
\end{proof}

\section{Existence of a suitable open subset in $\mathrm{End}_d^k$}\label{sec-existence}
The aim of this section is to prove the following existence statement.
\begin{theorem}\label{th-existence}
There exist $(a,\epsilon,\sigma_3,\ldots,\sigma_k)\in(\Rb_{>0})^k$ and a small perturbation $f\in\mathrm{Poly}_d^k$ of the map $f_0:\mathbb{C}^k\to\mathbb{C}^k$ given by
\[f_0(z,w,y_3,\ldots,y_k)=(e^{i\pi/4}z+\epsilon w,a(w^2-1),\sigma_3y_3,\ldots,\sigma_ky_k)\]
such that $f$ admits a neighborhood $\Omega$ in $\mathrm{End}_d^k$ which satisfies all the assumptions described in \S~\ref{sec-ass}. 
\end{theorem}
\begin{remark}\normalfont
Since the map $f$ in Theorem \ref{th-existence} belongs to $\mathrm{Poly}_d^k$, the result also provides a non-empty open subset of $\mathrm{Poly}_d^k$ satisfying all the assumptions of \S~\ref{sec-ass}.
\end{remark}

The structure of the section is as follows. In \S~\ref{sec-q}, we recall elementary results about the dynamics of $w\mapsto a(w^2-1)$ when $|a|$ is large. In \S~\ref{sec-ifs}, we begin to study the $2$-dimensional case which is the most important one. In particular, Lemma \ref{le-ifs} and Lemma \ref{le-blender} settle the blender property for the repelling hyperbolic set $\Lambda$. Observe that for $d=2$, it is delicate to obtain a saddle point and a repelling hyperbolic set which form a heterodimensional cycle. This explains why we have to work with the second or the fourth iterates of our maps. \S~\ref{sec-ifs} is also devoted to the study of the degeneracy of these maps when the parameter $a$ goes to infinity. This is the key ingredient to check Assumption \ref{10} of \S~\ref{sec-ass}, which is by far the most difficult to obtain. The case of higher dimension is considered in \S~\ref{sec-dim} where the parameters are chosen more carefully, in particular to linearize in family the dynamics near the two periodic points $p$ and $r$. \S~\ref{sec-tan} is devoted to establishing point iii) in Assumption \ref{10}. Finally, we prove Theorem \ref{th-existence} in \S~\ref{sec-veri}.

\subsection{Dynamics of $a(w^2-1)$}\label{sec-q}
For $a\in\Cb^*$ we consider the polynomial map $q_a(w):=a(w^2-1)$. For $|a|$ large enough, $q_a$ is hyperbolic with a Cantor set as Julia set. In what follows, we will construct a blender for a map of the form $(z,w)\mapsto(g(z,w),q_a^4(w))$. To this end, we need to consider open sets with specific (but simple) combinatorics.

From now on, we fix $a\in\Cb^*$ with $|a|>10.$
\begin{lemma}\label{le-U}
There exists a neighborhood $U_+$ (resp. $U_-$) of $1$ (resp. $-1$) such that $q_{a|U_\pm}$ is a biholomorphism between $U_\pm$ and $\Db_3,$
and
\[D(\pm1,|a|^{-1})\subset U_\pm\subset D(\pm1,3|a|^{-1}).\]
In particular, $q_a$ admits a unique fixed point $\tilde w(a)\in U_-$.
\end{lemma}

This implies that the Julia set of $q_a$ is a Cantor set, equal to $J_{q_a}=\cap_{n\geq0}q_a^{-n}(U_+\cup U_-)$. Its dynamics corresponds to a (one-sided) full $2$-shift. However, our construction will not use the entire $J_{q_a}$ but two smaller hyperbolic sets.

The first one is simply the unique fixed point $\tilde w(a)\in U_-$. For the second one, let $g_+\colon\Db_3\to U_+$ and $g_-\colon\Db_3\to U_-$ be the two inverse branches of $q_a$ obtained in Lemma \ref{le-U}. From this point, we define the following open sets:
\begin{itemize}
\item $V_-:=g_-(U_+)$ and $V_+:=g_+(U_-)$,
\item $V_1:=q_a^{-2}(V_+)\cap V_+$, $V_2:=q_a^{-2}(V_-)\cap V_+$, $V_3:=q_a^{-2}(V_+)\cap V_-$ and $V_4:=q_a^{-2}(V_-)\cap V_-$.
\end{itemize}
They satisfy $q_a^2(V_-)=q_a^2(V_+)=\Db_3$ and $q_a^4(V_i)=\Db_3$ for $i\in\{1,2,3,4\}$. The definition of $V_i$ ensures that the associated maximal invariant sets are equal, i.e.,
\begin{equation}\label{eq-cantor}
E:=\bigcap_{n\geq0}q_a^{-2n}(V_+\cup V_-)=\bigcap_{n\geq0}q_a^{-4n}(\cup_{i=1}^4V_i).
\end{equation}
It is also a Cantor set where $q_a^2$ is conjugate to a full $2$-shift. Observe that if $\{+,-\}$ is the alphabet for $q_{a|J_{q_a}}$ then $q^2_{a|E}$ corresponds to the $2$-shift associated to $\{+-,-+\}$. This alternation will play an important role in the proof of Theorem \ref{th-existence}. In particular, the second coordinates of $r$ in Assumption \ref{repelling} will be the point $w_0(a)$ which is the unique $2$-periodic point of $q_a$ in $V_4\subset V_-$. We will need the following simple fact in the proof of Lemma \ref{le-tan}.
\begin{lemma}\label{le-reel}
If $a$ is real then 
 the $2$-periodic point $w_0(a)\in V_-$ is also real. Actually, the two inverse branches, $g_+\colon\Db_3\to U_+$ and $g_-\colon\Db_3\to U_-$ map real points into $\Rb$. In particular, $w_l(a):=(g_+\circ g_-)^l(w_0(a))$, are real for all $l\geq0.$
\end{lemma}
\begin{proof}
This simply comes from the fact that $g_\pm(w)=\pm\sqrt{1+w/a}$ and $w_0(a)=\lim_{n\to\infty}(g_-\circ g_+)^n(0).$
\end{proof}

All the subsets defined in this section depend on $a$. If necessary,  we will write $U_\pm(a)$, $V_\pm(a)$, $V_i(a)$ or $E(a)$ to emphasize on these dependencies.

\subsection{Perturbations of product maps and the IFS at infinity}\label{sec-ifs}
The construction in Theorem \ref{th-existence} starts from a skew product
\[{\mathcal F}_\lambda(z,w)=(\alpha z+\epsilon w+\beta zw,q_a(w)),\]
where $\lambda=(a,\alpha,\beta,\epsilon)$. Such a map does not extend to $\Pb^2$. The case of $\mathrm{End}_d^k$ with a general $k\geq2$ will be considered in \S~\ref{sec-dim}. Several objects denoted with stylized uppercase letters in this section (e.g. $\mathcal F_\lambda$, ${\mathcal P}(\lambda)$, $\mathcal R(\lambda)$) will corresponds to lowercase letters in \S~\ref{sec-dim} (e.g. $f_\lambda$, $p(\lambda)$, $r(\lambda)$).

A first observation is that ${\mathcal F}_{(a,\alpha,\beta,\epsilon)}$ and ${\mathcal F}_{(a,\alpha,\beta,\epsilon')}$ are globally conjugate if $\epsilon\neq0\neq\epsilon'$ by $(z,w)\mapsto(Cz,w)$ for some $C\neq0.$ The role of the parameter $\epsilon\neq0$ is just to rescale the dynamics in order to have the blender property above $\Db.$

If $\tilde w(a)$ denotes the unique fixed point of $q_a$ in $U_-(a)$ then ${\mathcal F}_\lambda$ has a fixed point
\[{\mathcal P}(\lambda)=\left(\frac{-\epsilon\tilde w(a)}{\alpha+\beta\tilde w(a)-1},\tilde w(a)\right),\]
which is repelling in a vertical direction and whose multiplier in the horizontal direction is $\alpha+\beta\tilde w(a)$, very close to $\alpha-\beta$ when $|a|$ is large. On the other hand, by the choice of the sets $V_+(a)$ and $V_-(a)$, the dynamics in the horizontal direction of the second iterate $\mathcal F^2_\lambda$ is mainly a dilatation of factor $\alpha^2-\beta^2$ on $\Cb\times(V_+(a)\cup V_-(a))$. Hence, in what follows we will choose $\alpha$ and $\beta$ in order to have $|\alpha-\beta|<1$, which implies that $\mathcal P(\lambda)$ is saddle, and $|\alpha^2-\beta^2|>1$ which ensures the existence of a repelling hyperbolic set $\Lambda(\lambda)$ for $\mathcal F_\lambda^2$. This hyperbolic set will have a blender property if $\alpha$, $\beta$ and $\epsilon$ are well chosen and it will project on $E(a)$. In order to check transversality properties, we will make $a$ goes to infinity. In this situation, the set $E(a)$ degenerates to $\{-1,1\}$ and the dynamics on $\Lambda(\lambda)$ degenerates to (the inverse of) an iterated function system (IFS) with $2$ generators.

To be more precise, the second iterate of ${\mathcal F}_\lambda$ is
\[{\mathcal F}_\lambda^2(z,w)=(z(\alpha^2+\alpha\beta(w+q_a(w))+\beta^2wq_a(w))+\epsilon(\alpha w+q_a(w)+\beta wq_a(w)),q_a^2(w)).\]
In particular, since
\[V_+(a)\subset q_a(V_-)=U_+(a)\subset D(1,3|a|^{-1})\quad  and \]
\[ V_-(a)\subset q_a(V_+)=U_-(a)\subset D(-1,3|a|^{-1}),\]
if $\hat\lambda:=(\alpha,\beta,\epsilon)\in\Cb^3$ and $R>0$ are fixed then for $|a|>10$ large enough, $\mathcal F^2_\lambda(z,w)$ is arbitrarily close to $(\phi^+_{\hat\lambda}(z),q_a^2(w))$ (resp. $(\phi^-_{\hat\lambda}(z),q_a^2(w))$) on $\Db_R\times V_+(a)$ (resp. on $\Db_R\times V_-(a)$) where
\[\phi^+_{\hat\lambda}(z)=(\alpha^2-\beta^2)z-\epsilon(\beta+1-\alpha),\quad \phi^-_{\hat\lambda}(z)=(\alpha^2-\beta^2)z-\epsilon(\beta+\alpha-1).\]
 From this, we define
\[\phi_{1,\hat\lambda}:=\phi^+_{\hat\lambda}\circ\phi^+_{\hat\lambda},\ \phi_{2,\hat\lambda}:=\phi^-_{\hat\lambda}\circ\phi^+_{\hat\lambda},\ \phi_{3,\hat\lambda}:=\phi^+_{\hat\lambda}\circ\phi^-_{\hat\lambda}\ \text{ and }\ \phi_{4,\hat\lambda}:=\phi^-_{\hat\lambda}\circ\phi^-_{\hat\lambda}\]
in order to have $\mathcal F^4_\lambda(z,w)\simeq(\phi_{j,\hat\lambda}(z),q_a^4(w))$ on $\Db_R\times V_j(a)$.

Now, we fix a small real $A>0$ and take $\alpha_0:=\zeta(1+A)$ and $\beta_0:=2A\zeta$ where $\zeta\in\Sb^1$. This gives $\alpha_0-\beta_0=\zeta(1-A)$ and thus the fixed point ${\mathcal P}(a,\alpha_0,\beta_0,\epsilon)$ is saddle for $|a|$ large enough. On the other hand, $\alpha_0^2-\beta_0^2=\zeta^2(1+2A-3A^2)$ which has modulus larger than $1$ if $A<2/3$. For the constant $\zeta\in\Sb^1$, following \cite[Lemma 4.4]{Dujardin_blender}, we will take $\zeta=e^{i\pi/4}$ in order to have a blender property for $\Lambda(\lambda)$. The following result can be seen as the counterpart in our context of this lemma using the vocabulary of \cite{Taflin_blender}.
\begin{lemma}\label{le-ifs}
Let $\zeta:=e^{i\pi/4}$ and $\epsilon_0=(20(\zeta-1))^{-1}$. Let $A\in(0,1/10]$ be small enough and let set $\hat\lambda_0:=(\alpha_0,\beta_0,\epsilon_0)$ where
\[\alpha_0=\zeta(1+A)\ \ \text{ and }\ \ \beta_0=2A\zeta.\]
Then, there exist four open sets $H_j$, $j\in\{1,2,3,4\}$ such that
\[\Db_2=\cup_{j=1}^4H_j,\ \overline{\phi_{j,\hat\lambda_0}(H_j)}\subset\Db_2\ \text{ and } \overline\Db\subset H_j.\]
\end{lemma}
\begin{proof}
For $\hat\lambda_1:=(\zeta,0,\epsilon_0)$, an easy computation gives
\[\phi_{1,\hat\lambda_1}(z)=-z+\frac{1+i}{20},\ \phi_{2,\hat\lambda_1}(z)=-z+\frac{i-1}{20},\]
\[ \phi_{3,\hat\lambda_1}(z)=-z+\frac{1-i}{20} \text{ and } \phi_{4,\hat\lambda_1}(z)=-z-\frac{1+i}{20}.\]
Moreover, if we define
\[H_j=\Db_{4/3}\cup\{z\in\Db_2\ ;\ |\arg(z\zeta^{-2j+1})|<\pi/3\},\]
then $\overline{\phi_{j,\hat\lambda_1}(H_j)}\subset\Db_2$. Actually, this comes from the inequalities $|2e^{i\pi/3}-\sqrt2/20|<2$ and $|\sqrt2/20|<1$. Since this inclusion is stable under small perturbations, if $A\in(0,1/10]$ is small enough and $j\in\{1,2,3,4\}$ then $\overline{\phi_{j,\hat\lambda_0}(H_j)}\subset\Db_2$ when $\hat\lambda_0=(\zeta(1+A),2A\zeta,\epsilon_0)$. On the other hand, $\Db_2=\cup_{j=1}^4H_j$ and $\overline\Db\subset\cap_{j=1}^4H_j$ follow from the definition of $H_j.$
\end{proof}

Since the properties in Lemma \ref{le-ifs} are stable under perturbations of the $\phi_{j,\hat\lambda_0}$, they persist in a small neighborhood of $\hat\lambda_0=(\alpha_0,\beta_0,\epsilon_0)$. From now on, we denote by $\hat M$ such small neighborhood of $\hat\lambda_0$ which is connected and where, moreover, for all $\hat\lambda=(\alpha,\beta,\epsilon)\in\hat M$
\begin{equation}\label{def-M}
1/20<|\epsilon|<1/10,\ |\alpha-\beta|<1 \text{ and } |\alpha^2|-|\beta^2|>1+A
\end{equation}
In particular, if $R:=A^{-1}$, the maps $\phi_{\hat\lambda}^\pm$ satisfies $\overline{\Db_R}\subset\phi_{\hat\lambda}^\pm(\Db_R).$

The next step is to define the hyperbolic set $\Lambda$ with the blender property which appears in Assumptions \ref{hyperbolic} and \ref{blender}.
\begin{lemma}\label{le-blender}
There exist $\rho>100$ and $\delta>100$ such that, if $|a|>\delta$ then for every $(\alpha,\beta,\epsilon)\in\hat M$ the map ${\mathcal F}_{\lambda}$, with $\lambda:=(a,\alpha,\beta,\epsilon)$, satisfies the following properties.
\begin{itemize}
\item On both $\Db_R\times V_-(a)$ and $\Db_R\times V_+(a)$ the map $\mathcal F^2_\lambda$ is injective, contracts the cone field $C_\rho$ and is expanding.  Moreover
\[\overline{\Db_R\times(U_+(a)\cup U_-(a))}\subset \mathcal F^2_\lambda(\Db_R\times V_\pm(a)).\]
In particular, the set $\Lambda(\lambda):=\cap_{n\geq0}\mathcal F^{-2n}_\lambda(\Db_R\times(V_-(a)\cup V_+(a)))$ is a hyperbolic repelling invariant set for $\mathcal F^2_\lambda.$
\item For $i\in\{1,2,3,4\}$, any vertical graph in $H_i\times V_i(a)$ tangent to the cone $C_\rho$ intersects $\Lambda(\lambda).$
\end{itemize}
Moreover, both statements are stable under small $C^1$-perturbations of ${\mathcal F}_\lambda.$
\end{lemma}
\begin{proof}
Let $\hat\lambda=(\alpha,\beta,\epsilon)$ be in $\hat{M}$. The key ingredient is that if $|a|$ is large enough then $\mathcal F^2_\lambda$ is arbitrarily close to the product map $(\phi^\pm_{\hat\lambda},q_a^2)$ on $\Db_R\times V_\pm(a)$ and $\mathcal F^4_\lambda$ is arbitrarily close to $(\phi_{j,\hat\lambda},q_a^4)$ on $H_j\times V_j(a)$. This gives that $\mathcal F^2_\lambda$ is expanding on $\Db_R\times V_\pm(a)$ and also injective since $q_a^2$ is injective on $V_\pm(a)$. Moreover, $\mathcal F^2_\lambda$ contracts the cone field $C_\rho$ on $\Db_R\times V_\pm(a)$ for $|a|$ large since the derivative of $\mathcal F^2_\lambda$ in the vertical direction is bounded from below by $|a|^2$ while the derivative in the horizontal direction is uniformly bounded from above on $\Db_R\times V_\pm(a)$. Hence, for every $\rho>0$ there exists $\delta>0$ such that $\mathcal F^2_\lambda$ contracts the cone field $C_\rho$ on $\Db_R\times V_\pm(a)$ as soon as $|a|>\delta.$

 We also have $\overline{\Db_R\times(U_+(a)\cup U_-(a))}\subset \mathcal F^2_\lambda(\Db_R\times V_\pm(a))$ since $\overline{U_+(a)\cup U_-(a)}\subset\Db_3=q^2_a(V_\pm(a))$ and $\overline{\Db_R}\subset\phi_{\hat\lambda}^\pm(\Db_R)$. From this, it is classical that $\Lambda(\lambda):=\cap_{n\geq0}\mathcal F^{-2n}_\lambda(\Db_R\times(V_-(a)\cup V_+(a)))$ is a hyperbolic repelling set for $\mathcal F_\lambda^2$. Actually, it is easy to see that $\Lambda(\lambda)$ is homeomorphic, via the second canonical projection, to the corresponding set $E(a)$ for $q_a^2$ defined by \eqref{eq-cantor} which is a Cantor set. Observe that, as in \eqref{eq-cantor}, we also have $\Lambda(\lambda)=\cap_{n\geq0}\mathcal F_\lambda^{-4n}(\Db_R\times(\cup_{j=1}^4V_j(a))).$

The second statement is the counterpart of \cite[Lemma 4.5]{Dujardin_blender} or \cite[Proposition 3.3]{Taflin_blender} in our setting and we only sketch the proof. Let $H_j$, $j\in\{1,2,3,4\}$, be the four open subsets of $\Db_2$ defined in Lemma \ref{le-ifs}. Recall that $\Db_2=\cup_{j=1}^4H_j$, $\overline\Db\subset H_j$ and $\overline{\phi_{j,\hat\lambda}(H_j)}\subset\Db_2$, where this last inclusion comes from our choice of $\hat{M}$. In particular, there exists $r>0$ such that
$\cup_{j=1}^4\overline{\phi_{j,\hat\lambda}(H_j)}\subset\Db_{2-r}$ and thus, if $|a|$ is large enough and we set $\mathcal F^4_\lambda(z,w)=(F_\lambda(z,w),q_a^4(w))$ then
\[\bigcup_{j=1}^4\overline{F_\lambda(H_j\times V_j(a))}\subset\bigcup_{j=1}^4H_j.\]
Let $\eta>0$ denote the Lebesgue number of this open cover. If $\rho>0$ is large enough then the projection on the first coordinate of a vertical graph tangent to $C_\rho$ has diameter less than $\eta$. Hence, if $j_0\in\{1,2,3,4\}$ and $\Gamma_0\subset H_{j_0}\times V_{j_0}(a)$ is a vertical graph tangent to $C_\rho$ then $\mathcal F^4_\lambda(\Gamma_0)$ contains a vertical graph $\Gamma_1$ in $H_{j_1}\times V_{j_1}(a)$ for some $j_1\in\{1,2,3,4\}$ which is tangent to $C_\rho$. By induction, we obtain a sequence of vertical graph $\Gamma_n\subset\mathcal F_\lambda^{4n}(\Gamma_0)$ in some $H_{j_n}\times V_{j_n}(a)$ and thus $\Gamma_0$ intersects $\cap_{n\geq0}\mathcal F_\lambda^{-4n}(\Db_R\times(\cup_{j=1}^4V_j(a)))=\Lambda(\lambda).$
\end{proof}

By construction, each point in $\Lambda(\lambda)$ is associated to a unique word $\omega$ in $\Sigma:=\{-1,1\}^\Nb$. To be more precise, let $g_{\lambda,+}\colon\Db_R\times\Db_3\to\Db_R\times V_+(a)$ and $g_{\lambda,-}\colon\Db_R\times\Db_3\to\Db_R\times V_-(a)$ be the two inverse branches of $\mathcal F^2_\lambda$. If we identify the symbol $+$ with $1$ and $-$ with $-1$ then a word $\omega=(\omega_n)_{n\geq0}\in\Sigma:=\{-1,1\}^\Nb$ induces a dynamical system $(g^n_{\lambda,\omega})_{n\geq0}$ where
\[g^n_{\lambda,\omega}:=g_{\lambda,\omega_0}\circ\cdots\circ g_{\lambda,\omega_{n-1}}.\]
Since the maps $g_{\lambda,\pm}$ are contracting, the sequence $(g_{\lambda,\omega}^n(y))_{n\geq0}$ has a limit, denoted by $x_\omega(\lambda)$, which is independent of $y\in\Db_R\times\Db_3$. The hyperbolic set $\Lambda(\lambda)$ corresponds exactly to $\{x_{\omega}(\lambda)\ ;\ \omega\in\Sigma\}$ and the repelling point ${\mathcal R}(\lambda)$ in Assumption \ref{repelling} is, in this situation, $x_\omega(\lambda)$ where $\omega=(-1)_{n\geq0}$. Observe that $x_\omega(\lambda)$ depends holomorphically on $\lambda$ and continuously on $\omega$ with respect to the product topology on $\Sigma.$

In order to check Assumption \ref{10}, we are interested in parameters $\lambda=(a,\alpha,\beta,\epsilon)$ and $\omega\in\Sigma$ where $x_\omega(\lambda)$ is a preimage of ${\mathcal R}(\lambda)$ lying on $W_{{\mathcal P}(\lambda),loc}^u$ with an additional condition on the multipliers. Here, we consider the local unstable manifold in $\Db_R\times D(-1,1/2)$, which is a vertical graph, see Lemma \ref{le-unstable} below. In what follows, we will study the degeneracy of these relations $x_\omega(\lambda)\in W_{{\mathcal P}(\lambda),loc}^u$ when $|a|$ tends to infinity. The general picture is that $\Lambda(\lambda)$ degenerates to the limit set of the IFS generated by two affine maps and $W_{{\mathcal P}(\lambda),loc}^u$ converges to a vertical linear hypersurface. To study this degeneracy, we introduce the set $D(\infty,r):=\{\infty\}\cup\{a\in\Cb\ ;\ |a|>1/r\}$ and we fixe $\rho>100$ and $\delta>100$ as in Lemma \ref{le-blender}
\begin{lemma}\label{le-xomega}
For each $\omega \in \Sigma$, the map $x_\omega$ extends holomorphically to $D(\infty, 1/\rho) \times \hat{M}$. If $\hat{\lambda} = (\alpha, \beta, \epsilon)$, then this extension is given by  
\[
x_\omega(\infty, \hat{\lambda}) = (z_\omega(\hat{\lambda}), \omega_0),
\]  
where
\begin{equation}\label{eq-zomega}
z_\omega(\hat{\lambda}) = \epsilon \mu \left( \frac{\beta}{1 - \mu} + (1 - \alpha) h_\omega(\mu) \right),
\end{equation}
with $\mu:=(\alpha^2-\beta^2)^{-1}$ and  $h_\omega(\mu) := \sum_{n \geq 0} \omega_n \mu^n$. Moreover, the map $x_\omega$ depends continuously on $\omega$ with respect to the product topology on $\Sigma.$
\end{lemma}
\begin{proof}
When $a$ converges to infinity, the maps $g_{\lambda,\pm}$ converge to $(\ell_{\hat\lambda,\pm},\pm1)$ where
\begin{equation}\label{eq-ell}
\ell_{\hat\lambda,+}(z)=\mu z+\nu_+\ \text{ and }\ \ell_{\hat\lambda,-}(z)=\mu z+\nu_-,
\end{equation}
with $\hat\lambda=(\alpha,\beta,\epsilon)$, $\mu:=(\alpha^2-\beta^2)^{-1}$ and $\nu_\pm:=\mu\epsilon(\beta\pm(1-\alpha))$. Hence, since the maps $g_{\lambda,\pm}$ and $\ell_{\hat\lambda,\pm}$ are contractions, the point $x_\omega(\lambda)$ converges to $(z_\omega(\hat\lambda),\omega_0)$ where 
\[z_\omega(\hat\lambda):=\lim_{n\to\infty}\ell_{\hat\lambda,\omega_0}\circ\cdots\circ\ell_{\hat\lambda,\omega_{n-1}}(y)\]
for any $y\in\Db_R$. This gives the desired extension (which is holomorphic by the Riemann extension theorem).


Moreover, the definition of $\ell_\pm$ in \eqref{eq-ell} ensures that
\[z_\omega(\hat\lambda)=\sum_{n\geq0}\nu_{\omega_n}\mu^n.\]
Using the definition of $\nu_\pm$, this gives
\begin{equation*}
z_\omega(\hat\lambda)=\epsilon\mu\left(\frac{\beta}{1-\mu}+(1-\alpha)h_\omega(\mu)\right),
\end{equation*}
where $h_\omega(\mu):=\sum_{n\geq0}\omega_n\mu^n$.

The continuity of \( \omega \mapsto x_\omega \) follows again from the fact that \( g_{\lambda, \pm} \) are two contractions.
\end{proof}
On the other hand, $\mathcal F_\lambda$ contracts $C_\rho$ on $\Cb\times D(-1,1/2)$ thus when $a$ goes to infinity, the unstable manifold $W_{{\mathcal P}(a,\alpha,\beta,\epsilon),loc}^u$ converges to a vertical line. More precisely, we have the following result in family.
\begin{lemma}\label{le-unstable}
There exists a closed analytic subvariety $W$ of $D(\infty, 1/\rho) \times \hat{M}\times\Db_R\times D(-1,1/2)$ that is vertical (i.e., its closure in a neighborhood of $D(\infty, 1/\rho) \times \hat{M}\times\Db_R\times D(-1,1/2)$ is disjoint from $D(\infty, 1/\rho) \times \hat{M}\times(\partial\Db_R)\times D(-1,1/2)$). More precisely, $W$ corresponds to a family of analytic sets $(W_\lambda)_{\lambda\in D(\infty, 1/\rho) \times \hat{M}}$ such that $W_\lambda$ is tangent to $C_\rho$ and, denoting  $\lambda=(a,\alpha,\beta,\epsilon),$ satisfies
\begin{itemize}
\item $W_{{\mathcal P}(\lambda),loc}^u$ if $a\neq \infty,$
\item $\{z=\frac{\epsilon}{\alpha-\beta-1}\}$ if $a=\infty.$
\end{itemize}
\end{lemma}
\begin{proof}
For this proof, if $a\neq\infty$ then we denote by $G_a\colon D(-1,1/2)\to\Cb$ the inverse branch of $q_a(w)=a(w^2-1)$ defined by $G_a(w)=-\sqrt{1+w/a}$. Hence, the inverse branch $\mathcal G_\lambda\colon\Cb\times D(-1,1/2)\to\C\times U_-(a)$ of $\mathcal F_\lambda$ when $\lambda=(a,\alpha,\beta,\epsilon)$ is
$$\mathcal G_\lambda(z,w)=\left(\frac{z-\epsilon G_a(w)}{\alpha+\beta G_a(w)},G_a(w)\right).$$
Both definitions can be extended holomorphically to $a=\infty$ with $G_\infty(w):=-1,$ but $\mathcal G_\lambda$ can no longer be seen as an inverse branch there. For $n\geq1,$ we define
$$W_n:=\{(\lambda,z,w)\in D(\infty, 1/\rho)\times \hat{M}\times\Cb\times D(-1,1/2)\ ;\ \pi\circ\mathcal G^n_\lambda(z,w)=0\},$$
where $\pi$ is the projection onto the first coordinate. Outside $\{a=\infty\},$ $W_n$ is the graph transform of $W_0:=D(\infty, 1/\rho)\times \hat{M}\times\{0\}\times D(-1,1/2)$ by $\mathcal F^n(\lambda,z,w):=(\lambda,\mathcal F^n_\lambda(z,w))$ and thus, the fibers have to converge to $W_{{\mathcal P}(\lambda),loc}^u.$ Since the fibers of $W_0$ are tangent to $C_\rho$ and $\mathcal F_\lambda$ contracts this cone, the unstable manifolds are also tangent to $C_\rho.$ Thus, they have to be vertical in $\Db_R\times D(-1,1/2)$ since ${\mathcal P}(\lambda)\in \Db_{1/4}\times D(-1,1/2)$ and $\rho>100$.

Above $\{a=\infty\}$, the fibers of $W_n$ are the vertical lines $\{z=f^n_{\alpha,\beta,\epsilon}(0)\}$, where $f_{\alpha,\beta,\epsilon}(z)=(\alpha-\beta) z-\epsilon.$ As this map is a contraction whose fixed point is $\frac{\epsilon}{\alpha-\beta-1}$, these lines converge to $\{z=\frac{\epsilon}{\alpha-\beta-1}\}.$

The fact that each fiber of \( W_n \) is a graph tangent to \( C_\rho \) ensures that the sequence \( (W_n)_{n \geq 1} \) lies in a compact family of closed analytic subsets of \( D(\infty, 1/\rho) \times \hat{M} \times \mathbb{D}_R \times \mathbb{D}_3 \). Any limit value \( W \) of this sequence (which is actually unique by the discussion above) provides the desired set in the statement.
\end{proof}

These information help us to understand the relation $x_\omega(\lambda)\in W_{{\mathcal P}(\lambda),loc}^u$. The following result will in particular imply Assumption \ref{free}.
\begin{lemma}\label{le-xomega2}
Let $(\alpha_0,\beta_0,\epsilon_0)$ be as in Lemma \ref{le-ifs}. There exists a non-empty connected open neighborhood $M$ of $(\infty,\alpha_0,\beta_0,\epsilon_0)$ in $D(\infty,1/\delta)\times\hat M$ such that for each $\omega\in\Sigma$, the analytic set
\[X_\omega:=\left\{\lambda\in D(\infty,1/\delta)\times\hat M\ ;\ x_\omega(\lambda)\in W^u_{{\mathcal P}(\lambda),loc}\right\}\]
is a (possibly empty) hypersurface and each irreducible component of $X_\omega$ that intersects $M$ also intersects $\{\infty\}\times\hat M.$
\end{lemma}
\begin{proof}
It is clear that $X_\omega$ is an analytic set of codimension at most $1$ (if not empty). Let assume that for some $\omega\in\Sigma$ we have $X_\omega= D(\infty,1/\delta)\times\hat M$. In particular, with $a=\infty$ we have for all $(\alpha,\beta,\epsilon)\in\hat M$, and after simplification by $\epsilon$, that
\[\mu\left(\frac{\beta}{1-\mu}+(1-\alpha)h_\omega(\mu)\right)=\frac{1}{\alpha-\beta-1},\]
where $\mu=(\alpha^2-\beta^2)^{-1}$. As $\hat M$ is open and as the radius of convergence of $h_\omega$ is $1$, this equality should hold for all $(\alpha,\beta)\in\Cb^2$ with $|\alpha^2-\beta^2|>1$ which is impossible with $\alpha=2$ and $\beta=1$ since the right hand side diverges.

Observe that we have proved the stronger result that the intersections between each $X_\omega$ and $\{\infty\}\times\hat M$ are proper. This will allow us to prove the second statement by contradiction. Assume there exist a sequence $(\omega_n)_{n\geq0}$ in $\Sigma$ and a sequence $(\lambda_n)_{n\geq0}$ in $\hat M$ converging toward $\lambda_\infty:=(\infty,\alpha_0,\beta_0,\epsilon_0)$ such that $\lambda_n$ belongs to an irreducible component $C_{\omega_n}$ of $X_{\omega_n}$ which is disjoint from $\{\infty\}\times\hat M$. Up to a subsequence, $(\omega_n)_{n\geq0}$ converges to some $\omega_\infty\in\Sigma.$ By Lemma \ref{le-xomega}, $\omega\mapsto x_\omega$ is continuous, thus $X_{\omega_n}$ converges to $X_{\omega_\infty}$ and $C_{\omega_n}$ converges to a union of irreducible components $C_{\omega_\infty}$ of $X_{\omega_\infty}$. Since $\lambda_n\in C_{\omega_n}$, we must have $\lambda_\infty\in C_{\omega_\infty}$, i.e., $C_{\omega_\infty}\cap\{\infty\}\times\hat M\neq\varnothing$. The fact that $C_{\omega_n}$ is disjoint from $\{\infty\}\times\hat M$ then contradicts the persistence of proper intersections (see e.g. \cite[\textsection 12.3]{Chirka}).
\end{proof}
The next step is to check that we have a dense set of maps where points i), ii) and iv) in Assumption \ref{10} are simultaneously satisfy. We denote by $\chi_{{\mathcal P}(\lambda)}$ (resp. $\chi_{{\mathcal R}(\lambda)}$) the eigenvalue of $D_{{\mathcal P}(\lambda)}{\mathcal F}_\lambda$ (resp. $D_{{\mathcal R}(\lambda)}{\mathcal F}_\lambda^2$) with the smallest modulus. A first observation, already made in Remark \ref{rk-dense}, is that the condition $\overline{\langle\chi_{{\mathcal P}(\lambda)},\chi_{{\mathcal R}(\lambda)}\rangle}=\Cb^*$ is equivalent to $1$, $\theta$ and $t$ being independent over $\Qb$, where $\chi_{{\mathcal R}(\lambda)}=e^{2i\pi\theta}\chi_{{\mathcal P}(\lambda)}^{t}$. This is fulfilled by a dense subsets of $(\theta,t)\in\Rb^2$. Hence, in order to have i), ii) and iv) in Assumption \ref{10} simultaneously it is sufficient to have the ``transversality'' property described in Lemma \ref{le-ass11} below, between two families of hypersurfaces $(Y_{\zeta,t})_{\zeta\in\Sb^1, t\in\Rb}$ and  $(X_\omega)_{\omega\in\Sigma}$. If $\zeta\in\Sb^1$ and $t\in\Rb$, we set
\[Y_{\zeta,t}:=\left\{\lambda\in D(\infty,1/\delta)\times\hat M\ ;\ \chi_{{\mathcal R}(\lambda)}=\zeta\chi_{{\mathcal P}(\lambda)}^{t}\right\}.\]
Observe that $\chi_{{\mathcal R}(a,\alpha,\beta,\epsilon)}\simeq\alpha^2-\beta^2$ and $\chi_{{\mathcal P}(a,\alpha,\beta,\epsilon)}\simeq\alpha-\beta$ near $a=\infty$ so $Y_{\zeta,t}$ is actually a hypersurface. The family $(Y_{\zeta,t})_{\zeta\in\Sb^1, t\in\Rb}$ defined a (possibly singular) foliation which is not holomorphic. On the other hand, $(X_\omega)_{\omega\in\Sigma}$ is parametrized by a Cantor set and depends continuously on $\omega$. Furthermore, the blender property of $\Lambda(\lambda)$ ensures that $(X_\omega)_{\omega\in\Sigma}$ covers $D(\infty,1/\delta)\times\hat M$. Actually, many points belong to two $X_\omega$ and $X_{\omega'}$ at the same time, which will greatly simplify the verification of Assumption \ref{10}.
\begin{lemma}\label{le-doublons}
Possibly by reducing $M$, for each $\lambda\in M$
\begin{itemize}
\item $W^u_{{\mathcal P}(\lambda),loc}\subset\Db_{1/4}\times D(-1,1/2)$,
\item there exist two words $\omega,\omega'\in\Sigma$ such that $\omega'\neq\omega$ and $\lambda\in X_\omega\cap X_{\omega'}$.
\end{itemize}
\end{lemma}
\begin{proof}
The first point follows from the facts that ${\mathcal P}(\lambda)\in \Db_{1/4}\times D(-1,1/2)$ and that $W^u_{{\mathcal P}(\lambda),loc}$ is almost a straight vertical graph when $|a|$ is large.

For the second point, observe that both $V_3(a)$ and $V_4(a)$ are contained in $D(-1,1/2)$ hence, since each $H_j$ in Lemma \ref{le-ifs} contains $D_{1/4}\subset\overline\Db$, the local stable manifold $W^u_{{\mathcal P}(\lambda),loc}$ intersects $H_3\times V_3(a)$ and $H_4\times V_4(a)$ in two vertical graphs tangent to $C_\rho$. By Lemma \ref{le-blender}, there exist two intersections between $W^u_{{\mathcal P}(\lambda),loc}$ and $\Lambda(\lambda)$.
\end{proof}

\begin{lemma}\label{le-ass11}
Let $M$ be as in Lemma \ref{le-doublons}. Let $\omega,\omega'\in\Sigma$, $\lambda\in M$ and $(\zeta,t)\in\Sb^1\times\Rb$. If there exist irreducible components $Z_\omega$ and $Z_{\omega'}$ of $X_\omega$ and $X_{\omega'}$ respectively such that $\lambda\in Z_\omega=Z_{\omega'}\subset Y_{\zeta,t}$ then $\omega=\omega'$.
\end{lemma}
\begin{proof}
In this situation, by Lemma \ref{le-xomega2}, $Z_\omega$ intersects $\{\infty\}\times\hat M$. As we have seen in the proof of Lemma \ref{le-xomega2}, a point $\lambda=(\infty,\alpha,\beta,\epsilon)\in\{\infty\}\times\hat M$ is in $X_\omega$ if and only if 
\[\epsilon\mu\left(\frac{\beta}{1-\mu}+(1-\alpha)h_\omega(\mu)\right)=\frac{\epsilon}{\alpha-\beta-1},\]
where $\mu=(\alpha^2-\beta^2)^{-1}=\chi_{{\mathcal R}(\lambda)}^{-1}$ and $h_\omega(\mu)=\sum_{n\geq0}\omega_n\mu^n$. The relation $Z_\omega=Z_{\omega'}$ implies that on $\hat Z_\omega:=Z_\omega\cap\{\infty\}\times\hat M$, which has dimension at least $2$, $h_\omega(\mu)=h_{\omega'}(\mu)$. If $\omega\neq\omega'$ then these two power series are different and $\mu$ has to be constant on $\hat Z_\omega$. On the other hand, $Z_\omega\subset Y_{\zeta,t}$ hence $\chi_{{\mathcal R}(\lambda)}=\zeta\chi_{{\mathcal P}(\lambda)}^{t}$ on $\hat Z_\omega$. Since $\chi_{{\mathcal R}(\infty,\alpha,\beta,\epsilon)}=\alpha^2-\beta^2=\mu^{-1}$ and $\chi_{{\mathcal P}(\infty,\alpha,\beta,\epsilon)}=\alpha-\beta$, both are constant on $\hat Z_\omega$ and thus $\alpha$, $\beta$ are also constant. This contradicts the fact that $\hat Z_\omega$ has dimension at least $2$. Hence, $\omega=\omega'$.
\end{proof}

\subsection{Higher dimensions and degrees}\label{sec-dim}
The next step is to move to higher dimensions and higher degrees. Let $k\geq2$ and $d\geq2$. We denote by $[y_0:\cdots:y_k]$ the homogeneous coordinates on $\Pb^k$ and we will mainly work in the affine chart $y_0=1$. Since the two first coordinates are the most important for the dynamics, we take the convention of notation that
\[z=y_1,\ \ w=y_2\ \text{ and }\ y=(y_3,\ldots,y_k).\]
Recall that $N_d^k:=(k+1)\binom{k+d}{d}$ is the dimension of the set of $k+1$ homogeneous polynomials of degree $d$. We choose coordinates in $\Cb^{N_d^k}$ such that, if $\sigma=(\sigma_3,\ldots,\sigma_k)$ and $\tau=(\tau_3,\ldots,\tau_k)$ are in $\Cb^{k-2}$ then the parameter $\lambda=(a,\alpha,\beta,\epsilon,\sigma,\tau,0)\in\Cb^{N_d^k}$ corresponds to the map
\begin{equation}\label{eq-skew}
f_\lambda(z,w,y)=\left(\alpha z+\epsilon w+\beta zw+w\sum_{i=3}^k\tau_iy_i,q_a(w),\sigma_3y_3,\ldots,\sigma_ky_k\right).
\end{equation}
Observe that, when $\tau=0$ then this map is a product map, acting by ${\mathcal F}_{(a,\alpha,\beta,\epsilon)}$ on $(z,w)$ and by a diagonal matrix on $y$. When $\tau\neq0$ then it is a skew product of $\Cb\times\Cb^{k-1}$. In what follows, we will take $\sigma$ with $|\sigma_i|>1$ very large with respect to $\alpha$, $\beta$ and $\epsilon$ in order to ensure a dominated splitting. The choice of $\sigma$ will also depend on $a$ in order to obtain non-resonance conditions for the periodic points $p$ and $r$. The parameter $\tau$ will be chosen very small at the end and its only role will be to obtain the point iii) in Assumption \ref{10}.

To be more explicit, let $(a,\alpha,\beta,\epsilon)$ be in the set $M$ given by Lemma \ref{le-xomega2}, let $\tau=0$ and let $\sigma=(\sigma_3,\ldots,\sigma_k)\in\Cb^{k-2}$ be such that each $|\sigma_i|>1$ is large. In this situation, the corresponding map $f_\lambda$ has a fixed saddle point $p(\lambda)=(\mathcal P(\lambda),0)$, a period $2$ repelling point $r(\lambda)=(\mathcal R(\lambda),0)$ and a repelling hyperbolic set 
\[\Lambda(\lambda):=\bigcap_{n\geq0}f^{-2n}_\lambda(\Db_R\times(V_-(a)\cup V_+(a))\times\Db^{k-2})\]
which is equal to the product of the hyperbolic set of $\mathcal F_{(a,\alpha,\beta,\epsilon)}$ with $\{0\}$. If each $|\sigma_i|>1$ is large enough then the cone field $C_\rho:=\{(u_1,\ldots,u_k)\in\Cb^k\ ;\ \rho|u_1|\leq\max_{2\leq i\leq k}|u_k|\}$, where $\rho$ is given by Lemma \ref{le-blender}, is contracted by $f_\lambda$ (resp. $f_\lambda^2$) on $\Db_R\times(U_-(a)\cup U_+(a))\times\Db^{k-2}$  (resp. $\Db_R\times(V_-(a)\cup V_+(a))\times\Db^{k-2}$) and thus $\Lambda(\lambda)$ has the following blender property: for each $i\in\{1,2,3,4\}$, any vertical graph in $\overline\Db\times V_i(a)\times\Db^{k-2}$ tangent to $C_\rho$ intersects $\Lambda(\lambda)$. Moreover, a simple computation gives that the critical set of $f_\lambda$ is disjoint from $\Db_R\times(U_-(a)\cup U_+(a))\times\Db^{k-2}$ and that the stable manifold of $p(\lambda)$ is equal to $\Cb\times\{(\tilde w(a),0)\}$, where $\tilde w(a)$ is the unique fixed point of $q_{a}$ in $U_-(a).$

We also need non-resonance conditions for $p(\lambda)$ and $r(\lambda)$ and for that we will choose $(a,\alpha,\beta,\epsilon)\in M$ and $\sigma$ more carefully. When $|a|$ is very large then the eigenvalues of $D_{p(\lambda)}f_\lambda$ are close to $\alpha-\beta$, $-2a$ and $\sigma_3,\ldots,\sigma_k$. Those of $D_{r(\lambda)}f^{2}_\lambda$ are close to $\alpha^2-\beta^2$, $-4a^2$ and $\sigma_3^2,\ldots,\sigma_k^2$. In both cases, only the first two ones depend on $a$. Hence, we can first fix $\alpha_1$, $\beta_1$ and $\epsilon_1$ then $a_1\in\Rb_+$ then $\sigma_1=(\sigma_i)_{3\leq i\leq k}\in(\Rb_-)^{k-2}$ in order to have $(a_1,\alpha_1,\beta_1,\epsilon_1)\in M$ and for $f_{\lambda_1}$, $\lambda_1:=(a_1,\alpha_1,\beta_1,\epsilon_1,\sigma_1,0)\in\Cb^{N_d^k}$
\begin{itemize}
\item[\textbf{(1)}] the eigenvalues of $p(\lambda_1)$ satisfy the \emph{strong Sternberg} condition of order $3,$
\item[\textbf{(2)}] $\max_{3\leq i\leq k}|\sigma_i|<|a_1|<\min_{3\leq i\leq k}|\sigma_i|^2/4,$
\item[\textbf{(3)}] there is \emph{no resonance} between the eigenvalues of $D_{r(\lambda_1)}f^2_{\lambda_1}$ and they are all different.
\end{itemize}
A first remark is that we choose $a_1$ in $\Rb_+$ and $\sigma_1$ in $(\Rb_-)^{k-2}$ only to have specific cone contractions for Lemma \ref{le-tan}. For the other properties, recall that there is a resonance between $k$ complex numbers $(\eta_1,\ldots,\eta_k)$ if there exist $j\in\{1,\ldots,k\}$ and a multi-index $N=(N_1,\ldots,N_k)$ of non-negative integers such that $\sum_{i=1}^kN_i\geq2$ and $\left|\prod_{i=1}^k\eta_i^{N_i}\right|=|\eta_j|$. Notice that for repelling or attracting periodic points, the eigenvalues have no resonance for an open and dense set of parameters, provided there are no persistence relations between them, which is the case here. In this situation, the periodic point can be holomorphically linearized, with a linearization which depends holomorphically in the parameters. This can be seen in the proof of Latt\`es in \cite{lattes-BSMF} for $k=2$ and Berger and Reinke deal with a much more general setting in \cite{berger-reinke}. Observe that we also ask for different eigenvalues in order to locally follow the associated eigenspaces.

In the saddle case, the absence of resonance is no longer an open condition and in particular, it might be not possible to holomorphically linearized in family. Thus, we use the work of Sell \cite{sell} in order to have $C^1$-linearization in family. The strong Sternberg condition of order $3$ comes from \cite{sell} and is implied by the non-existence of resonance with multi-index $N=(N_1,\ldots,N_k)$ with $\sum_{i=1}^kN_i\leq3$ which is an open and dense property. The condition \textbf{(2)} above ensures that the spectral spreads as defined in \cite{sell} satisfy $\rho^-=1$ and $\rho^+<2$. Hence, \cite[Theorem 7]{sell} implies that the dynamics near $p(\lambda)$ can be $C^1$-linearized for $\lambda$ in a small neighborhood of $\lambda_1$ and with a linearization which depends continuously in the $C^1$-topology on $\lambda.$

Since all the properties above are stable under small $C^1$-perturbations there exists a small connected open neighborhood $\tilde M$ of $\lambda_1$ in $\Cb^{N_d^k}$ such that $p(\lambda)$, $r(\lambda)$ and $\Lambda(\lambda)$ can be followed holomorphically and, for each $\lambda\in\tilde M$, in addition to the linearization properties of $p(\lambda)$ and $r(\lambda)$ we just mentioned, we also have the following properties.
\begin{itemize}
\item If we set $U_\pm:=U_\pm(a_1)$, $V_\pm:=V_\pm(a_1)$, $\mathcal U_\pm:=\Db_R\times U_\pm\times\Db^{k-2}$ and $\mathcal V_\pm:=\Db_R\times V_\pm\times\Db^{k-2}$ then $f_\lambda$ (resp. $f_\lambda^2$) contracts to cone field $C_\rho$ on $\mathcal U_+\cup\mathcal U_-$ (resp. on $\mathcal V_+\cup\mathcal V_-$) and $\overline{\mathcal U_+\cup\mathcal U_-}\subset f_\lambda^2(\mathcal V_\pm)$ with $f^2_\lambda$ injective and expanding on $\mathcal V_+$ and on $\mathcal V_-.$
\item The critical set of $f_\lambda$ is disjoint from $\mathcal U_+\cup\mathcal U_-.$
\item Using inverse branches, each point in $\Lambda(\lambda)$ corresponds to a unique coding $\omega\in\Sigma$ and $(\lambda\mapsto x_\omega(\lambda))_{\omega\in\Sigma}$ gives the holomorphic motion of $\Lambda(\lambda).$
\item For each $i\in\{1,2,3,4\}$, any vertical graph in $\overline\Db\times V_i(a_1)\times\Db^{k-2}$ tangent to $C_\rho$ intersects $\Lambda(\lambda).$
\item $p(\lambda)$ is saddle and $W^u_{p(\lambda),loc}$ is a hypersurface intersecting $\Db_{1/4}\times V_-\times\Db^{k-2}$ as a vertical graph tangent to $C_\rho$. In particular, as in Lemma \ref{le-doublons}, $W^u_{p(\lambda),loc}$ intersects $\Lambda(\lambda)$ at two different points.
\item The stable manifold $W^s_{p(\lambda)}$ contains a subset close to $\Db_R\times\{(\tilde w(a_1),0)\}$ and in particular it intersects any vertical graphs in $\mathcal U_-$ tangent to $C_\rho$. 
\item $r(\lambda)$ is a repelling $2$-periodic point in $\Lambda(\lambda)$ with $k$ different eigenvalues.
\end{itemize}
In particular, we can define for $\omega\in\Sigma$
\[\tilde X_\omega:=\left\{\lambda\in\tilde M\ ;\ x_\omega(\lambda)\in W^u_{p(\lambda),loc} \right\},\]
and for $(\zeta,t)\in\Sb^1\times\Rb$, if $\chi_{r(\lambda)}$ (resp. $\chi_{p(\lambda)}$) is the eigenvalue of $D_{r(\lambda)}f^2_\lambda$ (resp. $D_{p(\lambda)}f_\lambda$) with the smallest modulus, 
\[\tilde Y_{\zeta,t}:=\left\{\lambda\in\tilde M\ ;\ \chi_{r(\lambda)}=\zeta\chi_{p(\lambda)}^t\right\}.\]
The following result is deduced from its counterpart on $M$.
\begin{lemma}\label{le-y}
For each $\omega\in\Sigma$ the set $\tilde X_\omega$ is a (possible empty) hypersurface of $\tilde M$. Moreover, there exists a connected open neighborhood $\tilde M'\subset\tilde M$ of $\lambda_1$ such that if $(\zeta,t)$ are in $\Sb^1\times\Rb$ and $\lambda\in\tilde M'$ is a regular point of $\tilde Y_{\zeta,t}$ then there is $\omega\in\Sigma$ such that $\lambda$ belongs to an irreducible component of $\tilde X_\omega$ which is not included in $\tilde Y_{\zeta,t}$.
\end{lemma}
\begin{proof}
The first point is a direct consequence of Lemma \ref{le-xomega2} since if $\tilde X_\omega=\tilde M$ then $X_\omega=D(\infty,1/\delta)\times\hat M$.

Exactly as in the proof of Lemma \ref{le-xomega2}, there exists a connected open neighborhood $\lambda_1$ in $\tilde M$ such that if an irreducible component of $\tilde X_\omega$ intersects $\tilde M'$ then it also intersects $M\times\{(\sigma_1,0)\}$.

Let $(\zeta,t)\in\Sb^1\times\Rb$ and $\lambda\in\tilde M'$ be a regular point of $Y_{\zeta,t}$. As we have already seen, there exists two different coding $\omega,\omega'\in\Sigma$ such that $\lambda\in\tilde X_\omega\cap\tilde X_{\omega'}$. Let $\tilde Z_\omega$ and $\tilde Z_{\omega'}$ be irreducible components of $\tilde X_\omega$ and $\tilde X_{\omega'}$ respectively containing $\lambda$. Assume by contradiction that both $\tilde X_\omega$ and $\tilde X_{\omega'}$ are contained in $\tilde Y_{\zeta,t}$. As $\lambda$ is a regular point of $\tilde Y_{\zeta,t}$ this implies that $\tilde Z_\omega=\tilde Z_{\omega'}\subset \tilde Y_{\zeta,t}$. Hence, since $\tilde Z_\omega$ intersects $M\times\{(\sigma_1,0)\}$, a similar result holds on $M$ which is not possible by Lemma \ref{le-ass11}.
\end{proof}
As a consequence, generically $p(\lambda)$ has plenty of homoclinic points.
\begin{lemma}\label{le-homoclinic}
The set of parameters in $\tilde M$ where Assumption \ref{homoclinic} holds is open and dense.
\end{lemma}
\begin{proof}
This set is clearly open. It remains to prove that it is dense. Notice that the set $\Sigma'$ of $\omega\in\Sigma$ coding for a point with dense orbit in $\Lambda(\lambda)$ is dense in $\Sigma$ and does not depend on $\lambda$. As each $\tilde X_\omega$ is a hypersurface, the set $\cup_{\omega\in\Sigma'}\tilde X_\omega$ is dense in $\tilde M$. Let $\omega\in\Sigma'$, $\lambda\in\tilde X_\omega$ and let $\Gamma$ be a small neighborhood of $x_\omega(\lambda)$ in $W^u_{p(\lambda),loc}$. For $n\geq1$ large enough, its image $f_\lambda^{2n}(\Gamma)$ contains a vertical graph in $\mathcal U_-$ and thus, as we have seen when $\tilde M$ was chosen, it intersects $W_{p(\lambda)}^s$. As the graph is vertical, the intersection is transverse. Moreover, we obtain in this way several different intersection points. Actually, the orbit of $x_{\omega(\lambda)}$ is dense in $\Lambda(\lambda)$ and, by the blender property, the projection on the first coordinate of this set contains $\overline\Db$. This, combined with the fact that the graphs above are tangent to $C_\rho$ with $\rho>100$, ensures that several of these intersection points are different from $p(\lambda)$.

Finally, all the dynamics above stay in $\mathcal U_-$ which is disjoint from the critical set of $f_\lambda$, by assumption on $\tilde M$.
\end{proof}

\subsection{Tangencial dynamics}\label{sec-tan}
In this part, we will prove that a property, which is robust in the $C^1$-topology and which implies iii) in Assumption \ref{10}, holds generically in the open set $\tilde M\subset\Cb^{N_d^k}$ obtained in \S~\ref{sec-dim}.

To fix some notations, let $\lambda\in\tilde M$ and $F_\lambda:=f^2_\lambda$. Since $\Lambda(\lambda)$ is a hyperbolic set for $F_\lambda$ with a dominated splitting, to each history $\hat x(\lambda)=(x_{n})_{n\leq0}$ in the natural extension $\hat\Lambda(\lambda)$ is associated a strong unstable 
subspace $E^{uu}_{\hat x(\lambda)}$. This subspace is simply obtained by
\begin{equation}\label{eq-uu}
E^{uu}_{\hat x(\lambda)}:=\lim_{n\to\infty}D_{x_{-n}}F_\lambda^{n}E^{v}=\lim_{n\to\infty}(D_{x_{-1}}F_\lambda\circ\cdots\circ D_{x_{-n}}F_\lambda)E^v.
\end{equation}
where $E^{v}=\{(u_1,\ldots,u_k)\in\Cb^k\ ;\ u_1=0\}$. The strong unstable manifold  can be constructed in a similar way using graph transform but we will not use it. These objects depend continuously on $\hat x$ and holomorphically on $\lambda\in\tilde M$. Actually, this is true as long as the hyperbolic set can be followed, a remark that will be used in Lemma \ref{le-tan}.

Observe that the natural extension $\hat\Lambda(\lambda)$ of $\Lambda(\lambda)$ corresponds to the two-sided full shift encoded by $\hat\Sigma:=\{-1,1\}^\Zb.$
For $l\in\Nb$ and $n\in\Zb$, we set $\omega_n(l)=1$ if $n<l$ and $\omega_n(l)=-1$ otherwise. If $\omega(l):=(\omega_n(l))_{n\in\Nb}$ and $\hat\omega(l):=(\omega_n(l))_{n\in\Zb}$ then $x_{\omega(l)}(\lambda)\in\Lambda(\lambda)$ is a preimage of $r(\lambda)$ by $F_\lambda^l$ and $x_{\hat\omega(l)}(\lambda)\in\hat\Lambda(\lambda)$ is a history of $x_{\omega(l)}(\lambda)$. 

\begin{lemma}\label{le-tan}
The set $T$ defined by
\[T:=\left\{\lambda\in\tilde M\ ; \ E^{uu}_{x_{\hat\omega(0)}(\lambda)} \text{ contains an eigenvector of }D_{r(\lambda)}F_\lambda\right\}\] 
 is a proper analytic subset of $\tilde M$.
\end{lemma}
\begin{proof}
First, recall that the eigenvalues of $D_{r(\lambda)}F_\lambda$ are all different thus we can follow the corresponding eigenspaces holomorphically in $\lambda$. Hence, the fact that $T$ is analytic comes from the holomorphic dependency of $E^{uu}_{\hat x(\lambda)}$ on $\lambda$. It remains to prove properness, i.e., to find $\lambda\in\tilde M$ outside $T$. Observe that as soon as $k\geq3$, the parameter $\lambda_1=(a_1,\alpha_1,\beta_1,\epsilon_1,\sigma_1,0)\in\Cb^{N_d^k}$ defined in \S~\ref{sec-dim} is in $T$ since $f_{\lambda_1}$ is a product map on $\Cb^2\times\Cb^{k-2}$. We perturb it as a skew product of $\Cb\times\Cb^{k-1}$
\[f_{\lambda_2}(z,w,y)=\left(\alpha_1 z+\epsilon_1 w+\beta_1 zw+w\sum_{i=3}^k\tau_iy_i,q_{a_1}(w),\sigma_3y_3,\ldots,\sigma_ky_k\right),\]
where $\lambda_2:=(a_1,\alpha_1,\beta_1,\epsilon_1,\sigma_1,\tau_1,0)\in\Cb^{N_d^k}$, with $\tau_1=(\tau_i)_{3\leq i\leq k}\in(\Rb_{>0})^{k-2}$ is small enough to have $\lambda_2\in\tilde M$. We will deform $f_{\lambda_2}$ along a path $(f_{\gamma(t)})_{t\in[0,1]}$ of maps of the form \eqref{eq-skew} such that $f_{\gamma(1)}=f_{\lambda_2}$ and $f_{\gamma(0)}$ does not satisfies the property defining $T$. 
Observe that the path $\gamma$ may go outside $\tilde M$ but the hyperbolic set $\Lambda$ can be followed all along $\gamma$, which is sufficient.

For $t\in[0,1]$ we define $\gamma(t):=(a_1,|\alpha_1|e^{it\arg(\alpha_1)},t\beta_1,|\epsilon_1|e^{it\arg(\epsilon_1)},\sigma_1,\tau_1,0)$. This path is chosen in such a way that $\gamma(0)=(a_1,|\alpha_1|,0,|\epsilon_1|,\sigma_1,\tau_1,0)$ is in $(\Rb_+)^{N_d^k}$ and that the map $F_{\gamma(t)}$ is still expanding on $\mathcal V_+\cup\mathcal V_-$ with $\overline{\mathcal V_+\cup\mathcal V_-}\subset F_{\gamma(t)}(\mathcal V_+)\cap F_{\gamma(t)}(\mathcal V_-)$ since $(\alpha_1,\beta_1,\epsilon_1)$ satisfies \eqref{def-M}. In particular, the hyperbolic set $\Lambda$ can be followed in a neighborhood of this path.

From this, an important remark is that, by Lemma \ref{le-reel}, $r(\gamma(0))=(z_0,w_0,0)$ with $w_0$ real close to $-1$ and, for $l\geq1$, $x_{\omega(l)}(\gamma(0))=(z_l,w_l,0)$ with $w_l$ real close to $1$. Moreover, if $(z,w)\in\Db_R\times(V_-\cup V_+)$ with $w$ real then $D_{(z,w,0)}F_{\gamma(0)}=(A_{i,j})$ is a real matrix with $A_{i,j}=0$ if $i\neq j$ and $i\neq1$, and, with the notation $w':=q_{a_1}(w),$
\[A_{1,1}=|\alpha_1|^2,\ A_{2,2}=4a_1^2ww',\ A_{1,2}=|\epsilon_1|(|\alpha_1|+2a_1w),\]
\[ A_{i,i}=\sigma_i^2\ \text{and}\ A_{1,i}=\tau_i(|\alpha_1|w+\sigma_iw')\]
for $i\in\{3,\ldots,k\}$. From this, using that $r(\gamma(0))=(z_0,w_0,0)$ with $w_0$ close to $-1$ and $w_0'$ close to $1$, it is easy to see that the eigenvectors associated to $r(\gamma(0))$ are proportional to $e_1$ and to $e_i+b_ie_1$ where $b_i>0$ for each $i\in\{2,\ldots,k\}$. Here,  $(e_1,\ldots,e_k)$ is the canonical basis of $\Cb^k$. On the other hand, for each $l\geq1$, $x_{\omega(l)}(\gamma(0))=(z_l,w_l,0)$ with $w_l$ real close to $1$ and $w_l'$ close to $-1$. Hence, a vector of the form $e_i+c_ie_1$ with $-1\leq c_i\leq0$ is send by $D_{(z_l,w_l,0)}F_{\gamma(0)}$ on a vector proportional to $e_i+c_i'e_1$ with
\[\left\{\begin{array}{l}
c_i'=\displaystyle\frac{|\alpha_1|^2c_i+\tau_i(|\alpha_1|w_l+\sigma_iw_l')}{\sigma_i^2}\ \text{if}\ i\in\{3,\ldots,k\}\ \text{and}\\
 c_2'=\displaystyle\frac{|\alpha_1|^2c_2+|\epsilon_1|(|\alpha_1|+2a_1w_l)}{4a_1^2w_lw_l'}.
 \end{array}\right.\]
In particular, since $a_1$ and every $\sigma_i$ are very large and $|\alpha_1|<2$, $|\epsilon_1|>1/20$, each $c_i'$ satisfies $-1\leq c_i'<0$. This implies that the subspace $E^{uu}_{x_{\hat\omega(0)}(\gamma(0))}$, which by \eqref{eq-uu} is equal to
\[\lim_{n\to\infty}\left(D_{x_{\omega(1)}(\gamma(0))}F_{\gamma(0)}\circ\cdots\circ D_{x_{\omega(n)}(\gamma(0))}F_{\gamma(0)}\right)E^v,\]
 is generated by $(e_i+d_ie_1)_{2\leq i\leq k}$ with $d_i<0$. Hence, it contains none of the eigenvectors of $D_{r(\gamma(0))}F_{\gamma(0)}$ described above. This conclude the proof.
\end{proof}

\subsection{Verification of the assumptions}\label{sec-veri}
We have now all the ingredients to prove Theorem \ref{th-existence}. Let $\lambda_1$, $\tilde M$, $U_\pm$, $\mathcal U_\pm$, $V_\pm$ and $\mathcal V_\pm$ be as in \S~\ref{sec-dim}. 
Recall that $f_{\lambda_1}$ is of the form
\[f_{\lambda_1}(z,w,y)=\left(\alpha_1 z+\epsilon_1 w+\beta_1 zw+w\sum_{i=3}^k\tau_iy_i,q_a(w),\sigma_3y_3,\ldots,\sigma_ky_k\right).\]
As we have said in the discussion in \S~\ref{sec-dim} where $\tilde M$ was chosen, for each $\lambda\in\tilde M$ the map $f_{\lambda}$ satisfies several assumptions of \S~\ref{sec-ass}. More precisely, if $\rho>0$ and $R>0$ as in Lemma \ref{le-blender}, then we already know that Assumptions \ref{verti}, \ref{saddle} and \ref{blender} hold in $\tilde M$. This is also true for Assumption \ref{hyperbolic}  except on the point about the small Julia set $J_k$, which is not necessarily well-defined for $f_\lambda$, and for Assumption \ref{repelling} except the part about the domain of linearization. For this last point, since the map $f^2_{\lambda}$ restricted to $\mathcal V_-$ is expanding and injective with $\mathcal U_-$ in its image so, if $h_\lambda\colon\mathcal U_-\to\mathcal V_-$ denotes its inverse then we can extend the linearization $\delta_\lambda$ on $\mathcal U_-$ in an injective way using the dynamics, i.e., if $x\in\mathcal U_-$ then $\delta_\lambda(x):=(D_{r(\lambda)}f^2_\lambda)^n\circ\delta_\lambda\circ h_\lambda^n(x)$ for $n\geq1$ large enough. We will come back to the small Julia set of Assumption \ref{hyperbolic} later.

Assumption \ref{free} is just a reformulation of the fact that $\tilde X_\omega$ is a proper analytic subset of $\tilde M$ and thus, it holds on $\tilde M$ by Lemma \ref{le-y}.

For the other assumptions, we consider a perturbation $f_{\tilde \lambda}$ in $\tilde M'$ of $f_{\lambda_1}$ defined by
\[f_{\tilde\lambda}(z,w,y)=f_{\lambda_1}(z,w,y)+c(z^d,w^d,y_3^d,\ldots,y_k^d),\]
where $c\in\Cb^*$ is very close to $0$. Observe that $f_{\tilde\lambda}$ is a regular polynomial endomorphism of $\Cb^k$ which is a skew product of $\Cb\times\Cb^{k-1}$ above a product map of $\Cb^{k-1}$. This implies that the small Julia set $J_k(f_{\tilde\lambda})$ is exactly the closure of the repelling periodic points of $f_{\tilde\lambda}$. Actually, this was proved by Jonsson \cite{Jonsson99} when $k=2$ where the key ingredient is a fibered formula for the equilibrium measure. This formula has been generalized in higher dimension by \cite[Corollary 1.2]{Dupont_Taflin} and thus the result of Jonsson also holds for $f_{\tilde\lambda}$. In particular, since the repelling periodic points are dense in the hyperbolic set $\Lambda(\tilde\lambda)$, this gives that $\Lambda(\tilde\lambda)\subset J_k(f_{\tilde\lambda})$. As this property persists under small perturbations (see e.g. \cite[Lemma 2.3]{Dujardin_blender}), Assumption \ref{hyperbolic} holds in a small neighborhood of $f_{\tilde\lambda}$ in $\mathrm{End}_d^k$. Moreover, one part of the critical set of $f_{\tilde\lambda}$ comes from the dynamics on the basis $\Cb^{k-1}$. A simple computation gives that the remaining part of this critical set is
\[C_{\tilde\lambda}:=\left\{(z,w,y)\in\Cb^k\ ;\ w=-\frac{\alpha_1+cdz^{d-1}}{\beta_1}\right\},\]
which is always transverse to fibers of the form $\{w=w_0,\ y=y_0\}$ except when $d\geq3$ and $w_0=-\alpha_1/\beta_1$. The stable manifold $W^s_{p(\tilde\lambda)}$ of the saddle point $p(\tilde\lambda)$ is an attracting basin in the invariant fiber $\{w=w_0,y=0\}$ where $w_0\simeq -1$ is the unique fixed point of $w\mapsto q_{a_1}(w)+cw^d$ in $U_-$ and thus by a classical result of Fatou it has to intersect $C_{\tilde\lambda}$ in a point of infinite orbit. Furthermore, since $w_0\simeq -1$ and $-\alpha_1/\beta_1$ has large modulus, they cannot be equal. This gives a transverse intersection between $W^s_{p(\tilde\lambda)}$ and the critical set. The skew product structure of $f_{\tilde\lambda}$ and the fact that $\{w=w_0,y=0\}$ is not a critical fiber ensures that the images of this transverse intersection remain transverse. This shows Assumption \ref{critical} is satisfied in a small neighborhood of $f_{\tilde\lambda}$. 

Observe that by Lemma \ref{le-homoclinic} and Lemma \ref{le-tan}, there exists near $\tilde\lambda$ a small  non-empty open set $\Omega$ in $\mathrm{End}_d^k\cap\tilde M'$ where, in addition to all the assumptions above, Assumption \ref{homoclinic} holds and which is disjoint from the set $T$ defined in Lemma \ref{le-tan}. Moreover, as by \cite[Corollary C]{fj-brolin} when $k=2$ and \cite[Lemma 5.4.5]{DinhSibonysuper} for $k\geq2$ the exceptional set is generically empty (or reduced to the hyperplane at infinity in the case of regular polynomial endomorphisms of $\Cb^k$, with the same proof than \cite[Lemma 5.4.5]{DinhSibonysuper}), we can also assume that it is the case for maps in $\Omega$, i.e., that Assumption \ref{exceptional} holds on $\Omega.$

Hence, all assumptions of \S~\ref{sec-ass} are satisfied on $\Omega$, except possibly Assumption \ref{10}. Let $\Omega'$ a non-empty open subset of $\Omega$. Let $\lambda'\in\Omega'$ be a regular point of the foliation $(\tilde Y_{\zeta,t})_{(\zeta,t)\in\Sb^1\times\Rb}$ defined just before Lemma \ref{le-y}. This lemma implies that there exists $\omega'=(\omega'_n)_{n\geq0}\in\Sigma$ such that $\lambda'\in\tilde X_{\omega'}$ and $\tilde X_{\omega'}$ is not contained in some $\tilde Y_{\zeta,t}$. In particular, if $\omega=(\omega_n)_{n\geq0}$ is very close to $\omega'$ then $\tilde X_{\omega}$ intersects $\Omega'$ and is not contained in some $\tilde Y_{\zeta,t}$. We defined such $\omega$ by
\[\omega_n=\omega'_n \text{ if } n\leq N_1,\ \ \ \omega_n=1 \text{ if } N_1<n\leq N_1+N_2 \text{ and } \omega_n=-1 \text{ if } n> N_1+N_2,\]
where $N_1$ and $N_2$ are two very large positive integers. The first condition ensures that $\omega$ is close enough to $\omega'$. The third one that the corresponding point $x_\omega(\lambda)\in\Lambda(\lambda)$ is a preimage by $f_\lambda^{2(N_1+N_2+1)}$ of the repelling periodic point $r(\lambda)$. This gives point i) of Assumption \ref{10} with $m:=2(N_1+N_2+1)$ for every $\lambda\in\Omega'\cap\tilde X_\omega$. By definition of $\tilde X_\omega$, $x_\omega(\lambda)\in W^u_{p(\lambda),loc}$ on this set and thus ii) also holds. The fact that $\tilde X_\omega$ is not contained in some $\tilde Y_{\zeta,t}$ implies that iv) is satisfied on a dense subset of $\Omega'\cap\tilde X_\omega$. Finally, recall $\Omega'$ is disjoint from the set $T$ defined in Lemma \ref{le-tan}. Recall that this implies that if $\omega(1):=(\omega_n(1))_{n\geq0}\in\Sigma$ and the history $\hat\omega(1):=(\omega_n(1))_{n\in\Zb}\in\{-1,1\}^\Zb$ are defined by $\omega_n(1)=1$ if $n<1$ and $\omega_n(1)=-1$ otherwise then the image by $D_{x_{\omega(1)}(\lambda)}f_\lambda^2 $ of the strong unstable subspace $E^{uu}_{x_{\hat\omega(1)}(\lambda)}$ associated to $x_{\hat\omega(1)}(\lambda)$ is a generic  hyperplane for $D_{r(\lambda)}f_\lambda^2$, i.e., does not contain any eigenvector of $D_{r(\lambda)}f_\lambda^2$. On the other hand, if $\lambda\in\Omega'\cap\tilde X_\omega$ and if $E$ denotes the tangent space of $W^u_{p(\lambda),loc}$ at $x_{\omega}(\lambda)$ then $E$ is in the cone $C_\rho$ and $D_{x_\omega(\lambda)}f_\lambda^{2(N_1+N_2)}E$ is very close to $E^{uu}_{x_{\hat\omega(1)}(\lambda)}$ if $N_2$ is large enough. Hence, $D_{x_\omega(\lambda)}f_\lambda^{2(N_1+N_2+1)}E$ is also a generic  hyperplane for $D_{r(\lambda)}f_\lambda^2$. Thus, the point iii) of Assumption \ref{10} is satisfied for all $\lambda\in\Omega'\cap\tilde X_\omega.$

\medskip

In conclusion, the open set $\Omega$ verifies all the assumptions of \S~\ref{sec-ass}.

\section{The fundamental height inequalities}\label{section3}

\subsection{Adelic metrics, height functions}\label{sec:canonical-height}
Let $X$ be a projective variety, and let $L_0,\ldots,L_k$ be $\mathbb{Q}$-line bundles on $X$, all defined over a number field $\mathbb{K}$. 

For any $i$, assume $L_i$ is equipped with an adelic continuous metric $\{\|\cdot\|_{v,i}\}_{v\in M_\mathbb{K}}$ and we denote $\bar{L}_i:=(L_i,\{\|\cdot\|_v\}_{v\in M_\mathbb{K}})$. Assume also $\bar{L}_i$ is  semi-positive. 
 Fix a place $v\in M_\mathbb{K}$. Denote by $X_v^\mathrm{an}$ the Berkovich analytification of $X$ at the place $v$. We also let $c_1(\bar{L}_i)_v$ be the curvature form of the metric $\|\cdot\|_{v,i}$ on $L_{i,v}^\mathrm{an}$.

For any closed subvariety $Y$ of dimension $q$, as observed by Chambert-Loir \cite{ACL}, the arithmetic intersection number $\left(\bar{L}_0\cdots\bar{L}_q|Y\right)$ is symmetric and multilinear with respect to the $\bar{L}_i$'s and is defined inductively by
\[\left(\bar{L}_0\cdots\bar{L}_q|Y\right)=\left(\bar{L}_1\cdots\bar{L}_q|\mathrm{div}(s)\cap Y\right)+\sum_{v\in M_\mathbb{K}}n_v\int_{Y_v^{\mathrm{an}}}\log\|s\|^{-1}_{v,0} \bigwedge_{j=1}^qc_1(\bar{L}_i)_v,\]
for any global section $s\in H^0(X,L_0)$. In particular, if $L_0$ is the trivial bundle and $\|\cdot\|_{v,0}$ is the trivial metric at all places but $v_0$, this gives
\[\left(\bar{L}_0\cdots\bar{L}_k|Y\right)=n_{v_0}\int_{Y_{v_0}^{\mathrm{an}}}\log\|s\|^{-1}_{v_0,0} \bigwedge_{j=1}^qc_1(\bar{L}_i)_{v_0}.\]
When $L$ is a big and nef $\mathbb{Q}$-line bundle endowed with a semi-positive continuous adelic metric $\bar{L}$, following Zhang~\cite{Zhang-positivity}, we can define $h_{\bar{L}}(Y)$ as
\[ h_{\bar{L}}(Y):=\frac{\left(\bar{L}^{q+1}|Y\right)}{(q+1)[\mathbb{K}:\mathbb{Q}]\deg_{Y}(L)},\]
where $\deg_Y(L)=(L|_Y)^k$ is the volume of the line bundle $L$ restricted to $Y$.

Recall that a sequence $(x_i)_i$ of points of $Y(\bar{\mathbb{Q}})$ is \emph{generic} if for any closed subvariety $W\subset Y$ defined over $\mathbb{K}$, there is $i_0\geq1$ such that $\mathsf{O}(x_i)\cap W=\varnothing$ for all $i\geq i_0$. By Zhang's inequalities~\cite{Zhang-positivity}, if $h_{\bar{L}}\geq0$ on $X(\bar{\mathbb{Q}})$, if we let 
\[e_1(\bar{L}):=\sup_{Z\subsetneq Y}\inf_{x\in (Y\setminus Z)(\bar{\mathbb{Q}})}h_{\bar{L}}(x),\]
where $Z$ ranges on strict subvarieties of $Y$ defined over $\bar{\mathbb{Q}}$, then we have
\begin{align}
e_1(\bar{L})\geq h_{\bar{L}}(Y)\geq \frac{1}{q+1}e_1(\bar{L}).\label{Zhang-forte}
\end{align}
In particular, there is a generic sequence $(x_i)_i$ of closed points of $Y(\bar{\mathbb{Q}})$ such that
\begin{align}
\liminf_{i\to\infty}h_{\bar{L}}(x_i)\geq h_{\bar{L}}(Y)\geq \frac{1}{q+1}\liminf_{j\to\infty}h_{\bar{L}}(x_j).\label{Zhang-faible}
\end{align}

\medskip

Let $X$ be a projective variety defined on a number field $\mathbb{K}$ and let $\bar{L}$ be an ample line bundle on $X$ endowed with an adelic semi-positive metric. Let $m\geq1$ be an integer and, for $1\leq i\leq m$, let $p_i:X^m\to X$ be the projection onto the $i$-th factor. Let $\bar{L}_m:=p_1^*(\bar{L})+\cdots+p_m^*(\bar{L})$.

We will use the next lemma.

\begin{lemma}\label{lm:Zhang-product}
For any subvariety $Y\subset X$ defined over $\mathbb{K}$, we have 
\[h_{\bar{L}_m}(Y^m)= m\cdot h_{\bar{L}}(Y).\]
\end{lemma}

\begin{proof}
For $m=1$, or if $q=\dim Y=0$ there is nothing to prove. We can assume that $L$ is very ample. Fix $m\geq2$ and set $q:=\dim Y\geq1$. For any $1\leq i\leq m$ and any line bundles $M_{q+2},\ldots,M_{qm}$ on $X^m$, we have
\[\left(c_1(p_i^*L)^{q+1}\cdot c_1(M_{q+2})\cdots c_1(M_{qm})\cdot\{Y^m\}\right)=0.\]
In particular,
\[\deg_{L_m}(Y^m)={{qm}\choose{q}}\left(c_1(p_1^*L)^q\cdots \left(c_1(p_2^*L)+\cdots+c_1(p_m^*L)\right)^{q(m-1)}\cdot\{Y^m\}\right),\]
 hence
\[\deg_{L_m}(Y^m)={{qm}\choose{q}}\deg_{L}(Y)\deg_{L_{m-1}}(Y^{m-1}).\]
Similarly, as the arithmetic intersection product is multilinear and symmetric, if we let $\pi_i:X^m\to X^{m-1}$ be the cancellation of the $i$th variable, we have
\begin{align*}
(\bar{L}_m^{qm+1}|Y^m) & =\sum_{i=1}^m{{qm+1}\choose{q+1}}(p_i^*\bar{L}^{q+1}\cdot \pi_i^*\bar{L}_{m-1}^{q(m-1)}|Y^m).
\end{align*}
Let $s_1,\ldots,s_{q+1}$ be sections of $L$ such that $\mathrm{div}(s_1)\cap\cdots\cap\mathrm{div}(s_{q+1})\cap Y=\varnothing$ and let $Z_0:=Y$ and for $1\leq j\leq q$, $Z_j:=Z_{j-1}\cap\mathrm{div}(s_j)$. Following \cite{ACL}, as $Y^m=p_i^{-1}(Y)\cap \pi_i^{-1}(Y^{m-1})$, we have
\begin{align*}
(p_i^*\bar{L}^{q+1}\cdot \pi_i^*\bar{L}_{m-1}^{q(m-1)}|& Y^m)   = \  (p_i^*\bar{L}^{q}\cdot \pi_i^*\bar{L}_{m-1}^{q(m-1)}|\pi_i^{-1}(Y^{m-1})\cap p_i^{-1}(Z_1))\\
& + \sum_{v\in M_\mathbb{K}}n_v\int_{(Y^m)^{\mathrm{an}}_v}\log\|p_i^*s_1\|_{p_i^*\bar{L},v}c_1(p_i^*\bar{L})_v^q\wedge c_1(\pi_i^*\bar{L}_m)_v^{q(m-1)},
\end{align*}
which rewrites as 
\begin{align*}
(p_i^*\bar{L}^{q+1}\cdot \pi_i^*\bar{L}_{m-1}^{q(m-1)}|Y^m)  = &\ \ \  (p_i^*\bar{L}^{q}\cdot \pi_i^*\bar{L}_{m-1}^{q(m-1)}|\pi_i^{-1}(Y^{m-1})\cap p_i^{-1}(Z_1))\\
& + \deg_{L_{m-1}}(Y^{m-1})\sum_{v\in M_\mathbb{K}}n_v\int_{Y^{\mathrm{an}}_v}\log\|s_1\|_{\bar{L},v}c_1(\bar{L})_v^q.
\end{align*}
Similarly, for any $1\leq j\leq q-1$ one can write 
\begin{align*}
(p_i^*\bar{L}^{q-j+2}\cdot \pi_i^*\bar{L}_{m-1}^{q(m-1)}| & \pi_i^{-1}(Y^{m-1})\cap p_i^{-1}(Z_{j-1}))= \\
& \ \ \ (p_i^*\bar{L}^{q-j+1}\cdot \pi_i^*\bar{L}_{m-1}^{q(m-1)}|\pi_i^{-1}(Y^{m-1})\cap p_i^{-1}(Z_j))\\
& +\deg_{L_{m-1}}(Y^{m-1})\sum_{v\in M_\mathbb{K}}n_v\int_{(Z_{j-1})^{\mathrm{an}}_v}\log\|s_j\|_{\bar{L},v}c_1(\bar{L})_v^{q-j+1}.
\end{align*}
Summing up over the $q+1$ terms we get
\begin{align*}
(p_i^*\bar{L}^{q+1}\cdot \pi_i^*\bar{L}_{m-1}^{q(m-1)}|Y^m)  = \deg_{L_{m-1}}(Y^{m-1})\cdot(\bar{L}^{q+1}|Y).
\end{align*}
Together with the above, this gives
\begin{center}
$(\bar{L}_m^{qm+1}|Y^m) =m{{qm+1}\choose{q+1}}\deg_{L_{m-1}}(Y^{m-1})\cdot(\bar{L}^{q+1}|Y)$.
\end{center}
Since by definition,
\begin{align*}
h_{\bar{L}}(Y)=\frac{(\bar{L}^{q+1}|Y)}{[\mathbb{K}:\mathbb{Q}](q+1)\deg_L(Y)} \ \text{and} \ h_{\bar{L}_m}(Y^m)=\frac{(\bar{L}_m^{qm+1}|Y^m)}{[\mathbb{K}:\mathbb{Q}](qm+1)\deg_{L_m}(Y^m)},
\end{align*}
the proof is complete.
\end{proof}

\subsection{Dynamics over number fields}

Let $X$ be a projective variety, $f:X\to X$ a morphism and $L$ be an ample line bundle on $X$, all defined over a number field $\mathbb{K}$. Recall that we say $(X,f,L)$ is a \emph{polarized endomorphism} of degree $d>1$ if $f^*L\simeq L^{\otimes d}$, i.e. $f^*L$ is linearly equivalent to $L^{\otimes d}$. Let $k:=\dim X$.

It is known that polarized endomorphisms defined over the field $\mathbb{K}$ admit a \emph{canonical metric}. This is an adelic semi-positive continuous metric on $L$, which can be built as follows: let $\mathscr{X}\to\mathrm{Spec}(\mathscr{O}_\mathbb{K})$ be an $\mathscr{O}_\mathbb{K}$-model of $X$ and $\bar{\mathscr{L}}$ be a model of $L$ endowed with a model metric, for example $\bar{\mathscr{L}}=\iota^*\bar{\mathcal{O}}_{\mathbb{P}^N}(1)$, where $\iota: X\hookrightarrow \mathbb{P}^N$ is an embedding inducing $L$ and $\bar{\mathcal{O}}_{\mathbb{P}^N}(1)$ is the naive metrization.
We then can define $\bar{L}$ as
\[\bar{L}:=\lim_{n\to\infty}\frac{1}{d^n}(f^n)^*\bar{\mathscr{L}}|_\mathbb{K}.\]
This metrization induces the \emph{canonical height} $\widehat{h}_f$ of $f$: for any closed point $x\in X(\bar{\mathbb{Q}})$ and any section $\sigma\in H^0(X,L)$ which does not vanish at $x$, we let
\[\widehat{h}_{f}(x):=\frac{1}{[\mathbb{K}:\mathbb{Q}]\deg(x)}\sum_{v\in M_\mathbb{K}}\sum_{y\in \mathsf{O}(x)}n_v\log\|\sigma(y)\|_v^{-1},\]
where $\mathsf{O}(x)$ is the Galois orbit of $x$ in $X$. The function $\widehat{h}_f:X(\bar{\mathbb{Q}})\to \mathbb{R}$ satisfies $\widehat{h}_f\circ f=d\cdot \widehat{h}_f$, $\widehat{h}_f\geq0$ and $\widehat{h}_f(x)=0$ if and only if $x$ is preperiodic under iteration of $f$, i.e. if there are $n>m\geq0$ such that $f^n(x)=f^m(x)$. Note that $\widehat{h}_f$ can also be defined as
\[\widehat{h}_f(x)=\lim_{n\to\infty}\frac{1}{d^n}h_{X,L}(f^n(x)),\]
where $h_{X,L}$ is any Weil height function on $X$ associated with the ample line bundle $L$. 

\medskip

If $Y$ is a subvariety of dimension $q\geq0$ defined over $\bar{\mathbb{Q}}$, we define
\[\widehat{h}_f(Y):=h_{\bar{L}}(Y)=\frac{\left(\bar{L}^{q+1}|Y\right)}{(q+1)[\mathbb{K}:\mathbb{Q}]\deg_{Y}(L)}\]
(observe that when $Y=\{x\}$ has dimension $0$, both definitions coincide i.e. both definitions of $\widehat{h}_f$ coincide).
This satisfies $\widehat{h}_f(f_*(Y))=d\widehat{h}_f(Y)$, where $f_*(Y)$ is the image of $Y$ by $f$ counted with multiplicity as a cycle on $X$. In particular, if $Y$ is preperiodic under iteration of $f$, i.e. if there are $n>m\geq0$ such that $f^n(Y)= f^m(Y)$, then $\widehat{h}_f(Y)=0$.

\subsection{Canonical height and height on the base}

We now let $(\mathcal{X},f,\mathcal{L},\mathcal{Y})$ be a dynamical pair of degree $d\geq2$ parametrized by a smooth projective variety $S$, with regular part $S^0_\mathcal{Y}$. 
Let $\mathcal{Y}^0:=\pi|_{\mathcal{Y}}^{-1}(S^0_\mathcal{Y})$.
We also assume $(\mathcal{X},f,\mathcal{L})$, $\mathcal{Y}$ and $S$ are all defined over a number field $\mathbb{K}$. In what follow, we fix an embedding $\iota:\mathbb{K} \hookrightarrow \C$ for which we define the different bifurcation currents.

\begin{definition}
Let $m\geq\dim S$. If the measure $\mu_{f,\mathcal{Y}}$ is non-zero, we define the $m$-\emph{higher order canonical height}  $\widehat{\mathcal{H}}^{(m)}_{f,\mathcal{M}}(\mathcal{Y})$ of the family $\mathcal{Y}$, relative to $\mathcal{M}$, as
\[\widehat{\mathcal{H}}^{(m)}_{f,\mathcal{M}}(\mathcal{Y}):=\frac{\mathrm{Vol}_{f}^{(m)}(\mathcal{Y})}{\dim\mathcal{Y}^{[m]}\cdot \deg_{f,\mathcal{M}}^{(m)}(\mathcal{Y})}.\]
Otherwise, we let $\widehat{\mathcal{H}}^{(m)}_{f,\mathcal{M}}(\mathcal{Y}):=0$.
\end{definition}

\begin{remark}\normalfont
\begin{enumerate}
	\item Observe that both $\mathrm{Vol}_{f}^{(m)}(\mathcal{Y})$ and $\deg_{f,\mathcal{M}}^{(m)}(\mathcal{Y})$ are geometric quantities that do not depend on the choice of a place (hence we can take another embedding $\iota:\mathbb{K} \hookrightarrow \C$).
\item If $\dim S=1$, we have $\widehat{\mathcal{H}}^{(1)}_f(\mathcal{Y})=\widehat{h}_{f_\eta}(Y_\eta)$, see~\S~\ref{sec:def-currents} for the definition of the geometric canonical height,
\item The quantity $\widehat{\mathcal{H}}^{(m)}_{f,\mathcal{M}}(\mathcal{Y})$ is well-defined by Proposition~\ref{prop-volumes} and satisfies $\widehat{\mathcal{H}}^{(m)}_{f,\mathcal{M}}(\mathcal{Y})>0$ for all $m\geq\dim S$ if and only if $\mu_{f,\mathcal{Y}}$ is non-zero.
\end{enumerate}
\end{remark}
 We prove here the following which is inspired from \cite[Theorem~1.4 and Proposition~10.1]{Gao-Habegger} and \cite[Theorem~1.6]{Dimitrov-Gao-Habegger}:
\begin{theorem}\label{tm:Gao-Habegger}
Assume that $\mu_{f,\mathcal{Y}}$ is non-zero. Let $m\geq\dim S$ and $\mathcal{M}$ be any ample $\mathbb{Q}$-line bundle on $S$ of volume $1$. Then, for any $0<\varepsilon<\widehat{\mathcal{H}}^{(m)}_{f,\mathcal{M}}(\mathcal{Y})$, there is a non-empty Zariski open subset $\mathcal{U}\subset (\mathcal{Y}^{[m]})^0$ and a constant $C\geq1$ depending only on $(\mathcal{X},f,\mathcal{L})$, $\mathcal{Y}$, $\mathcal{M}$, $m$ and $\varepsilon$ such that
\[h_{S,\mathcal{M}}(\pi_{[m]}(x))\leq \frac{1}{\widehat{\mathcal{H}}^{(m)}_{f,\mathcal{M}}(\mathcal{Y})-\varepsilon}\sum_{j=1}^{m}\widehat{h}_{f_{\pi(x_j)}}(x_j)+C,\]
for any $x=(x_1,\ldots,x_{m})\in\mathcal{U}(\bar{\mathbb{Q}})$.
\end{theorem}
As explained in the introduction, one can use the theory of adelic line bundles on quasi-projective varieties set up by Yuan and Zhang~\cite{YZ-adelic} to obtain such an inequality (see Theorem~6.2.2 therein). The hypothesis that  $\mu_{f,\mathcal{Y}}$ is non zero in Theorem~\ref{tm:Gao-Habegger} is equivalent to the non degeneracy of $\mathcal{Y}^{[m]}$ in \cite[Theorem 6.2.2]{YZ-adelic}.
 However, the proof we give here allows us to have explicit constants in the inequality.
\begin{proof}
Fix $0<\varepsilon<\widehat{\mathcal{H}}^{(m)}_{f,\mathcal{M}}(\mathcal{Y})$ and $C\geq1$. Take $n\geq1$ such that $d^n\left(\widehat{\mathcal{H}}^{(m)}_{f,\mathcal{M}}(\mathcal{Y})-\varepsilon/2\right)>1$ and $C\leq d^n\varepsilon/4$. Choose integers $M,N\geq1$ such that
\begin{center}
$N\left(d^n\widehat{\mathcal{H}}^{(m)}_{f,\mathcal{M}}(\mathcal{Y})- C\right)>M\geq N d^n\left(\widehat{\mathcal{H}}^{(m)}_{f,\mathcal{M}}(\mathcal{Y})-\varepsilon/2\right)$.
\end{center}

We use Lemma~\ref{lm:goodmodel}: increasing $n$ if necessary, if $\mathcal{Y}_{n}^{(m)}:=(F_n^{(m)})_*(\psi_n^{(m)})^*\mathcal{Y}^{[m]}$, we deduce from Lemma~\ref{cor:GV} that the quantity
\begin{align*}
\frac{\left(\{\mathcal{Y}_{n}^{(m)}\}\cdot \left(Nc_1(\mathcal{L}^{[m]})\right)^{\dim\mathcal{Y}^{[m]}}\right)}{\dim\mathcal{Y}^{[m]}\left(\{\mathcal{Y}_{n}^{(m)}\}\cdot \left(N c_1(\mathcal{L}^{[m]})\right)^{\dim\mathcal{Y}^{[m]}-1}\cdot \left(M c_1(\pi_{[m]}^*\mathcal{M})\right)\right)}
\end{align*}
is bounded from below by $\frac{N}{M}\left(d^n\widehat{\mathcal{H}}^{(m)}_{f,\mathcal{M}}(\mathcal{Y})- C\right)>1$.
Let $\mathcal{Y}_n:=(\psi_n^{(m)})^{*}(\mathcal{Y}^{[m]})$, $ \mathcal{L}_n:=\left((F_n)^*\mathcal{L}^{[m]}\right)|_{\mathcal{Y}_n}$, and $\mathcal{M}_n:=\left((\pi_{[m]}\circ\psi_n^{(m)})^*\mathcal{M}\right)|_{\mathcal{Y}_n}$. By construction, the line bundles $\mathcal{L}_n$ and $\mathcal{M}_n$ are nef on $\mathcal{Y}_n$ and the above inequality implies
\[\left((N\mathcal{L}_n)^{\dim\mathcal{Y}_n}\right)>\dim\mathcal{Y}_n\left((M\mathcal{M}_n)\cdot(N\mathcal{L}_n)^{\dim\mathcal{Y}_n-1}\right)\]
by the projection formula.
We thus can apply Siu's bigness criterion~\cite[Theorem~2.2.15]{Laz} and find that $N\mathcal{L}_n-M\mathcal{M}_n$ is a big line bundle on $\mathcal{Y}_n$. In particular, there exist $\ell\geq1$ and a non-empty Zariski open set $\mathcal{U}_n\subset \mathcal{Y}_n$ such that for any $x\in \mathcal{U}_n(\bar{\mathbb{Q}})$,
\[h_{\mathcal{Y}_n,\ell(N\mathcal{L}_n-M\mathcal{M}_n)}(x)\geq -C_1\]
for some constant $C_1$ depending only on $n$. Now we use successively functorial properties of Weil height functions, see~e.g.~\cite{Silvermandiophantine}. First, we find that for any $y\in \mathcal{U}_n(\bar{\mathbb{Q}})$,
\begin{align*}
h_{\mathcal{Y}_n,\ell(N\mathcal{L}_n-M\mathcal{M}_n)}(y) & = \ell\left( Nh_{\mathcal{Y}^{[m]},\mathcal{L}^{[m]}}(F^n(y))-Mh_{S,\mathcal{M}}(\pi_{[m]}\circ\psi_n(y)\right)+O(1).
\end{align*}
Since $F_n=(f^{[m]})^n\circ \psi_n^{(m)}$ on the non-empty Zariski open set $\mathcal{U}_n\cap (\psi_n^{(m)})^{-1}((\mathcal{Y}^{[m]})^0)$, since $\psi_n$ is an isomorphism from $\mathcal{U}_n\cap (\psi_n^{(m)})^{-1}((\mathcal{Y}^{[m]})^0)$ to its image $\mathcal{U}^1:=\psi_n^{(m)}\left(\mathcal{U}_n\right)\cap(\mathcal{Y}^{[m]})^0$, and since $\mathcal{Y}^{[m]}$ is a subvariety of $\mathcal{X}^{[m]}$ we deduce that, for any $x\in \mathcal{U}^1(\bar{\mathbb{Q}})$,
\begin{align*}
h_{\mathcal{Y}_n,\ell(N\mathcal{L}_n-M\mathcal{M}_n)}(\psi_n^{-1}(x)) & = \ell\left( Nh_{\mathcal{X}^{[m]},\mathcal{L}^{[m]}}(f^n(x))-Mh_{S,\mathcal{M}}(\pi_{[m]}(x))\right)+O(1).
\end{align*}
In particular, the above gives
\begin{align*}
h_{\mathcal{X}^{[m]},\mathcal{L}^{[m]}}((f^{[m]})^n(x) ) & \geq \frac{M}{N} h_{S,\mathcal{M}}(\pi_{[m]}(x))-C_2\\
& \geq d^n\left(\widehat{\mathcal{H}}^{(m)}_{f,\mathcal{M}}(\mathcal{Y})-\varepsilon/2\right)h_{S,\mathcal{M}}(\pi_{[m]}(x))-C_2,
\end{align*}
for any $x\in \mathcal{U}(\bar{\mathbb{Q}})$, where $C_2$ is a constant depending on $n$. This rewrites as
\begin{align*}
\frac{1}{d^n}h_{\mathcal{X}^{[m]},\mathcal{L}^{[m]}}((f^{[m]})^n(x))& \geq \left(\widehat{\mathcal{H}}^{(m)}_{f,\mathcal{M}}(\mathcal{Y})-\varepsilon/2\right)h_{S,\mathcal{M}}(\pi_{[m]}(x))-C_3,
\end{align*}
for any $x\in\mathcal{U}^1(\bar{\mathbb{Q}})$, where $C_3$ depends on $n$. We now use an estimate of Call and Silverman~\cite[Theorem~3.1]{CS-height}: there is a constant $C_4>0$ depending only on $(\mathcal{X},f,\mathcal{L})$ and $\mathcal{M}$ such that for any $x\in \mathcal{X}^0(\bar{\mathbb{Q}})$, 
\[\left|\widehat{h}_{f_{\pi(x)}}(x)-h_{\mathcal{X},\mathcal{L}}(x)\right|\leq C_4\left(h_{S,\mathcal{M}}(\pi(x))+1\right).\]
By functorial properties of heights \[h_{\mathcal{X}^{[m]},\mathcal{L}^{[m]}}(x)=\sum_{j=1}^mh_{\mathcal{X},\mathcal{L}}(x_j)+O(1), \quad x=(x_1,\ldots,x_m)\in \mathcal{X}^{[m]}(\bar{\mathbb{Q}}),\]
and the construction of the canonical height gives
\[\widehat{h}_{f^{[m]}_{\pi_{[m]}(x)}}(x)=\sum_{i=1}^m\widehat{h}_{f_{\pi(x_i)}}(x_i), \quad x=(x_1,\ldots,x_m)\in \mathcal{X}^{[m]}(\bar{\mathbb{Q}}),\]
since $\pi_{[m]}(x)=\pi(x_i)$ for any $i$ by construction.
Applying this inequality to $(f^{[m]})^n(x)$ and using that $\widehat{h}_{f_{\pi(x_i)}}(f^n(x_i))=d^n\widehat{h}_{f_{\pi(x_i)}}(x_i)$ for any $i$, we find
\[\sum_{j=1}^m\widehat{h}_{f_{\pi(x_j)}}(x_j)\geq \left(\widehat{\mathcal{H}}_{f,\mathcal{M}}(\mathcal{Y})-\frac{\varepsilon}{2} -\frac{C_4}{d^n}\right)h_{S,\mathcal{M}}(\pi_{[m]}(x))-C_3 -\frac{C_4}{d^n},\]
for any $x=(x_1,\ldots,x_m)\in \mathcal{U}^1(\bar{\mathbb{Q}})$. Up to increasing $n$, we can assume $C_4\leq d^n\varepsilon/2$, which gives the expected inequality.
\end{proof}

As an immediate application of Theorem~\ref{tm:Gao-Habegger}, we have

\begin{corollary}\label{cor:Gao-Habegger}
Fix $m\geq\dim S$ and assume $\mathrm{Vol}^{(m)}_{f}(\mathcal{Y})>0$. Let $\mathcal{M}$ be any ample $\mathbb{Q}$-line bundle on $S$ of volume $1$ and let $0<\varepsilon<\widehat{\mathcal{H}}_{f,\mathcal{M}}^{(m)}(\mathcal{Y})$. There is a non-empty Zariski open subset $\mathcal{U}\subset \mathcal{Y}^0$ and a constant $C\geq1$ depending only on $(\mathcal{X},f,\mathcal{L})$, $\mathcal{Y}$, $\mathcal{M}$, $m$ and $\varepsilon$ such that
\[h_{S,\mathcal{M}}(\pi(x))\leq \frac{m}{\widehat{\mathcal{H}}^{(m)}_{f,\mathcal{M}}(\mathcal{Y})-\varepsilon}\widehat{h}_{f_{\pi(x)}}(x)+C, \quad x\in\mathcal{U}(\bar{\mathbb{Q}}).\]
\end{corollary} 

\begin{proof} Fix $\mathcal{M}$ and $0<\varepsilon<\widehat{\mathcal{H}}_{f,\mathcal{M}}^{(m)}(\mathcal{Y})$ and let $\mathcal{B}$ be the set of points $x\in \mathcal{Y}$ such that 
	  	  \[h_{S,\mathcal{M}}(\pi(x))> \frac{m}{\widehat{\mathcal{H}}_{f,\mathcal{M}}^{(m)}(\mathcal{Y})-\varepsilon}\widehat{h}_{f_{\pi(x)}}(x)+C\]
	where $C$ is the constant given by Theorem~\ref{tm:Gao-Habegger}. We deduce that
	\[\forall (x_1, \dots, x_m)\in \mathcal{B}^{[m]},\  h_{S,\mathcal{M}}(\pi_{[m]}(x)) > \frac{1}{\widehat{\mathcal{H}}_{f,\mathcal{M}}^{(m)}(\mathcal{Y})-\varepsilon}\sum_{j=1}^{m}\widehat{h}_{f_{\pi(x_j)}}(x_j)+C,\]
	so that $\mathcal{B}^{[m]}$ is necessarily contained in a strict Zariski closed subset of $\mathcal{Y}^{[m]}$ by Theorem~\ref{tm:Gao-Habegger}, hence $\mathcal{B}$ is contained in a strict Zariski closed subset of $\mathcal{Y}$.
	\end{proof}

\subsection{General dynamical heights as moduli heights}
Let $(\mathcal{X},f,\mathcal{L},\mathcal{Y})$  be a dynamical pair of degree $d\geq2$ parametrized by a smooth projective variety $S$, with regular part $S^0_\mathcal{Y}$, all defined over a number field $\mathbb{K}$. 

When $Z$ is a subvariety of $S^0_\mathcal{Y}$, we let $\mathcal{Y}_Z:=\pi|_\mathcal{Y}^{-1}(Z)$ and we define $(\mathcal{X}_Z,f_Z,\mathcal{L}_Z)$ as the family of polarized endomorphisms induced by restriction of $(\mathcal{X},f,\mathcal{L})$ to $\mathcal{X}_Z:=\pi^{-1}(Z)$.

\medskip

We prove here the following

\begin{theorem}\label{tm:ample-height}
let $(\mathcal{X},f,\mathcal{L})$, $S$ and $\mathcal{Y}$ be all defined over a number field $\mathbb{K}$. Fix an embedding $\iota:\mathbb{K} \to \C$ for which we define the different bifurcation currents and assume that $\mu_{f,\mathcal{Y}}\neq0$. Then, there is a non-empty Zariski open subset $U\subset S^0_\mathcal{Y}$ such that for any ample height $h$ on $U$, there are constants $C_1,C_2>0$ and $C_3,C_4\in\mathbb{R}$ such that
\[C_1 h(t)+C_3 \leq \widehat{h}_{f_t}(Y_t) \leq C_2 h(t)+C_4 \quad \text{for all} \ \ t\in U(\bar{\mathbb{Q}}).\]
 Moreover, for any archimedean place $v\in M_\mathbb{K}$, any irreducible component $Z$ of $S^0_\mathcal{Y}\setminus U$ satisfies $T_{f,\mathcal{Y},v}^{(\dim Z)}\wedge[Z]=0$.
\end{theorem}

We are now in position to prove Theorem~\ref{tm:critical-ample-height}.

\begin{proof}[Proof of Theorem~\ref{tm:critical-ample-height}]
We work at the archimedean place of $\mathbb{Q}$.
It follows from  Proposition~\ref{Joe'suggestion} (or from Theorem~\ref{tm:mu-interior}) that the bifurcation measure $\mu_{f,\mathrm{Crit}}$ is non-zero on $\mathscr{M}_{d}^{k}(\mathbb{C})$, hence it is sufficient to apply Theorem~\ref{tm:ample-height} to conclude the proof of Theorem~\ref{tm:critical-ample-height}.
\end{proof}

The proof of Theorem~\ref{tm:ample-height} splits into two distinct parts that are summarized in two Propositions below.
We first use Zhang's inequalities over number fields to deduce the following from Corollary~\ref{cor:Gao-Habegger}:

\begin{proposition}\label{cor:use-Zhang}
Let $\mathcal{M}$ be an ample $\mathbb{Q}$-line bundle on $S$ of volume $1$ and assume $\mathrm{Vol}^{(\dim S)}_{f}(\mathcal{Y})>0$. There are constants $C_1>0$ and $C_2\geq1$ depending only on $(\mathcal{X},f,\mathcal{L},\mathcal{Y})$ and $\mathcal{M}$ and a non-empty Zariski open subset $V\subset S^0_\mathcal{Y}$ defined over $\bar{\mathbb{Q}}$ such that
\[h_{S,\mathcal{M}}(t)\leq C_1\widehat{h}_{f_t}(Y_t)+C_2, \quad t\in V(\bar{\mathbb{Q}}).\]
Moreover, for any irreducible component $Z$ of $S^0_\mathcal{Y}\setminus V$, we have $\mathrm{Vol}_{f_Z}(\mathcal{Y}_Z)=0$.
\end{proposition}

\begin{proof}
Let $q:=\dim Y_\eta$. Fix $0<\varepsilon<\widehat{\mathcal{H}}_{f,\mathcal{M}}^{(\dim S)}(\mathcal{Y})$. Let $\mathcal{U}$ be the Zariski open subset in Corollary~\ref{cor:Gao-Habegger}, let $V$ be the set of $t\in\pi(\mathcal{U})$ so that $U_t:=\mathcal{U}\cap Y_t$ is non-empty a Zariski open subset of $Y_t$. The set $V$ is a Zariski open subset of $S^0$ and for any $t\in V(\bar{\mathbb{Q}})$, we have
\[h_{S,\mathcal{M}}(t)\leq \frac{\dim S}{\widehat{\mathcal{H}}_{f,\mathcal{M}}(\mathcal{Y})-\varepsilon}\left(\widehat{h}_{f_t}(x)+C\right), \quad x\in U_t(\bar{\mathbb{Q}}).\]
Taking the infimum of $\widehat{h}_{f_t}(x)$ over $x\in U_t(\bar{\mathbb{Q}})$ and using Zhang's inequalities \eqref{Zhang-faible} gives 
\[h_{S,\mathcal{M}}(t)\leq \frac{\dim S}{\widehat{\mathcal{H}}_{f,\mathcal{M}}(\mathcal{Y})-\varepsilon}\left((q+1)\widehat{h}_{f_t}(Y_t)+C\right).\]
This is the wanted inequality, but we may have restricted too much the open set.

To conclude, we can proceed exactly the same way on any irreducible component $Z$ of $S^0_\mathcal{Y}\setminus V$, where $\mathrm{Vol}_{f_Z}^{(\dim Z)}(\mathcal{Y}_Z)>0$. In finitely many steps, we end with the expected result.
\end{proof}

We now use another description of the height $h_{\bar{\mathcal{L}}}(Y_t)$, when $t\in S^0_{\mathcal{Y}}(\bar{\mathbb{Q}})$, using Chow forms as in \cite{Hutz-good}. The next is probably well-known, but we include a proof for the sake of completeness.

\begin{lemma}\label{lm:Chow-form}
Let $S$ be a projective variety, let $\pi:\mathcal{Y}\to S$ be a surjective morphism, both defined over a number field $\mathbb{K}$. Let $\bar{\mathcal{L}}$ be a relatively ample line bundle on $\mathcal{Y}$ endowed with an adelic relatively semi-positive metrization. Let $S^0_\mathcal{Y}\subseteq S$ be a Zariski open set such that $\pi$ is flat over $S^0_\mathcal{Y}$.

For any ample line bundle $\mathcal{M}$ on $S$, there are constants $C_1,C_2>0$ such that
\[h_{\bar{\mathcal{L}}}(Y_t)\leq C_1h_{S,\mathcal{M}}(t)+C_2, \quad t\in S^0_{\mathcal{Y}}(\bar{\mathbb{Q}}).\]
\end{lemma}

\begin{proof}
Up to replacing $\mathcal{L}$ by a large multiple and up to changing the metrization on $\mathcal{L}$, we may assume that there is an embedding $\iota:\mathcal{Y}\hookrightarrow \mathbb{P}^N_S$ such that $\bar{\mathcal{L}}=\iota^*\bar{\mathcal{O}}_{\mathbb{P}^N}(1)$, so that 
\[h_{\bar{\mathcal{L}}}(Y_t)=h_{\mathbb{P}^N}(\iota_*(Y_t)), \quad \text{for all} \ t\in S^0_\mathcal{Y}(\bar{\mathbb{Q}}),\]
where $h_{\mathbb{P}^N}$ is the naive height function on $\mathbb{P}^N$.
This is where Chow forms are used, to give a different description of $h_{\mathbb{P}^N}(Y_t)$, which makes easier the expected inequality.

\medskip

 For any irreducible subvariety $Y\subset\mathbb{P}^N$ of dimension $q\geq1$, in the Grassmannian $G(N-k-1,N)$ of linear subspaces of dimension $N-k-1$ of $\mathbb{P}^N$, the set
\[\mathcal{Z}_Y:=\{V\in G(N-k-1,N)\, ; \ V \cap Y\neq\varnothing\}\]
is an irreducible hypersurface. Moreover, in the Pl\"ucker coordinates, we have $\mathcal{Z}_Y=\{\mathcal{R}_Y=0\}$, where $\mathcal{R}_Y$ is a homogeneous polynomial satisfying the following properties, see,~e.g.,~\cite{Dalbec-Sturmfels,Hutz-good}:
\begin{enumerate}
\item if $Y$ is defined over $\bar{\mathbb{Q}}$, then $\mathcal{R}_Y$ is also defined over $\bar{\mathbb{Q}}$,
\item $\deg(\mathcal{R}_Y)=\deg(Y)$,
\item $h_{\mathbb{P}^N}(Y)=h([a_0:\cdots:a_M])$, where $a_0,\ldots,a_M$ are the coefficients of $\mathcal{R}_Y$.
\end{enumerate}
Coming back to our situation, the above gives
\[h_{\bar{\mathcal{L}}}(Y_t)=h([a_0(t):\cdots:a_M(t)]), \quad t\in S^0_{\mathcal{Y}}(\bar{\mathbb{Q}}).\]
We now observe that the map $A:t\in S^0_\mathcal{Y}\mapsto [a_0(t):\cdots:a_M(t)]\in \mathbb{P}^M$ is regular and defined over $\bar{\mathbb{Q}}$, i.e. $A\in \bar{\mathbb{Q}}[S^0_\mathcal{Y}]$. This observation is true by construction of the Chow form, see, e.g., \cite[\S3]{Sturmfels}. The lemma follows.
\end{proof}

As an application of Call and Silverman's fundamental work~\cite{CS-height} and from Lemma~\ref{lm:Chow-form}, we prove Theorem~\ref{tm:ample-height}:
\begin{proof}[Proof of Theorem~\ref{tm:ample-height}]
The left hand side inequality is proved in Proposition~\ref{cor:use-Zhang}. We now prove the right hand side inequality. Fix any closed point $t\in S^0(\bar{\mathbb{Q}})$.
By Zhang's inequality~\eqref{Zhang-faible}, if $(x_j)$ is a generic sequence of closed points of $Y_t(\bar{\mathbb{Q}})$, 
we have
\[\widehat{h}_{f_t}(Y_t)\leq\liminf_{j\to\infty}\widehat{h}_{f_t}(x_j) \quad \text{and} \quad \frac{1}{q+1}\liminf_{j\to\infty}h_{\bar{\mathcal{L}}}(x_j)\leq h_{\bar{\mathcal{L}}}(Y_t).\]
We now apply \cite[Theorem~3.1]{CS-height}: there exists constants $C,C'>0$ depending only on $(\mathcal{X},f,\mathcal{L},\mathcal{Y})$ and on $\mathcal{M}$ such that
\[\left|\widehat{h}_{f_t}(x)-h_{\bar{\mathcal{L}}}(x)\right|\leq Ch_\mathcal{M}(t)+C',\]
for all $x\in Y_t(\bar{\mathbb{Q}})$. The above implies
\[\widehat{h}_{f_t}(Y_t)\leq(q+1)h_{\bar{\mathcal{L}}}(Y_t)+Ch_\mathcal{M}(t)+C'.\]
The conclusion follows from Lemma~\ref{lm:Chow-form} above.
\end{proof}

\section{Two dynamical equidistribution results}\label{section_equi}
The purpose of this section is to state two arithmetic equidistribution's theorems on quasi-projective varieties. Equivalent statements were already obtained by Yuan and Zhang using their theory of adelic line bundles on quasi-projective varieties (\cite[Theorems 6.2.3 and 6.3.5]{YZ-adelic}). Nevertheless, for the sake of completeness, we choose to follow the works of the first author \cite{Good-height} and K\"uhne \cite{Kuhne} which are both of a more local nature.

\subsection{Good height functions on quasi-projective varieties}\label{sec:goodheight}
Let $V$ be a smooth quasi-projective variety defined over a number field $\mathbb{K}$ and let $\mathbb{K}\hookrightarrow \C$ be an embedding and let $h:V(\bar{\mathbb{Q}})\to\mathbb{R}$ be a function. A sequence $(F_i)_i$ of Galois-invariant finite subsets of $V(\bar{\mathbb{Q}})$ is 
\begin{itemize}
\item \emph{generic} if for any subvariety $Z\subset V$ defined over $\mathbb{K}$, there is $i_0$ such that $F_i\cap Z=\varnothing$ for $i\geq i_0$, and
\item $h$-\emph{small} if $h(F_i):=\frac{1}{\# F_i}\sum_{x\in F_i}h(x)\to0$, as $i\to\infty$.
\end{itemize}

 As in \cite{Good-height},  we say $h$ is a \emph{good height}  at the complex place if for any $n\geq0$, there is a projective model $X_n$ of $V$ together with a birational morphism $\psi_n:X_n\to X_0$ which is an isomorphism above $V$ and  a big and nef $\mathbb{Q}$-line bundle $L_n$ on $X_n$ endowed with an adelic semi-positive continuous metrization $\bar{L}_n$,
such that the following holds :
\begin{enumerate}
\item[$(1)$] For any generic $h$-small sequence $(F_i)_i$ of Galois-invariant finite subsets of $V(\bar{\mathbb{Q}})$, the sequence
$\varepsilon_n(\{F_i\}_i):=\limsup_ih_{\bar{L}_n}(\psi_n^{-1}(F_i))-h_{\bar{L}_n}(X_n)$ satisfies $\varepsilon_n(\{F_i\})\to0$ as $n\to\infty$,
\item[$(2)$] the sequence of volumes $\mathrm{vol}(L_n)$ converges to $\mathrm{vol}(h)>0$ as $n\to\infty$ and if $c_1(\bar{L}_n)$ is the curvature form of $\bar{L}_n$ on $X_{n}(\C)$, then the sequence of finite measures $\left(\mathrm{vol}(L_n)^{-1}(\psi_n)_*c_1(\bar{L}_n)^k\right)_n$ converges weakly on $V(\C)$ to a probability measure $\mu$,
\item[$(3)$] If $k:=\dim V>1$, for any ample line bundle $M_0$ on $X_0$ and any adelic semi-positive continuous metrization $\bar{M}_0$ on $M_0$, there is a constant $C\geq0$ such that 
\[\left(\psi_n^*(\bar{M}_0)\right)^j\cdot \left(\bar{L}_n\right)^{k+1-j}\leq C,\]
for any $2\leq j\leq k+1$ and any $n\geq0$.
\end{enumerate}
We say that $\mathrm{vol}(h)$ is the \emph{volume} of $h$ and that $\mu$ is the measure \emph{induced by $h$} over the complex numbers.

%
%

~

The first author proved in \cite[Theorem~1]{Good-height} the next result:
 
\begin{theorem}\label{tm:equidistrib}
For any $h$-small and generic sequence $(F_m)_m$ of Galois-invariant finite subsets of $V(\bar{\mathbb{Q}})$, the probability measure $\mu_{F_m}$ on $V(\C)$ which is  equidistributed on $F_m$ converges to $\mu$ in the weak sense of measures, i.e. for any $\varphi\in\mathscr{C}^0_c(V(\C))$, we have
\[\lim_{m\to\infty}\frac{1}{\# F_m}\sum_{y\in F_m}\varphi(y)=\int_{V(\C)}\varphi\,\mu.\]
\end{theorem}

\subsection{A dynamical relative equidistribution Theorem}

When $\pi:\mathcal{A}\to S$ is a family of abelian varieties defined over a number field $\mathbb{K}$, where $S$ is a smooth projective variety, and $\mathcal{Y}\subset \mathcal{A}$ is a non-degenerate subvariety also defined over $\mathbb{K}$, K\"uhne~\cite{Kuhne} proposes and proves a \emph{Relative Equidistribution Conjecture} which, in turn, says that if there is a generic sequence $\{x_i\}_i$ in $\mathcal{Y}^0(\bar{\mathbb{Q}})$ with $\widehat{h}_{\mathcal{A}}(x_i)\to0$, then the measure $\mu_{x_i}$ on $\mathcal{Y}^{0}(\mathbb{C})$ equidistributed on the Galois orbit $\mathsf{O}(x_i)$ converges weakly on $\mathcal{Y}^{0}(\mathbb{C})$ to a given probability measure $\mu$.

~

We want here to prove the next dynamical generalization of K\" uhne's Relative Equidistribution Conjecture:

\begin{theorem}[Dynamical Relative Equidistribution]\label{tm-dyn-REC}
Let $(\mathcal{X},f,\mathcal{L})$ be a family of polarized endomorphisms parametrized by a smooth projective variety $S$ and let $\mathcal{Y}\subset\mathcal{X}$ be a family of subvarieties of $\mathcal{X}$. Assume $\mu_{f,\mathcal{Y}}$ is non zero on $S^0(\mathbb{C})$. 

Then for any $m\geq\dim S$ and any $\varphi\in\mathscr{C}^0_c((\mathcal{Y}^{[m]})^{0}(\mathbb{C}),\mathbb{R})$ and any generic and $\widehat{h}_{f^{[m]}}$-small sequence $\{F_i\}_i$ of Galois invariant subsets of $(\mathcal{Y}^{[m]})^{0}(\bar{\mathbb{Q}})$, we have
\[\lim_{i\to\infty}\frac{1}{\# F_i}\sum_{x\in F_i}\varphi(x)=\frac{1}{\mathrm{Vol}_{f}^{(m)}(\mathcal{Y})}\int_{(\mathcal{Y}^{[m]})^{0}(\mathbb{C})}\varphi\, \widehat{T}_{f^{[m]}}^{\dim \mathcal{Y}^{[m]}}.\]
\end{theorem}

As mentioned above, we could also have used \cite[Theorem~6.2.3]{YZ-adelic}. To do that, we need to check that $\mathcal{Y}^{[m]}$ is non-degenerate which, in turn, is equivalent to the non-vanishing of the bifurcation measure (\cite[Lemma 5.4.4]{YZ-adelic}).

\begin{proof}
We fix an archimedean place of $\mathbb{K}$ and a corresponding embedding $\mathbb{K}\hookrightarrow \mathbb{C}$.
By Theorem~\ref{tm:equidistrib}, all there is to prove is that $\widehat{h}_{f^{[m]}}$ is a good height function on $(\mathcal{Y}^{[m]})^0$ and to show its induced measure on $(\mathcal{Y}^{[m]})^0(\mathbb{C})$ is indeed $\widehat{T}_{f^{[m]}}^{\dim \mathcal{Y}
^{[m]}}$.

Let $\mathcal{M}$ be an ample $\mathbb{Q}$-line bundle on $S$ of volume $1$ and let $\mathcal{L}_0:=\mathcal{L}^{[m]}+\pi_{[m]}^*\mathcal{M}$. The line bundle $\mathcal{L}_0$ is ample on $\mathcal{X}^{[m]}$.
Recall Call and Silverman's result~\cite[Theorem~3.1]{CS-height} guarantees the existence of $C\geq1$ such that
\[\left|\widehat{h}_f(x)-h_{\mathcal{X},\mathcal{L}}(x)\right| \leq C\left(h_S(\pi(x))+1\right),\]
for all $x\in \mathcal{X}^0(\bar{\mathbb{Q}})$.
As in the proof of Theorem~\ref{tm:Gao-Habegger}, using that $\widehat{h}_f\circ f=d\cdot\widehat{h}_f$, we deduce that up to changing the constant $C$, we have
\[\left|\widehat{h}_{f^{[m]}}(x)-\frac{1}{d^n}h_{\mathcal{X}^{[m]},\mathcal{L}^{[m]}}((f^{[m]})^n(x))\right| \leq \frac{C}{d^n}\left(h_S(\pi_{[m]}(x))+1\right),\]
for any $x\in (\mathcal{X}^{[m]})^0(\bar{\mathbb{Q}})$ and any $n\geq0$. As $((f^{[m]})^n)^*\mathcal{L}_0=d^n\mathcal{L}^{[m]}+\pi_{[m]}^*\mathcal{M}$, this implies
\[\left|\widehat{h}_{f^{[m]}}(x)-\frac{1}{d^n}h_{\mathcal{X}^{[m]},\mathcal{L}_0}((f^{[m]})^n(x))\right| \leq \frac{C}{d^n}\left(h_S(\pi_{[m]}(x))+2\right),\]
for any $x\in (\mathcal{X}^{[m]})^0(\bar{\mathbb{Q}})$ and any $n\geq0$. 
We now use Theorem~\ref{tm:Gao-Habegger}: there is a non-empty Zariski open set $\mathcal{V}\subset(\mathcal{Y}^{[m]})^0$ such that for any $x\in\mathcal{V}(\bar{\mathbb{Q}})$, we have
\[h_S(\pi_{[m]}(x))\leq \frac{2}{\widehat{\mathcal{H}}^{(m)}_{f,\mathcal{M}}(\mathcal{Y})}\left(\widehat{h}_{f^{[m]}}(x)+1\right).\]
We thus have a constant $C_2>0$ such that for any $x\in \mathcal{V}(\bar{\mathbb{Q}})$ and any $n\geq0$,
\begin{align}
\left|\widehat{h}_{f^{[m]}}(x)-\frac{1}{d^n}h_{\mathcal{X}^{[m]},\mathcal{L}^{[m]}}((f^{[m]})^n(x))\right| \leq \frac{C_2}{d^n}\left(\widehat{h}_{f^{[m]}}(x)+1\right).\label{eq:assumption1-good}
\end{align}
We now use Lemma~\ref{lm:goodmodel}: let $F_n,\psi_n:\mathcal{X}_n\to \mathcal{X}^{[m]}$ be such that $F_n=(f^{[m]})^n\circ\psi_n$ on $\psi_n^{-1}(\mathcal{X}^0)$ with $\psi_n$ birational. We also let $\pi_n:\mathcal{X}_n\to S$ be the structure morphism induced by $\pi_{[m]}$, i.e. such that $\pi_n=\pi_{[m]}\circ\psi_n$. Choose a model metric $\bar{\mathcal{M}}$ on $\mathcal{M}$ with $h_{\bar{\mathcal{M}}}\geq0$ on $S(\bar{\mathbb{Q}})$. We endow $\mathcal{L}$ with a metrization $\bar{\mathcal{L}}$ coming from the embedding
$\iota:\mathcal{X}\hookrightarrow \mathbb{P}^N\times S$ for which (a multiple) of $\mathcal{L}$ is $\iota^*\bar{\mathcal{O}}_{\mathbb{P}^N}(1)$, where $\bar{\mathcal{O}}_{\mathbb{P}^N}(1)$ is the naive metrization. Define
\begin{center}
$\bar{\mathcal{L}}_0:=\bar{\mathcal{L}}^{[m]}+(\pi_{[m]})^*\bar{\mathcal{M}}$.
\end{center}
We then let $\mathcal{Y}_n:=\psi^{-1}_n(\mathcal{Y}^{[m]})$ and
\[\bar{\mathcal{L}}_n:=\frac{1}{d^n}\left(F_n^*\bar{\mathcal{L}}_0
\right)|_{\mathcal{Y}_n}=\frac{1}{d^n}\left(F_n^*\bar{\mathcal{L}}^{[m]}\right)|_{\mathcal{Y}_n}+\frac{1}{d^n}(\pi_n^*\bar{\mathcal{M}})|_{\mathcal{Y}_n}, \quad n\geq0.\]
By construction the map $F_n$ is a generically finite morphism. Since $\bar{\mathcal{L}}_0$ is an adelic semi-positive continuous ample line bundle, $\bar{\mathcal{L}}_n$ is thus an adelic semi-positive continuous big and nef $\mathbb{Q}$-line bundle on $\mathcal{Y}_n$. Moreover, by construction, we have
\[h_{\mathcal{Y}_n,\bar{\mathcal{L}}_n}(\psi_n^{-1}(x))=\frac{1}{d^n}h_{\mathcal{X}^{[m]},\mathcal{L}_0}((f^{[m]})^n(x)),\]
for any $n\geq0$ and any $x\in(\mathcal{Y}^{[m]})^0(\bar{\mathbb{Q}})$. Note also that, by construction, $h_{\mathcal{Y}_n,\bar{\mathcal{L}}_n}\geq0$ on $\mathcal{Y}_n(\bar{\mathbb{Q}})$, so that \cite[Lemma~6]{Good-height} gives
\[h_{\bar{\mathcal{L}}_n}(\mathcal{Y}_n)\geq0.\]
We combine this inequality with the inequality~\eqref{eq:assumption1-good}: this implies that for any generic sequence $\{F_i\}_i$ of Galois invariant subsets of $(\mathcal{Y}^{[m]})^0(\bar{\mathbb{Q}})$ with $\widehat{h}_{f^{[m]}}(F_i)\to0$, we have
\[\limsup_{i\to\infty}\left(h_{\mathcal{Y}_n,\bar{\mathcal{L}}_n}(\psi_n^{-1}(F_i))-h_{\mathcal{Y}_n,\bar{\mathcal{L}}_n}(\mathcal{Y}_n)\right)\leq\limsup_{i\to\infty}h_{\mathcal{Y}_n,\bar{\mathcal{L}}_n}(\psi_n^{-1}(F_i))\leq 2\frac{C_2}{d^n}.\]
We now let $\omega$ and $\rho$ be the respective curvature forms $\omega:=c_1(\bar{\mathcal{L}}|_{\mathcal{Y}})$ and $\rho=c_1(\bar{\mathcal{M}})$ on $\mathcal{Y}(\mathbb{C})$ and $S(\mathbb{C})$ respectively. Then $\omega$ is a smooth form on $\mathcal{Y}(\mathbb{C})$ representing $c_1(\mathcal{L}|_{\mathcal{Y}})$, and if we denote as before $p_j:\mathcal{Y}^{[m]}\to\mathcal{Y}$ the projection onto the $j$-th factor of the fiber-product, the curvature form of $\bar{\mathcal{L}}_n$ satisfies as forms on $(\mathcal{Y}^{[m]})(\mathbb{C})$:
\begin{align*}
c_1(\bar{\mathcal{L}}_n)^{\dim \mathcal{Y}^{[m]}} & =d^{-n\dim\mathcal{Y}^{[m]}}\psi_n^*\left(((f^{[m]})^n)^*\left(p_1^*(\omega)+\cdots+p_m^*(\omega)+\pi_{[m]}^*(\rho)\right)^{\dim\mathcal{Y}^{[m]}}\right)
\end{align*}
so that, if $\omega_{m}:=p_1^*(\omega)+\cdots+p_m^*(\omega)$, we have as measures on $(\mathcal{Y}^{[m]})(\mathbb{C})$:
\[(\psi_n)_*\left(c_1(\bar{\mathcal{L}}_n)^{\dim\mathcal{Y}^{[m]}}\right)=d^{-n\dim\mathcal{Y}^{[m]}}\left(((f^{[m]})^n)^*\left(\omega_{m}+\pi_{[m]}^*(\rho)\right)^{\dim\mathcal{Y}^{[m]}}\right).\]
Now, as $d^{-n}((f^{[m]})^n)^*\omega_{m}$ converges to $\widehat{T}_{f^{[m]}}$ with a uniform convergence of local potentials and as we have $((f^{[m]})^n)^*\pi_{[m]}^*(\rho)=\pi_{[m]}^*(\rho)$ by construction, the following holds in the weak sense of measures on $(\mathcal{Y}^{[m]})(\mathbb{C})$:
\[d^{-n\dim\mathcal{Y}^{[m]}}\left(((f^{[m]})^n)^*\left(\omega_{m}+\pi_{[m]}^*(\rho)\right)^{\dim\mathcal{Y}^{[m]}}\right)\to\widehat{T}_{f^{[m]}}^{\dim\mathcal{Y}^{[m]}}.\]
Finally,  the volume of $\mathcal{L}_n$ can be computed as
\[\mathrm{deg}_{\mathcal{Y}_n}(\mathcal{L}_n)=\mathrm{Vol}_{f}^{(m)}(\mathcal{Y})+O\left(\frac{1}{d^n}\right).\]
Indeed, by definition of $\mathcal{L}_0$ and by Lemma~\ref{lm:GV}, we find
\begin{align*}
\mathrm{deg}_{\mathcal{Y}_n}(\mathcal{L}_n) &  =\left(c_1(\mathcal{L}_n)^{\dim\mathcal{Y}^{[m]}}\cdot\{\mathcal{Y}_n\}\right)=\left(d^{-n\dim \mathcal{Y}^{[m]}}(F_n)^*c_1(\mathcal{L}_0)^{\dim\mathcal{Y}^{[m]}}\cdot \{\mathcal{Y}^{[m]}\}\right)\\
 &  = \left(d^{-n\dim \mathcal{Y}^{[m]}}(F_n)^*c_1(\mathcal{L}^{[m]})^{\dim\mathcal{Y}^{[m]}}\cdot \{\mathcal{Y}^{[m]}\}\right)+ O\left(\frac{1}{d^n}\right)\\
& = \mathrm{Vol}^{(m)}_{f}(\mathcal{Y})+ O\left(\frac{1}{d^n}\right).
\end{align*}
Our assumption that $\mathrm{Vol}^{(m)}_{f}(\mathcal{Y})>0$ thus implies $\lim\limits_{n\to\infty}\deg_{\mathcal{Y}_n}(\mathcal{L}_n)=\mathrm{Vol}^{(m)}_{f}(\mathcal{Y})>0$.

\medskip

 To prove that $\widehat{h}_{f^{[m]}}$ is a good height function on $(\mathcal{Y}^{[m]})^0(\bar{\mathbb{Q}})$, the last thing to check is condition (3) introduced in Section~\ref{sec:goodheight}.  Let $\pi_n:\mathcal{Y}_n\to S$ be the morphism induced by $\pi_{[m]}:\mathcal{Y}^{[m]}\to S$. 
 
Let $\bar{M}_0$ be an ample adelic semi-positive continuous line bundle on $\mathcal{Y}_0$. Then $\psi_n^*\bar{M}_0$ is a big and nef $\mathbb{Q}$-line bundle on $\mathcal{Y}_n$ and $\psi_n^*\bar{M}_0$ is a semi-positive adelically metrized line bundle on $\mathcal{Y}_n$. Let $\bar{\mathcal{E}}:=d^{-1}F_1^*\bar{\mathcal{L}}-\psi_1^*\bar{\mathcal{L}}$. Then, there is $\bar{D}$ effective on $S$ such that $-d^{-1}\pi_1^*\bar{D}\leq \bar{\mathcal{E}}\leq d^{-1}\pi_1^*\bar{D}$.

\medskip

By construction, we can assume there is a birational morphism $\phi_n:\mathcal{Y}_{n+1}\to\mathcal{Y}_n$ with $\pi_n\circ\phi_{n+1}=\pi_{n+1}$ and that $\psi_{n+1}=\psi_n\circ \phi_{n+1}$. Without loss of generality, we can also assume $\psi_n=\phi_1\circ\cdots\circ\phi_n$ and there is a morphism $g_n:\mathcal{Y}_{n+1}\to\mathcal{Y}_1$ such that 
\begin{center}
$\phi_{1}\circ g_n= F_n\circ \phi_{n+1}$ and $F_1\circ g_n=F_{n+1}$ on $\mathcal{Y}_{n+1}$.
\end{center} 
We have $d^{-n}g_n^*(\bar{\mathcal{E}})\leq  d^{-(n+1)}g_n^*\pi_{1}^*\bar{D}=d^{-(n+1)}\pi_{n+1}^*\bar{D}$.
In particular, one sees that
 \begin{align*}
\bar{\mathcal{L}}_{n+1}-\phi_{n+1}^*\bar{\mathcal{L}}_n= & \, \frac{1}{d^n}g_{n}^*\left(\bar{\mathcal{E}}\right)+\left(\frac{1}{d^{n+1}}-\frac{1}{d^n}\right)\pi_{n+1}^*(\bar{\mathcal{M}})\\
& \leq \frac{1}{d^{n+1}}\pi_{n+1}^*(\bar{D}+\bar{\mathcal{M}}).
\end{align*}
Hence $\bar{\mathcal{L}}_{n+1}\leq \phi_{n+1}^*(\bar{\mathcal{L}}_n+d^{-(n+1)}\pi_n^*(\bar{D}+\bar{\mathcal{M}}))$.
An immediate induction gives
 \begin{align}
\bar{\mathcal{L}}_{n}\leq \psi_{n}^*\left(\bar{\mathcal{L}}_0+\frac{d}{d-1}\pi_{[m]}^*(\bar{D}+\bar{\mathcal{M}})\right).\label{ineg-good-height}
\end{align}
Let $P:=\dim\mathcal{Y}^{[m]}$ and pick $0\leq \ell\leq P+1$. For all $n\geq0$, \eqref{ineg-good-height} gives
\begin{align*}
\left(\left.\left(\psi_n^*(\bar{M}_0)\right)^\ell\cdot\left(\bar{\mathcal{L}}_n\right)^{P+1-\ell}\right|\mathcal{Y}_n\right) & \leq \left(\left.\left(\psi_n^*(\bar{M}_0)\right)^\ell\cdot\psi_{n}^*\left(\bar{\mathcal{L}}_0+\frac{d}{d-1}\pi_{[m]}^*(\bar{D}+\bar{\mathcal{M}})\right)^{P+1-\ell}\right|\mathcal{Y}_n\right)\\
& \leq \left(\left.\left(\bar{M}_0\right)^\ell\cdot\left(\bar{\mathcal{L}}_0+\frac{d}{d-1}\pi_{[m]}^*(\bar{D}+\bar{\mathcal{M}})\right)^{P+1-\ell}\right|\mathcal{Y}^{[m]}\right),
\end{align*}
where we used the projection formula and that $(\psi_n)_*(\mathcal{Y}_n)=\mathcal{Y}^{[m]}$. This proves hypothesis (3) of section~\ref{sec:goodheight} is satisfied as the last quantity is independent of $n$ and the proof of Theorem~\ref{tm:equidistrib} is complete.
\end{proof}

\subsection{Parametric equidistribution}
For any finite Galois invariant subset $F\subset S^0(\bar{\mathbb{Q}})$, we define $h_{f,\mathcal{Y}}(F)$ as
\[h_{f,\mathcal{Y}}(F):=\frac{1}{\# F}\sum_{t\in F}\widehat{h}_{f_t}(Y_t).\]
As usual, we say a sequence $F_i$ of finite Galois invariant subsets of $S^0(\bar{\mathbb{Q}})$ is \emph{$h_{f,\mathcal{Y}}$-small} if $h_{f,\mathcal{Y}}(F_i)\to0$.

\begin{corollary}\label{tm:small-distrib}
Let $(\mathcal{X},f,\mathcal{L},\mathcal{Y})$ be a dynamical pair parametrized by a smooth projective variety $S$ with regular part $S^0$, all defined over a number field $\mathbb{K}$. Assume $\mathrm{Vol}_f^{(\dim S)}(\mathcal{Y})>0$. Assume also there is a generic and $h_{f,\mathcal{Y}}$-small sequence $\{F_i\}_i$ of finite Galois invariant subsets of $S^0(\bar{\mathbb{Q}})$. Then for any $\varphi\in\mathscr{C}^0_c(S^{0}(\mathbb{C}),\mathbb{R})$, we have
\[\lim_{i\to\infty}\frac{1}{\# F_i}\sum_{t\in F_i}\varphi(t)=\int_{S^{0}(\mathbb{C})}\varphi \, \frac{\mu_{f,\mathcal{Y}}}{\mathrm{Vol}_f(\mathcal{Y})}.\]
\end{corollary}
Observe that this result corresponds to \cite[Theorem 6.3.5]{YZ-adelic}. Our approach is based on Zhang inequalities over number fields and Theorem~\ref{tm-dyn-REC}, whereas Yuan and Zhang rely on properties of metrics on the Deligne pairing on adelic line bundles.

\begin{proof}
Fix $m\geq\dim S$ and fix $i$ and let $t\in F_i$. Zhang's inequalities~\eqref{Zhang-faible} imply there exists a generic sequence $\{x_j^{(t)}\}_j$ of $Y_{t}^{[m]}(\bar{\mathbb{Q}})$ such that we have 
\[\limsup_{j\to\infty}\widehat{h}_{f^{[m]}}(x_j^{(t)})\leq (q+1)\widehat{h}_{f^{[m]}_{t}}(Y^{[m]}_{t})~.\]
For any $i,j$, we define a finite Galois invariant subset $Z_j^i$ of $(\mathcal{Y}^{[m]})^0(\bar{\mathbb{Q}})$ by letting
\[Z_j^{i}:=\bigcup_{t\in F_i}\mathsf{O}(x_j^{(t)}).\]
By the above, and by Lemma~\ref{lm:Zhang-product}, we deduce that
\[\liminf_{j\to\infty}\widehat{h}_{f^{[m]}}(Z_j^i)\leq (q+1)h_{f^{[m]},\mathcal{Y}^{[m]}}(F_i)= m(q+1)\cdot h_{f,\mathcal{Y}}(F_i).\]
Take $\varepsilon_i>0$ such that $\varepsilon_i\to0$ as $i\to\infty$ and such that $(q+1)^2m\cdot h_{f,\mathcal{Y}}(F_i)\leq \varepsilon_i$ for any $i$. For any $i\geq1$, there is an infinite sequence $(j_n(i+1))_n$, extracted from $(j_n(i))_n$ such that, for any $n\geq0$, we have $\widehat{h}_{f^{[m]}}(Z_{j_n(i)}^i)\leq 2\varepsilon_i$.
We deduce there exists a sequence $\{Z_i\}_i$ of finite Galois invariant finite subsets of $(\mathcal{Y}^{[m]})^0(\bar{\mathbb{Q}})$ such that $\pi_{[m]}(Z_i)=F_i$ and such that for any $t\in F_i$, we have $\mathsf{O}(x_{j(i)}^{(t)})\subset \pi_{[m]}^{-1}\{t\}$ for some $j(i)\geq j_0(i)$. Moreover, by construction we can choose $Z_i$ generic, and we have
\begin{align*}
0\leq \frac{1}{\# Z_i}\sum_{x\in Z_i}\widehat{h}_{f^{[m]}}(x)\leq  2\varepsilon_i.
\end{align*}
As $\varepsilon_i\to0$ and $\{Z_i\}_i$ is generic, Theorem~\ref{tm-dyn-REC} implies 
\[\frac{1}{\# Z_i}\sum_{x\in Z_i}\delta_{x,v}\to \mu_v,\]
where $\mu_v$ is a probability measure on $(\mathcal{Y}^{[m]})_v^{0,\mathrm{an}}$ which satisfies
\begin{center}
$(\pi_{[m]})_*(\mu_v)=\mu_{f,[\mathcal{Y}],v}(S^{0,\mathrm{an}}_v)^{-1}\mu_{f,[\mathcal{Y}],v}$, 
\end{center}
when $v$ is archimedean.
Let $\nu_v:=(\pi_{[m]})_*(\mu_v)$ and take $\varphi\in\mathscr{C}^0_v(S_v^{0,\mathrm{an}},\mathbb{R})$. Then
\[\frac{1}{\# F_i}\sum_{t\in F_i}\varphi(t)=\frac{1}{\# F_i}\sum_{t\in F_i}\frac{1}{\#\mathsf{O}(x_{j(i)}^{(t)})}\sum_{x\in\mathsf{O}(x_{j(i)}^{(t)})}\varphi(\pi_{[m]}(x))=\frac{1}{\# Z_i}\sum_{x\in Z_i}\varphi(\pi_{[m]}(x)).\]
We now use that $\nu_v=(\pi_{[m]})_*(\mu_v)$ so that
\[\int_{S_v^{0,\mathrm{an}}}\varphi \nu_v=\int_{(\mathcal{Y}^{[m]})_v^{0,\mathrm{an}}}(\varphi\circ \pi_{[m]})\mu_v.\]
Finally, if $v$ is archimedean, since $(\pi_{[m]})_*(\mu_v)=\left(\mu_{f,[\mathcal{Y}],v}(S^{0,\mathrm{an}}_v)\right)^{-1}\cdot\mu_{f,\mathcal{Y},v}$, we have $\nu_v=\left(\mu_{f,\mathcal{Y},v}(S^{0,\mathrm{an}}_v)\right)^{-1}\cdot\mu_{f,\mathcal{Y},v}$ and the proof is complete.
\end{proof}

\section{Sparsity and uniformity: proof of the main results}\label{section_harvest}
We are now interested in applying all the above results in two specific situations, where we study the variations of the dynamics of the critical set.
\subsection{Sparsity of PCF maps of $\mathbb{P}^k$}
We focus the family introduced in Section~\ref{sec:moduli-space} (which plays the role of a universal family here),
which is a family $(\mathbb{P}^k_{S},f,\mathcal{O}_{\mathbb{P}^k}(1))$ of degree $d$ endomorphisms of $\mathbb{P}^k$ parametrized by a projective model $S$ of finite branched cover of $\mathscr{M}_d^k$ with regular part $\mathcal{U}_d^k$ -- if we follow the notations introduced above -- which is defined over $\mathbb{Q}$, see Lemma~\ref{good-family-end}.

\medskip 

The critical variety $\mathrm{Crit}(f)\subsetneq \mathbb{P}^k_S$ satisfies $\pi(\mathrm{Crit}(f))=S$, where $\pi:\mathbb{P}^k_S\to S$ is the canonical projection, and $\pi|_{\mathrm{Crit}(f)}$ is flat and projective over a Zariski open subset $S^0\subseteq\mathscr{M}_d^k$. Moreover, for any $t\in S^0$, the fiber $\mathrm{Crit}(f_t)=\pi|_{\mathrm{Crit}(f)}^{-1}(t)$ is the critical locus of $f_t$. Moreover, up to reducing the open set $S^0$, we can assume $\mathrm{Crit}(f_t)$ is irreducible for all $t\in S^0$.

\bigskip

We are now in position to prove Theorem~\ref{tm:not-Zariski-dense}.

\begin{proof}[Proof of Theorem~\ref{tm:not-Zariski-dense}]
Recall that being PCF is a property which is invariant under conjugacy, hence the set of PCF maps $f\in\mathrm{End}_d^k$ is Zariski dense in $\mathrm{End}_d^k$ if and only if their conjugacy classes are Zariski dense in $\mathscr{M}_d^k$. 
To prove Theorem~\ref{tm:not-Zariski-dense}, we proceed by contradiction. Assume PCF maps are Zariski dense in $\mathrm{End}_d^k$. Then they are Zariski dense in $\mathcal{U}_d^k$. In this case, those which are defined over $\bar{\mathbb{Q}}$ are countable and Zariski dense.

We thus can find a generic sequence $(t_n)_{n\in\mathbb{N}}$ of PCF parameters $t_n\in\mathcal{U}_d^k(\bar{\mathbb{Q}})$. 
Note here that the bifurcation measure $\mu_{f,\mathrm{Cirt}}$ of the family we consider is the the pull-back of the measure of the moduli space by the canonical projection $\Pi:\mathcal{U}_d^k\to\mathscr{M}_d^k$ which is finite and whose image $\mathscr{U}_d^k$ contains a non-empty Zariski open subset of $\mathscr{M}_d^k$. In particular, $\mu_{f,\mathrm{Crit}}$ is non-zero, and $\mathrm{supp}(\mu_{f,\mathrm{Crit}})$ contains a non-empty analytic open subset $\Omega$ which contains no PCF parameters, by Theorem~\ref{tm:mu-interior}. 

Let now $\mu_n$ be the measure of $\mathcal{U}_{d}^{k}(\mathbb{C})$ equidistributed on the Galois orbit $\mathsf{O}(t_n)$ of $t_n$. By the parametric equidistribution Theorem (see Corollary~\ref{tm:small-distrib}), we have
\[\mu_n:=\frac{1}{\mathrm{Card}(\mathsf{O}(t_n))}\sum_{t\in\mathsf{O}(t_n)}\delta_t\to\mu_{f,\mathrm{Crit}}, \quad \text{as} \ n\to\infty.\]
In particular, in the analytic topology of $\mathcal{U}_{d}^{k}(\mathbb{C})$, the support of $\mu_{f,\mathrm{Crit}}$ is accumulated by PCF classes. In particular, PCF parameters are dense in $\Omega$. This is a contradiction.
\end{proof}

\subsection{Height gap and uniformity for regular maps of the affine space}

In this section, we focus on the case when $\mathcal{X}=\mathbb{P}^k\times S$ and where there is a hyperplane $H_\infty\subset \mathbb{P}^k$ such that $f_t^{-1}(H_\infty)=H_\infty$ for all $t\in S^0$. We call such a family of \emph{regular polynomial endomorphisms of the affine space} $\mathbb{A}^k$, see~\cite{bedford-jonsson}. 
Choosing an affine chart, we can assume the hyperplane $H_\infty$ is the hyperplane at infinity of $\mathbb{A}^k$ in $\mathbb{P}^k$. When $(\mathbb{P}^k\times S,f,\mathcal{O}_{\mathbb{P}^k}(1))$ is such a family of regular polynomial endomorphisms, we let
\[G_{f_t}(z)=G_f(z,t):=\lim_{n\to\infty}\frac{1}{d^n}\log^+\|f_t^n(z)\|, \]
for all $z\in \mathbb{C}^k$ and all $t\in S^0(\mathbb{C})$.

We let $\mathcal{Y}\subset\mathbb{P}^k\times S$ be an irreducible hypersurface that projects surjectively onto $S$ and which intersects properly $H_\infty\times S$. Up to reducing the Zariski open set $S^0$, we can assume $\mathcal{Y}$ is flat over $S^0$ and $Y_t\neq H_\infty$ for all $t\in S^0$. 

\begin{definition}
The \emph{polynomial bifurcation measure} $\mu^{\mathrm{pol}}_{f\mathcal{Y}}$ of the pair $(\mathbb{P}^k\times S,f,\mathcal{O}_{\mathbb{P}^k}(1),\mathcal{Y})$ is the Monge-Amp\`ere measure associated to the function $G_{f,\mathcal{Y}}:S^0(\mathbb{C})\to\mathbb{R}^+$ defined by
\[G_{f,\mathcal{Y}}(t):=\int_{\mathbb{C}^k}G_{f_t} (dd^cG_{f_t}(z))^{k-1}\wedge [Y_t], \quad t\in S^0(\mathbb{C}),\]
i.e. $\mu^{\mathrm{pol}}_{f\mathcal{Y}}:=(dd^cG_{f,\mathcal{Y}})^{\dim S}$ as a measure on $S^0(\C)$.
\end{definition}
The measure $\mu^{\mathrm{pol}}_{f\mathcal{Y}}$ detects phenomena which occur in the affine space. However, it does not in general allow to collect all the informations that $\mu_{f,\mathcal{Y}}$ carries. However, as measures on $S^0$, we have
\[\mu_{f,\mathcal{Y}}\geq \mu^{\mathrm{pol}}_{f,\mathcal{Y}}.\]

We now prove here the following which is a sufficient condition to get a height gap, and then to deduce uniformity in a Bogomolov-type statement.  To pursue the parallel with \cite{YZ-adelic}, this Theorem in turn says that if the Deligne pairing $\bar{M}$ relatively to $\pi|_{\mathcal{Y}}:\mathcal{Y}\to S$ of the adelic line bundle $\bar{L}_f$ has strictly positive arithmetic volume, then one has a uniform Bogomolov-type statement in the total space.

\begin{theorem}\label{prop-empty-interior}
Let $(\mathbb{P}^k\times S,f,\mathcal{O}_{\mathbb{P}^k}(1))$ be a family of regular polynomial endomorphisms of degree $d$ of the affine space parametrized by $S$, and  let $\mathcal{Y}\subset\mathbb{P}^k\times S$ be an irreducible hypersurface such that $\pi|_\mathcal{Y}:\mathcal{Y}\to S$ is surjective,  all defined over a number field. Assume $\mathcal{Y}$ intersects properly $H_\infty\times S$.
Assume also
\[\int_{S(\mathbb{C})}G_{f,\mathcal{Y}}\cdot\mu^{\mathrm{pol}}_{f,\mathcal{Y}}>0.\]
Then there exists $Z\subsetneq S$ Zariski-closed, $\varepsilon>0$, and an integer $N\geq1$, such that for all $t\in (S^0\setminus Z)(\bar{\mathbb{Q}})$, there exists a strict subvariety $W_t\subsetneq Y_t$ with $\deg(W_t)\leq N$ and such that 
\[\{z\in Y_t(\bar{\mathbb{Q}})\, : \ \widehat{h}_{f_t}(z)\leq \varepsilon\}\subset W_t.\]
\end{theorem}

\begin{remark}\normalfont
By Zhang's inequalities~\eqref{Zhang-forte}, this in particular implies that 
\[\widehat{h}_{f_t}(Y_t)\geq \frac{\varepsilon}{k}>0, \quad \text{for all} \quad t\in (S^0\setminus Z)(\bar{\mathbb{Q}}).\]
\end{remark}

The key ingredient is the next lemma, which is of purely complex analytic nature. Again, in the parallel with \cite{YZ-adelic}, it says that if the Deligne pairing $\bar{M}$ has positive arithmetic volume, then $\bar{L}_f^{[m+1]}|\mathcal{Y}^{[m+1]}$ also has positive arithmetic volume. 
\begin{lemma}\label{lm:interior}
Let $(\mathbb{P}^k\times S,f,\mathcal{O}_{\mathbb{P}^k}(1))$ be a complex family of regular polynomial endomorphisms of the affine space of degree $d$ parametrized by $S$ of dimension $m$.
Then
\[\int_{\mathbb{C}^{k(m+1)}\times S}G_{f^{[m+1]}} (dd^cG_{f^{[m+1]}})^{km+k-1}\wedge[\mathcal{Y}^{[m+1]}]\geq \int_S G_{f,\mathcal{Y}}\cdot\mu^{\mathrm{pol}}_{f,\mathcal{Y}} .\]
\end{lemma}
\begin{proof}
We denote by $p_i:(\mathbb{P}^k)^{m+1}\times S\to\mathbb{P}^k\times S$ the projection onto the $i$-th factor of the fiber product and by $\pi_i:(\mathbb{P}^k)^{m+1}\times S\to(\mathbb{P}^k)^{m}\times S$ the projection consisting in forgetting the $i$-th factor. By construction $f^{[m]}$ and $f^{[m+1]}$ are families of regular polynomial endomorphisms of the affine spaces $\mathbb{A}^{km}$ and $\mathbb{A}^{k(m+1)}$ respectively. Moreover, for any $1\leq i\leq m+1$, we have
\[G_{f^{[m+1]}}=\sum_{j=1}^{m+1}G_{f}\circ p_j=G_f\circ p_i+G_{f^{[m]}}\circ\pi_i.\]
Using that $[\mathcal{Y}^{[m+1]}]=p_1^*[\mathcal{Y}]\wedge \pi_1^*[\mathcal{Y}^{[m]}]$, we find
\begin{align*}
I:= & \int_{\mathbb{C}^{k(m+1)}\times S}G_{f^{[m+1]}} (dd^cG_{f^{[m+1]}})^{km+k-1}\wedge[\mathcal{Y}^{[m+1]}]\\
\geq & \int_{\mathbb{C}^{k(m+1)}\times S}(G_f\circ p_1)\cdot (p_1^*dd^cG_{f}+\pi_1^*dd^cG_{f^{[m]}})^{km+k-1}\wedge[\mathcal{Y}^{[m+1]}]\\
\geq & \int_{\mathbb{C}^{k(m+1)}\times S}(G_f\circ p_1)\cdot (p_1^*dd^cG_{f})^{k-1}\wedge(\pi_1^*dd^cG_{f^{[m]}})^{km}\wedge[\mathcal{Y}^{[m+1]}]\\
= & \int_{\mathbb{C}^{k(m+1)}\times S}p_1^*\left(G_f(dd^cG_{f})^{k-1}\wedge[\mathcal{Y}]\right)\wedge\pi_1^*\left((dd^cG_{f^{[m]}})^{km}\wedge[\mathcal{Y}^{[m]}]\right)\\
= & \int_{\mathbb{C}^{km}\times S}\left(\int_{\mathbb{C}^{k}}G_{f_t}(dd^cG_{f_t})^{k-1}\wedge[Y_t]\right)(dd^cG_{f^{[m]}})^{km}\wedge[\mathcal{Y}^{[m]}]\\
= & \int_{\mathbb{C}^{km}\times S}(G_{f,\mathcal{Y}}\circ \pi_{[m]})\cdot (dd^cG_{f^{[m]}})^{km}\wedge[\mathcal{Y}^{[m]}].
\end{align*}
\begin{claim}
For $m\geq\dim S$, there is $C(m)\geq1$ such that
\[(\pi_{[m]})_*\left((dd^cG_{f^{[m]}})^{km}\wedge[\mathcal{Y}^{[m]}]\right)=C(m)\mu^{\mathrm{pol}}_{f,\mathcal{Y}}.\]
\end{claim}
According to the Claim above, we find 
\[I\geq \int_S G_{f,\mathcal{Y}}\cdot\mu^{\mathrm{pol}}_{f,\mathcal{Y}},\]
which concludes the proof.
\end{proof}
All there is left to do is to prove the Claim.
\begin{proof}[Proof of the Claim]
We first prove that $dd^cG_{f,\mathcal{Y}}=\pi_*\left((dd^cG_f)^k\wedge[\mathcal{Y}]\right)$ using a slicing argument. Indeed, if $\phi$ is a smooth compactly supported $(\dim S -1,\dim S-1)$-form on $S^0(\mathbb{C})$, we have
\begin{align*}
\int_{\mathcal{X}^0(\mathbb{C})}\pi^*\phi\wedge (dd^c G_f)^k\wedge[\mathcal{Y}]  & =\int_{\mathcal{X}^0(\mathbb{C})} G_f(dd^c G_f)^{k-1}\wedge[\mathcal{Y}]\wedge\pi^*(dd^c\phi)\\
& =\int_{S^0(\mathbb{C})} \left(\int_{\pi^{-1}\{t\}}G_{f_t}(dd^c G_{f_t})^{k-1}\wedge (\iota_t)^*[\mathcal{Y}]\right)\cdot dd^c\phi,
\end{align*}
where $\iota_t:X_t\to \mathcal{X}$ is the natural injection, so that $\iota_t^*[\mathcal{Y}]=[Y_t]$, hence
\begin{align*}
\int_{\mathcal{X}^0(\mathbb{C})}\pi^*\phi\cdot (dd^c G_f)^k\wedge[\mathcal{Y}] 
& =\int_{S^0(\mathbb{C})} G_{f,\mathcal{Y}}\cdot dd^c\phi=\int_{S^0(\mathbb{C})} dd^cG_{f,\mathcal{Y}}\wedge \phi.
\end{align*}
To conclude, we proceed as in the proof of Proposition~\ref{prop-higher-currents}. 
\end{proof}

Now, when $X$ is a projective variety and $L$ is a line bundle on $X$, we denote by $L^{\boxtimes N}$ the induced line bundle on $X^N$, i.e. $L^{\boxtimes N}=\tau_1^*L+\cdots +\tau_N^*L$, where $\tau_i:X^N\to X$ is the canonical projection onto the $i$-th coordinate. We will also use the next Lemma due to Gao, Ge and K\"uhne~\cite[Lemma~4.3]{GGK}.
\begin{lemma}\label{lm:GGK}
Let $X$ be an irreducible projective variety with a very ample line bundle $L$, defined over an algebraically closed field $K$ and $N\geq2$. Let $Z\subsetneq X^N$ be a proper closed subvariety. There exists a constant
\[B=B(N,\dim X,\deg_L(X),\deg_{L^{\boxtimes N}}(Z))>0,\]
such that for any subset $\Sigma\subset X(K)$ with $\Sigma^N\subseteq Z(K)$, there exists a proper closed
subvariety $X'$ of $X$ with $\Sigma\subset X'(K)$ and $\deg_L(X')<B$.
\end{lemma}

We are now in position to prove Theorem~\ref{prop-empty-interior}.
\begin{proof}[Proof of Theorem~\ref{prop-empty-interior}]
For any $v\in M_\mathbb{K}$, recall that the Green function of $f^{[m+1]}$ is
\[G_{f^{[m+1]},v}(x):=\lim_{n\to\infty}\sum_{j=1}^{m+1} \frac{1}{d^n}\log^+\|f^n\circ p_i(x)\|_v, \quad x\in\mathbb{A}^{k(m+1)}(\bar{\mathbb{Q}})\times S^0(\bar{\mathbb{Q}}). \]
One can show that for any $x\in \mathbb{A}^{k(m+1)}(\bar{\mathbb{Q}})\times S^0(\bar{\mathbb{Q}})$, we have 
\[\widehat{h}_{f^{[m+1]}}(x)=\frac{1}{[\mathbb{L}:\mathbb{K}]}\sum_{v\in M_\mathbb{K}}\sum_{\sigma\in\mathrm{Gal}(\mathbb{L}/\mathbb{K})}n_vG_{f^{[m+1]},v}(\sigma(x)),\]
where $\mathbb{L}$ is any finite extension of $\mathbb{K}$ so that $x\in\mathbb{A}^{k(m+1)}(\mathbb{L})\times S^0(\mathbb{L})$. In particular, for a given place $v\in M_\mathbb{K}$, we deduce that
\begin{align}
\frac{n_{v}}{[\mathbb{L}:\mathbb{K}]}\sum_{\sigma\in\mathrm{Gal}(\mathbb{L}/\mathbb{K})}G_{f^{[m+1]},v}(\sigma(x))\leq \widehat{h}_{f^{[m+1]}}(x).\label{ineg-heigh-poly}
\end{align}
We proceed by contradiction, assuming that there is a Zariski dense subset of small points, i.e. for all $\varepsilon>0$, the set
\[E_\varepsilon:=\{x\in \mathcal{Y}^{[m+1]}(\bar{\mathbb{Q}})\, : \ \widehat{h}_{f^{[m+1]}}(x)\leq\varepsilon\}\]
is Zariski dense in $\mathcal{Y}^{[m+1]}(\bar{\mathbb{Q}})$. In particular, there exists a generic sequence $(x_n)\in (\mathcal{Y}^{[m+1]})^0(\bar{\mathbb{Q}})$ such that $\widehat{h}_{f^{[m+1]}}(x_n)\to0$ as $n\to \infty$. 
Let now $v_0\in M_\mathbb{K}$ be an archimedean place. Since we will now work only at the place $v_0$, we forget the subscript $v_0$ in the rest of the proof.

By construction of the Green current $\widehat{T}_{f^{[m+1]}}$ and of the Green function $G:=G_{f^{[m+1]}}$, as measures on $\mathbb{C}^{k(m+1)}\times S^0(\mathbb{C})$, we have
\[\widehat{T}^{k(m+1)-1}_{f^{[m+1]}}\wedge[\mathcal{Y}^{[m+1]}]=(dd^cG_{f^{[m+1]}})^{k(m+1)-1}\wedge[\mathcal{Y}^{[m+1]}].\]
In particular, Lemma~\ref{lm:interior} says that
\[\int_{(\mathcal{Y}^{[m+1]})^0_{\star}(\mathbb{C})}G\cdot\mu_{m+1} >0,\]
where $\mu_{m+1}=\widehat{T}_{f^{[m+1]}}^{\dim\mathcal{Y}^{[m+1]}}$ and $(\mathcal{Y}^{[m+1]})^0_{\star}=(\mathcal{Y}^{[m+1]})^0\cap(\mathbb{A}^{k(m+1)}\times S)$.

As $G$ is continuous and non-negative on $(\mathcal{Y}^{[m+1]})^0_\star(\mathbb{C})$, we deduce that there exists a non-empty open analytic set $\Omega\Subset (\mathcal{Y}^{[m+1]})^0_\star(\mathbb{C})$ such that $G>0$ on $\Omega$ and such that $\mu_m(\Omega)>0$. Let $\chi:(\mathcal{Y}^{[m+1]})^0_\star(\mathbb{C})\to\mathbb{R}_+$ be a smooth compactly supported function with $\chi=1$ on $\Omega$ and $0\leq \chi\leq 1$. The function $\phi:=G\cdot \chi$ is thus continuous, compactly supported, and $G\geq\phi$. 
 We now apply the Equidistribution Theorem~\ref{tm:equidistrib}:
\[\lim_{n\to\infty}\frac{1}{\mathrm{Card}(\mathsf{O}(x_n))}\sum_{y\in\mathsf{O}(x_n)}\phi(y)=\int_{(\mathcal{Y}^{[m+1]})_\star^0(\mathbb{C})} \phi\mu_{m+1}.\]
In particular, there is $n_0\geq1$ such that for any $n\geq n_0$, we have 
\[\frac{1}{\mathrm{Card}(\mathsf{O}(x_n))}\sum_{y\in\mathsf{O}(x_n)}\phi(y)\geq \frac{1}{2}\int_{(\mathcal{Y}^{[m+1]})^0_\star(\mathbb{C})} \phi\mu_{m+1}\geq \frac{1}{2}\int_{\Omega} G\mu_{m+1}>0.\]
Moreover, for any finite extension $\mathbb{L}_n$ of $\mathbb{K}$ with $x_n\in\mathcal{Y}^{[m+1]}(\mathbb{L}_n)$,
\begin{align*}
\frac{1}{[\mathbb{L}_n:\mathbb{K}]}\sum_{\sigma\in\mathrm{Gal}(\mathbb{L}_n/\mathbb{K})}G_{f^{[m+1]}}(\sigma(x_n)) & = \frac{1}{\mathrm{Card}(\mathsf{O}(x_n))}\sum_{y\in\mathsf{O}(x_n)}G(y)\\
& \geq \frac{1}{\mathrm{Card}(\mathsf{O}(x_n))}\sum_{y\in\mathsf{O}(x_n)}\phi(y),
\end{align*}
where we used that $G\geq\phi$. Together with \eqref{ineg-heigh-poly}, this gives
\[\widehat{h}_{f^{[m+1]}}(x_n)\geq \frac{n_{v_0}}{2}\int_{\Omega} G\mu_{m+1}>0,\]
for any $n\geq n_0$. This is a contradiction since $\widehat{h}_{f^{[m+1]}}(x_n)\to0$ as $n\to\infty$ by hypothesis. 

\medskip

We have thus proved there exists $\varepsilon_0>0$ such that the set $E_{\varepsilon_0}$ is not Zariski dense in $\mathcal{Y}^{[m+1]}(\bar{\mathbb{Q}})$. In particular, there is a proper Zariski closed subset $V\subsetneq \mathcal{Y}^{[m+1]}$ which is defined over $\bar{\mathbb{Q}}$ and that contains $E_{\varepsilon_0}$. If $Z:=\pi_{[m+1]}(V)\subsetneq S$ is a proper closed subvariety of $S$, then for any $t\in (S^0\setminus Z)(\bar{\mathbb{Q}})$, we have $\widehat{h}_{f^{[m+1]}}\geq\varepsilon_0$ on $Y^{m+1}_t(\bar{\mathbb{Q}})$. It is in particular true on $\Delta:=\{(z,\ldots,z)\, : \ z\in Y_t(\bar{\mathbb{Q}})\}$. Let $\varepsilon:=\varepsilon_0/(m+1)$. This gives
\[\sum_{j=1}^{m+1}\widehat{h}_{f_t}(z)=\widehat{h}_{f^{[m+1]}}(z,\ldots,z,t)\geq\varepsilon_0,\]
which rewrites as $\widehat{h}_{f_t}\geq \varepsilon=\varepsilon_0/(m+1)$ on $Y_t(\bar{\mathbb{Q}})$.

Assume now that $\pi(V)=S$ and let $Z\subsetneq S$ be the proper closed subvariety  of $S$ such that $\pi_{[m+1]}$ is flat on each irreducible component of $V$ over $S^0\setminus Z$. Pick now $t\in (S^0\setminus Z)(\bar{\mathbb{Q}})$. By definition, the line bundle $L_t:=\mathcal{O}_{\mathbb{P}^k}(1)|_{Y_t}$ is very ample and  the set $V_t:=V\cap Y_t^{m+1}$ is a proper closed subvariety of $Y_t^{m+1}$ with $D:=\deg_{L_t}(V_t)$ independent of $t$.
Let
\[\Sigma_t:=\left\{z\in Y_t(\bar{\mathbb{Q}})\, : \ \widehat{h}_{f_t}(z)\leq \varepsilon\right\}.\]
where $\varepsilon=\varepsilon_0/(m+1)$ as above.
The conclusion follows from Lemma~\ref{lm:GGK}.
\end{proof}

\subsection{Uniformity in the moduli space $\mathscr{P}_d^2$} 
As before, we focus on the good family $(\mathbb{P}^2_{S},f,\mathcal{O}_{\mathbb{P}^2}(1))$ of degree $d$ endomorphisms of $\mathbb{P}^2$ which is defined over $\mathbb{Q}$ introduced in Lemma~\ref{good-family-poly} and let, as before, $\mathcal{V}_d^2$ be its maximal regular part (see \S~\ref{sec:moduli-space}).

\medskip

We also study here the variation of the canonical height of the critical locus. However, when $f:\mathbb{A}^2\to\mathbb{A}^2$ is a degree $d$ regular polynomial endomorphism,  $L_\infty$ is an irreducible component of the critical locus of $f$ and $f$ induces an endomorphism of $L_\infty$, we denote by $f_{L_\infty}$. This induces a map
\[r:\mathcal{V}_d^2\longrightarrow \mathcal{U}_d^{1}\]
defined by $r(t)=f_{t,L_\infty}$. This map is well defined and surjective and, for every $g\in \mathcal{U}_d^{1}$, the set $r^{-1}(g)$ consists of conjugacy classes of regular polynomial endomorphisms whose restriction to $L_\infty$ are conjugate to $g$. It thus is a subvariety of $\mathcal{V}_d^2$ of dimension $\dim\mathscr{P}_d^2-\dim\mathscr{M}_d^1>0$.

\medskip

In the present situation, one sees that $\mathrm{Crit}(f_t)$ decomposes as
\[\mathrm{Crit}(f_t)=L_\infty\cup C_{f_t}\]
where $C_{f_t}\cap \mathbb{A}^2=\{z\in\mathbb{A}^2\, : \ \det(D_zf_t)=0\}=\mathrm{Crit}(f_t)\cap\mathbb{A}^2$. We now let 
\begin{center}
$\mathcal{C}:=\{(z,t)\in \mathbb{P}^2\times \mathcal{V}_d^2\, : \ z\in C_{f_t}\}$.
\end{center}
The next key lemma is a consequence of Theorem~\ref{tm:mu-interior} (see~Theorem~\ref{th-existence}).
\begin{lemma}\label{lm:polynomialmubif}
There exists a non-empty open set $\Omega\subset\mathcal{V}_d^2(\mathbb{C})$ that is contained in $\mathrm{supp}(\mu_{f,\mathcal{C}}^{\mathrm{poly}})$. In particular, 
\[\int_{S(\mathbb{C})}G_{f,\mathcal{Y}}\cdot\mu^{\mathrm{pol}}_{f,\mathcal{Y}}>0.\]
\end{lemma}
\begin{proof}
Write $P:=\dim \mathcal{P}_d^2$ for simplicity. First, as currents on $\mathbb{C}^{2P}\times \mathcal{V}_d^2(\C)$, we have
\[\widehat{T}_{f^{[P]}}=dd^cG_{f^{[P]}}.\]
In particular, as measures on $\mathbb{C}^{2P}\times \mathcal{V}_d^2(\C)$, we also have
\[\widehat{T}_{f^{[P]}}^{2P}\wedge[\mathrm{Crit}^{[P]}]=\widehat{T}_{f^{[P]}}^{2P}\wedge[\mathcal{C}^{[P]}]=\left(dd^cG_{f^{[P]}}\right)^{2P}\wedge[\mathcal{C}^{[P]}].\]
By construction, the points of the support of $\widehat{T}_{f^{[P]}}^{2P}\wedge[\mathrm{Crit}^{[P]}]$ constructed in Theorem~\ref{th-existence} belong to $\mathbb{C}^{2P}\times \mathcal{V}_d^2(\C)$. In particular, they belong to the support of $\left(dd^cG_{f^{[P]}}\right)^{2P}\wedge[\mathcal{C}^{[P]}]$. We conclude by pushing forward this measure by $\pi_{[P]}$ that there exists $\Omega\subset \mathrm{supp}(\mu_{f,\mathcal{C}}^{\mathrm{poly}})$.

We prove that the integral is strictly positive by contradiction. If the integral vanishes, for any $\varepsilon>0$, the set of points $t\in \Omega$ such that $G_{f,\mathcal{Y}}(t)\leq\varepsilon$ is dense in $\Omega$. As $G_{f,\mathcal{Y}}\geq0$, this implies the continuous function $G_{f,\mathcal{Y}}:S^0(\mathbb{C})\to\mathbb{R}$ 
is constant equal to zero on $\Omega$. This is a contradiction since $(dd^c G_{f,\mathcal{Y}})^{P}$ would be zero on $\Omega$.
\end{proof}

We are now in position to prove the following result.

\begin{theorem}\label{tm:gap}
Fix $d\geq2$. There are constants $B(d)\geq1$ and $\varepsilon(d)>0$ and a non-empty Zariski open subset $U\subset \mathrm{Poly}_d^2$ such for any $f\in U(\bar{\mathbb{Q}})$, then
\[\# \{z\in C_{f}(\bar{\mathbb{Q}})\, : \ \widehat{h}_f(z)\leq \varepsilon(d)\}\leq B(d).\]
\end{theorem}

\begin{proof}
Let $(\mathbb{P}^2_S,f,\mathcal{O}_{\mathbb{P}^2}(1))$ be the family introduced in Lemma~\ref{good-family-poly}. Lemma~\ref{lm:polynomialmubif} with Theorem~\ref{prop-empty-interior} imply that there are $\varepsilon>0$, $B\geq1$ and a non-empty Zariski open set $\mathscr{U}\subset\mathcal{V}_d^2$ such that for any $t\in\mathscr{U}(\bar{\mathbb{Q}})$ 
\[ \#  \{z\in C_{f_t}(\bar{\mathbb{Q}})\, : \ \widehat{h}_{f_t}(z)\leq \varepsilon\}\leq B.\]
Now, recall that if two maps $f,g\in \mathrm{Poly}_d^2(\bar{\mathbb{Q}})$ are conjugate by $\phi\in \mathrm{Aut}(\mathbb{A}^2)$, i.e. if $f\circ \phi=\phi\circ g$, then $\widehat{h}_f\circ \phi=\widehat{h}_g$ and $\phi^{-1}(C_f)=C_g$ so that
\[\{z\in C_{f}(\bar{\mathbb{Q}})\, : \ \widehat{h}_{f}(z)\leq \varepsilon\}=\phi\left(\{z\in C_{g}(\bar{\mathbb{Q}})\, : \ \widehat{h}_{g}(z)\leq \varepsilon\}\right).\]
As $\phi$ is an automorphism of $\mathbb{P}^2$, the conclusion follows.
\end{proof}
%
To conclude, it remains to prove Theorem~\ref{tm:uniformity}.

\begin{proof}[Proof of Theorem~\ref{tm:uniformity} ]
Observe that the statement is a direct consequence of Theorem~\ref{tm:gap} in $\bar{\Q}$: 
there exists a constant $B(d)\geq1$, $\varepsilon(d)>0$ and a non-empty Zariski open subset $U\subset \mathrm{Poly}_d^2$ such for any $f\in U(\bar{\mathbb{\Q}})$, we have
\[\# \{z\in C_f(\bar{\mathbb{Q}})\, : \ \widehat{h}_f(z)\leq \varepsilon(d)\}\leq B(d).\]
As preperiodic points of $f\in U(\bar{\mathbb{Q}})$ are those $z\in \mathbb{P}^2(\bar{\mathbb{Q}})$ with $\widehat{h}_f(z)=0$, this implies
	\[\#\mathrm{Preper}(f)\cap C_f\leq B(d),\]
for any $f\in U(\bar{\mathbb{Q}})$.

Now, let $f_0\in  U(\mathbb{\C})$ with $\#\mathrm{Preper}(f_0)\cap C_{f_0}\geq B(d)+1$. Write $\mathrm{Preper}(f_0)\cap C_{f_0}=\{z_1,\ldots,z_N\}$ and let $n_i>m_j\geq0$ be minimal such that $f_0^{n_i}(z_i)=f_0^{m_i}(z_i)$ for $1\leq i\leq N$.
For any $i$, the set
\[X_i:=\{(f,z)\in \mathrm{Poly}_d^2\times \mathbb{A}^2, \ f^{n_i}(z)=f^{m_i}(z) \ \text{and}\ \det(D_zf)=0\}.\]
is a closed subvariety of $\mathrm{Poly}_d^2\times \mathbb{A}^2$ which is defined over $\bar{\mathbb{Q}}$.
For $1\leq j\leq N$, let $p_j:\mathrm{Poly}_d^2\times (\mathbb{A}^2)^N\to\mathrm{Poly}_d^2\times \mathbb{A}^2$ be the map defined by $p_j(f,z_1,\ldots,z_N)=(f,z_j)$ and set 
\[X:=\bigcap_{j=1}^Np_j^{-1}(X_j).\]
Then $X$ is a closed subvariety of $\mathrm{Poly}_d^2\times (\mathbb{A}^2)^N$ which is defined over $\bar{\mathbb{Q}}$. Let 
\[\Delta:=\bigcup_{i\neq j}\{(f,z_1,\ldots,z_N)\in \mathrm{Poly}_d\times (\mathbb{A}^2)^N\, : \ z_i=z_j\}.\]
$\Delta$ is also a closed subvariety of $\mathrm{Poly}_d\times (\mathbb{A}^2)^N$.
Our assumption on $f_0$ guarantees that $(f_0,z_1,\ldots,z_N)\in X(\C)$ so that $(X\setminus \Delta)(\C)\neq\varnothing$. As $X$ is defined over $\bar{\mathbb{Q}}$, this implies $(X\setminus \Delta)(\bar{\mathbb{Q}})\neq\varnothing$. This is  a contradiction.  
\end{proof}

\bibliographystyle{short}
\bibliography{biblio}
\end{document}